\documentclass[11pt,reqno]{amsart}
\usepackage{hyperref}

\DeclareMathOperator{\diag}{diag}

\DeclareMathOperator{\blkdiag}{blkdiag}

\usepackage{amsmath}
\usepackage{mathabx}
\usepackage{subcaption}
\usepackage{graphicx}
\usepackage{mathtools}
\usepackage{epstopdf}
\usepackage{leftindex}
\usepackage{amssymb}

\usepackage{xcolor}
\usepackage{booktabs}
\usepackage{multirow}

\usepackage{algorithm}
\usepackage{algorithmic}

\let\underbrace\LaTeXunderbrace

\DeclareMathOperator{\sech}{sech}

\graphicspath{{figures/}}

\begin{document}
\large

\title[PinT for Exponential Integrators]
{Exp-ParaDiag: Time-Parallel Exponential Integrators for Parabolic PDEs}
 
\author[G. Garai]{Gobinda Garai}
\address{
Institute of Mathematics,
Czech Academy of Sciences, Prague, Czech Republic}
\email{garai@math.cas.cz}

\author[N. Chamakuri]{Nagaiah Chamakuri}
\address{
School of Mathematics,
Indian Institute of Science Education and Research Thiruvananthapuram, India}
\email{nagaiah.chamakuri@iisertvm.ac.in}

\thanks{Corresponding authors: Gobinda Garai and Nagaiah Chamakuri}

\begin{abstract}
This paper introduces Exp-ParaDiag, a novel time-parallel method that combines the strength of exponential integrators into the ParaDiag framework. We develop and analyze Exp-ParaDiag based on first and second order accurate exponential integrators. We establish the convergence of the proposed methods both as preconditioned fixed-point iterations and as preconditioners within the GMRES framework. Furthermore, we extend the Exp-ParaDiag formulation to achieve sixth-order temporal accuracy using exponential integrators. The proposed approach is also generalized to nonlinear problems, for which convergence is rigorously demonstrated. A series of numerical experiments is presented to validate the theoretical results and to illustrate the robustness and efficiency of the developed methods.
\end{abstract}

\maketitle
\numberwithin{equation}{section}
\newtheorem{theorem}{Theorem}[section]
\newtheorem{lemma}[theorem]{Lemma}
\newtheorem{definition}[theorem]{Definition}
\newtheorem{proposition}[theorem]{Proposition}
\newtheorem{remark}[theorem]{Remark}
\allowdisplaybreaks

\section{Introduction} \label{intro}
Parabolic partial differential equations (PDEs) are fundamental in modeling time-dependent phenomena such as heat diffusion, fluid dynamics, and phase transitions, yet their numerical solution often demands significant computational resources due to the inherent sequential nature of time-stepping methods. In recent decades, the pursuit of efficient solvers has spurred the development of time-parallel techniques, which aim to distribute temporal computations between multiple processors, thereby accelerating simulations of large-scale systems. A comprehensive overview of parallel-in-time (PinT) methods can be found in \cite{gander201550, gander2024lunet}. Notable among these PinT methods are the diagonalization-based strategies. A diagonalization-based approach for space-time systems was initially proposed in \cite{maday2008parallelization}, incorporating non-uniform variable time steps. The study presented in \cite{mcdonald2018preconditioning} introduces a block circulant preconditioner for the all-at-once formulation, with the resulting system solvable via diagonalization. The ParaDiag framework, introduced in \cite{gander2019convergence} at the differential level, is based on diagonalization and periodic-like initial conditions to achieve temporal parallelism by diagonalizing a reformulated temporal discretization matrix in an all-at-once system. The ParaDiag approach, employed both as a fixed-point iteration and as a preconditioner within the generalized minimal residual (GMRES) method, has been successfully applied to various PDEs, as demonstrated in \cite{lin2021all, wu2021parallel, liu2020fast, garai2024diagonalization}.

Simultaneously, exponential integrators (EI) have emerged as powerful tools for solving stiff PDEs by exactly integrating the linear part of the system via the matrix exponential, potentially enabling larger time steps without sacrificing stability \cite{boyd2001chebyshev, canuto2012spectral, fornberg1999fast}. Examples of such methods include Exponential Time Differencing (ETD) schemes and integrating factor (IF)-based methods. Despite their advantages, this approach presents two primary challenges: computing the matrix exponential—particularly in high-dimensional settings—and efficiently performing matrix exponential–vector products.
Recent advances in large-scale computing and the development of parallel algorithms for matrix exponentiation \cite{dimitrov2017parallel, moler2003nineteen}, along with the introduction of Krylov subspace techniques for approximating matrix exponentials \cite{hochbruck1997krylov, saad1992analysis}, have made it feasible to apply exponential integrators to large-scale problems. In addition, considerable progress has been made in efficient methods for computing the action of a matrix exponential on a vector \cite{acebron2020highly, al2011computing}.
These developments have enabled the successful application of exponential integrators to a wide range of PDEs, including the Navier–Stokes equations \cite{li2018exponential}, the shallow water equations \cite{gaudreault2016efficient}, and phase-field models \cite{ju2022generalized, garai2024integrating}, among others.

In this work, our model problem has the following form:
\begin{equation}\label{model_problem}
u_t = \mathcal{L}u + \mathcal{N}(u),
\end{equation}
with suitable initial and boundary conditions, which will be specified later for the particular problem under consideration. Here \( \mathcal{L} \) is a linear operator, and \( \mathcal{N}(u) \) represents a nonlinear term. We consider the following forms of the operator $\mathcal{L}$: 
\begin{equation}\label{model_problem_linear}
    \mathcal{L} := a \Delta  - c,  
\end{equation}
where \( a \) and \( c\) are strictly positive numbers.
We assume that \( \mathcal{N} \) satisfies a one-sided Lipschitz condition with constant \( M \), i.e.,
\begin{equation}\label{one-sided_lip}
\langle \mathcal{N}(\mathbf{u}_1) - \mathcal{N}(\mathbf{u}_2), \mathbf{u}_1 - \mathbf{u}_2 \rangle \leq -M \|\mathbf{u}_1 - \mathbf{u}_2\|^2, \quad \forall t \in [0, T], \quad \mathbf{u}_1, \mathbf{u}_2 \in \mathbb{R}^{m},
\end{equation}
where \( \langle \cdot, \cdot \rangle \) denotes the standard Euclidean inner product.

In this study, we propose a novel framework that combines the strengths of exponential integrators with the ParaDiag strategy, leading to an Exponential-ParaDiag (Exp-ParaDiag) method. By integrating exponential time discretization within the ParaDiag structure, we develop a scalable space-time solver capable of efficiently handling both stiff and non-stiff PDEs. We examine the proposed Exp-ParaDiag approach from two complementary perspectives: first, as a standalone fixed-point iteration scheme; and second, as an effective preconditioner within the GMRES iterative solver. Furthermore, the framework is systematically extended to support higher-order time discretization schemes and nonlinear PDEs, demonstrating its flexibility, robustness, and potential for broad applicability in solving complex evolution equations.

The paper is organized as follows. Section \ref{Section2} introduces the Exp-ParaDiag method at the differential level, along with an analysis of its convergence behavior. In Section \ref{Section3}, we examine Exp-ParaDiag both as a preconditioned fixed-point method and as a preconditioner for GMRES, presenting the corresponding convergence results. Section \ref{Section4} focuses on extending Exp-ParaDiag to a second-order discretization scheme, developing both a preconditioned fixed-point method and a GMRES preconditioner, accompanied by convergence analysis. In Section \ref{bdfs}, we extend the Exp-ParaDiag formulation to higher-order EI. Section \ref{Section5} extends the proposed methods to nonlinear problems. Finally, Section \ref{Section6} reports a comprehensive set of numerical experiments designed to assess the performance of the proposed methods, illustrating their accuracy and robustness across a range of test problems.

\section{Formulation and Convergence of Exp-ParaDiag}\label{Section2}
In this section, we present and analyze the ParaDiag method based on exponential integrators in the linear setting. 
At the differential level, the ParaDiag formulation takes the following form:
\begin{equation}\label{cts_wr}
\begin{cases}
u_t^k = \mathcal{L}u^k, & (x, t) \in \Omega \times (0, T], \\
u^k(x, 0) = \alpha u^k(x, T) - \alpha u^{k-1}(x, T) + u_0(x), & x \in \Omega, \\
u^k = 0, & (x, t) \in \partial \Omega \times (0, T],
\end{cases}
\end{equation}
where $\alpha$ is the free parameter and $k$ is the iteration number.
 Applying the centered finite difference method in space on a uniformly discretized interval with mesh size $h$ and $N_x$ interior points, employing a first-order  EI for time with uniform time step $\Delta t$ and $N_t=T/\Delta t$, we get the following fully-discrete iterative system corresponding to \eqref{cts_wr} :
\begin{equation}\label{disc_wr}
\begin{cases}
\mathbf{u}^k_{n}=A \mathbf{u}^k_{n-1}, \text{where}\;  A = \exp(\Delta t \mathcal{L}_h), \text{and}\; n=1, 2, \cdots, N_t,\\
\mathbf{u}^k_0 = \alpha \mathbf{u}^k_{N_t} - \alpha \mathbf{u}^{k-1}_{N_t} + \mathbf{u}_0,
\end{cases}
\end{equation}
where $\mathbf{u}_n$ denotes solution at $n$-th time step. Here $\mathcal{L}_h=a\Delta_h-cI_x \in\mathbb{R}^{N_x \times N_x}$, with $\Delta_h$ denotes discrete Laplacian corresponding to homogeneous Dirichlet boundary conditions and $I_x$ is the identity matrix. The space-time system corresponding to \eqref{disc_wr} is given by  
\begin{equation}\label{st_system}
    \left( I_t \otimes I_x - C_0^{\alpha} \otimes A \right) \mathcal{U}^{k}= \mathbf{b}^{k-1},
\end{equation}
where $\mathcal{U}=\left(\mathbf{u}_1, \mathbf{u}_2, \cdots, \mathbf{u}_{N_t} \right)^{\top}\in\mathbb{R}^{N_xN_t}, \mathbf{b}^{k-1}=\left(-\alpha A\mathbf{u}_{N_t}^{k-1} + A\mathbf{u}_0, \mathbf{0}, \cdots, \mathbf{0} \right)^{\top} \in\mathbb{R}^{N_xN_t}$ and 
\begin{equation}\label{alpha_circ}
C_0^{\alpha} = 
\begin{bmatrix}
 0 & 0 &\cdots & 0  & \alpha\\
 1 & 0& \cdots & 0 & 0\\
 0 & 1 &\cdots & 0 & 0\\
 \vdots & \vdots & \ddots & \vdots & \vdots\\
 0 & 0 &   \cdots & 1 & 0
\end{bmatrix} \in\mathbb{R}^{N_t \times N_t}.
\end{equation}
We note that in \eqref{st_system} the all-at-once temporal matrix 
$C_0^{\alpha}$ differs from that obtained when the backward Euler (BE) scheme is used as the time integrator. Moreover, the temporal coupling 
structure is modified in the present formulation, which leads to a different space--time system compared to the classical 
BE-based formulation \cite{gander2019convergence, mcdonald2018preconditioning, lin2021all, wu2021parallel, liu2020fast}.
The $\alpha$-circulant matrix $C_0^{\alpha}$ is diagonalizable, as shown in \cite{bini2005numerical}, and can be expressed in the form 
$C_0^{\alpha}=\mathcal{P}_{\alpha}\mathcal{D}_0\mathcal{P}^{-1}_{\alpha}$, where $\mathcal{P}_{\alpha}=\Gamma_{\alpha}^{-1}\mathbb{F}^*$, with $\Gamma_\alpha = \operatorname{diag}\left(1, \alpha^{1/N_t}, \dots, \alpha^{(N_t - 1)/N_t} \right)$ and $\mathbb{F} = \frac{1}{\sqrt{N_t}} \left[ \omega_0^{(l_1 - 1)(l_2 - 1)} \right]_{l_1, l_2 = 1}^{N_t},$ where $\omega_0 = e^{\frac{2\pi \mathrm{i}}{N_t}}.$ The diagonal entries of $\mathcal{D}_0$, denoted $\lambda_{0, n}$, are explicitly given by $\lambda_{0,n} = \alpha^{\frac{1}{N_t}} e^{\frac{2(n-1)\pi \mathrm{i}}{N_t}}$, for $n=1, 2, \cdots, N_t$.
Using the diagonalizability of $C_0^{\alpha}$, we factor the space-time system $\left( I_t \otimes I_x - C_0^{\alpha} \otimes A \right)$ as $(\mathcal{P}_{\alpha}\otimes I_x)\left( I_t \otimes I_x - \mathcal{D}_0 \otimes A \right)(\mathcal{P}^{-1}_{\alpha}\otimes I_x)$. 
Then the system \eqref{st_system} can be solved using the following three steps:
\begin{equation*}
\begin{aligned}
\text{Step-(1)} \quad & S_1 = \left( \mathbb{F} \otimes I_x \right) \left( \Gamma_{\alpha} \otimes I_x \right) \mathbf{b}^{k-1}, \\[.05em]
\text{Step-(2)} \quad & S_{2,n} = \left(  I_x - \lambda_{0,n}A \right)^{-1} S_{1,n}, \quad n = 1, 2, \dots, N_t, \\[.05em]
\text{Step-(3)} \quad & \mathcal{U}^{k} = \left( \Gamma_{\alpha}^{-1} \otimes I_x \right) \left( \mathbb{F}^* \otimes I_x \right) S_2.
\end{aligned}
\end{equation*} 
The first and third steps involve only matrix-vector multiplications, which can be executed efficiently using fast Fourier transforms (FFTs). The most computationally intensive part is the second step; however, it is naturally parallelizable across the $N_t$ time steps, making it well-suited for parallel implementation. The next section discusses the convergence behavior of the proposed method.

\subsection{Convergence in a discrete setting}

To observe the error contraction of the iterative procedure, we denote the fully discrete error $\mathbf{e}_n^k=\mathbf{u}_n^k -\mathbf{u}_n$. Then, we have the following error relations 
\begin{equation}\label{err_disc_wr}
\begin{cases}
\mathbf{e}^k_{n}=A \mathbf{e}^k_{n-1}, \text{for}\;  n = 1, 2, \cdots, N_t, \\
\mathbf{e}^k_0 = \alpha \mathbf{e}^k_{N_t} - \alpha \mathbf{e}^{k-1}_{N_t}.
\end{cases}
\end{equation}
Using \eqref{err_disc_wr}$_1$ in \eqref{err_disc_wr}$_2$ we obtain 
\begin{equation}\label{reccurence_relation1}
    \mathbf{e}^k_0 = \frac{-\alpha e^{T\mathcal{L}^{\mathrm{rd}}_h}}{1-\alpha e^{T\mathcal{L}^{\mathrm{rd}}_h}} \mathbf{e}^{k-1}_{0}.
\end{equation}
As $\mathcal{L}_h$ is orthogonally diagonalizable, we have $\mathcal{L}_h=V_hD_hV_h^{-1}$, where $D_h=\diag(\lambda_1, \lambda_2,\cdots, \lambda_{N_x})$. The $\lambda_j$'s for $j=1, 2, \cdots, N_x$ satisfies $\lambda_j\in[\lambda_{\min}, \lambda_{\max}] \subset (-\infty, 0)$, where $\lambda_{\min}=\min\limits_{1\leq j\leq N_x}\lambda_j$ and $\lambda_{\max}=\max\limits_{1\leq j\leq N_x}\lambda_j$. Using diagonalizability of $\mathcal{L}_h$, we have $e^{T\Delta_h}=V_he^{TD_h}V_h^{-1}$. Then from  \eqref{reccurence_relation1} and taking norms, we have 
\begin{equation}\label{reccurence_relation2}
    \parallel V_h\mathbf{e}^k_0\parallel  \leq \max\limits_{z\in(-\infty, 0)} \vert\psi(z)\vert  \parallel V_h\mathbf{e}^{k-1}_{0}\parallel, 
\end{equation}
where $\psi(z)=\frac{\alpha e^{z}}{1-\alpha e^z}$ with $z=\lambda T$ and $\lambda\in\sigma(\mathcal{L}_h)$ ( here $\sigma(\mathcal{L}_h)$ denotes the spectrum of $\mathcal{L}_h$). So, the convergence of Exp-ParaDiag depends on the quantity $\max\limits_{z\in(-\infty, 0)} \vert \psi(z)\vert$. Clearly, $\vert\psi(z)\vert  \leq \frac{\vert\alpha\vert e^{z}}{1-\vert\alpha\vert e^z} \leq \frac{\vert\alpha\vert e^{\lambda_{\max}T}}{1-\vert\alpha\vert e^{\lambda_{\max}T}}$. As $V_h$ is orthogonal, \eqref{reccurence_relation2} yields
\begin{equation}\label{reccurence_relation3}
    \parallel \mathbf{e}^k_0\parallel  \leq \left( \frac{\vert\alpha\vert e^{\lambda_{\max}T}}{1-\vert\alpha\vert e^{\lambda_{\max}T}} \right)  \parallel \mathbf{e}^{k-1}_{0}\parallel.
\end{equation}
The above discussion is summarized in the following theorem.
\begin{theorem}\label{thm_1st}
    The discrete error $\mathbf{e}_n^k$ of the Exp-ParaDiag method satisfies the following error contraction relation
    \begin{equation}
        \max\limits_{1\leq n\leq N_t} \parallel \mathbf{e}_n^k\parallel \leq \left( \frac{\vert\alpha\vert e^{\lambda_{\max}T}}{1-\vert\alpha\vert e^{\lambda_{\max}T}} \right)^k \parallel \mathbf{e}_0^0\parallel. 
    \end{equation}
\end{theorem}
\begin{proof}
    Taking norm in \eqref{err_disc_wr}$_1$, we get $\parallel \mathbf{e}^k_{n}\parallel  \leq \parallel A\parallel  \parallel \mathbf{e}^{k}_{n-1}\parallel < \parallel \mathbf{e}^{k}_{n-1}\parallel < \parallel \mathbf{e}^{k}_{0}\parallel$, as $\parallel A\parallel<1$. Next, using \eqref{reccurence_relation3} in the above inequality, we have the required result. 
\end{proof}
\begin{remark}\label{remark1}
Theorem \ref{thm_1st} clearly holds in higher dimensions for the $\mathcal{L}_h$ on a uniform grid, as the properties of eigenvalues remain the same in higher dimensions.   
\end{remark}

\subsection{Optimization of the free parameter $\alpha$}
In practice, the free parameter $\alpha$ usually takes values close to zero, but there is no clearly defined optimal choice. Here, we present one approach to optimizing $\alpha$ by taking the Fourier transform in space of the error equation corresponding to  \eqref{cts_wr}. Letting $\widehat{e}^k(t) = \widehat{u}(t) - \widehat{u}^k(t)$, we obtain the following error relation in Fourier space: 
\begin{equation}
\widehat{e}^k(0) = 
\left( \frac{-\alpha e^{-(a\omega^2  + c) T}}{1 - \alpha e^{-(a\omega^2  + c) T}} \right)
\widehat{e}^{k-1}(0),
\end{equation}
where $\omega$ is the Fourier variable.
So, to optimize $\alpha$, we aim to solve the following min-max problem:
\begin{equation}\label{minmax}
    \alpha_{\text{opt}}:= \min\limits_{\alpha\in(0,1)} \max\limits_{\omega\in[-\omega_{\max}, \omega_{\max}]} \left| \frac{-\alpha e^{-(a\omega^2 +  c) T}}{1 - \alpha e^{-(a\omega^2 + c) T}}  \right|,
\end{equation}
where $\omega_{\max}=\frac{\pi}{h}$ is the highest Fourier frequency in a numerical grid. We find $\alpha_{\text{opt}}$ by solving the min-max problem \eqref{minmax} numerically using \texttt{fminsearch}. The formula for obtaining $\alpha_{\text{opt}}$ in higher dimensions is straightforward.

\section{Preconditioner for $\mathcal{O}(\Delta t)$ Scheme}\label{Section3}
In this section, we present the procedure outlined in Section~\ref{Section2} as a preconditioned fixed-point Richardson iteration. By applying the first-order accurate exponential integrator described in \eqref{disc_wr} to the linear part of the model problem $u_t=\mathcal{L}u$, subject to homogeneous Dirichlet boundary conditions, we obtain the following all-at-once system:
\begin{equation}\label{all_at_once_1storder}
    \mathcal{Q}\mathbf{u}=\mathbf{f}, \text{where}\; \mathcal{Q}:= I_t \otimes I_x - C_0 \otimes A, \text{with}\; A=\exp(\Delta t \mathcal{L}_h),
\end{equation}
where $\mathbf{u}=\left(\mathbf{u}_1, \mathbf{u}_2, \cdots, \mathbf{u}_{N_t} \right)^{\top}\in\mathbb{R}^{N_xN_t}$, $\mathbf{f}=(A\mathbf{u}_0, 0, \cdots, 0)^{\top}$, and $C_0$ is $C_0^{\alpha}$ with $\alpha=0$ in \eqref{alpha_circ}. 
The preconditioned iterative scheme for solving the all-at-once system \eqref{all_at_once_1storder} takes the following form:
    \begin{equation}\label{pcond-1st-order}
        \mathbf{u}^{k+1}=\mathbf{u}^k +\mathbf{s}^k, \; \mathcal{Q}_{\alpha}\mathbf{s}^k=\mathbf{f}-\mathcal{Q}\mathbf{u}^k, \text{where}\; \mathcal{Q}_{\alpha} := I_t \otimes I_x - C_0^{\alpha} \otimes A. 
    \end{equation}
We now aim to solve the residual system $\mathcal{Q}_{\alpha}\mathbf{s}^k=\mathbf{f}-\mathcal{Q}\mathbf{u}^k$ by PinT fashion, which described below:
\begin{equation}\label{eq:exp_ParaDiag_steps}
\left\{
\begin{aligned}
&\text{Step-(a):} && S_a = \left( \mathbb{F} \otimes I_x \right) [\left( \Gamma_{\alpha} \otimes I_x \right)\left(\mathbf{f}-\mathcal{Q}\mathbf{u}^k \right)], \\
&\text{Step-(b):} && S_{b,n} = \left(  I_x - \lambda_{0,n}A \right)^{-1} S_{a,n}, \quad n = 1, 2, \dots, N_t, \\[.05em]
&\text{Step-(c):} && \mathbf{s}^k = \left( \Gamma_{\alpha}^{-1} \otimes I_x \right) \left( \mathbb{F}^* \otimes I_x \right)S_{b,n},
\end{aligned}
\right.
\end{equation}
where the definitions of the involved matrices were described earlier. Next, we analyze the error estimate of the fixed-point iteration procedure.


\subsection{Eigenvalue estimate of the preconditioned system}\label{subsec_1storder}
Let us denote the error at $k$-th iteration as $\mathbf{e}^k:=\mathbf{u}-\mathbf{u}^k$. Then we have the following error relation.
\begin{equation}\label{err_1st_order}
    \mathbf{e}^k = (I_{xt} - \mathcal{Q}_{\alpha}^{-1} \mathcal{Q}) \mathbf{e}^{k-1},
\end{equation}
where $I_{xt}\in\mathbb{R}^{N_xN_t \times N_xN_t}$ is the identity matrix.
Denoting $\mathbf{d}^k:=\mathbf{u}^{k+1}-\mathbf{u}^k$ as the difference between consecutive iterates, then we have $\mathbf{d}^k = (I_{xt} - \mathcal{Q}_{\alpha}^{-1} \mathcal{Q}) \mathbf{d}^{k-1}$. Letting $\mathbf{r}^{k}:=\mathbf{f}-\mathcal{Q}\mathbf{u}^k$ as the residual, we obtain $\mathbf{r}^k = (I_{xt} - \mathcal{Q}\mathcal{Q}_{\alpha}^{-1} ) \mathbf{r}^{k-1}$. But $I_{xt} - \mathcal{Q}\mathcal{Q}_{\alpha}^{-1}$ is similar to $I_{xt} - \mathcal{Q}_{\alpha}^{-1} \mathcal{Q}$, so must have same spectrum. Thus it is enough to study the spectral distribution of $I_{xt} - \mathcal{Q}_{\alpha}^{-1} \mathcal{Q}$. Now, using the diagonalizability of $A$, we have 
\[
\begin{array}{cc}
\mathcal{Q}= \left( I_t \otimes V_h \right) \underbrace{\left( I_t \otimes I_x - C_0 \otimes e^{\Delta t D_h} \right)}_{\widebar{\mathcal{Q}}} \left( I_t \otimes V_h^{-1} \right),\\
\mathcal{Q}_{\alpha}= \left( I_t \otimes V_h \right) \underbrace{\left( I_t \otimes I_x - C_0^{\alpha} \otimes e^{\Delta t D_h} \right)}_{\widebar{\mathcal{Q}}_{\alpha}} \left( I_t \otimes V_h^{-1} \right).
\end{array}
\]
Thus $\mathcal{Q}$ is similar to $\widebar{\mathcal{Q}}$, and $\mathcal{Q}_{\alpha}$ is similar to $\widebar{\mathcal{Q}}_{\alpha}$. Hence, it is enough to study the spectrum of $I_{xt} - \widebar{\mathcal{Q}}_{\alpha}^{-1} \widebar{\mathcal{Q}}$. 
\begin{theorem}\label{pcond_estimate_first_order}
   The spectrum of $I_{xt} - \widebar{\mathcal{Q}}_{\alpha}^{-1} \widebar{\mathcal{Q}}$ is bounded above by $\frac{|\alpha| e^{\lambda_{\max}T} }{1-|\alpha| e^{\lambda_{\max}T} }$.
\end{theorem}
\begin{proof}
Observe that the term $I_{xt} - \widebar{\mathcal{Q}}_{\alpha}^{-1} \widebar{\mathcal{Q}}$ can be written as $I_{xt} - \left( I_{xt} + \widebar{\mathcal{Q}}^{-1} (\widebar{\mathcal{Q}}_{\alpha} -\widebar{\mathcal{Q}})\right)^{-1}$. Then the inverse of block Toeplitz matrix $\widebar{\mathcal{Q}}$ and the structure of low-rank matrix $\mathcal{H}:=\widebar{\mathcal{Q}}_{\alpha} -\widebar{\mathcal{Q}}$ are given below:
\[
\begin{array}{cc}
   \widebar{\mathcal{Q}}^{-1} = 
\begin{bmatrix}
 I_x &  & &   &  \\
 E & I_x&  &  & \\
 E^2 & E& I_x &  & \\
 \vdots & \vdots & \ddots & \ddots &\\
 E^{N_t-1} &  E^{N_t-2} &   \cdots &E & I_x
\end{bmatrix},

\mathcal{H} = 
\begin{bmatrix}
 \mathbf{0} & \mathbf{0} &\cdots & \mathbf{0}  & -\alpha E\\
  \mathbf{0} &  \mathbf{0}& \cdots &  \mathbf{0} &  \mathbf{0}\\
  \mathbf{0} &  \mathbf{0} &\cdots &  \mathbf{0} &  \mathbf{0}\\
 \vdots & \vdots & \ddots & \vdots & \vdots\\
  \mathbf{0} &  \mathbf{0} &   \cdots &  \mathbf{0} &  \mathbf{0}
\end{bmatrix},

\end{array}
\]
where $E=e^{\Delta t D_h}$. A straightforward calculation yields $I_{xt} + \widebar{\mathcal{Q}}^{-1} \mathcal{H}$ and $\left( I_{xt} + \widebar{\mathcal{Q}}^{-1} \mathcal{H}\right)^{-1}$ as follows
\begin{equation}
\begin{array}{cc}
  I_{xt} + \widebar{\mathcal{Q}}^{-1} \mathcal{H} = 
\begin{bmatrix}
 I_x &  & &   & -\alpha E \\
  & I_x&  &  & -\alpha E^2\\
  & & I_x &  & -\alpha E^3 \\
  &  &  & \ddots &\vdots\\
  &  &    & & I_x-\alpha E^{N_t}
\end{bmatrix},

\left( I_{xt} + \widebar{\mathcal{Q}}^{-1} \mathcal{H}\right)^{-1} = 
\begin{bmatrix}
 I_x &  & &   & \frac{\alpha E}{I_x-\alpha E^{N_t}} \\
  & I_x&  &  & \frac{\alpha E^2}{I_x-\alpha E^{N_t}}\\
  & & I_x &  & \frac{\alpha E^3}{I_x-\alpha E^{N_t}} \\
  &  &  & \ddots &\vdots\\
  &  &    & & \frac{1}{I_x-\alpha E^{N_t}}
\end{bmatrix}.
\end{array}
\end{equation}
So the iteration matrix $I_{xt} - \widebar{\mathcal{Q}}_{\alpha}^{-1} \widebar{\mathcal{Q}}$ has the following form
\begin{equation}\label{pcond_system}
I_{xt} - \widebar{\mathcal{Q}}_{\alpha}^{-1} \widebar{\mathcal{Q}} = 
\begin{bmatrix}
 \mathbf{0} &  & &   & -\frac{\alpha E}{I_x-\alpha E^{N_t}} \\
  & \mathbf{0}&  &  & -\frac{\alpha E^2}{I_x-\alpha E^{N_t}}\\
  & & \mathbf{0} &  & -\frac{\alpha E^3}{I_x-\alpha E^{N_t}} \\
  &  &  & \ddots &\vdots\\
  &  &    & & I_x-\frac{1}{I_x-\alpha E^{N_t}}
\end{bmatrix}.
\end{equation}
Hence the non-zero spectrum of $I_{xt} - \widebar{\mathcal{Q}}_{\alpha}^{-1} \widebar{\mathcal{Q}}$ is $\sigma \left( 
I_x-\frac{1}{I_x-\alpha E^{N_t}} \right)$. Thus the spectral radius of $I_{xt} - \widebar{\mathcal{Q}}_{\alpha}^{-1} \widebar{\mathcal{Q}}$ is $\max\limits_{\lambda\in\sigma(\mathcal{L}_h)} \left| 1- \frac{1}{1-\alpha e^{\lambda T}} \right| \leq \frac{|\alpha| e^{\lambda_{\max}T} }{1-|\alpha| e^{\lambda_{\max}T} }$. Hence the theorem.
\end{proof}
\begin{remark}
     Theorem \ref{pcond_estimate_first_order} extends naturally to higher spatial dimensions, and its validity remains unaffected by the increase in dimensionality.
\end{remark}

\subsection{Convergence Analysis in GMRES Setting}
In this section, we examine the convergence behavior of the GMRES method when applied to the linear system \eqref{all_at_once_1storder}, with the preconditioner $\mathcal{Q}_{\alpha}$.
Our primary objective is to understand how the choice of this preconditioner influences the performance and convergence properties of GMRES in the all-at-once solution framework.
To this end, we begin by focusing on the case where the preconditioner 
has the structure of an identity-plus-low-rank matrix. This particular structure facilitates the convergence of the preconditioned system.
\begin{theorem}\label{thm_gmres1}
  The GMRES method applied to the preconditioned system $\mathcal{Q}_{\alpha}^{-1}\mathcal{Q}\mathbf{u}=\mathcal{Q}_{\alpha}^{-1}\mathbf{f}$ converges in at most \( N_x + 1 \) iterations.
\end{theorem}
\begin{proof}
Following the steps in Theorem \ref{pcond_estimate_first_order} it is easy to observe that the structure of $I_{xt} - \mathcal{Q}_{\alpha}^{-1} \mathcal{Q}$ has the exact same form as in \eqref{pcond_system} with the new definition of $E$ as $E=\exp(\Delta t \mathcal{L}_h)$. Clearly, the rank of $I_{xt} - \mathcal{Q}_{\alpha}^{-1} \mathcal{Q}$ is $N_x$ (as the blocks in the last column are invertible). So, the preconditioned system is identity plus a low-rank matrix of rank $N_x$. Thus, GMRES will converge in at most $N_x+1$ iterations. 
\end{proof}

Even though Theorem \ref{thm_gmres1} guarantees the convergence of the preconditioned GMRES method, it does not provide a sharp bound, as the number of iterations increases with \( N_x \). To address this, we consider an analysis that takes into account the clustering of eigenvalues and the diagonalizability of the preconditioned system.
Specifically, we observe that the eigenvalues of the preconditioned matrix \( \mathcal{Q}_{\alpha}^{-1} \mathcal{Q} \) are clustered around 1. This clustering is established in the following lemma. We denote the preconditioned residual as  $\mathbf{r}_{g,\alpha}^k = \mathcal{Q}_{\alpha}^{-1} \mathbf{f} - \mathcal{Q}_{\alpha}^{-1} \mathcal{Q} \mathbf{u}^k$ after $k$-th iterations.
\begin{lemma}\label{gmres_centered_at1}
   The eigenvalues of the preconditioned system $\mathcal{Q}_{\alpha}^{-1} \mathcal{Q}$ are concentrated around $1$ with radius $\frac{\alpha}{1-\alpha}$. 
\end{lemma}
\begin{proof}
    Observe that $\sigma\left( \mathcal{Q}_{\alpha}^{-1} \mathcal{Q}\right)=\{ 1\} \cup\sigma \left( \frac{1}{I_x-\alpha E^{N_t}} \right)$ with $E=\exp(\Delta t \mathcal{L}_h)$. Then we have $\max\limits_{\lambda\in \sigma\left( \mathcal{Q}_{\alpha}^{-1} \mathcal{Q}\right)} |\lambda-1| = \max\limits_{\lambda\in \{1\}\cup \sigma\left( \frac{1}{I_x-\alpha E^{N_t}} \right)} |\lambda-1| \leq \max\limits_{\lambda\in \sigma\left( \frac{1}{I_x-\alpha E^{N_t}} \right)} |\lambda-1| \leq \max\limits_{\lambda\in\sigma(\mathcal{L}_h)} \left| \frac{\alpha e^{\lambda T}}{1-\alpha e^{\lambda T}}\right| \leq \frac{\alpha}{1-\alpha}$. Hence the Lemma.
\end{proof}
\begin{lemma}\label{diag_of_pcond}
    The preconditioned system $\mathcal{Q}_{\alpha}^{-1} \mathcal{Q}$ is diagonalizable.
\end{lemma}
\begin{proof}
    Using the diagonalizability of $A$, the system $\mathcal{Q}_{\alpha}^{-1} \mathcal{Q}$ can be written as $\mathcal{Q}_{\alpha}^{-1} \mathcal{Q}=\left(I_t\otimes V_h \right)\widebar{\mathcal{Q}}_{\alpha}^{-1} \widebar{\mathcal{Q}} \left(I_t\otimes V_h^{-1} \right)$ with $\widebar{\mathcal{Q}}=\left( I_t \otimes I_x - C_0 \otimes e^{\Delta t D_h} \right)$ and $\widebar{\mathcal{Q}}_{\alpha}=\left( I_t \otimes I_x - C_0^{\alpha} \otimes e^{\Delta t D_h} \right)$. Now its easy to observe that the system $\widebar{\mathcal{Q}}_{\alpha}^{-1} \widebar{\mathcal{Q}}$ can be diagonalized as $\widebar{\mathcal{Q}}_{\alpha}^{-1} \widebar{\mathcal{Q}}=\widebar{V}\widebar{D}\widebar{V}^{-1}$, where 
    \[
    \begin{array}{cc}
    \widebar{V}=
\begin{bmatrix}
 I_x &  &  &   & &\frac{1}{E^{N_t-1}} \\
  & I_x&  &  & &\frac{1}{ E^{N_t-2}}\\
  &  & & \ddots& &\vdots\\
  & &  &  &I_x &\frac{1}{ E} \\
  &  &  &  &  &I_x
\end{bmatrix},
\widebar{D}=
\begin{bmatrix}
 I_x &  &  &   & & \\
  & I_x&  &  & &\\
  &  & & \ddots& &\\
  & &  &  &I_x & \\
  &  &  &  &  &\frac{1}{I_x-\alpha E^{N_t}}
\end{bmatrix}.
\end{array}
\]
Here $E=e^{\Delta t D_h}$, with $D_h$ being the diagonal matrix corresponding to $\mathcal{L}_h$. Therefore $\mathcal{Q}_{\alpha}^{-1} \mathcal{Q}$ is diagonalizable with $\mathcal{Q}_{\alpha}^{-1} \mathcal{Q}=\widehat{V}\widebar{D}\widehat{V}^{-1}$, where $\widehat{V}=\left(I_t\otimes V_h \right)\widebar{V}$.
\end{proof}

\begin{theorem}\label{thm_gmres22}
 Let the GMRES method be applied to the preconditioned system $\mathcal{Q}_{\alpha}^{-1}\mathcal{Q}\mathbf{u}=\mathcal{Q}_{\alpha}^{-1}\mathbf{f}$. Then the residuals satisfy the error estimate
 $
 \parallel \mathbf{r}_{g,\alpha}^k\parallel \leq \kappa(\widehat{V}) \left(\frac{\alpha}{1-\alpha} \right)^k \parallel \mathbf{r}_{g,\alpha}^0\parallel
 $
 where $\kappa(\widehat{V})=\parallel \widehat{V}\parallel \parallel \widehat{V}^{-1}\parallel$.
\end{theorem}
\begin{proof}
    By Lemma \ref{diag_of_pcond}, we have the diagonalizability of the preconditioned system. Then, from the classical theory, we get
    $\parallel \mathbf{r}_{g,\alpha}^k\parallel \leq \kappa(\widehat{V}) \min\limits_{p\in P_k(0)}\max\limits_{1\leq i\leq N_x} |p(\lambda_i)| \parallel \mathbf{r}_{g,\alpha}^0\parallel$ for the preconditioned GMRES residual, where $P_k$ is the set of polynomials of degree $k$ or less. Then, by using \cite[Section 6.11]{saad2003iterative} and Lemma \ref{gmres_centered_at1}, we obtain the required result.
\end{proof}
\begin{remark}\label{gmres_ade}
    We briefly examine the convergence behavior of preconditioned GMRES when the operator $\mathcal{L}$ includes an advection term, that is, $\mathcal{L} = a\Delta u - \mathrm{b}\nabla u - cu$. Using periodic boundary conditions and central difference discretizations for the involved operators, we obtain the discrete $\mathcal{L}_h$, which is non-symmetric but diagonalizable. It is now evident from the proof of Theorem \ref{thm_gmres1} that the argument remains valid even when the operator $\mathcal{L}_h$ includes the advection term, and thus Theorem \ref{thm_gmres1} still holds. Similarly, Lemmas \ref{gmres_centered_at1} and \ref{diag_of_pcond} can be proved in this case as well, and therefore Theorem \ref{thm_gmres22} remains valid.
\end{remark}
Although Theorem \ref{thm_gmres22} provides convergence estimates, accurately estimating the condition number \( \kappa \) remains challenging in practice.
This leads to the consideration of an alternative framework to establish convergence bounds for GMRES. In particular, we leverage the results of Elman \cite{elman1982iterative}.
Following this approach, we demonstrate that the convergence rate of the preconditioned GMRES method is independent of the mesh parameters, given that the parameter \( \alpha \) is chosen sufficiently small. This mesh-independent convergence ensures the efficiency of the iterative solver even as the discretization is refined. Before proving the main results, we present a few auxiliary lemmas.

\begin{lemma}\label{positivity}
    The function $g_{\alpha}(y)=4(1-\alpha y^{N_t})(1-y^2)-\alpha^2(y^2-y^{2N_t})$  is strictly positive for $\alpha\in(0, \frac{2}{1+\sqrt{2}})$ with $y\in(0, 1)$.
\end{lemma}
\begin{proof}
    Treating the function $g_{\alpha}$ as a quadratic in $\alpha$, we have $g_{\alpha}=-\xi_1\alpha^2-4\xi_2\alpha+4\xi_3$ where $\xi_1=(y^2-y^{2N_t}), \xi_2=y^{N_t}(1-y^2)$ and $\xi_3=(1-y^2)$. This is a concave-down quadratic in \( \alpha \). At \( \alpha = 0 \), $g_{\alpha}(0) = 4(1 - y^2) > 0,$ so for each fixed \( y \), there are two real roots $\alpha_1(y) < 0 < \alpha_2(y),$ and \( g_{\alpha}(y) > 0 \) if and only if $\alpha \in \left( \alpha_1(y), \alpha_2(y) \right).$ Thus for positivity of $g_{\alpha}$ for all $y\in(0, 1)$, we therefore need $\alpha < \inf\limits_{y \in (0,1)} \alpha_2(y)$. A direct computation yields $\alpha_2(y)=\frac{2\xi_3}{\sqrt{\xi_2^2 + \xi_1\xi_3} + \xi_2} = \frac{2}{y^{N_t} + y\sqrt{\sum\limits_{j=0}^{N_t-1}y^{2j}}}.$ It is easy to observe that $\alpha_2(y)$ is monotonically decreasing and $\inf\limits_{y \in (0,1)} \alpha_2(y)=\frac{2}{1+\sqrt{N_t}}$. As $N_t\geq 2$, we have the required result.

\end{proof}

Let us define symmetric part of any square matrix $Z$ as $\mathrm{H}(Z)=\frac{Z+ Z^*}{2}$. 

\begin{lemma}\label{pos_definite}
 $\mathrm{H}(\mathcal{Q}_{\alpha}^{-1} \mathcal{Q})$ is positive definite for $\alpha\in(0, \frac{2}{1+\sqrt{2}})$.
\end{lemma}
\begin{proof}
Observe that $\mathrm{H}(\mathcal{Q}_{\alpha}^{-1} \mathcal{Q})=\left( I_t\otimes V_h\right) \mathrm{H}(\widebar{\mathcal{Q}}_{\alpha}^{-1} \widebar{\mathcal{Q}}) \left( I_t\otimes V_h^{-1}\right)$. Thus $\mathrm{H}(\mathcal{Q}_{\alpha}^{-1} \mathcal{Q})$ is similar to $\mathrm{H}(\widebar{\mathcal{Q}}_{\alpha}^{-1} \widebar{\mathcal{Q}})$, hence spectrum must be the same. $\mathrm{H}(\widebar{\mathcal{Q}}_{\alpha}^{-1} \widebar{\mathcal{Q}})$ has the following form 
\begin{equation}
    \mathrm{H}(\widebar{\mathcal{Q}}_{\alpha}^{-1} \widebar{\mathcal{Q}}):=
    \begin{bmatrix}
        \mathfrak{A} & \mathfrak{B}\\ \mathfrak{B}^{\top} & \mathfrak{D}
    \end{bmatrix},
\end{equation}
\begin{equation}
\text{where}\, \mathfrak{A}=
\begin{bmatrix}
 I_x &  & &   \\
  & I_x&  & \\
  &  &  \ddots\\
  &  &    & I_x
\end{bmatrix} \in \mathbb{R}^{J\times J},
\mathfrak{B}=
\begin{bmatrix}
  \frac{\alpha E}{2(I_x-\alpha E^{N_t})} \\
   \frac{\alpha E^2}{2(I_x-\alpha E^{N_t})}\\
 \vdots\\
\frac{\alpha E^{N_t-1}}{2(I_x-\alpha E^{N_t})}
\end{bmatrix} \in \mathbb{R}^{J\times N_x},
\text{and}\; \mathfrak{D}=\frac{1}{I_x-\alpha E^{N_t}},
\end{equation}
where $J=(N_t-1)N_x$. Our goal here is to show $\mathrm{H}(\widebar{\mathcal{Q}}_{\alpha}^{-1} \widebar{\mathcal{Q}})$ positive definite. For that it is enough to show $\mathfrak{A}\succ 0$ and the Schur complement $\mathrm{H}(\widebar{\mathcal{Q}}_{\alpha}^{-1} \widebar{\mathcal{Q}})/\mathfrak{A}:=\mathfrak{D}-\mathfrak{B}^{\top}\mathfrak{A}^{-1}\mathfrak{B}\succ 0$. Clearly, $\mathfrak{A}\succ 0$. Next we show $\mathrm{H}(\widebar{\mathcal{Q}}_{\alpha}^{-1} \widebar{\mathcal{Q}})/\mathfrak{A}\succ 0$. Now $\mathfrak{D}-\mathfrak{B}^{\top}\mathfrak{A}^{-1}\mathfrak{B} = \frac{1}{I_x-\alpha E^{N_t}} - \frac{\alpha^2}{4}\frac{1}{(I_x-\alpha E^{N_t})^2}\left( E^2 + E^4+\cdots+\left( E^{N_t-1}\right) ^2\right)$. As right hand side of $\mathfrak{D}-\mathfrak{B}^{\top}\mathfrak{A}^{-1}\mathfrak{B}$ is diagonal, the eigen-structure has the form $\zeta(\lambda)=\frac{1}{1-\alpha e^{\lambda T}}-\frac{\alpha^2}{4}\left( \frac{1-e^{2\lambda T}}{1-e^{2\lambda \Delta t}}-1 \right)\frac{1}{(1-\alpha e^{\lambda T})^2}$. Now $\zeta(\lambda)$ can be written as $\zeta(\lambda)=\frac{g_{\alpha}(\lambda)}{4(1-\alpha e^{\lambda T})^2(1-e^{2\lambda \Delta t})}$. So it is enough to show $g_{\alpha}(\lambda)>0$, which is true using the Lemma \ref{positivity} with the change of variable $y=e^{\lambda \Delta t}$. 
\end{proof}

\begin{lemma}\label{E_PD}
    The matrices $E^m$ for $m=1,2,\cdots$ are positive definite, where $E=\exp(\Delta t \mathcal{L}_h)$. 
\end{lemma}
\noindent The proof of Lemma \ref{E_PD} is straightforward. 
\begin{lemma}\label{alpha_pm}
For $y\in (0, 1)$, the functions $\alpha_{\pm}(y):=\frac{1}{2}\left( \frac{\alpha y}{1-\alpha y} \pm \sqrt{\tau(y)}\right)$, where $\tau(y)= \frac{\alpha^2}{(1-\alpha y)^2} \left( \sum\limits_{j=1}^{N_t-1} y^{2j} + y^2\right)$ satisfies $|\alpha_{\pm}(y)|<1$ for $\alpha\in(0, \frac{2}{3+\sqrt{2}})$.    
\end{lemma}
\begin{proof}
    Simplifying $\alpha_{\pm}(y)$, we obtain $\alpha_{\pm}(y)= \frac{\alpha }{2(1-\alpha y)}\left( y\pm \sqrt{\tau(y)} \right)$. Clearly, $\sqrt{\tau(y)} \geq y$ for all $y \in (0, 1)$. Therefore, we have $y + \sqrt{\tau(y)} \geq 0$ and $y - \sqrt{\tau(y)} \leq 0$. Also, it is easy to observe that $y + \sqrt{\tau(y)}$ is increasing in $y$, while $y - \sqrt{\tau(y)}$ is decreasing in $y$. Since $\frac{\alpha}{1 - \alpha y}$ is strictly increasing in $y$ and positive for $\alpha > 0$, it follows that $\alpha_+(y)$ is monotonically increasing and $\alpha_-(y)$ is monotonically decreasing. Thus $\sup\limits_{y\in(0, 1)}\alpha_+(y)=\frac{\alpha}{2(1-\alpha)}(1+\sqrt {N_t})$ and $\inf\limits_{y\in(0, 1)}\alpha_-(y)=\frac{\alpha}{2(1-\alpha)}(1-\sqrt {N_t})$. Then $|\alpha_{+}(y)|<1\iff 0<\alpha<\frac{2}{3+\sqrt{N_t}}$, and $|\alpha_{-}(y)|<1\iff \alpha<\frac{2}{1+\sqrt{N_t}}$. Thus $|\alpha_{\pm}(y)|<1 \iff 0<\alpha<\frac{2}{3+\sqrt{N_t}}$. As $N_t\geq 2$ we have the required result.   
\end{proof}

\begin{lemma}\label{norm_diference_estimate}
    For $\alpha\in(0, \frac{2}{3+\sqrt{2}})$, one has $\parallel \mathrm{H}(\mathcal{Q}_{\alpha}^{-1} \mathcal{Q}-I_{xt}) \parallel<1$.
\end{lemma}
\begin{proof}
As $\mathrm{H}(\mathcal{Q}_{\alpha}^{-1} \mathcal{Q}-I_{xt})$ is symmetric, spectrum estimation is sufficient. Now $\mathrm{H}(\mathcal{Q}_{\alpha}^{-1} \mathcal{Q}-I_{xt})$ can be written as $\mathrm{H}(\mathcal{Q}_{\alpha}^{-1} \mathcal{Q}-I_{xt})=\mathbb{P}\mathbb{D}\mathbb{P}^{-1}$, where $\mathbb{D}=\blkdiag(0, 0,\cdots,0, G_-, G_+)$. The blocks $G_{\pm}$ are given by  $G_{\pm}=\frac{1}{2}\left(\Gamma_{N_t}-I_{xt} \pm\sqrt{\sum\limits_{j=1}^{N_t-1} \Gamma_j^2 + (I_{xt}-\Gamma_{N_t})^2}\right)$, with the definitions $\Gamma_j=\frac{\alpha E^j}{I_{xt}-\alpha E^{N_t}}$ for $j=1, 2, \cdots, N_t-1$, and $\Gamma_{N_t}=\frac{1}{I_{xt}-\alpha E^{N_t}}$. The matrix inside the square root is positive definite. This follows from the Neumann series expansion of the \(\Gamma_j\) terms, which is valid since \(\|\alpha E^{N_t}\| < 1\) by Lemma~\ref{E_PD}. Furthermore, the positive definiteness of the powers of \(E\), also established in Lemma~\ref{E_PD}, ensures that each \(\Gamma_j\) contributes positively. Therefore, the square root is well-defined. 
Using the diagonalizability of $E$, it is easy to observe that $G_{\pm}$ is diagonalizable. So, the eigenvalue of $G_{\pm}$ has the form $\alpha_{\pm}(y):=\frac{1}{2}\left( \frac{1}{1-\alpha y} -1\pm \sqrt{\tau(y)}\right)$ with $y=e^{\lambda \Delta t}$, where $\tau(y)=\frac{\alpha^2}{(1-\alpha y)^2}\sum\limits_{j=1}^{N_t-1}y^{2j} + \left(\frac{\alpha y}{1-\alpha y} \right)^2$. Then, by applying Lemma~\ref{alpha_pm}, the desired result follows.
\end{proof}

\begin{lemma}[See \cite{elman1982iterative}, thm. 5.4, \cite{eisenstat1983variational}]\label{lemma_gmres_elkman}
Let $\Lambda$ be a real square matrix with $\mu(\Lambda) > 0$, where $\mu(\Lambda)$ denotes the leftmost eigenvalue of $\mathrm{H}(\Lambda)$. Then, when solving $\Lambda \mathbf{y} = \mathbf{c}$ using GMRES, the residual $\mathbf{r}_g^k = \mathbf{c} - \Lambda \mathbf{y}^k$ at $k$-th iteration satisfies
 $
 \parallel \mathbf{r}_{g}^{k}\parallel \leq \left(1 -\frac{\mu(\Lambda)^2}{\parallel  \Lambda\parallel^2} \right)^{k/2} \parallel \mathbf{r}_{g}^{0}\parallel.
 $
\end{lemma}

\begin{theorem}\label{gmres3_thm}
Let GMRES applied to $\mathcal{Q}_{\alpha}^{-1} \mathcal{Q} \mathbf{u} = \mathcal{Q}_{\alpha}^{-1} \mathbf{f}$. Then the residual $\mathbf{r}_{g,\alpha}^k = \mathcal{Q}_{\alpha}^{-1} \mathbf{f} - \mathcal{Q}_{\alpha}^{-1} \mathcal{Q} \mathbf{u}^k$ after $k$ iterations satisfies 
    $\parallel \mathbf{r}_{g, \alpha}^{k}\parallel \leq \left(4\alpha(1-\alpha) \right)^{k/2} \parallel \mathbf{r}_{g, \alpha}^{0}\parallel$ for $\alpha\in(0, \frac{2}{3+\sqrt{2}})$.
\end{theorem}
\begin{proof}
We apply Lemma \ref{lemma_gmres_elkman} for our preconditioned system. For that $\mu(\mathrm{H}(\mathcal{Q}_{\alpha}^{-1} \mathcal{Q}))$ must be positive, which is guaranteed by Lemma \ref{pos_definite} for $\alpha\in(0,\frac{2}{1+\sqrt{2}})$. Now one has to estimate $\mu(\mathrm{H}(\mathcal{Q}_{\alpha}^{-1} \mathcal{Q}))^2$ and $\parallel  \mathcal{Q}_{\alpha}^{-1} \mathcal{Q}\parallel^2$. The term $\mathrm{H}(\mathcal{Q}_{\alpha}^{-1} \mathcal{Q})$ can be rewritten as $I_{xt} + \mathrm{H}(\mathcal{Q}_{\alpha}^{-1} \mathcal{Q}-I_{xt})$. Using Lemma \ref{norm_diference_estimate} we have $\parallel \mathrm{H}(\mathcal{Q}_{\alpha}^{-1} \mathcal{Q}) \parallel \geq \parallel I_{xt} \parallel -\parallel \mathrm{H}(\mathcal{Q}_{\alpha}^{-1} \mathcal{Q}-I_{xt}) \parallel \geq \parallel I_{xt} \parallel -\parallel \mathcal{Q}_{\alpha}^{-1} \mathcal{Q}-I_{xt} \parallel \geq 1-\frac{\alpha}{1-\alpha}=\frac{1-2\alpha}{1-\alpha}$. So, $\mu(\mathrm{H}(\mathcal{Q}_{\alpha}^{-1} \mathcal{Q}))^2 \geq \left(\frac{1-2\alpha}{1-\alpha} \right)^2$ for $\alpha\in(0,\frac{2}{3+\sqrt{2}})$.

\noindent We have $I_{xt}-\mathcal{Q}_{\alpha}^{-1} \mathcal{Q}=L_r$, where $L_r$ has the same structure as in \eqref{pcond_system} with $E=\exp(\Delta t \Delta_h)$. Then $\parallel \mathcal{Q}_{\alpha}^{-1} \mathcal{Q}\parallel = \sqrt{\rho \left[\left( \mathcal{Q}_{\alpha}^{-1} \mathcal{Q}\right) \left( \mathcal{Q}_{\alpha}^{-1} \mathcal{Q}\right)^{\top} \right]}=\sqrt{\rho\left( I_{xt} - L_r-L_r^{\top} + L_rL_r^{\top}\right)}$, which is further bounded above by $1+\rho(L_r)$. Now $\sigma(L_r)=\{0\}\cup \sigma \left( I_x-\frac{1}{I_x-\alpha E^{N_t}} \right)$. Thus $\rho(L_r)$ is bounded above by $\frac{\alpha}{1-\alpha}$. So, we obtain $\parallel \mathcal{Q}_{\alpha}^{-1} \mathcal{Q}\parallel\leq \frac{1}{1-\alpha}$. 
Thus the term $1 -\frac{\mu(\mathrm{H}(\mathcal{Q}_{\alpha}^{-1} \mathcal{Q}))^2}{\parallel  \mathcal{Q}_{\alpha}^{-1} \mathcal{Q}\parallel^2} \leq 1- (1-2\alpha)^2=4\alpha(1-\alpha)$. Hence the theorem.
\end{proof}

\section{Preconditioner for $\mathcal{O}((\Delta t)^2)$ Scheme}\label{Section4}
In this section, we develop a preconditioner based on the second-order backward differentiation formula (BDF2) in time. The second-order accurate exponential integrator for the linear part of our model problem \eqref{model_problem} is given by
\begin{equation}\label{2nd_order_scheme}
    \mathbf{u}_{n}=\frac{4}{3}A\mathbf{u}_{n-1} - \frac{1}{3}B\mathbf{u}_{n-2}, \text{where}\; A=\exp(\Delta t \mathcal{L}_h)\; \text{and} \; B=A^2.
\end{equation}
Using \eqref{2nd_order_scheme}, we have the following space-time system
\begin{equation}\label{all_at_once_2ndorder}
    \mathcal{R}\mathbf{v}=\mathbf{f}_2, \text{where}\; \mathcal{R}:= I_{t-1} \otimes I_x - \frac{4}{3}C_0 \otimes A + \frac{1}{3} C_1 \otimes B,
\end{equation}
where $\mathbf{v}=\left(\mathbf{u}_2, \mathbf{u}_3, \cdots, \mathbf{u}_{N_t} \right)^{\top}\in\mathbb{R}^{N_x(N_t-1)}$, $\mathbf{f}_2=(\frac{4}{3}A\mathbf{u}_1 - \frac{1}{3}B\mathbf{u}_0, -\frac{1}{3}B\mathbf{u}_1, 0, \cdots, 0)^{\top}$ and $C_1$ is given by
\begin{equation}\label{alpha_circ_c1}
C_1 = 
\begin{bmatrix}
 0 & 0 &\cdots & 0  & 0 & 0\\
 0 & 0 &\cdots & 0  & 0 & 0\\
 1 & 0& \cdots & 0 & 0 & 0\\
 0 & 1 &\cdots & 0 & 0 & 0\\
 \vdots & \vdots & \ddots & \vdots & \vdots\\
 0 & 0 &   \cdots & 1 & 0 & 0
\end{bmatrix} \in\mathbb{R}^{(N_t-1) \times (N_t-1)}.
\end{equation}
The preconditioned iterative system to solve \eqref{all_at_once_2ndorder} is given by
    \begin{equation}\label{pcond-2nd-order}
        \mathbf{v}^{k+1}=\mathbf{v}^k +\mathbf{s}^k, \; \mathcal{R}_{\alpha}\mathbf{s}^k=\mathbf{f}_2-\mathcal{R}\mathbf{v}^k, \text{where}\; \mathcal{R}_{\alpha} := I_{t-1} \otimes I_x - \frac{4}{3}C_0^{\alpha} \otimes A + \frac{1}{3} C_1^{\alpha} \otimes B. 
    \end{equation}
    The $\alpha$-circulant matrix $C_1^{\alpha}$ is a modified version of $C_1$, with the only differences being the entries $C_1(1, N_t - 2) = \alpha$ and $C_1(2, N_t-1) = \alpha.$ Note that the structures of $C_0$ and $C_0^{\alpha}$ remain as before, but with a reduced size.
    The matrices \( C_0^{\alpha} \) and \( C_1^{\alpha} \) are simultaneously diagonalizable, with 
\( C_0^{\alpha} = \mathcal{P}_{\alpha} \mathcal{D}_0 \mathcal{P}_{\alpha}^{-1} \) and 
\( C_1^{\alpha} = \mathcal{P}_{\alpha} \mathcal{D}_1 \mathcal{P}_{\alpha}^{-1} \),
where the expressions for \( \mathcal{P}_{\alpha} \) and \( \mathcal{D}_0 \) are discussed in Section~\ref{Section2}, with the replacement of $N_t$ by $N_t-1$. 
The matrix \( \mathcal{D}_1 \) takes the form
$
\mathcal{D}_1 = \operatorname{diag}(\lambda_{1,1}, \lambda_{1,2}, \dots, \lambda_{1,N_t-1}),
$
where \( \lambda_{1,n} \) corresponds to the \( n \)-th eigenvalue associated with \( C_1^{\alpha} \) and is given by $\lambda_{1,n} = \alpha^{\frac{2}{N_t-1}} e^{\frac{4(n-1)\pi \mathrm{i}}{N_t-1}}$, for $n=1, 2, \cdots, N_t-1$.
Next we want to solve the residual system $\mathcal{R}_{\alpha}\mathbf{s}^k=\mathbf{f}_2-\mathcal{R}\mathbf{v}^k$ by PinT fashion, which described below:
\begin{equation}\label{eq:exp_ParaDiag_steps_2ndorder}
\left\{
\begin{aligned}
&\text{Step-(a):} && S_p = \left( \mathbb{F} \otimes I_x \right) [\left( \Gamma_{\alpha} \otimes I_x \right)\left(\mathbf{f}_2-\mathcal{R}\mathbf{v}^k \right)], \\
&\text{Step-(b):} && S_{q,n} = \left(  I_x - \frac{4}{3}\lambda_{0,n}A + \frac{1}{3}\lambda_{1,n}B \right)^{-1} S_{p,n}, \quad n = 2, 3, \dots, N_t, \\[.05em]
&\text{Step-(c):} && \mathbf{s}^k = \left( \Gamma_{\alpha}^{-1} \otimes I_x \right) \left( \mathbb{F}^* \otimes I_x \right)S_{q,n},
\end{aligned}
\right.
\end{equation}
In Equation~\ref{eq:exp_ParaDiag_steps_2ndorder}, Step-(a) and Step-(c) can be efficiently executed using the FFT, while Step-(b) remains highly parallelizable. We now proceed to establish the convergence estimates for the proposed method.

\subsection{Eigenvalue estimate of the preconditioned system}
Here, we discuss the convergence of the fixed point preconditioned Richardson scheme \eqref{pcond-2nd-order}. Let us denote the error at $k$-th iteration as $\mathbf{e}^k:=\mathbf{v}-\mathbf{v}^k$. Then we have the following error relation
\begin{equation}\label{err_1st_order}
    \mathbf{e}^k = (I_{xt-1} - \mathcal{R}_{\alpha}^{-1} \mathcal{R}) \mathbf{e}^{k-1},
\end{equation}
where $I_{xt-1}\in\mathbb{R}^{N_x(N_t-1) \times N_x(N_t-1)}$ is the identity matrix. Using the diagonalizability of spatial operator $\mathcal{L}_h$ we rewrite $\mathcal{R}$ and $\mathcal{R}_{\alpha}$ as
\[
\begin{array}{cc}
\mathcal{R}= \left( I_{t-1} \otimes V_h \right) \underbrace{\left( I_{t-1} \otimes I_x - \frac{4}{3}C_0 \otimes e^{\Delta t D_h} + \frac{1}{3}C_1 \otimes e^{2\Delta t D_h} \right)}_{\widebar{\mathcal{R}}} \left( I_{t-1} \otimes V_h^{-1} \right),\\
\mathcal{R}_{\alpha}= \left( I_{t-1} \otimes V_h \right) \underbrace{\left( I_{t-1} \otimes I_x - \frac{4}{3}C_0^{\alpha} \otimes e^{\Delta t D_h} + \frac{1}{3}C_1^{\alpha} \otimes e^{2\Delta t D_h} \right)}_{\widebar{\mathcal{R}}_{\alpha}} \left( I_{t-1} \otimes V_h^{-1} \right).
\end{array}
\]
Thus, $\mathcal{R}$ is similar to $\widebar{\mathcal{R}}$, and $\mathcal{R}_{\alpha}$ is similar to $\widebar{\mathcal{R}}_{\alpha}$. Hence, it is enough to study the spectrum of $I_{xt-1} - \widebar{\mathcal{R}}_{\alpha}^{-1} \widebar{\mathcal{R}}$. Let $z=\lambda \Delta t$ for $\lambda\in\sigma(D_h)$. Then we have 
\[
\begin{array}{cc}
\widebar{\mathcal{R}}= \underbrace{\left( I_{t-1} - \frac{4}{3}e^zC_0  + \frac{1}{3}e^{2z}C_1  \right)}_{\mathcal{T}_z}  \otimes I_x,\;
\widebar{\mathcal{R}}_{\alpha}= \underbrace{\left( I_{t-1} - \frac{4}{3}e^zC_0^{\alpha}  + \frac{1}{3}e^{2z}C_1^{\alpha}  \right)}_{\mathcal{T}_{z, \alpha}}  \otimes I_x.
\end{array}
\]
Therefore, we only need to study the spectral distribution of $\mathcal{M}_z=I_{t-1}- \mathcal{T}_{z, \alpha}^{-1}\mathcal{T}_z$ for the convergence of the method. 
\begin{theorem}\label{thm_bdf2}
    The spectral bound of $\mathcal{M}_z$ satisfies the estimate $\rho(\mathcal{M}_z) \leq \frac{\alpha}{1-\alpha}$.
\end{theorem}
\begin{proof}
    Clearly, $\mathcal{M}_z=I_{t-1} - \left[ I_{t-1} + \mathcal{T}_z^{-1}(\mathcal{T}_{z, \alpha}-\mathcal{T}_z) \right]^{-1} = I_{t-1} - \left[ I_{t-1} + \mathcal{T}_z^{-1}\mathcal{H}_{z, \alpha} \right]^{-1}$, where the explicit forms of $\mathcal{T}_z$ and $\mathcal{H}_{z, \alpha}$ are given below:
    \[
\begin{array}{cc}
   \mathcal{T}_z = 
\begin{bmatrix}
 a_0 & 0 &\cdots & \cdots& \cdots & 0\\
 a_1 &   a_0& \cdots & \cdots & \cdots& 0\\
 a_2& a_1 &  a_0 &\cdots & \cdots & 0\\
 0 &  a_2& a_1 &  a_0 &\cdots & 0 \\
 \vdots & \vdots & \ddots & \ddots & \ddots &\vdots\\
 0 & 0 &  \cdots&  a_2& a_1 &  a_0 \\
\end{bmatrix},
\mathcal{H}_{z, \alpha} = \frac{\alpha}{3}
\begin{bmatrix}
 0 & 0 &\cdots & e^{2z}  & -4e^{z}\\
 0 & 0& \cdots & 0 & e^{2z}\\
 0 & 0 &\cdots & 0 & 0\\
 \vdots & \vdots & \ddots & \vdots & \vdots\\
 0 & 0 &   \cdots & 0 & 0
\end{bmatrix},
\end{array}
\]
where $a_0=1, a_1=-\frac{4}{3}e^z$ and $a_2=\frac{1}{3}e^{2z}$. So, $\sigma(\mathcal{M}_z) = 1-\frac{1}{1+ \sigma( \mathcal{T}_z^{-1}\mathcal{H}_{z, \alpha})}$. Thus to get a bound we need $ \sigma(\mathcal{T}_z^{-1}\mathcal{H}_{z, \alpha})$. Following \cite[p.~172]{linz1985analytical} we have inverse of lower triangular Toeplitz matrix $\mathcal{T}_z$ as 
\[
\mathcal{B}_z = 
\begin{bmatrix}
 b_0 & 0 &\cdots & \cdots & 0\\
 b_1 &   b_0& \cdots & \cdots & 0\\
 b_2& b_1&  b_0 &\cdots & 0\\
 \vdots & \vdots & \ddots & \ddots  &\vdots\\
 b_{N_t-2} & b_{N_t-3} &  \cdots& b_1 &  b_0 \\
\end{bmatrix},
\]
where $b_0=1$ and $b_{m}=-\sum\limits_{l=0}^{m-1}a_{m-l}b_l$ for $1\leq m\leq N_t-2$. Solving this difference equation yields the explicit expressions of $b_m$ as $b_m=\frac{3}{2}e^{mz}-\frac{1}{2}\frac{e^{mz}}{3^m}$ for $1\leq m\leq N_t-2$. Now $\mathcal{T}_z^{-1}\mathcal{H}_{z, \alpha}$ has the following form:
\[
\mathcal{T}_z^{-1}\mathcal{H}_{z, \alpha} = \frac{\alpha}{3}
\begin{bmatrix}
 0  &\cdots & 0 &b_0e^{2z} & -4b_0e^z\\
 0 & \cdots & 0 &b_1e^{2z} & -4b_1e^z+b_0e^{2z}\\
 \vdots & \vdots & \vdots & \vdots  &\vdots\\
 0 & \cdots &  0&b_{N_t-2}e^{2z} & -4b_{N_t-2}e^{z} + b_{N_t-3}e^{2z} \\
\end{bmatrix}.
\]
Clearly, $\mathcal{T}_z^{-1}\mathcal{H}_{z, \alpha}$ has $N_t-3$ zero eigenvalues and two non-zero eigenvalue. The non-zero eigenvalue corresponds to the roots of the following equation
\[
\det
\begin{bmatrix}
 \frac{\alpha}{3}b_{N_t-3}e^{2z} -r & \frac{\alpha}{3}\left(-4b_{N_t-3}e^{z} + b_{N_t-4}e^{2z}\right) \\
 \frac{\alpha}{3}b_{N_t-2}e^{2z} & \frac{\alpha}{3}\left(-4b_{N_t-2}e^{z} + b_{N_t-3}e^{2z} \right) -r 
\end{bmatrix}=0.
\]
Solving the above system for $r$ we obtain
\begin{equation}\label{r_pm}
    r_{\pm}=\frac{\alpha}{3} \left\{ \left(b_{N_t-3}e^{2z} - 2b_{N_t-2}e^{z} \right)\pm \sqrt{4b_{N_t-2}^2e^{2z} - 4b_{N_t-2}b_{N_t-3}e^{3z} + b_{N_t-2}b_{N_t-4}e^{4z}}\right\}.
\end{equation}
Putting the values of $b_m$'s in \eqref{r_pm} and simplifying the expression we get $r_{+}=-\alpha\left(\frac{e^z}{3}\right)^{N_t-1}$ and $r_{-}=-\alpha e^{(N_t-1)z}$. Thus $\sigma(\mathcal{M}_z) =\underbrace{\{0, \cdots, 0 \}}_{N_t-3} \cup \left\{ \frac{-\alpha e^{(N_t-1)z}}{1-\alpha e^{(N_t-1)z}} \right \} \cup \left\{ \frac{-\alpha\left(\frac{e^z}{3}\right)^{N_t-1}}{1-\alpha\left(\frac{e^z}{3}\right)^{N_t-1}} \right \}$. As $z\in\mathbb{R}$ and $z<0$, its clear from $\sigma(\mathcal{M}_z)$ that we have $\rho(\mathcal{M}_z) \leq \frac{\alpha}{1-\alpha}$. Hence the proof.
\end{proof}

\subsection{Analysis of the preconditioned system in GMRES setting}
To analyze the convergence of the GMRES method of the preconditioned system $\mathcal{R}_{\alpha}^{-1}\mathcal{R}$ we rewrite  $\mathcal{R}_{\alpha}^{-1}\mathcal{R}$ as $\mathcal{R}_{\alpha}^{-1}\mathcal{R} = \left[ I_{xt-1} +\mathcal{R}^{-1}\mathcal{L}_{\alpha} \right]^{-1}$ where $\mathcal{L}_{\alpha}:=\mathcal{R}_{\alpha}-\mathcal{R}$. The inverse of the lower triangular block Toeplitz matrix is given as 
\[
\mathcal{R}^{-1} = 
\begin{bmatrix}
 R_0 & 0 &\cdots & \cdots & 0\\
 R_1 &   R_0& \cdots & \cdots & 0\\
 R_2& R_1&  R_0 &\cdots & 0\\
 \vdots & \vdots & \ddots & \ddots  &\vdots\\
 R_{N_t-2} & R_{N_t-3} &  \cdots& R_1 &  R_0 \\
\end{bmatrix},
\]
where $R_0=I_x$ and $R_{m}=\frac{4}{3}AR_{m-1}-\frac{1}{3}BR_{m-2}$ for $1\leq m\leq N_t-2$. Solving this difference equation we get the explicit expressions of $R_m$'s as $R_m=\frac{3}{2}A^m - \frac{1}{2}\left( \frac{A}{3} \right)^m$ for $1\leq m\leq N_t-2$.
Now $\mathcal{R}^{-1}\mathcal{L}_{\alpha}$ has the following form:
\[
\mathcal{R}^{-1}\mathcal{L}_{\alpha} = \frac{\alpha}{3}
\begin{bmatrix}
 0  &\cdots & 0 &R_0B & -4R_0A\\
 0 & \cdots & 0 &R_1B & -4R_1A+R_0B\\
 \vdots & \vdots & \vdots & \vdots  &\vdots\\
 0 & \cdots &  0&R_{N_t-2}B & -4R_{N_t-2}A + R_{N_t-3}B \\
\end{bmatrix}
=
\begin{bmatrix}
 0  &\cdots & 0 &\Xi_0 & \Sigma_0\\
 0 & \cdots & 0 &\Xi_1 & \Sigma_1\\
 \vdots & \vdots & \vdots & \vdots  &\vdots\\
 0 & \cdots &  0&\Xi_{N_t-2} & \Sigma_{N_t-2}\\
\end{bmatrix},
\]
where $\Xi_j$ and $\Sigma_j$ have the obvious definition for $j=1, 2, \cdots, N_t-2$. Then the preconditioned system $\mathcal{R}_{\alpha}^{-1}\mathcal{R}$ has the following structure
\[
\mathcal{R}_{\alpha}^{-1}\mathcal{R}=
\begin{bmatrix}
 I_{x} & 0 &\cdots & 0 &-\frac{\Xi_0 + \Xi_0\Sigma_{N_t-2}-\Xi_{N_t-2}\Sigma_0}{\mathbb{S}} & -\frac{\Sigma_0-\Xi_0\Sigma_{N_t-3}+\Xi_{N_t-3}\Sigma_0}{\mathbb{S}}\\
 0 & I_{x}&\cdots & 0 &-\frac{\Xi_1 + \Xi_1\Sigma_{N_t-2}-\Xi_{N_t-2}\Sigma_1}{\mathbb{S}} & -\frac{\Sigma_1-\Xi_1\Sigma_{N_t-3}+\Xi_{N_t-3}\Sigma_1}{\mathbb{S}}\\
 \vdots & \vdots & \vdots & \vdots  &\vdots&\vdots\\
  0 & 0&\cdots & I_{x} &-\frac{\Xi_{N_t-4} + \Xi_{N_t-4}\Sigma_{N_t-2}-\Xi_{N_t-2}\Sigma_{N_t-4}}{\mathbb{S}} & -\frac{\Sigma_{N_t-4}-\Xi_{N_t-4}\Sigma_{N_t-3}+\Xi_{N_t-3}\Sigma_{N_t-4}}{\mathbb{S}}\\
 0 & 0&\cdots &  0&\frac{I_{x}+\Sigma_{N_t-2}}{\mathbb{S}} & -\frac{\Sigma_{N_t-3}}{\mathbb{S}}\\
 0 & 0&\cdots &  0&-\frac{\Xi_{N_t-2}}{\mathbb{S}} & \frac{I_{x}+\Xi_{N_t-3}}{\mathbb{S}}\\
\end{bmatrix},
\]
where $\mathbb{S}=I_x + \Xi_{N_t-3} + \Sigma_{N_t-2} + \Xi_{N_t-3}\Sigma_{N_t-2} - \Xi_{N_t-2}\Sigma_{N_t-3}$. Putting the values of $\Xi_j$ and $\Sigma_j$ for $j=N_t-2, N_t-3$ and using the values of $R_m$'s in the expression of $\mathbb{S}$ we obtain $\mathbb{S}=I_{x} + \alpha A^{N_t-1}\left( -1 + \frac{1}{3^{N_t-1}}\right)+\alpha^2 A^{2(N_t-1)}\frac{1}{3^{N_t-1}}$. So the system $\mathcal{R}_{\alpha}^{-1}\mathcal{R}$ can be written as $\mathcal{R}_{\alpha}^{-1}\mathcal{R}=I_{xt-1} + \Pi$, where $\Pi$ is given by 
\[
\Pi=
\begin{bmatrix}
 0 & 0 &\cdots & 0 &-\frac{\Xi_0 + \Xi_0\Sigma_{N_t-2}-\Xi_{N_t-2}\Sigma_0}{\mathbb{S}} & -\frac{\Sigma_0-\Xi_0\Sigma_{N_t-3}+\Xi_{N_t-3}\Sigma_0}{\mathbb{S}}\\
 0 & 0&\cdots & 0 &-\frac{\Xi_1 + \Xi_1\Sigma_{N_t-2}-\Xi_{N_t-2}\Sigma_1}{\mathbb{S}} & -\frac{\Sigma_1-\Xi_1\Sigma_{N_t-3}+\Xi_{N_t-3}\Sigma_1}{\mathbb{S}}\\
 \vdots & \vdots & \vdots & \vdots  &\vdots&\vdots\\
  0 & 0&\cdots & 0 &-\frac{\Xi_{N_t-4} + \Xi_{N_t-4}\Sigma_{N_t-2}-\Xi_{N_t-2}\Sigma_{N_t-4}}{\mathbb{S}} & -\frac{\Sigma_{N_t-4}-\Xi_{N_t-4}\Sigma_{N_t-3}+\Xi_{N_t-3}\Sigma_{N_t-4}}{\mathbb{S}}\\
 0 & 0&\cdots &  0&-I_x+\frac{I_{x}+\Sigma_{N_t-2}}{\mathbb{S}} & -\frac{\Sigma_{N_t-3}}{\mathbb{S}}\\
 0 & 0&\cdots &  0&-\frac{\Xi_{N_t-2}}{\mathbb{S}} & -I_x+\frac{I_{x}+\Xi_{N_t-3}}{\mathbb{S}}\\
\end{bmatrix}.
\]
\begin{lemma}
  The blocks in the last two columns of $\Pi$ are invertible.  
\end{lemma}
\begin{theorem}\label{thm_gmres1_2ndorder}
  The GMRES takes at most $2N_x+1$ iterations to converge.
\end{theorem}
\begin{proof}
  Clearly, the rank of $\Pi$ is $2N_x$ (as the blocks in the last two columns are invertible). So, the preconditioned system $\mathcal{R}_{\alpha}^{-1}\mathcal{R}$ is identity plus a low-rank matrix of rank $2N_x$. Thus, GMRES will converge in at most $2N_x+1$ iterations. 
\end{proof}

Next, we discuss the concentration of eigenvalues of the preconditioned system.
Let $\mathbb{G}=\left(I_{2x}\otimes\frac{1}{\mathbb{S}}\right)\left( I_{2x} + \mathcal{W}\right)$, where $I_{2x}=\diag(I_x, I_x)$ and $\mathcal{W}$ is given by 
\[
\mathcal{W}=
\begin{bmatrix}
    \Sigma_{N_t-2} & -\Sigma_{N_t-3}\\ -\Xi_{N_t-2} & \Xi_{N_t-3}
\end{bmatrix},
\]
where $\Sigma_{N_t-2}=\alpha\left( -\frac{3}{2} + \frac{1}{2}\frac{1}{3^{N_t-1}}\right)A^{N_t-1}, \Sigma_{N_t-3}=\alpha\left( -\frac{3}{2} + \frac{1}{2}\frac{1}{3^{N_t-2}}\right)A^{N_t-2}, \Xi_{N_t-2}=\alpha\left( \frac{1}{2} - \frac{1}{2}\frac{1}{3^{N_t-1}}\right)A^{N_t}$, \text{and} $\Xi_{N_t-3}=\alpha\left( \frac{1}{2} - \frac{1}{2}\frac{1}{3^{N_t-2}}\right)A^{N_t-1}$. 
\begin{lemma}\label{eig_G}
   The eigenvalues of $\mathbb{G}$ are given by $\frac{1 + s_+}{1 + \nu}$ and $\frac{1 + s_-}{1 + \nu}$, where \( s_+ \), \( s_- \), and \( \nu \) are defined in the proof.
\end{lemma}
\begin{proof}
As $\mathcal{L}_h$ is diagonalizable, the matrix $\mathbb{G}$ can be written as \\$\mathbb{G}=\begin{bmatrix}
    V_h & 0\\ 0& V_h
\end{bmatrix}
\begin{bmatrix}
    \mathbb{D}_{\mathbb{S}}^{-1} & 0\\0 &\mathbb{D}_{\mathbb{S}}^{-1}
\end{bmatrix}
\left(I_{2x} + \mathcal{D}_{\mathcal{W}} \right)
\begin{bmatrix}
    V_h^{-1} & 0\\ 0& V_h^{-1}
\end{bmatrix}$,
where $\mathcal{D}_{\mathcal{W}}$ is given by
\[
\mathcal{D}_{\mathcal{W}} =
\begin{bmatrix}
    \alpha\left( -\frac{3}{2} + \frac{1}{2}\frac{1}{3^{N_t-1}}\right)E^{N_t-1} & -\alpha\left( -\frac{3}{2} + \frac{1}{2}\frac{1}{3^{N_t-2}}\right)E^{N_t-2}\\ -\alpha\left( \frac{1}{2} - \frac{1}{2}\frac{1}{3^{N_t-1}}\right)E^{N_t} & \alpha\left( \frac{1}{2} - \frac{1}{2}\frac{1}{3^{N_t-2}}\right)E^{N_t-1}
\end{bmatrix},
\]
and $\mathbb{D}_{\mathbb{S}}=I_{x} + \alpha E^{N_t-1}\left( -1 + \frac{1}{3^{N_t-1}}\right)+\alpha^2 E^{2(N_t-1)}\frac{1}{3^{N_t-1}}$. Next, for the permutation matrix $\mathbb{P}\in\mathbb{R}^{2N_x\times 2N_x}$ that rearranges the rows and columns, the matrix $\mathcal{D}_{\mathcal{W}}$ satisfies $\mathbb{P}^{\top}\mathcal{D}_{\mathcal{W}}\mathbb{P}=\blkdiag(b_{d_1}, b_{d_2},\cdots, b_{d_{N_x}})$, where the blocks $b_{d_j}$ are given by 
\[
b_{d_j} =
\begin{bmatrix}
    \alpha\left( -\frac{3}{2} + \frac{1}{2}\frac{1}{3^{N_t-1}}\right)e^{\lambda_j\Delta t(N_t-1)} & -\alpha\left( -\frac{3}{2} + \frac{1}{2}\frac{1}{3^{N_t-2}}\right)e^{\lambda_j\Delta t(N_t-2)}\\ -\alpha\left( \frac{1}{2} - \frac{1}{2}\frac{1}{3^{N_t-1}}\right)e^{\lambda_j\Delta t N_t} & \alpha\left( \frac{1}{2} - \frac{1}{2}\frac{1}{3^{N_t-2}}\right)e^{\lambda_j\Delta t(N_t-1)}
\end{bmatrix},
\]
where $\lambda_j\in\sigma(\mathcal{L}_h)$ for $j=1, 2, \cdots, N_x$.
Under a permutation similarity transformation, the eigenvalues remain unchanged; thus, it suffices to study the eigenvalues of the blocks \( b_{d_j} \).
By solving the quadratic equation \( \det(b_{d_j} - s I_2) = 0 \) for \( s \), we obtain the eigenvalues as
\[
s_+=-\alpha e^{\lambda_j\Delta t(N_t-1)}\frac{1}{3^{N_t-1}},\; s_-=-\alpha e^{\lambda_j\Delta t(N_t-1)}.
\]
So the eigen-structure of $\mathbb{G}$ has the form $\frac{1+s_+}{1+\nu}$ and $\frac{1+s_-}{1+\nu}$, where $\nu=\alpha\left(-1+\frac{1}{3^{N_t-1}} \right)e^{\lambda_j\Delta t(N_t-1)}+\alpha^2\frac{1}{3^{N_t-1}}e^{2\lambda_j\Delta t(N_t-1)}$. So, the proof is complete.
\end{proof}
For notational simplicity, we use \( s_+, s_- \), and \( \nu \) with a generic \( \lambda \), by replacing \( \lambda_j \) with \( \lambda \) in their definitions.
\begin{lemma}\label{gmres_centered_at1_2nd_order}
   The eigenvalues of the preconditioned system are concentrated around $1$ with radius $\frac{\alpha +\alpha^2}{1-\alpha-\alpha^2}$. 
\end{lemma}
\begin{proof}
    Its clear that $\sigma\left( \mathcal{R}_{\alpha}^{-1} \mathcal{R}\right)=\{ 1\} \cup\sigma \left( \mathbb{G} \right)$. Then we have $\max\limits_{\lambda\in \sigma\left( \mathcal{R}_{\alpha}^{-1} \mathcal{R}\right)} |\lambda-1| = \max\limits_{\lambda\in \{1\}\cup \sigma\left( \mathbb{G} \right)} |\lambda-1| \leq \max\limits_{\lambda\in\sigma(\mathcal{L}_h)} \max\left\{\left| \frac{s_+-\nu}{1+\nu}\right|, \left| \frac{s_--\nu}{1+\nu}\right| \right\}$. Now we estimate $\left| \frac{s_+-\nu}{1+\nu}\right| \text{and} \left| \frac{s_--\nu}{1+\nu}\right|$.
    Letting $y=e^{\lambda\Delta t}$ the expression 
    for $\left| \frac{s_--\nu}{1+\nu}\right|$ becomes 
    \begin{equation}\label{eig_est1}
    \left| \frac{s_--\nu}{1+\nu}\right|= \left| \frac{\alpha\frac{1}{3^{N_t-1}}y^{N_t-1}(1+\alpha y^{N_t-1})}{1+\nu}\right|.
  \end{equation}
    As $\lambda\in \mathbb{R}$ with $\lambda<0$ then $y\in(0,1)$. So its clear that the numerator in \eqref{eig_est1} is bounded above by $\frac{\alpha}{3^{N_t-1}}(1+\alpha)$. Now we rewrite $1+\nu=1-\nu_0$, where $\nu_0=\alpha y^{N_t-1}\left(1-\frac{1}{3^{N_t-1}}-\frac{\alpha}{3^{N_t-1}}y^{N_t-1} \right)$. Next, we show $|\nu_0|<1$ for $y<1$. Using triangle inequality and $y<1$, we obtain 
  $|\nu_0|\leq \alpha y^{N_t-1} \left( \left |1-\frac{1}{3^{N_t-1}} \right| + \frac{\alpha}{3^{N_t-1}}y^{N_t-1} \right) < \alpha \left( 1+(\alpha -1)\frac{1}{3^{N_t-1}} \right)<1$ for $N_t\geq 2$ and $\alpha\leq 1$. Thus $|1+\nu|\geq 1-|\nu_0| > 1-\alpha\left( 1+\frac{\alpha}{3^{N_t-1}} \right)>1-\alpha-\alpha^2$. Thus from \eqref{eig_est1} we have 
    \begin{equation}\label{eig_est2}
    \left| \frac{s_--\nu}{1+\nu}\right|\leq \frac{\frac{\alpha}{3^{N_t-1}}(1+\alpha)}{1-\alpha-\alpha^2}<\frac{\alpha +\alpha^2}{1-\alpha-\alpha^2}:=\epsilon_0.
    \end{equation}
    For the term $\left| \frac{s_+-\nu}{1+\nu}\right|$ we have 
    \begin{equation}\label{eig_est3}
    \left| \frac{s_+-\nu}{1+\nu}\right|= \left| \frac{\alpha y^{N_t-1}(-1+\frac{2}{3^{N_t-1}}+\frac{\alpha}{3^{N_t-1}} y^{N_t-1})}{1+\nu}\right| \leq  \frac{\alpha (|-1+\frac{2}{3^{N_t-1}}|+|\frac{\alpha}{3^{N_t-1}} y^{N_t-1}|}{|1+\nu |}.
  \end{equation}
  For $N_t\geq 2$ one has $|-1+\frac{2}{3^{N_t-1}}|\leq 1$. Using this in the numerator in \eqref{eig_est3}, we find the numerator is bounded above by $\alpha(1+\frac{\alpha}{3^{N_t-1}})$, which is further bounded by $\alpha+\alpha^2$. So, in this case, we also have the upper bound $\epsilon_0$ for \eqref{eig_est3}. Hence, we have the result.
  
\end{proof}

\begin{lemma}\label{diag_of_pcond2}
    The preconditioned system $\mathcal{R}_{\alpha}^{-1} \mathcal{R}$ is diagonalizable.
\end{lemma}
\begin{proof}
    To prove the statement, we assume that $N_t$ is odd such that $N_t-1$ is even (since \( N_t \) can be chosen freely). Then $\mathcal{R}_{\alpha}^{-1} \mathcal{R}$ can be rewritten as 
        $\mathcal{R}_{\alpha}^{-1} \mathcal{R}=I_{xt-1} + \widetilde{\Pi}$, where $\widetilde{\Pi}$ is regrouping of the blocks of $\Pi$, is given as 
        \[
\widetilde{\Pi}=
\begin{bmatrix}
 0 & 0 &\cdots & 0 &\Theta_1\\
 0 & 0&\cdots & 0 &\Theta_2\\
 \vdots & \vdots & \vdots & \vdots  &\vdots\\
 0 & 0&\cdots &  0&\Theta_{\frac{N_t-1}{2}-1} \\
 0 & 0&\cdots &  0&\Theta_{\frac{N_t-1}{2}} \\
\end{bmatrix},
\]
 where $\Theta_j\in\mathbb{R}^{2N_x\times 2N_x}$ for $j=1,2,\cdots \frac{N_t-1}{2}$. The new blocks $\Theta_j$ are formed by taking blocks of the last two columns of $\Pi$, each with two row blocks. We are particularly interested in $\Theta_{\frac{N_t-1}{2}}$, which has the form $\Theta_{\frac{N_t-1}{2}}=-I_{2x} + \mathbb{G}$. Using Lemma \ref{eig_G} and the analysis of Lemma \ref{gmres_centered_at1_2nd_order} its clear that $\Theta_{\frac{N_t-1}{2}}$ is invertible. Then $\widetilde{\Pi}$ can be written as $\widetilde{\Pi}=\widetilde{V}\widetilde{D}\widetilde{V}^{-1}$ where $\widetilde{V} \text{and}\; \widetilde{D}$ are given as
  \[
    \begin{array}{cc}
    \widetilde{V}=
\begin{bmatrix}
 I_{2x} &  &  &   & &\Theta_1\Theta_{\frac{N_t-1}{2}}^{-1}\\
  & I_{2x}&  &  & &\Theta_2\Theta_{\frac{N_t-1}{2}}^{-1}\\
  &  & & \ddots& &\vdots\\
  & &  &  &I_{2x} &\Theta_{\frac{N_t-1}{2}-1}\Theta_{\frac{N_t-1}{2}}^{-1} \\
  &  &  &  &  &I_{2x}
\end{bmatrix},
\widetilde{D}=
\begin{bmatrix}
 0&  &  &   & & \\
  & 0&  &  & &\\
  &  & & \ddots& &\\
  & &  &  &0 & \\
  &  &  &  &  &\Theta_{\frac{N_t-1}{2}}
\end{bmatrix}.
\end{array}
\]
Furthermore, $\Theta_{\frac{N_t-1}{2}}$ is diagonalizable as it has disctinct eigenvalues, so $\widetilde{D}=V_{\theta}D_{\theta}V_{\theta}^{-1}$.
Thus we obtain that the system $\mathcal{R}_{\alpha}^{-1} \mathcal{R}=\widetilde{V}V_{\theta}\left( I_{xt-1} + D_{\theta}\right)V_{\theta}^{-1}\widetilde{V}^{-1}$. Hence, the lemma.
\end{proof}
We denote the preconditioned residual as  $\mathbf{r}_{g,\alpha}^k = \mathcal{R}_{\alpha}^{-1} \mathbf{f}_2 - \mathcal{R}_{\alpha}^{-1} \mathcal{R} \mathbf{v}^k$ after $k$-th iterations.
\begin{theorem}\label{thm_gmres2}
Let the GMRES method be applied to the preconditioned system $ \mathcal{R}_{\alpha}^{-1} \mathcal{R} \mathbf{v}=\mathcal{R}_{\alpha}^{-1} \mathbf{f}_2$. Then the residual  satisfies the error estimate
 $
 \parallel \mathbf{r}_{g,\alpha}^k\parallel \leq \kappa_0 \left(\frac{\alpha +\alpha^2}{1-\alpha-\alpha^2} \right)^k \parallel \mathbf{r}_{g,\alpha}^0\parallel,
 $
 where $\kappa_0=\parallel \widetilde{V}V_{\theta}\parallel \parallel (\widetilde{V}V_{\theta})^{-1}\parallel$.
\end{theorem}
\begin{proof}
    By Lemma \ref{diag_of_pcond2}, we have the diagonalizability of the preconditioned system. Then one has 
    $\parallel \mathbf{r}_{g,\alpha}^k\parallel \leq \kappa_0 \min\limits_{p\in P_k(0)}\max\limits_{1\leq i\leq N_x} |p(\lambda_i)| \parallel \mathbf{r}_{g,\alpha}^0\parallel$ for the preconditioned GMRES residual, where $P_k$ is the set of polynomials of degree $k$ or less. Then by using \cite[Section 6.11]{saad2003iterative} and Lemma \ref{gmres_centered_at1_2nd_order} we obtain the required result.
\end{proof}

In this case as well, a convergence result analogous to the first-order case can be established by leveraging Elman's estimate. However, due to the complexity of the matrix expressions involved in the second-order formulation, we state the result as a conjecture rather than providing a complete proof.
\begin{theorem}\label{gmres3_thm_2ndorder}
Let GMRES applied to $\mathcal{R}_{\alpha}^{-1} \mathcal{R} \mathbf{u} = \mathcal{R}_{\alpha}^{-1} \mathbf{f}_2$. Then the residual $\mathbf{r}_{g,\alpha}^k = \mathcal{R}_{\alpha}^{-1} \mathbf{f}_2 - \mathcal{R}_{\alpha}^{-1} \mathcal{R} \mathbf{u}^k$ after $k$ iterations satisfies 
    $\parallel \mathbf{r}_{g, \alpha}^{k}\parallel \leq \left(4(\alpha+\alpha^2)(1-\alpha-\alpha^2) \right)^{k/2} \parallel \mathbf{r}_{g, \alpha}^{0}\parallel.
   $
\end{theorem}

\section{Exp-ParaDiag Formulation for $\mathcal{O}((\Delta t)^s)$ Scheme for $3\leq s\leq 6$}\label{bdfs}
We extend the Exp-ParaDiag formulation to BDF schemes \cite{suli2003introduction} up to order six and study the convergence of preconditioned GMRES numerically. While GMRES convergence can, in principle, follow earlier analysis, it becomes increasingly involved. However, fixed-point convergence follows similar lines as in Theorem \ref{thm_bdf2}.
First, we define several identity matrices $ I_{t-j} \in \mathbb{R}^{{N_t-j} \times {N_t-j}} $ for $ j = 2, 3, 4, 5 $. Additionally, we define circulant matrices $ C_2, C_3, C_4, $ and $ C_5 $, each generated from the following first column vectors:
\[
\begin{aligned}
C_2 &: [0,\ 0,\ 0,\ 1,\ 0,\ \dots,\ 0]^{\top}, \\
C_3 &: [0,\ 0,\ 0,\ 0,\ 1,\ 0,\ \dots,\ 0]^{\top}, \\
C_4 &: [0,\ 0,\ 0,\ 0,\ 0,\ 1,\ 0,\ \dots,\ 0]^{\top}, \\
C_5 &: [0,\ 0,\ 0,\ 0,\ 0,\ 0,\ 1,\ 0,\ \dots,\ 0]^{\top}.
\end{aligned}
\]
Their corresponding $\alpha$-criculant matrices, denoted by $C_j^{\alpha}$, are based on the circulant matrices $C_j$'s for $j=2, 3, 4, 5$, with modifications to $C_j$'s as $C_2(1, N_t-4)=C_2(2, N_t-3)=C_2(3, N_t-2)=\alpha$,
 $C_3(1, N_t-6)=C_3(2, N_t-5)=C_3(3, N_t-4)=C_3(4, N_t-3)=\alpha$, $C_3(1, N_t-8)=C_3(2, N_t-7)=C_3(3, N_t-6)=C_3(4, N_t-5)=C_3(5, N_t-4)=\alpha$ and $C_4(1, N_t-10)=C_4(2, N_t-9)=C_4(3, N_t-8)=C_4(4, N_t-7)=C_4(5, N_t-6)=C_4(6, N_t-5)=\alpha$.
We begin by presenting the Exp-ParaDiag formulation as a GMRES preconditioner using the BDF3 method. The third-order accurate EI for \eqref{model_problem} is given by
\begin{equation}\label{3rd_order_scheme}
    \mathbf{u}_{n}=\frac{18}{11}A\mathbf{u}_{n-1} - \frac{9}{11}A^2\mathbf{u}_{n-2} +  \frac{2}{11}A^3\mathbf{u}_{n-3}, \text{for}\; n\geq 3.
\end{equation}
Using \eqref{3rd_order_scheme}, we have the following space-time system
\begin{equation}\label{all_at_once_3rdorder}
    \mathcal{R}_3\mathbf{v}_3=\mathbf{f}_3, \text{where}\; \mathcal{R}_3:= I_{t-2} \otimes I_x - \frac{18}{11}C_0 \otimes A + \frac{9}{11} C_1 \otimes A^2 -  \frac{2}{11} C_2 \otimes A^3,
\end{equation}
where $\mathbf{v}_3=\left(\mathbf{u}_3, \mathbf{u}_4, \cdots, \mathbf{u}_{N_t} \right)^{\top}\in\mathbb{R}^{N_x(N_t-2)}$, $\mathbf{f}_3=(\frac{18}{11}A\mathbf{u}_2 -\frac{9}{11}A^2\mathbf{u}_1 + \frac{2}{11}A^3\mathbf{u}_0, -\frac{9}{11}A^2\mathbf{u}_2 + \frac{2}{11}A^3\mathbf{u}_1, \frac{2}{11}A^3\mathbf{u}_2, 0, \cdots, 0)^{\top}$. The preconditioner to solve \eqref{all_at_once_3rdorder} by GMRES is given by $\mathcal{R}_3^{\alpha}:= I_{t-2} \otimes I_x - \frac{18}{11}C_0^{\alpha} \otimes A + \frac{9}{11} C_1^{\alpha} \otimes A^2 -  \frac{2}{11} C_2^{\alpha} \otimes A^3$. 
The fourth-order accurate EI for \eqref{model_problem} is given by
\begin{equation}\label{fourth_order_scheme}
    \mathbf{u}_{n}=\frac{48}{25}A\mathbf{u}_{n-1} - \frac{36}{25}A^2\mathbf{u}_{n-2} +  \frac{16}{25}A^3\mathbf{u}_{n-3} - \frac{3}{25}A^4\mathbf{u}_{n-4}, \text{for}\; n\geq 4.
\end{equation}
Using \eqref{fourth_order_scheme}, we have the following space-time system
\begin{equation}\label{all_at_once_4thorder}
    \mathcal{R}_4\mathbf{v}_4=\mathbf{f}_4, \text{where}\; \mathcal{R}_4:= I_{t-3} \otimes I_x - \frac{48}{25}C_0 \otimes A + \frac{36}{25} C_1 \otimes A^2 -  \frac{16}{25} C_2 \otimes A^3 + \frac{3}{25} C_3 \otimes A^4,
\end{equation}
where $\mathbf{v}_4=\left(\mathbf{u}_4, \mathbf{u}_5, \cdots, \mathbf{u}_{N_t} \right)^{\top}\in\mathbb{R}^{N_x(N_t-3)}$, $\mathbf{f}_4=( \frac{48}{25}A\mathbf{u}_3 - \frac{36}{25}A^2\mathbf{u}_2 + \frac{16}{25}A^3\mathbf{u}_1 - \frac{3}{25}A^4\mathbf{u}_0, -\frac{36}{25}A^2\mathbf{u}_3 +\frac{16}{25}A^3\mathbf{u}_2 - \frac{3}{25}A^4\mathbf{u}_1, \frac{16}{25}A^3\mathbf{u}_3 - \frac{3}{25}A^4\mathbf{u}_2, -\frac{3}{25}A^4\mathbf{u}_3, 0, \cdots, 0)^{\top}$. The preconditioner to solve \eqref{all_at_once_4thorder} by GMRES is given by $\mathcal{R}_4^{\alpha}:= I_{t-3} \otimes I_x - \frac{48}{25}C_0^{\alpha} \otimes A + \frac{36}{25} C_1^{\alpha} \otimes A^2 -  \frac{16}{25} C_2^{\alpha} \otimes A^3 + \frac{3}{25} C_3^{\alpha} \otimes A^4$. 
The fifth-order accurate EI for \eqref{model_problem} is given by
\begin{equation}\label{fifth_order_scheme}
    \mathbf{u}_{n}=\frac{300}{137}A\mathbf{u}_{n-1} - \frac{300}{137}A^2\mathbf{u}_{n-2} +  \frac{200}{137}A^3\mathbf{u}_{n-3} - \frac{75}{137}A^4\mathbf{u}_{n-4} + \frac{12}{137}A^5\mathbf{u}_{n-5}, \text{for}\; n\geq 5.
\end{equation}
Using \eqref{fifth_order_scheme}, we have the following space-time system
\begin{multline}\label{all_at_once_5thorder}
    \mathcal{R}_5\mathbf{v}_5=\mathbf{f}_5, \text{where}\; \mathcal{R}_5:= I_{t-4} \otimes I_x - \frac{300}{137}C_0 \otimes A + \frac{300}{137} C_1 \otimes A^2 -  \frac{200}{137} C_2 \otimes A^3 \\+ \frac{75}{137} C_3 \otimes A^4 -  \frac{12}{137} C_4 \otimes A^5,
\end{multline}
where $\mathbf{v}_5=\left(\mathbf{u}_5, \mathbf{u}_6, \cdots, \mathbf{u}_{N_t} \right)^{\top}\in\mathbb{R}^{N_x(N_t-4)}$, $\mathbf{f}_5=( \frac{300}{137}A\mathbf{u}_4 - \frac{300}{137}A^2\mathbf{u}_3 + \frac{200}{137}A^3\mathbf{u}_2 - \frac{75}{137}A^4\mathbf{u}_1 + \frac{12}{137}A^5\mathbf{u}_0, - \frac{300}{137}A^2\mathbf{u}_4 + \frac{200}{137}A^3\mathbf{u}_3 - \frac{75}{137}A^4\mathbf{u}_2 + \frac{12}{137}A^5\mathbf{u}_1, \frac{200}{137}A^3\mathbf{u}_4 - \frac{75}{137}A^4\mathbf{u}_3 + \frac{12}{137}A^5\mathbf{u}_2, - \frac{75}{137}A^4\mathbf{u}_4 + \frac{12}{137}A^5\mathbf{u}_3, \frac{12}{137}A^5\mathbf{u}_4, 0, \cdots, 0)^{\top}$. The preconditioner to solve \eqref{all_at_once_5thorder} by GMRES is given by $\mathcal{R}_5^{\alpha}:= I_{t-4} \otimes I_x - \frac{300}{137}C_0^{\alpha} \otimes A + \frac{300}{137} C_1^{\alpha} \otimes A^2 -  \frac{200}{137} C_2^{\alpha} \otimes A^3 + \frac{75}{137} C_3^{\alpha} \otimes A^4 - \frac{12}{137} C_4^{\alpha} \otimes A^5$. 
The sixth-order accurate EI for \eqref{model_problem} is given by
\begin{equation}\label{sixth_order_scheme}
    \mathbf{u}_{n}=\frac{360}{147}A\mathbf{u}_{n-1} - \frac{450}{147}A^2\mathbf{u}_{n-2} +  \frac{400}{147}A^3\mathbf{u}_{n-3} - \frac{225}{147}A^4\mathbf{u}_{n-4} + \frac{72}{147}A^5\mathbf{u}_{n-5} -\frac{10}{147}A^6\mathbf{u}_{n-6}, \text{for}\; n\geq 6.
\end{equation}
Using \eqref{sixth_order_scheme}, we have the following space-time system
\begin{multline}\label{all_at_once_6thorder}
    \mathcal{R}_6\mathbf{v}_6=\mathbf{f}_6, \text{where}\; \mathcal{R}_6:= I_{t-5} \otimes I_x - \frac{360}{147}C_0 \otimes A + \frac{450}{147} C_1 \otimes A^2 -  \frac{400}{147} C_2 \otimes A^3 \\+ \frac{225}{147} C_3 \otimes A^4 -  \frac{72}{147} C_4 \otimes A^5 + \frac{10}{147} C_5 \otimes A^6,
\end{multline}
where $\mathbf{v}_6=\left(\mathbf{u}_6, \mathbf{u}_7, \cdots, \mathbf{u}_{N_t} \right)^{\top}\in\mathbb{R}^{N_x(N_t-5)}$, $\mathbf{f}_6=( \frac{360}{147}A\mathbf{u}_5 - \frac{450}{147}A^2\mathbf{u}_4 + \frac{400}{147}A^3\mathbf{u}_3 - \frac{225}{147}A^4\mathbf{u}_2 + \frac{72}{147}A^5\mathbf{u}_1 - \frac{10}{147}A^6\mathbf{u}_0, - \frac{450}{147}A^2\mathbf{u}_5 + \frac{400}{147}A^3\mathbf{u}_4 - \frac{225}{147}A^4\mathbf{u}_3 + \frac{72}{147}A^5\mathbf{u}_2 - \frac{10}{147}A^6\mathbf{u}_1, \frac{400}{147}A^3\mathbf{u}_5 - \frac{225}{147}A^4\mathbf{u}_4 + \frac{72}{147}A^5\mathbf{u}_3 - \frac{10}{147}A^6\mathbf{u}_2, - \frac{225}{147}A^4\mathbf{u}_5 + \frac{72}{147}A^5\mathbf{u}_4 - \frac{10}{147}A^6\mathbf{u}_3,  \frac{72}{147}A^5\mathbf{u}_5 - \frac{10}{147}A^6\mathbf{u}_4, - \frac{10}{147}A^6\mathbf{u}_5, 0, \cdots, 0)^{\top}$. The preconditioner to solve \eqref{all_at_once_6thorder} by GMRES is given by $\mathcal{R}_6^{\alpha}:= I_{t-5} \otimes I_x - \frac{360}{147}C_0^{\alpha} \otimes A + \frac{450}{147} C_1^{\alpha} \otimes A^2 -  \frac{400}{147} C_2^{\alpha} \otimes A^3 + \frac{225}{147} C_3^{\alpha} \otimes A^4 - \frac{72}{147} C_4^{\alpha} \otimes A^5 + \frac{10}{147} C_5^{\alpha} \otimes A^6$.

\section{Formulation and Convergence in the Nonlinear Setting}\label{Section5}
We now extend the Exp-ParaDiag framework to handle nonlinear problems. To this end, we employ IF techniques for temporal discretization (one can use an ETD-based approach) and build our method upon these discretization techniques.
A first order accurate in time EI based on the IF approach for \eqref{model_problem} is given by 
\begin{equation}\label{fully_discrete_nonlinear}
\mathbf{u}_{n} = e^{\Delta t \mathcal{L}_h} \mathbf{u}_{n-1}  +\Delta t  \mathcal{N}(\mathbf{u}_{n}).
\end{equation}
The corresponding Exp-ParaDiag formulation is: 
\begin{equation}\label{fully_disc_wr_nonlinear}
\begin{cases}
\mathbf{u}_{n}^k = e^{\Delta t \mathcal{L}_h} \mathbf{u}_{n-1}^k  +\Delta t  \mathcal{N}(\mathbf{u}_{n}^k)\;, n=1, 2, \cdots, N_t,\\
\mathbf{u}^k_0 = \alpha \mathbf{u}^k_{N_t} - \alpha \mathbf{u}^{k-1}_{N_t} + \mathbf{u}_0.
\end{cases}
\end{equation}
Now we formulate the usual three-step solution procedure for \eqref{fully_disc_wr_nonlinear}. Gathering all the space-time points, we have the following space-time nonlinear equation in $\mathcal{U}^k$ for \eqref{fully_disc_wr_nonlinear} 
\begin{equation}\label{newton_1}
    \Psi_1(\mathcal{U}^k)= (I_t \otimes I_x)\mathcal{U}^k - (C_0^{\alpha} \otimes A)\mathcal{U}^k - \Delta t \mathcal{N}_0(\mathcal{U}^k) -{\mathbf{b}^{k-1}},
\end{equation}
where $\mathcal{N}_0(\mathcal{U}^k):= \left( \mathcal{N}(\mathbf{u}_{1}^k), \mathcal{N}(\mathbf{u}_{2}^k),\cdots, \mathcal{N}(\mathbf{u}_{n}^k) \right)^{\top}$ is the nonlinear space-time vector.
The inexact Newton to solve \eqref{newton_1} is given by: for $m=1, 2,\cdots$ compute the following 
\begin{equation}\label{newton_2}
    \Psi_1'(\mathcal{U}^k_m)\mathcal{U}_{m+1}^k=  - \Delta t \blkdiag \left( \mathcal{N}'(\mathbf{u}_{1}^k), \cdots, \mathcal{N}'(\mathbf{u}_{n}^k) \right)\mathcal{U}_m^k  + \Delta t \mathcal{N}_0(\mathcal{U}_m^k)+{\mathbf{b}^{k-1}}:=\mathbf{res}_m,
\end{equation}
where $\mathbf{b}^{k-1}$ is defined earlier and the inexact Jacobian approximation is taken as $\widebar{\mathcal{N}}(\mathcal{U}^k) := I_t \otimes \left( \frac{1}{N_t} \sum_{n=1}^{N_t} \mathcal{N}'(\mathbf{u}_{n}^k) \right)
$ \cite{gander2017time, gander2019convergence} to build $\Psi_1'(\mathcal{U}^k_m)=(I_t \otimes I_x) - (C_0^{\alpha} \otimes A) - \Delta t \widebar{\mathcal{N}}(\mathcal{U}^k_m)$. Then we have the following three steps to solve \eqref{newton_2}, which described below:
\begin{equation}\label{eq:exp_ParaDiag_steps_newton}
\left\{
\begin{aligned}
&\text{Step-(a):} && S_p = \left( \mathbb{F} \otimes I_x \right) [\left( \Gamma_{\alpha} \otimes I_x \right)\mathbf{res}_m], \\
&\text{Step-(b):} && S_{q,n} = \left(  I_x - \lambda_{0,n}A - \Delta t \widebar{\mathcal{N}}(\mathcal{U}^k_m)\right)^{-1} S_{p,n}, \quad n =1, 2, 3, \dots, N_t, \\[.05em]
&\text{Step-(c):} && \mathbf{s}^k = \left( \Gamma_{\alpha}^{-1} \otimes I_x \right) \left( \mathbb{F}^* \otimes I_x \right)S_{q,n}.
\end{aligned}
\right.
\end{equation}
In the context of approximating the space-time Jacobian, we propose the approximation of the space-time nonlinear terms as $\widebar{\mathcal{N}}(\mathcal{U}_m^k) := I_t \otimes \diag\left( \mathcal{N}'(\mathbf{u}_{0}) \right)$ in \eqref{eq:exp_ParaDiag_steps_newton}. We show the corresponding numerical experiments for that in Section \ref{Section6}.

We now demonstrate the convergence properties of the PinT method \eqref{fully_disc_wr_nonlinear}. Let $\mathfrak{e}_n^{k}:=\mathbf{u}_n^{k}-\mathbf{u}_n$ denote the error at each iteration of the PinT method. 
\begin{lemma}\label{inn_pro_lemma}
    For $\mathbf{x}, \mathbf{y}\in\mathbb{R}^{d}$, we have $\langle e^{\Delta t\mathcal{L}_h}\mathbf{x}, \mathbf{y} \rangle \leq \parallel e^{\Delta t\mathcal{L}_h}\parallel \langle\mathbf{x}, \mathbf{y}\rangle < \langle \mathbf{x}, \mathbf{y} \rangle$.
\end{lemma}

\begin{theorem}\label{thm_nonlin}
 The error $\mathfrak{e}_{n}^k$ at $k$-th iteration corresponding to \eqref{fully_disc_wr_nonlinear} satisfies the estimate
\begin{equation}\label{est_nonlin}
  \max_{1\leq n\leq N_t} \parallel\mathfrak{e}_{n}^k\parallel \leq \left( \frac{\vert\alpha\vert e^{-\frac{2TM}{1+ 2\Delta t M }}}{1-\vert\alpha\vert e^{-\frac{2TM}{1+ 2\Delta t M }}} \right)^k \parallel\mathfrak{e}_{0}^0\parallel.
\end{equation}
\end{theorem}
\begin{proof}
The error equation for the first order accurate scheme \eqref{fully_discrete_nonlinear} takes the form
\begin{equation}\label{fully_discrete_nonlinear_error}
\mathfrak{e}_{n}^k = e^{\Delta t \mathcal{L}_h} \mathfrak{e}_{n-1}^k  +\Delta t  \left( \mathcal{N}(\mathbf{u}_{n}^k) - \mathcal{N}(\mathbf{u}) \right).
\end{equation}
Taking inner-product of $\mathfrak{e}_{n}^k$ with \eqref{fully_discrete_nonlinear} and applying the one-sided Lipschitz condition \eqref{one-sided_lip} and using Lemma \ref{inn_pro_lemma} we obtain 
\begin{equation}\label{fully_discrete_nonlinear_error2}
\langle \mathfrak{e}_{n}^k,  \mathfrak{e}_{n}^k\rangle\leq \langle \mathfrak{e}_{n-1}^k,  \mathfrak{e}_{n}^k\rangle - \Delta t M \parallel \mathfrak{e}_{n}^k\parallel^2.
\end{equation}
 From \eqref{fully_discrete_nonlinear} we have 
\begin{equation}\label{full_discrete_nonli_err_2}
    \langle \mathfrak{e}_{n}^k - \mathfrak{e}_{n-1}^k,  \mathfrak{e}_{n}^k\rangle\leq - \Delta t M \parallel \mathfrak{e}_{n}^k\parallel^2.
\end{equation}
Using the identity $(\mathfrak{a}-\mathfrak{b})\mathfrak{a} \geq \frac{\mathfrak{a}^2}{2} - \frac{\mathfrak{b}^2}{2}$ in \eqref{full_discrete_nonli_err_2} we get
\begin{equation}\label{full_discrete_nonli_err_3}
    \frac{ \parallel\mathfrak{e}_{n}^k\parallel^2 - \parallel\mathfrak{e}_{n-1}^k\parallel^2}{2}\leq - \Delta t M \parallel \mathfrak{e}_{n}^k\parallel^2.
\end{equation}
Equation \eqref{full_discrete_nonli_err_3} yields
\begin{equation}\label{full_discrete_nonli_err_4}
  \parallel\mathfrak{e}_{n}^k\parallel^2 \leq \frac{1}{1+ 2\Delta t M }\parallel \mathfrak{e}_{n-1}^k\parallel^2 \leq \left( e^{-\frac{2\Delta tM}{1+ 2\Delta t M }}\right)^{N_t}\parallel \mathfrak{e}_{0}^k\parallel^2 = e^{-\frac{2TM}{1+ 2\Delta t M }}\parallel \mathfrak{e}_{0}^k\parallel^2 .
\end{equation}
From the error equation corresponding to \eqref{fully_disc_wr_nonlinear}$_2$ and using \eqref{full_discrete_nonli_err_4} we obtain
\begin{equation}\label{full_discrete_nonli_err_5}
  \parallel\mathfrak{e}_{0}^k\parallel  \leq |\alpha|\parallel \mathfrak{e}_{N_t}^k\parallel+ |\alpha|\parallel \mathfrak{e}_{N_t}^{k-1}\parallel \leq |\alpha| e^{-\frac{2TM}{1+ 2\Delta t M }}\parallel \mathfrak{e}_{0}^k\parallel + |\alpha| e^{-\frac{2TM}{1+ 2\Delta t M }}\parallel \mathfrak{e}_{0}^{k-1}\parallel,
\end{equation}
which yields
\begin{equation}\label{full_discrete_nonli_err_6}
  \parallel\mathfrak{e}_{0}^k\parallel  \leq \frac{ |\alpha| e^{-\frac{2TM}{1+ 2\Delta t M }}}{1- |\alpha| e^{-\frac{2TM}{1+ 2\Delta t M }}}\parallel \mathfrak{e}_{0}^{k-1}\parallel.
\end{equation}
Now using \eqref{full_discrete_nonli_err_6} in \eqref{full_discrete_nonli_err_4}, we have the required estimate in \eqref{est_nonlin}.
\end{proof}




\section{Numerical Illustration}\label{Section6}
This section presents the numerical results for the proposed Exp-ParaDiag methods, applied to both linear and nonlinear time-dependent PDEs. The PinT iterations are initialized with zeros and proceed until the error measured in the $ L^\infty(0, T; L^2(\Omega)) $ norm between the reference corresponding sequential solution \( u \) and the Exp-ParaDiag solution \( u^k \) at iteration \( k \)---is reduced below a tolerance of \( 10^{-10} \). The presented numerical results were obtained on a desktop PC equipped with a 4.2 GHz processor and 64 GB of RAM, with the code implemented in MATLAB 2025.

\subsection{Experiment in Linear Settings}
We begin by discussing the selection of the free parameter $\alpha$ based on the min-max problem in \eqref{minmax}. The optimal values, $\alpha_{\text{opt}}$, corresponding to different combinations of $a$, $c$, and $T$ are presented in Table \ref{table_alpha_opt}.
\begin{table}[ht]
\centering
\caption{Optimal values of $\alpha$ according to \eqref{minmax}}
\begin{tabular}{ccccccccccc}
\hline
\multirow{2}{*}{ $(N_x, T)$} & \multicolumn{4}{c}{ $\alpha_{\text{opt}}$: the optimized $\alpha$}  \\
\cline{2-5} \cline{7-10}
 & $(a,c)=(1, 0)$ & $(a,c)=(10^{-3},0)$ & $(a,c)=(10^{-3},0.1)$ & $(a,c)=(10^{-5}, 2)$ &  \\
\hline
(32, 0.5) & $0.0001953125$ & $0.0001953125$ & $0.0001953125$ & $0.0001953125$  \\
(64, 2) & $0.0001953125$ & $0.0001953125$ & $0.0001953125$ & $0.0001953125$  \\
(128, 8) & $0.0001953125$ & $0.0001953125$ & $0.0001953125$ & $0.0001953125$  \\
(256, 20) & $0.0001953125$ & $0.0001953125$ & $0.0001953125$ & $0.0001953125$  \\
\hline
\end{tabular}
\label{table_alpha_opt}
\end{table}
It can be observed that $\alpha_{\text{opt}}$ is independent of the parameters $a$, $c$, $N_x$, and $T$. The values of $\alpha_{\text{opt}}$ in Table \ref{table_alpha_opt} support the preference for smaller $\alpha$, consistent with the choices commonly made by researchers in this field.
For the 1D computation, we consider the initial condition \( u_0(x) = e^{-30x^2} \) on the domain \( \Omega = (-1, 1) \). For the 2D computation, the initial condition is given by \( u_0(x, y) = e^{-20[(x - 1/2)^2 + (y - 1/2)^2]} \) on the domain \( \Omega = (0, 1) \times (0, 1) \). Unless otherwise stated, both the domain and the initial condition remain fixed throughout.
 \subsubsection{Experiment of Exp-ParaDiag with $\mathcal{O}(\Delta t)$ Scheme}
 In the first two plots of Figure~\ref{err_bound}, we show a comparison between the theoretical and numerical error contraction, as predicted by Theorem~\ref{thm_1st}, with \( T = 1 \), \( \Delta t = h = 1/128 \), and \( \alpha_{\text{opt}} = 0.0001953125 \). The results indicate that the discrete error bound is sharp.
\begin{figure}[h!]
    \centering
    \subfloat{{\includegraphics[height=3.5cm,width=4cm]{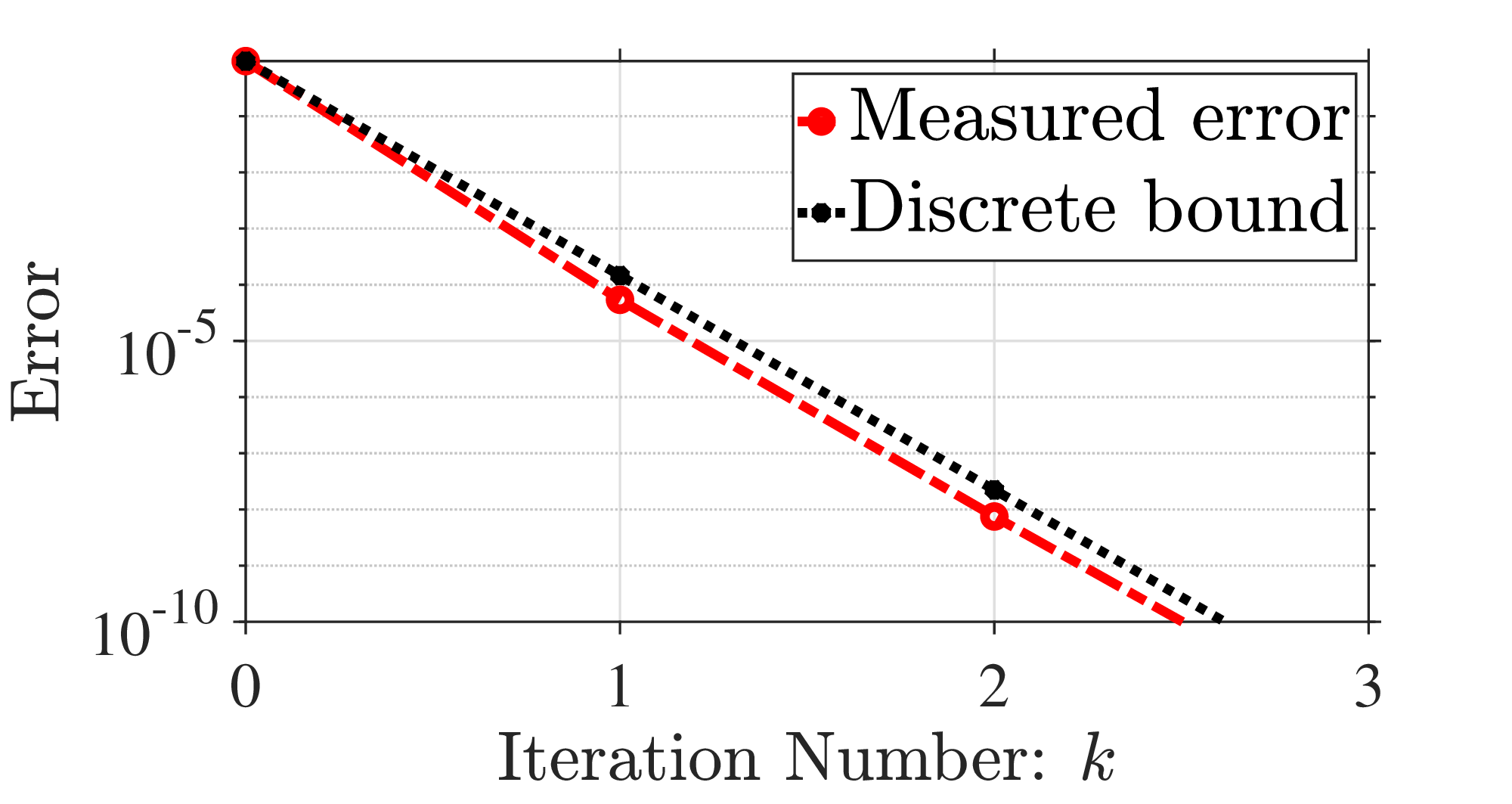} }}
    \subfloat{{\includegraphics[height=3.5cm,width=4cm]{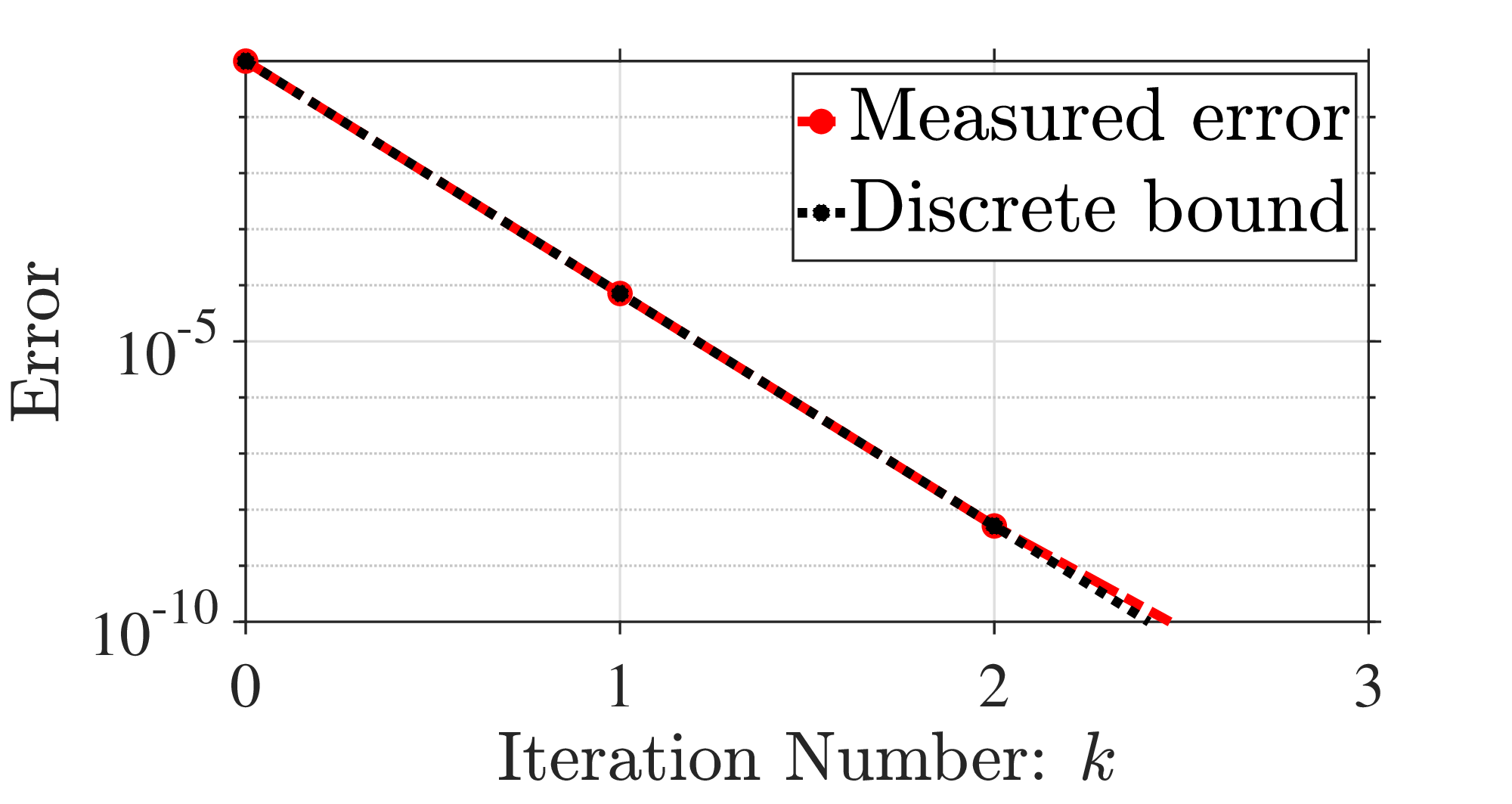} }}
    \subfloat{{\includegraphics[height=3.5cm,width=4cm]{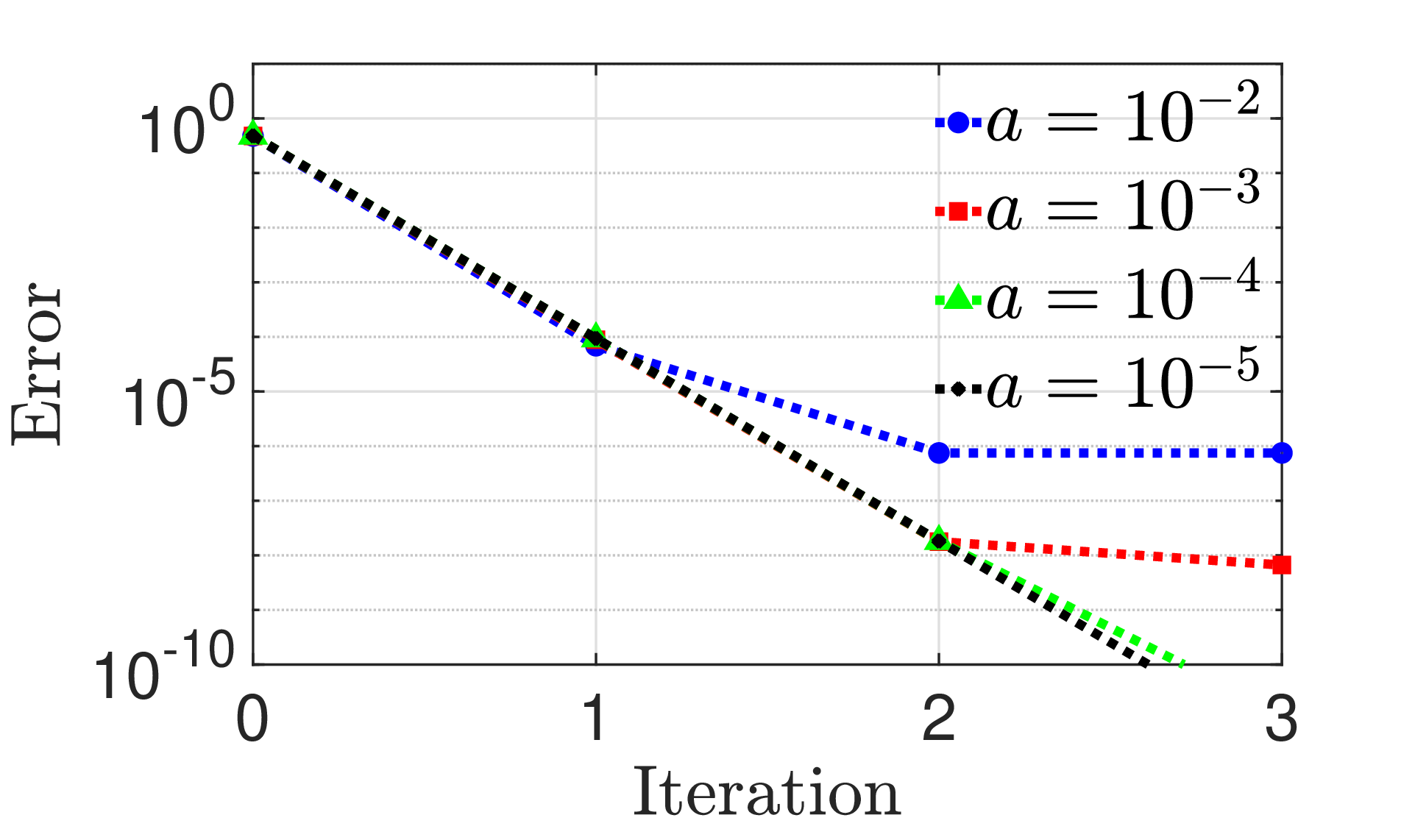} }}
    \subfloat{{\includegraphics[height=3.5cm,width=4cm]{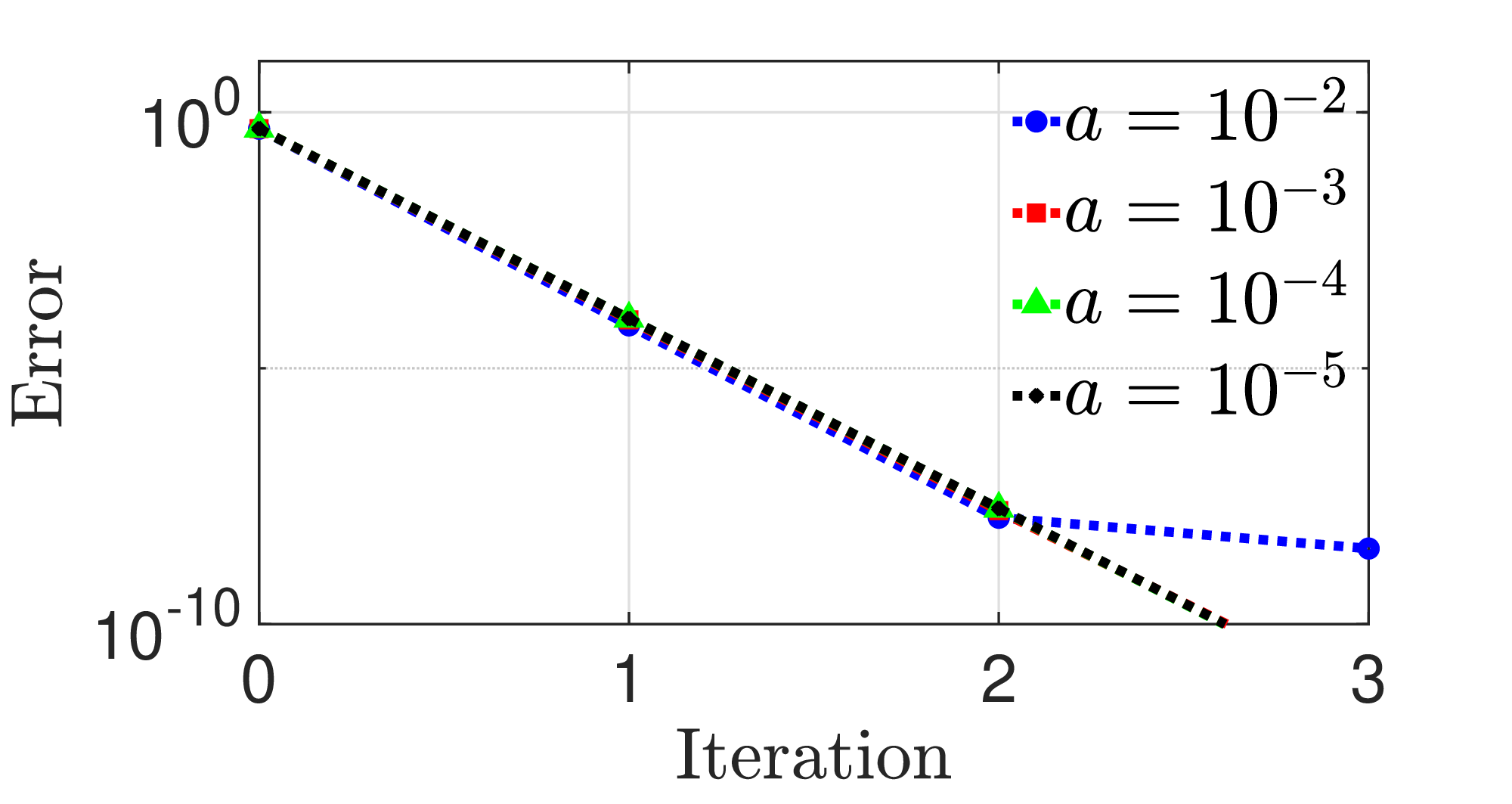} }}
    \caption{ First: error comparison for $(a, c)=(0.1, 0)$ ; Second: error comparison for $(a, c)=(10^{-5}, 1)$; Third: 2nd order accurate Pad\'{e} approximation; Fourth: 3rd order accurate Pad\'{e} approximation.}
    \label{err_bound}
\end{figure}
Next, we present an experiment using the Pad\'{e} approximation of \( A \) in Exp-ParaDiag, with exact \( A \) used in the reference solution.
To approximate the matrix exponential $A=\exp(\Delta t \mathcal{L}_h)$, we employ Pad\'{e} approximations, which provide rational approximations with favorable stability and accuracy properties. Specifically, we consider the Pad\'{e} approximations given by
\[
\exp(\Delta t \mathcal{L}_h) \approx \left(I + \tfrac{1}{2}\Delta t \mathcal{L}_h\right)\left(I - \tfrac{1}{2}\Delta t \mathcal{L}_h\right)^{-1} + \mathcal{O}((\Delta t \mathcal{L}_h)^3),
\]
which achieves second-order accuracy, and the approximation of the form,
\[
\exp(\Delta t \mathcal{L}_h) \approx \left(I + \tfrac{2}{3}\Delta t \mathcal{L}_h + \tfrac{1}{6}(\Delta t \mathcal{L}_h)^2\right)\left(I - \tfrac{1}{3}\Delta t \mathcal{L}_h\right)^{-1} + \mathcal{O}((\Delta t \mathcal{L}_h)^4),
\]
which is third-order accurate. In the last two plots of Figure~\ref{err_bound}, we illustrate the convergence behavior of Exp-ParaDiag using second- and third-order accurate Pad\'{e} approximations of \( A \), for varying diffusion coefficients \( a \), with \( c = 0 \) in \eqref{model_problem_linear}, and fixed parameters \( T = 2 \), \( \Delta t = h = 1/128 \), and \( \alpha = \alpha_{\text{opt}} \). As the parameter \( a \) decreases, the norm \( \| \Delta t \mathcal{L}_h \| \) becomes smaller, leading to more accurate convergence. Furthermore, higher-order Pad\'{e} approximations yield improved convergence, especially for larger values of \( a \).
\begin{figure}[h!]
    \centering
    \subfloat{{\includegraphics[height=3.5cm,width=4cm]{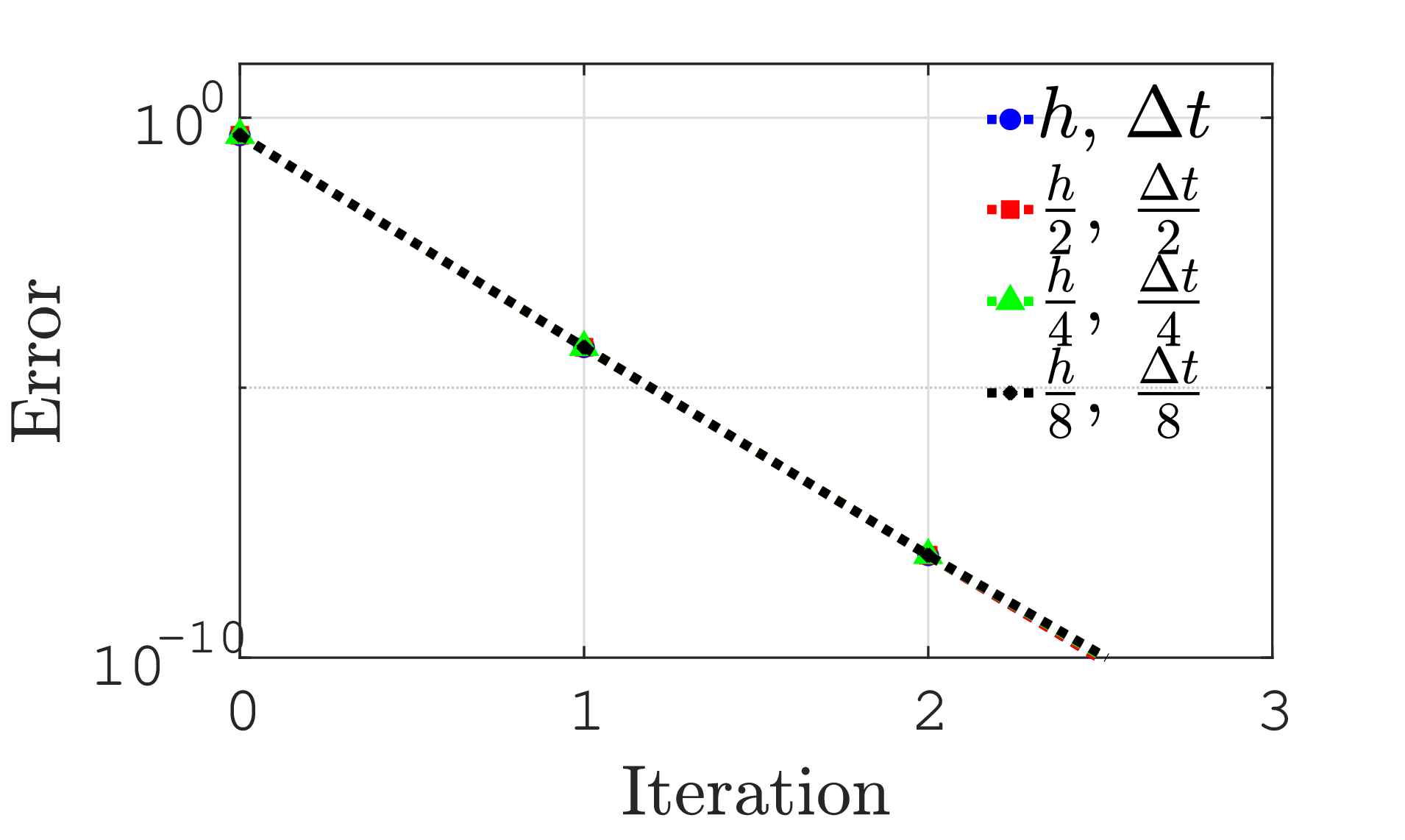} }}
    \subfloat{{\includegraphics[height=3.5cm,width=4cm]{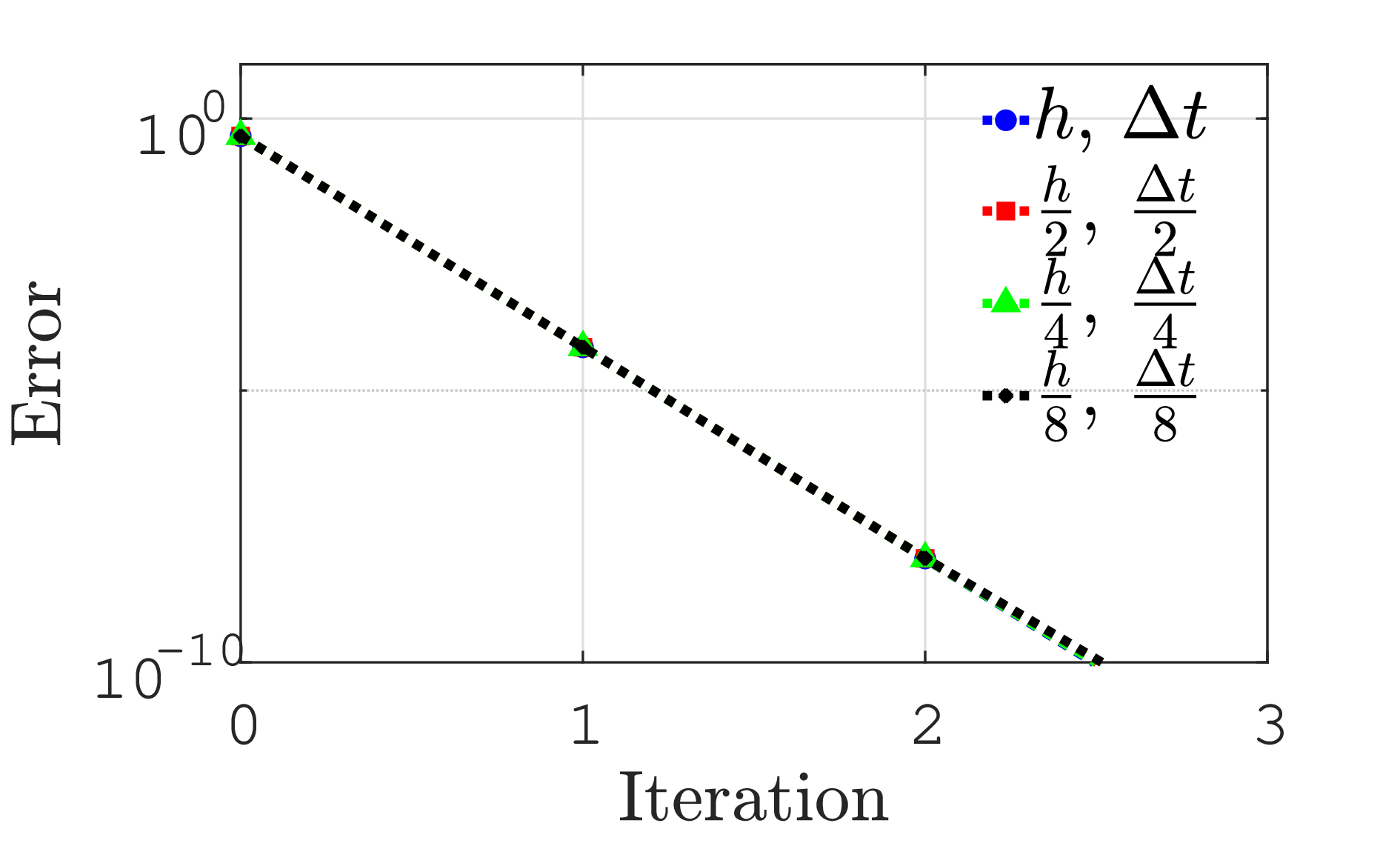} }}
    \subfloat{{\includegraphics[height=3.5cm,width=4cm]{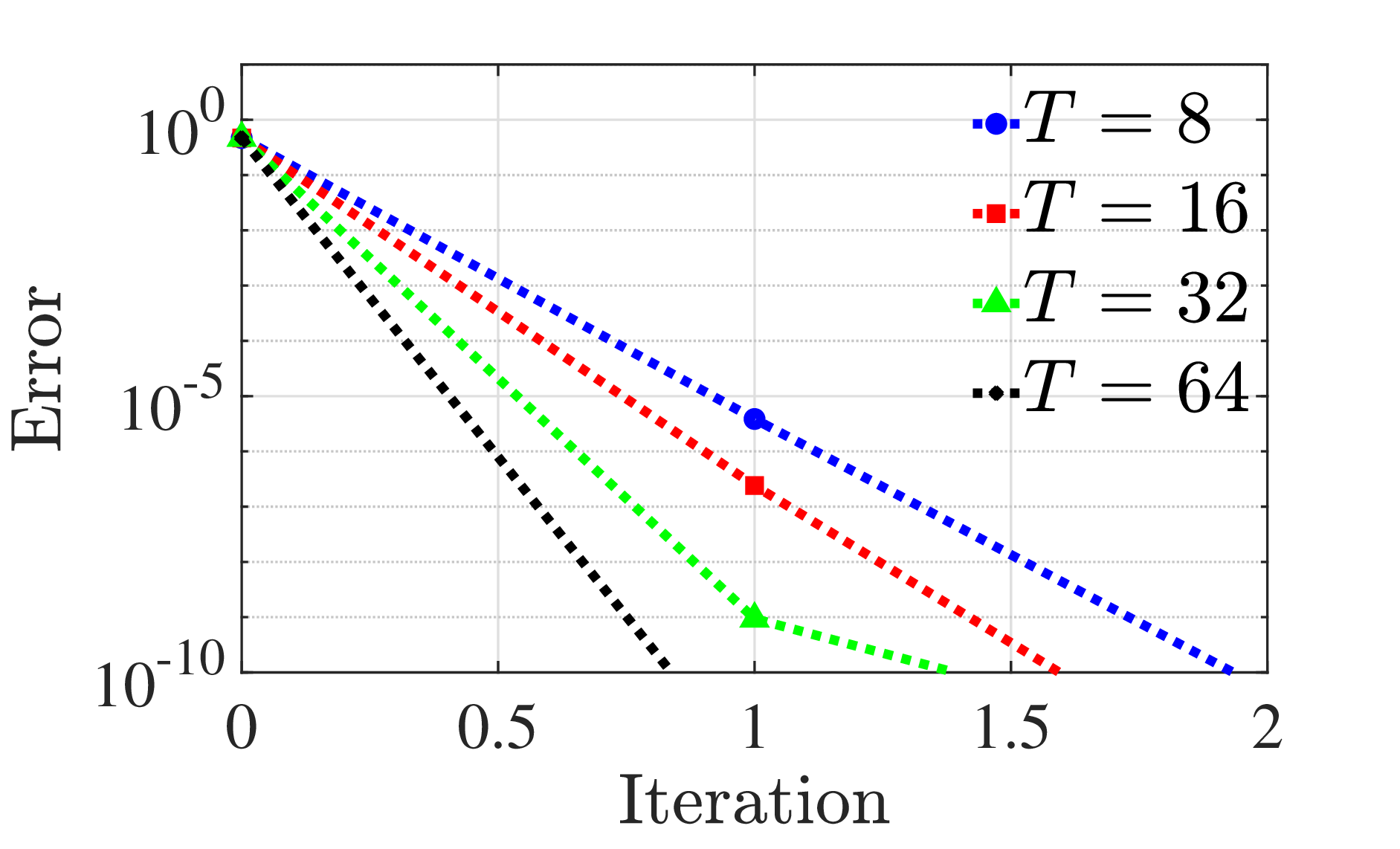} }}
    \subfloat{{\includegraphics[height=3.5cm,width=4cm]{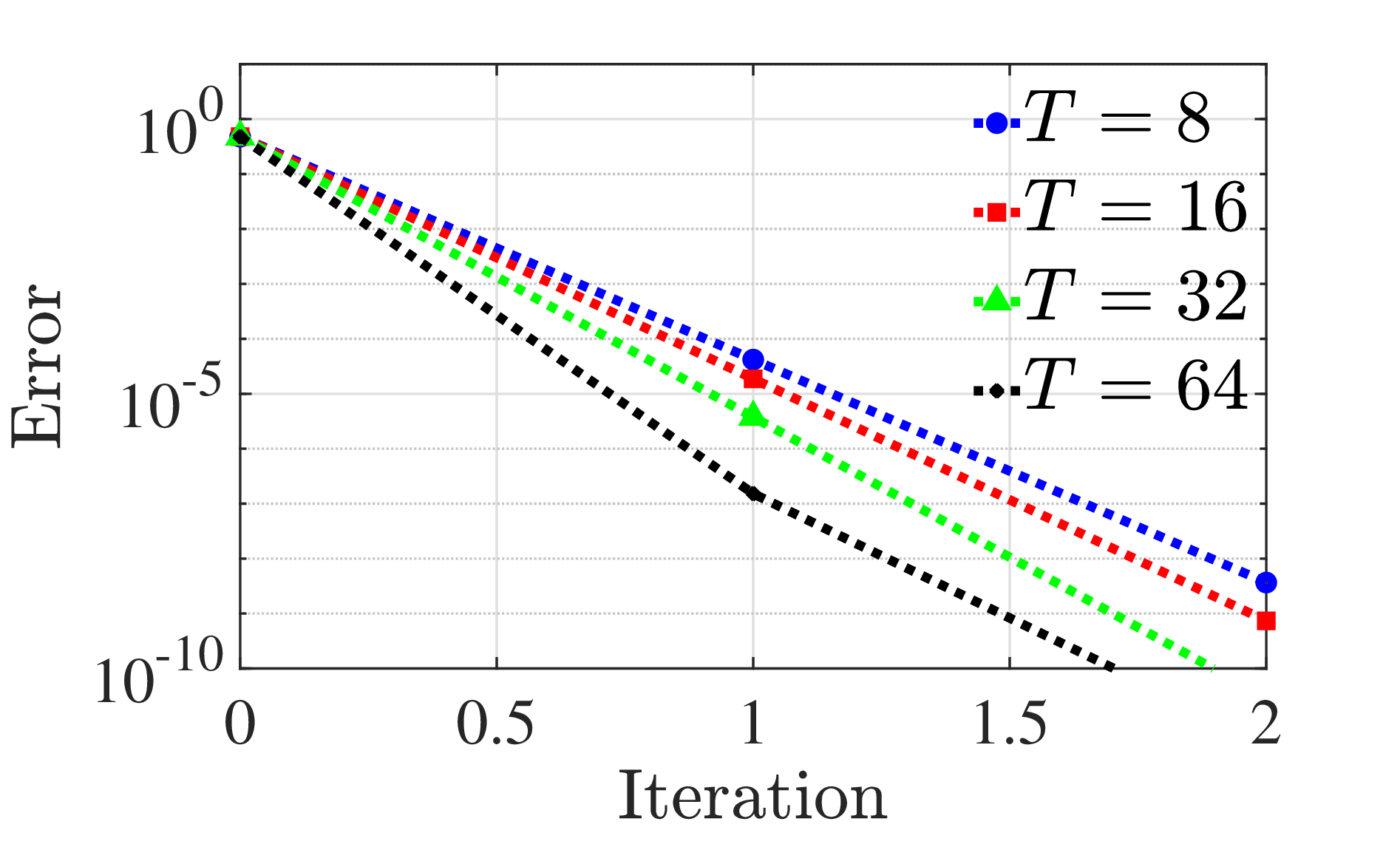} }}
    \caption{ First: Mesh refinement for $(a, c)=(0.01, 0.1)$ and $T=4$; Second: Mesh refinement for $(a, c)=(0.00001, 0.1)$ and $T=4$; Third: Long-time simulation for $(a, c)=(0.1, 0.1)$; Fourth:  Long-time simulation for $(a, c)=(0.00001, 0.1)$.}
    \label{err_diff_dx_dt_t}
\end{figure}
In the first two plots of Figure~\ref{err_diff_dx_dt_t}, we demonstrate the mesh independence of the proposed method. The initial grid is set to \( h = \Delta t = 1/32 \) and $\alpha=\alpha_{\text{opt}}$. The results show that the Exp-ParaDiag method remains insensitive to grid refinement across different diffusion coefficients and time window lengths.
In the last two plots of Figure~\ref{err_diff_dx_dt_t}, we display the convergence behavior of the Exp-ParaDiag method for large time windows and different diffusion coefficients, using \( \Delta t = h = 1/128 \) and \( \alpha = \alpha_{\text{opt}} \). The results show that the method converges rapidly. Observe that, in the case of $T=64$, we are solving for $1,048,576$ unknowns and converging in one iteration.
\begin{figure}[h!]
    \centering
    \subfloat{{\includegraphics[height=3.5cm,width=4cm]{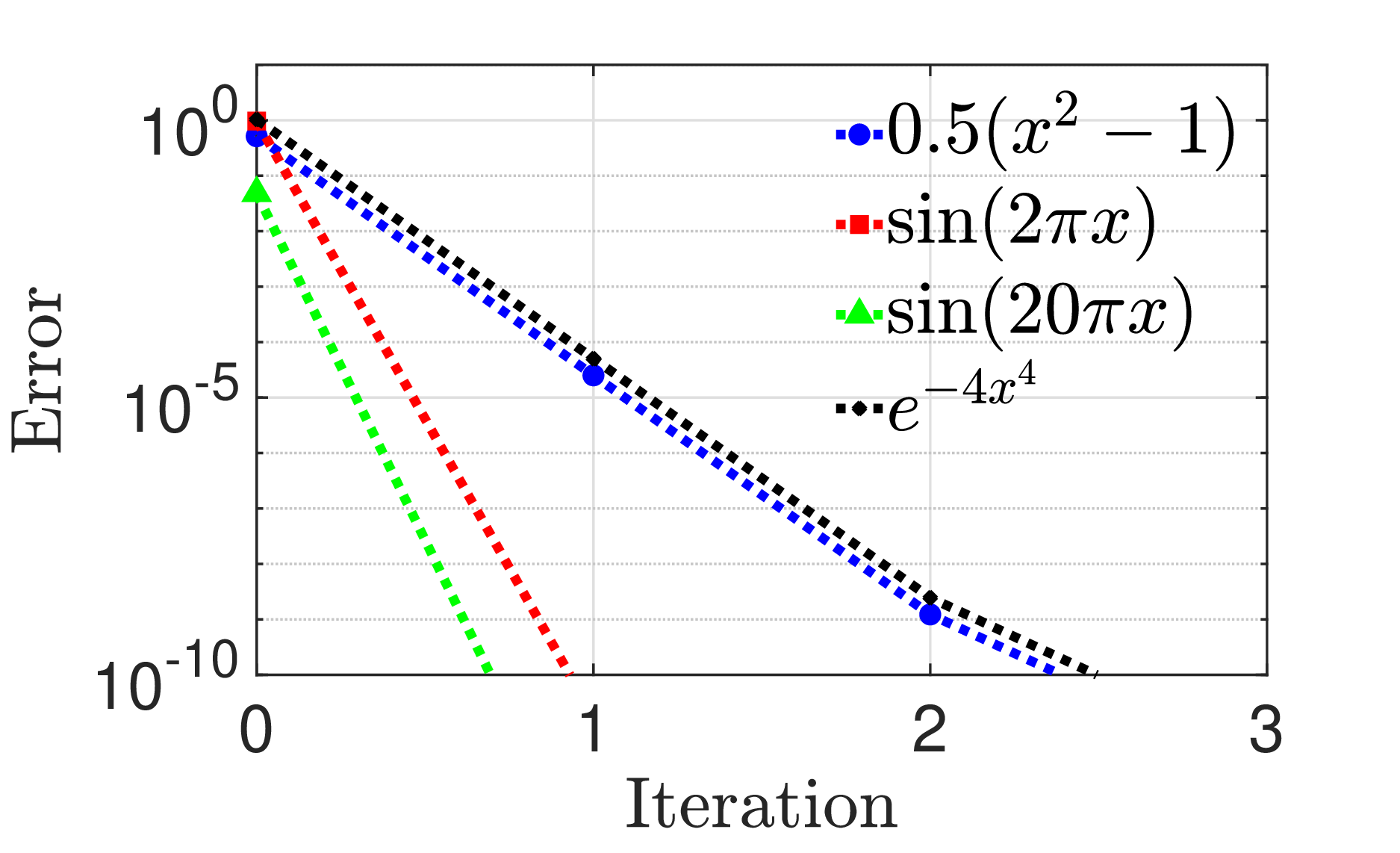} }}
    \subfloat{{\includegraphics[height=3.5cm,width=4cm]{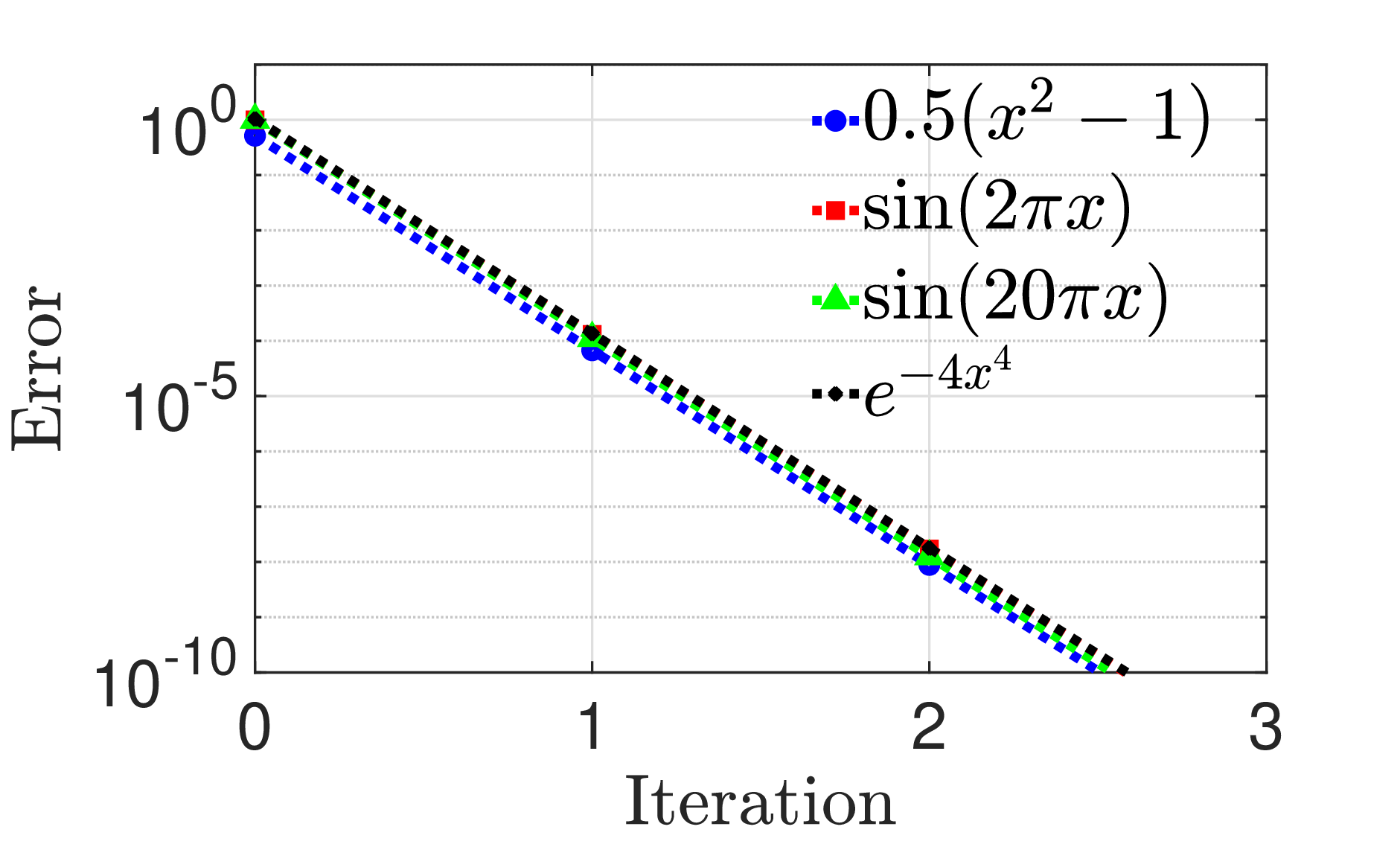} }}
    \subfloat{{\includegraphics[height=3.5cm,width=4cm]{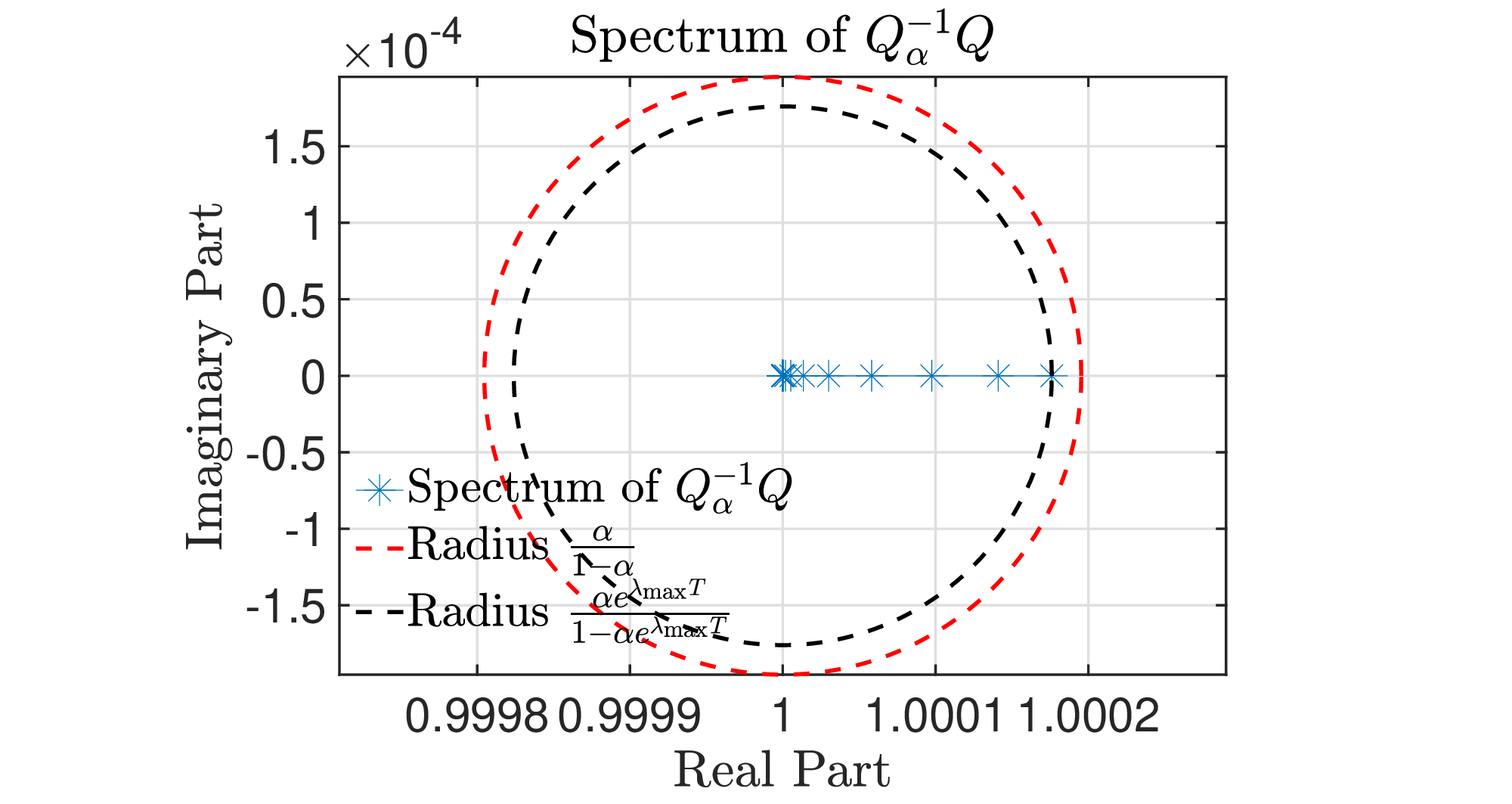} }}
   \subfloat{{\includegraphics[height=3.5cm,width=4cm]{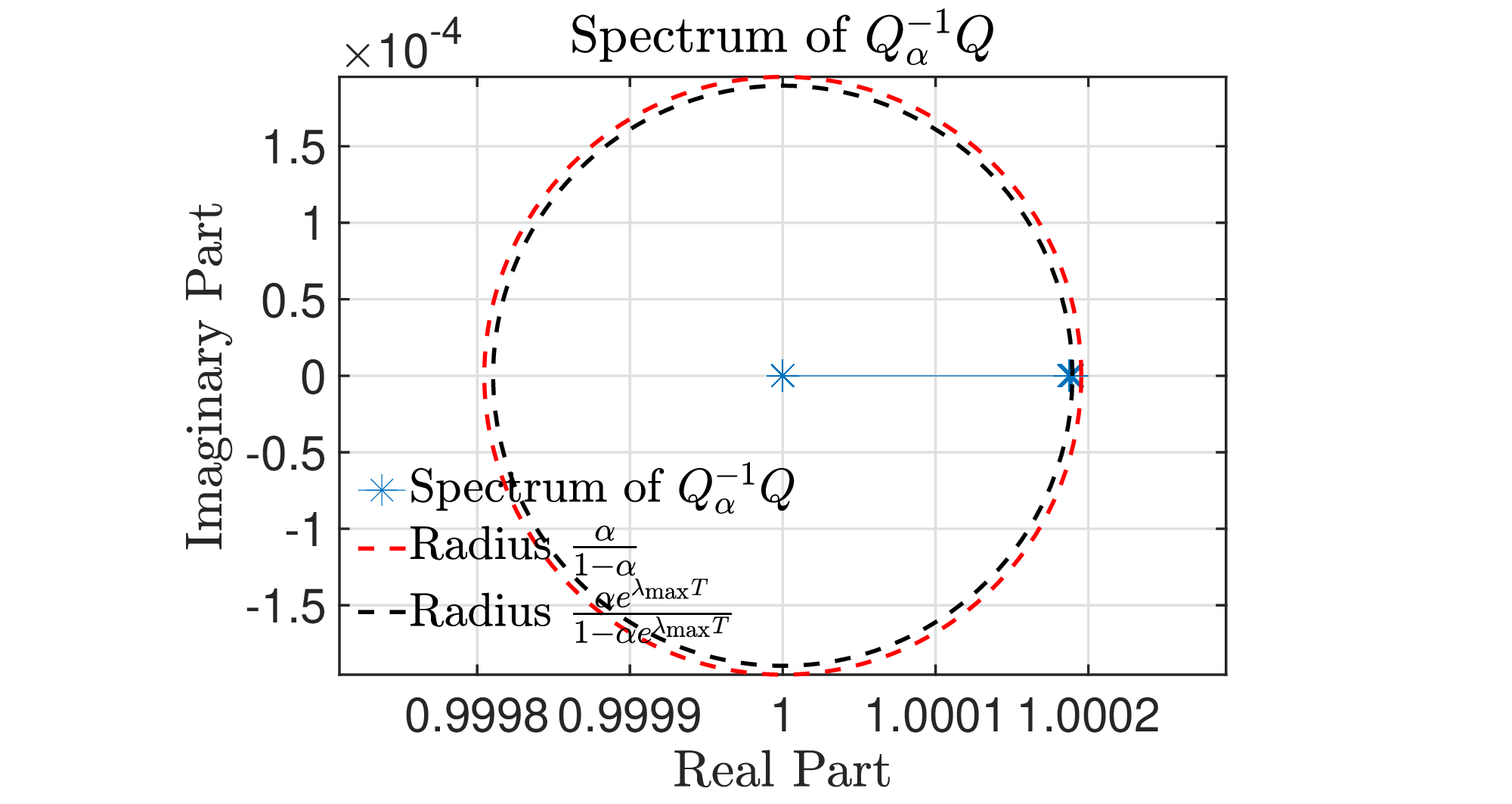} }}
    \caption{ First: convergence for $(a, c)=(0.1, 0.1)$; Second: convergence for $(a, c)=(0.00001, 0.1)$; Third: Spectrum of $Q_{\alpha}^{-1}Q$ and its bound for $(a, c)=(0.1, 0.1)$; Fourth: Spectrum of $Q_{\alpha}^{-1}Q$ and its bound for $(a, c)=(0.00001, 0.1)$.}
    \label{diff_initial_diff_bound}
\end{figure}
In the first two plots of Figure~\ref{diff_initial_diff_bound}, we present the convergence behavior for different initial guesses, including polynomial functions, low- and high-frequency sine modes, and exponential functions, using $\Delta t=h=1/128, \alpha=\alpha_{\text{opt}}$ and $T=4$. The method performs well for all choices of initial guess and across various diffusion coefficients. 
In the last two plots of Figure~\ref{diff_initial_diff_bound}, we display the spectrum of \( Q_{\alpha}^{-1}Q \) alongside the theoretical bounds provided by Theorem~\ref{pcond_estimate_first_order} and Lemma~\ref{gmres_centered_at1}, for various diffusion coefficients with \( N_x = 63 \), \( N_t = 30 \), \( \alpha = \alpha_{\text{opt}} \), and \( T = 0.3 \). The results show that the theoretical bounds closely capture the spectrum.
\begin{figure}[h!]
    \centering
    \subfloat{{\includegraphics[height=3.5cm,width=4cm]{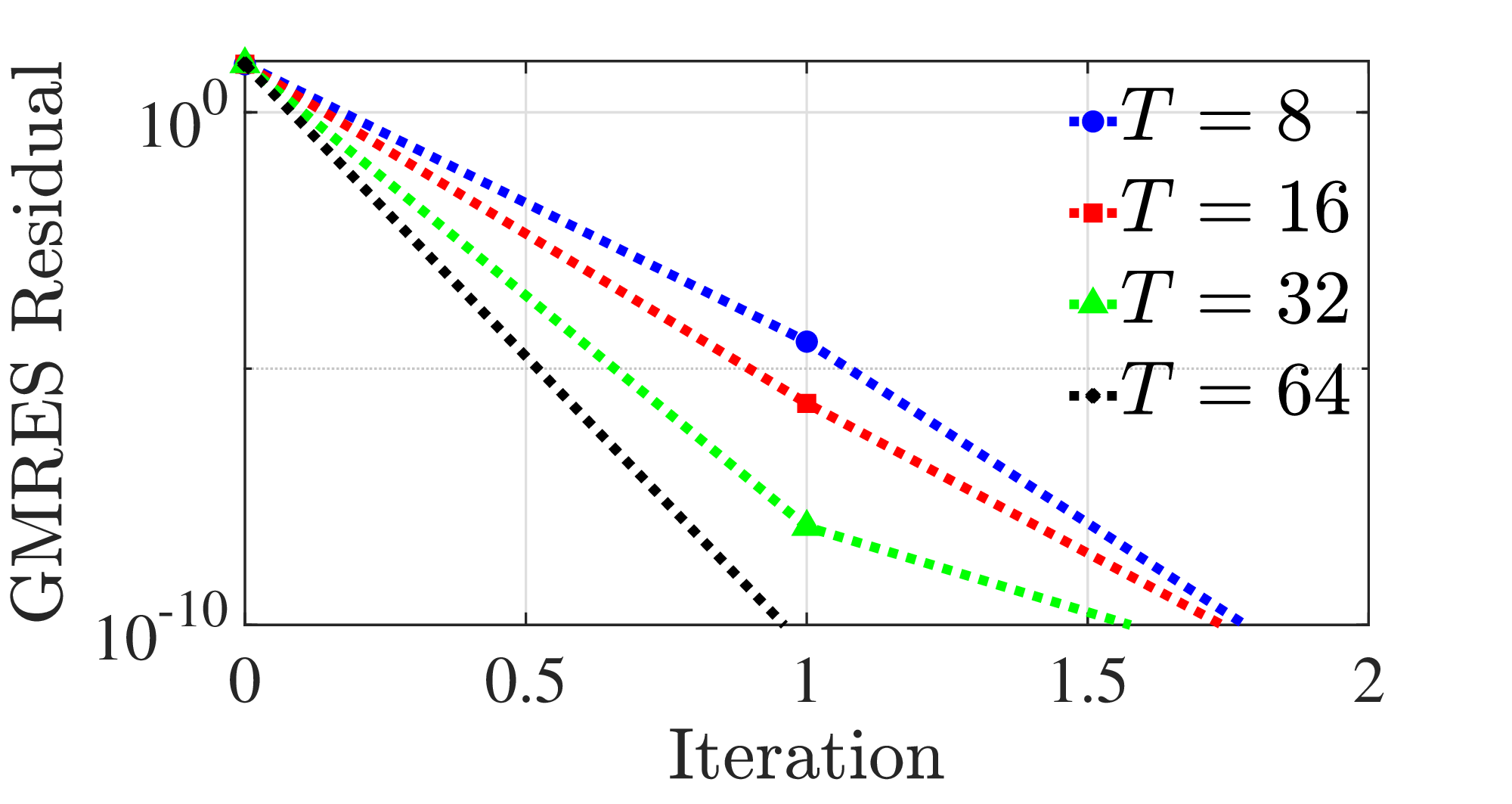} }}
    \subfloat{{\includegraphics[height=3.5cm,width=4cm]{figures/gmres_1d_diff_T_a1.eps} }}
    \subfloat{{\includegraphics[height=3.5cm,width=4cm]{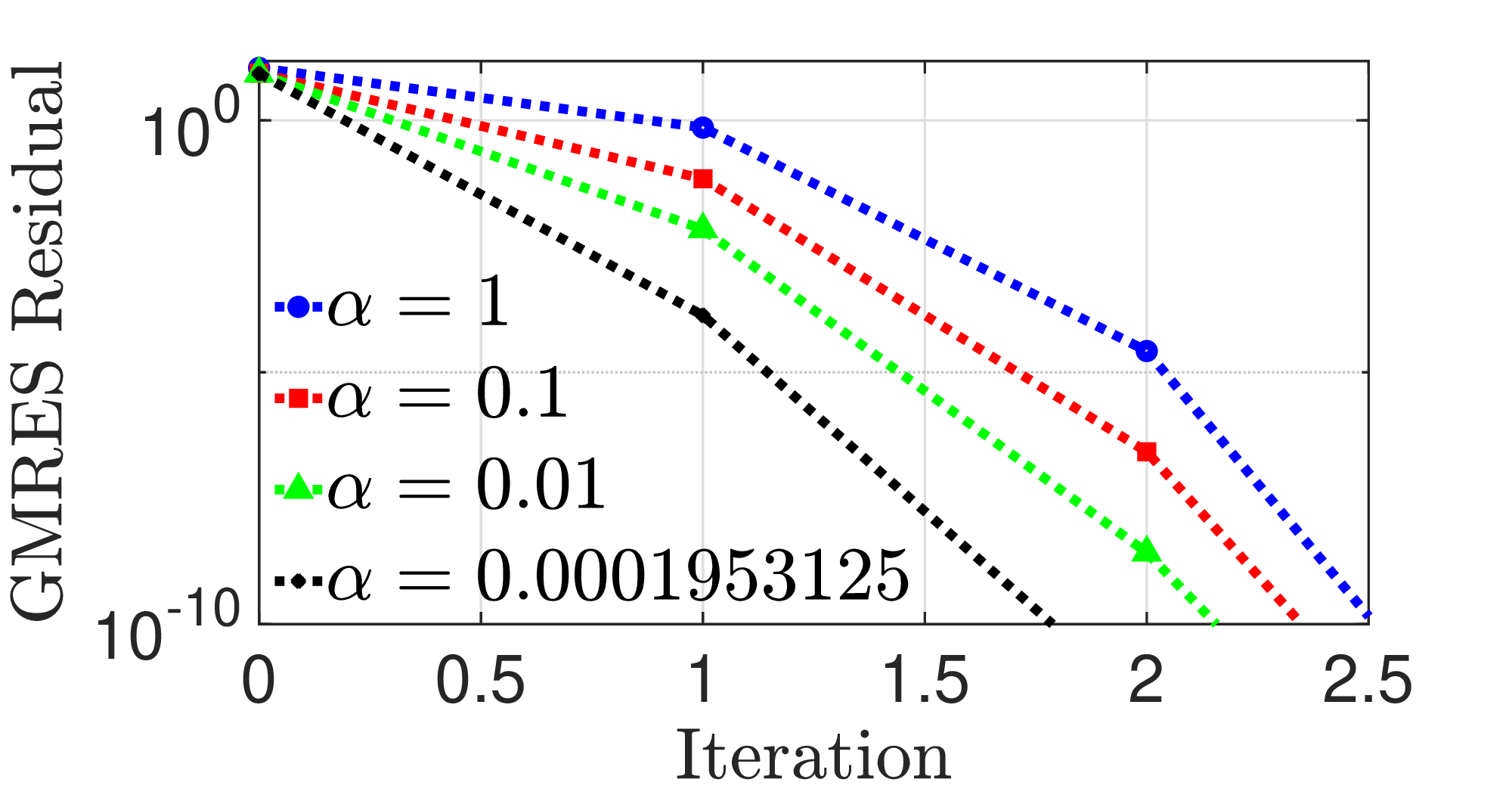} }}
   \subfloat{{\includegraphics[height=3.5cm,width=4cm]{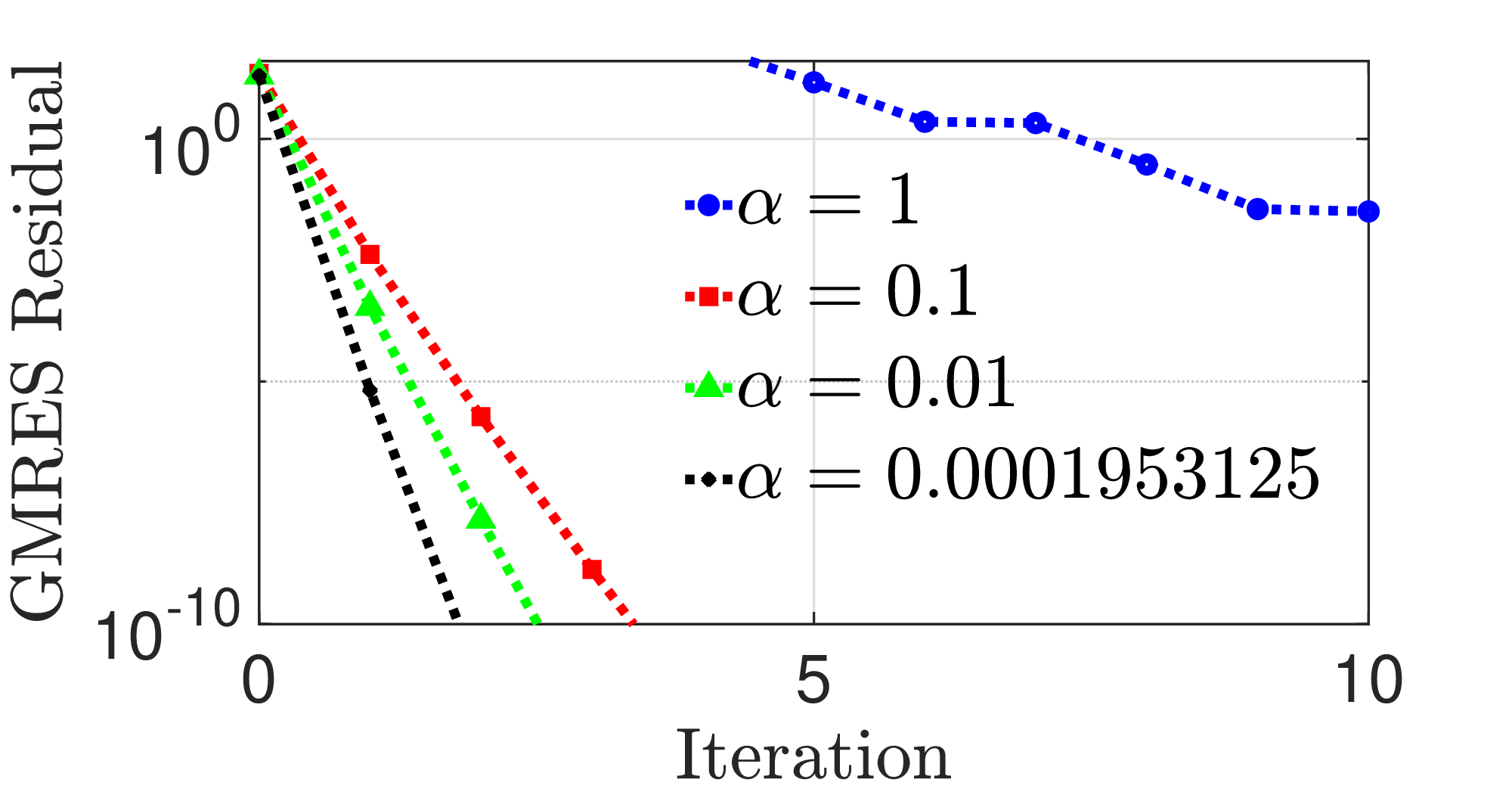} }}
    \caption{ First: convergence for various $T$ with $(a, c)=(0.1, 0.1)$; Second: convergence for various $T$ with $(a, c)=(0.00001, 0.1)$; Third: convergence for various $\alpha$ with $(a, c)=(0.1, 0.1)$; Fourth: convergence for various $\alpha$ with $(a, c)=(0.00001, 0)$.}
    \label{gmres_1d_diff_t}
\end{figure}
In the first two plots of Figure~\ref{gmres_1d_diff_t}, we plot the convergence of the preconditioned GMRES for different time window lengths and varying values of \( a \), with \( h = 1/128 \), \( \Delta t = 0.01 \), and \( \alpha = \alpha_{\text{opt}} \). The results show rapid convergence, with the method effectively acting as a direct solver.
In the last two plots of Figure~\ref{gmres_1d_diff_t}, we show the convergence of the preconditioned GMRES method for different values of \( \alpha \), with varying diffusion coefficient \( a \) and reaction term \( c \), using \( h = 1/128 \), \( \Delta t = 0.01 \), and \( T = 4 \). It is evident that \( \alpha = \alpha_{\text{opt}} \) provides the most effective choice for the free parameter \( \alpha \). In contrast, setting \( \alpha = 1 \) in the purely diffusive case leads to significantly slower convergence.
 \begin{figure}[h!]
    \centering
    \subfloat{{\includegraphics[height=3.5cm,width=4cm]{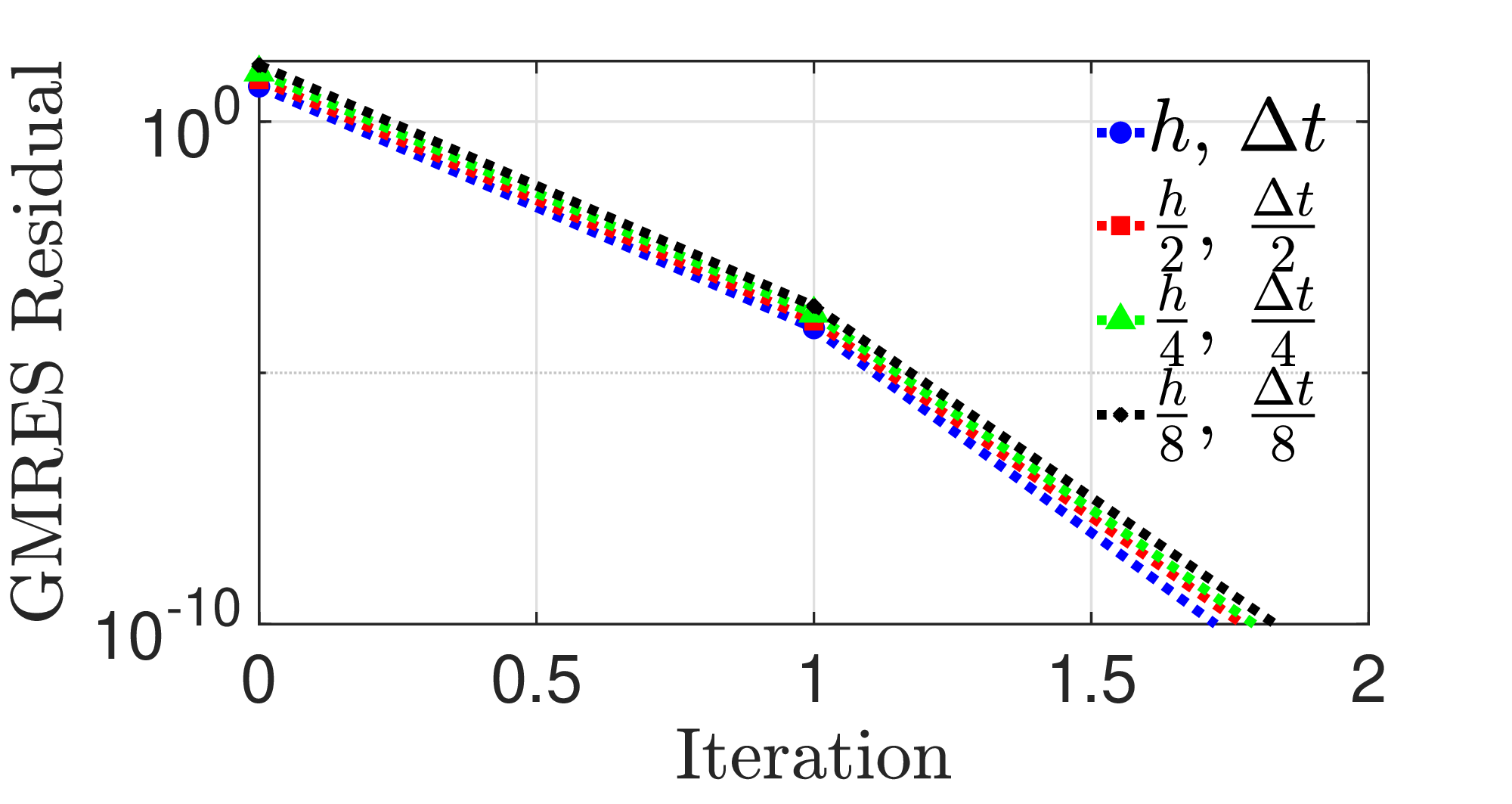} }}
    \subfloat{{\includegraphics[height=3.5cm,width=4cm]{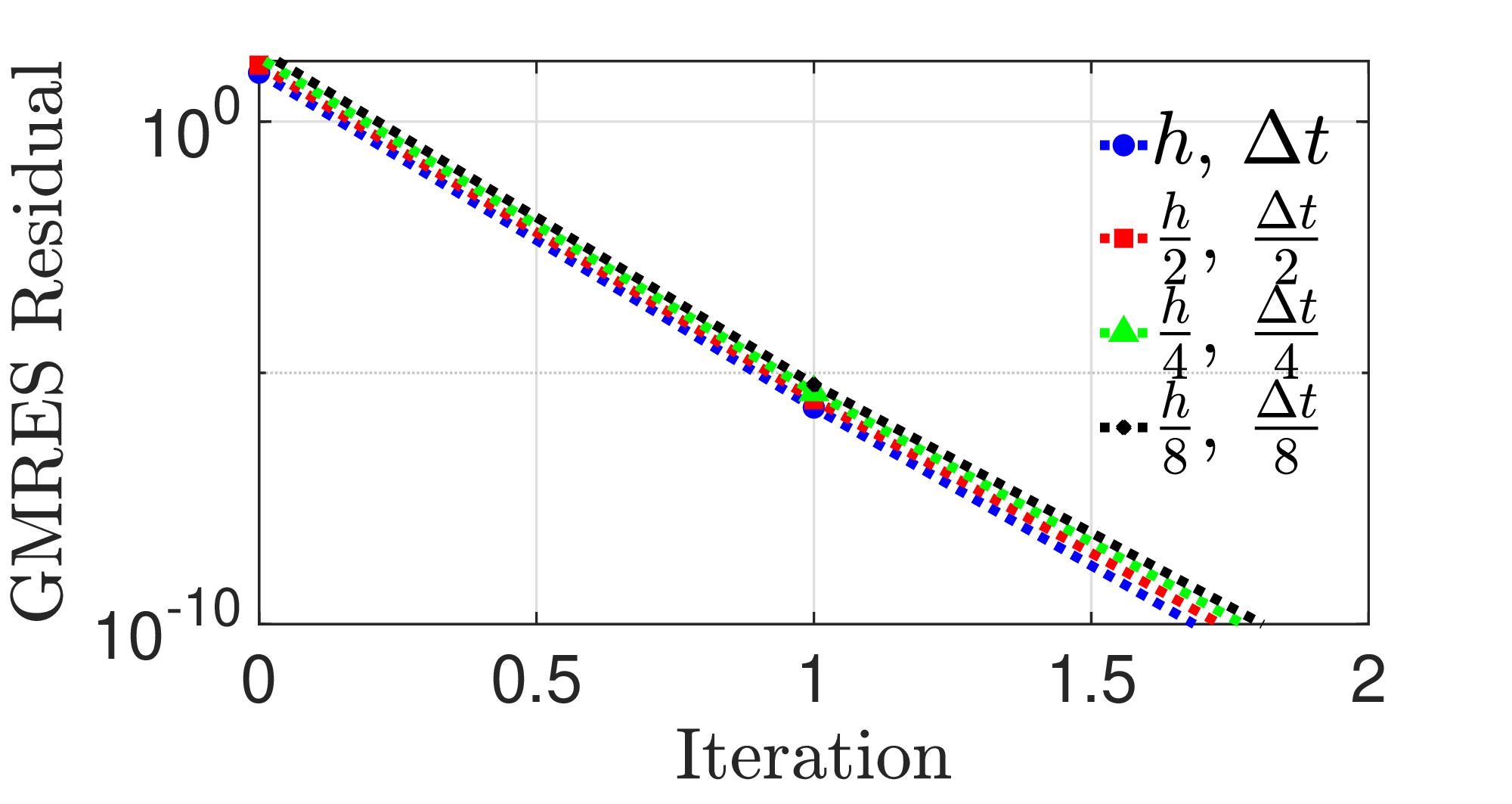} }}
    \subfloat{{\includegraphics[height=3.5cm,width=4cm]{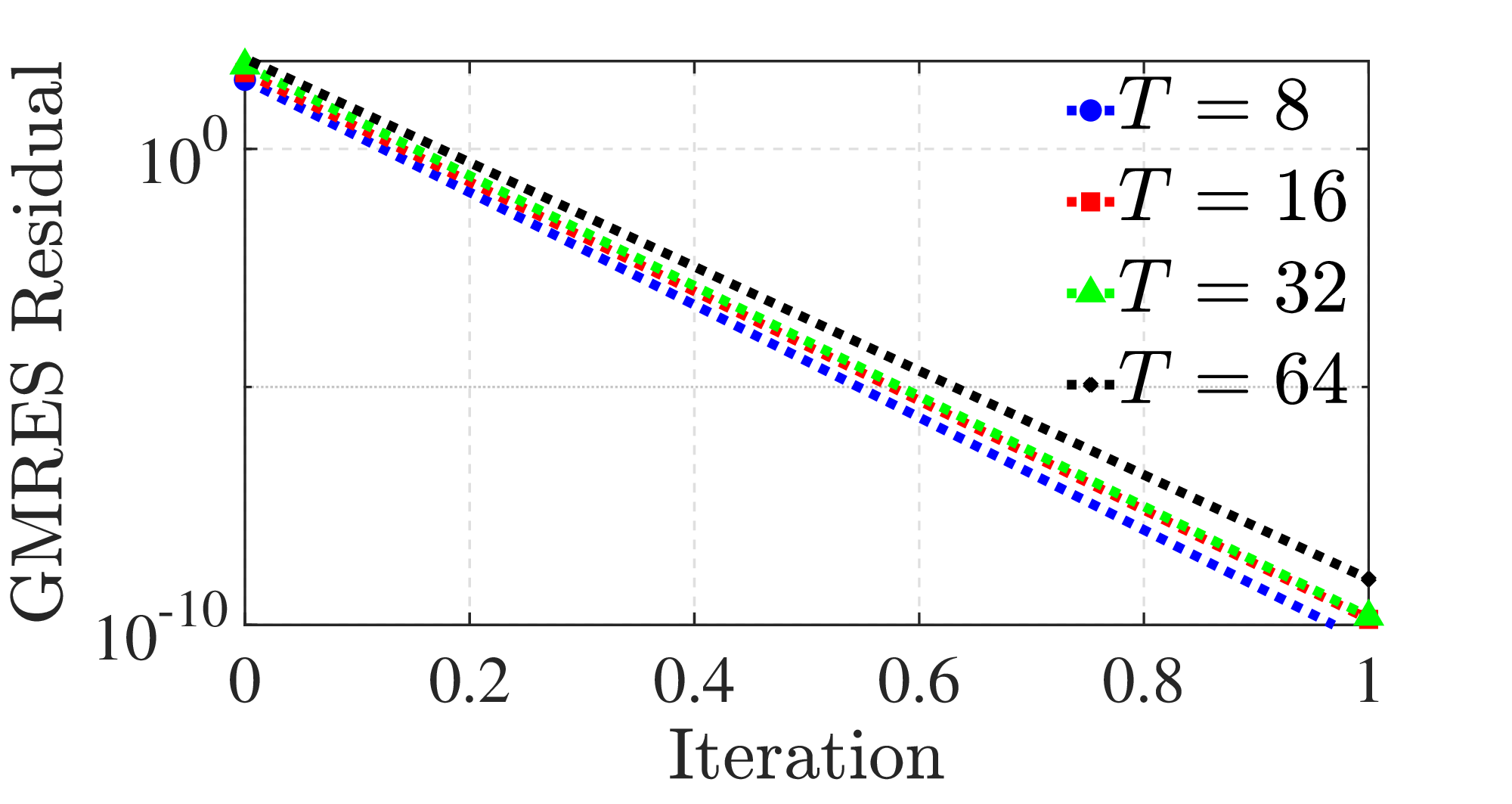} }}
   \subfloat{{\includegraphics[height=3.5cm,width=4cm]{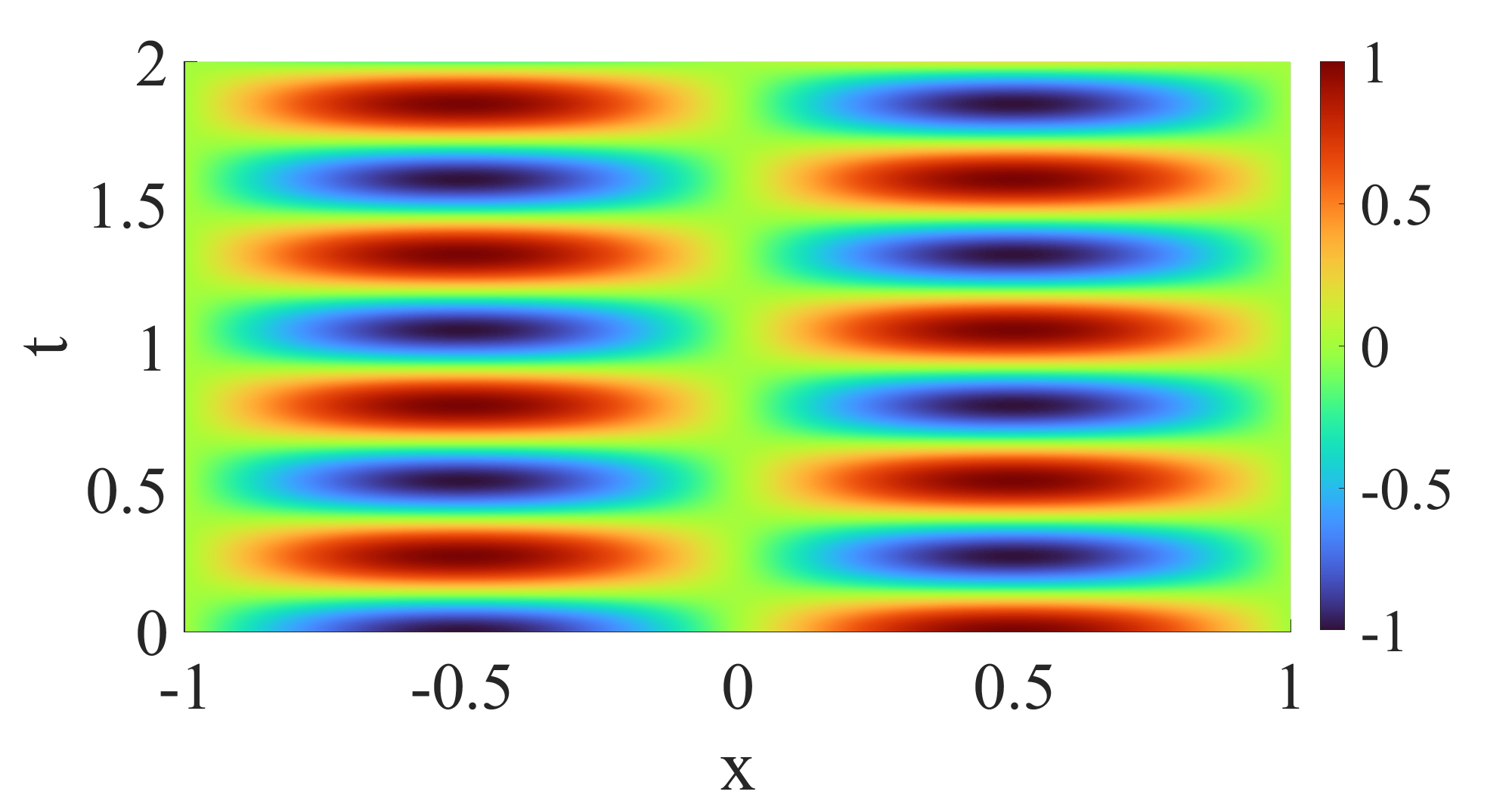} }}
    \caption{ First: mesh independence with $(a, c)=(0.1, 0.1)$; Second: mesh independence with $(a, c)=(0.00001, 0.1)$; Third: convergence for different $T$; Fourth: real part of the solution to the SE.}
    \label{gmres_1d_diff_dxdt}
\end{figure}
The first and second plots of Figure~\ref{gmres_1d_diff_dxdt} illustrate the mesh-independence of the preconditioned GMRES method for different values of \( a \), with \( T = 4 \) and \( \alpha = \alpha_{\text{opt}} \). Starting from a coarse mesh with \( (h, \Delta t) = (1/32, 1/32) \), we observe that the number of iterations remains unchanged under successive mesh refinements, confirming the robustness of the method with respect to discretization parameters.

\noindent Next, we consider the linear Schrödinger equation (SE) by setting \( a = \mathrm{i} \) and \( c = 2\mathrm{i} \) in \eqref{model_problem_linear}, where \( \mathrm{i} = \sqrt{-1} \). We study the convergence behavior of the proposed method in this purely oscillatory, complex-valued setting with the initial condition \( u_0 = \sin(\pi x) \). In the third plot of Figure~\ref{gmres_1d_diff_dxdt}, we present the convergence of the preconditioned GMRES applied to the SE for different time window lengths, using \( h = 1/128 \), \( \Delta t = 0.01 \), and \( \alpha = \alpha_{\text{opt}} \). The results show that the proposed method converges in just one iteration. 
In the rightmost plot, we show the real part of the solution to the SE for \( h = 1/128 \), \( \Delta t = 0.01 \), \( T = 2 \), and \( \alpha = \alpha_{\text{opt}} \).

\subsubsection{Experiment for $\mathcal{O}((\Delta t)^2)$ Scheme}
In this section, we investigate the convergence behavior of the Exp-ParaDiag method constructed using the BDF2 scheme, both as a fixed-point iteration and as a preconditioner for GMRES. The leftmost plot in Figure~\ref{bdf2_gmres} displays error curves for various values of the free parameter $\alpha$ for $h = \frac{1}{128}$, $\Delta t = 0.01$, and $T = 2$, including the optimal choice $\alpha = \alpha_{\text{opt}}$. It is evident that smaller values of $\alpha$ lead to improved convergence of the Exp-ParaDiag method when used as a fixed-point iteration.
In the second plot, we show the convergence of the Exp-ParaDiag method as a fixed-point iteration for different time window sizes with $h = \frac{1}{128}$, $\Delta t = 0.01$, and $\alpha = \alpha_{\text{opt}}$. We observe that the convergence is independent of the time window size.
\begin{figure}[h!]
    \centering
    \subfloat{{\includegraphics[height=3.5cm,width=4cm]{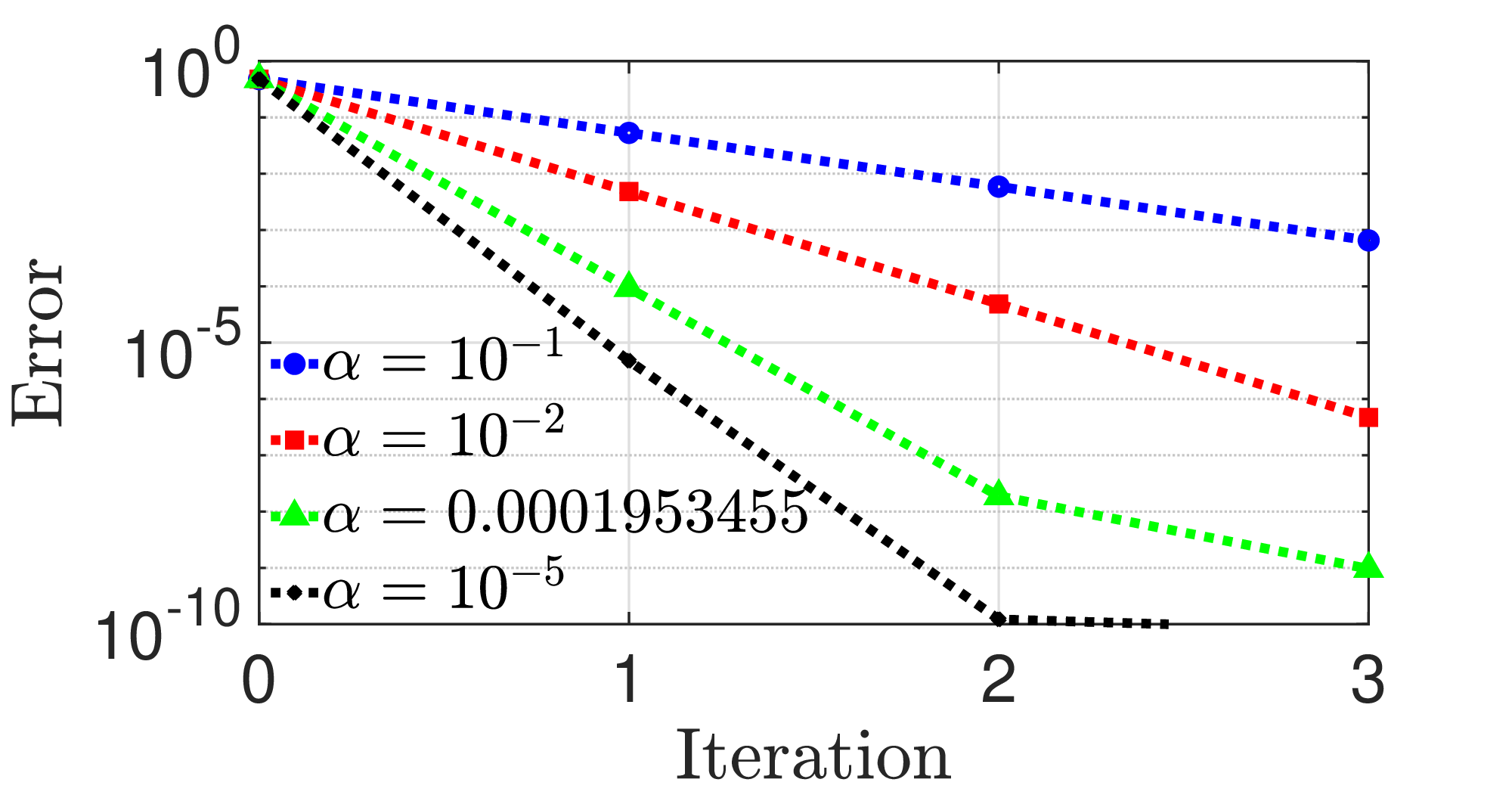} }}
    \subfloat{{\includegraphics[height=3.5cm,width=4cm]{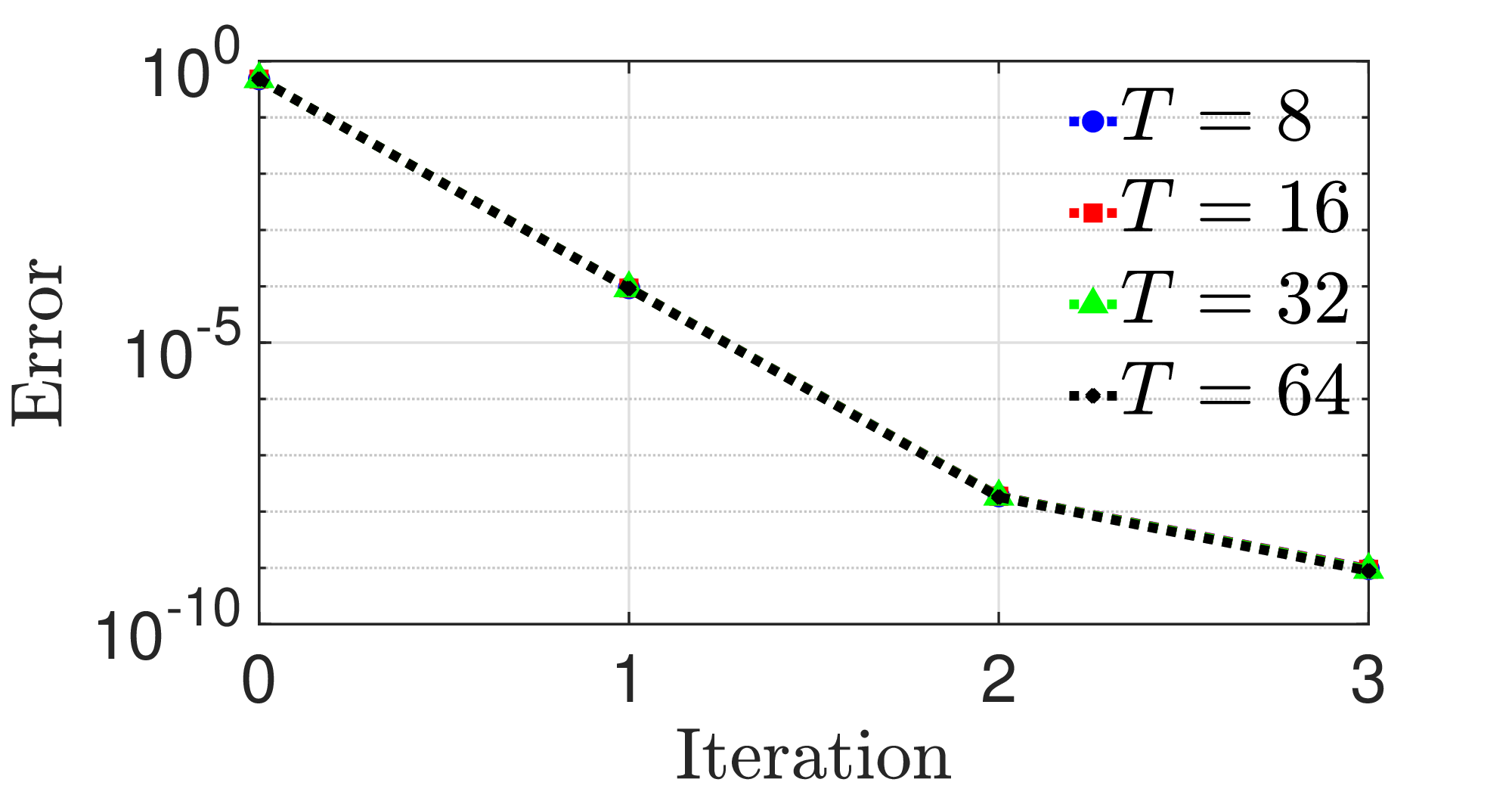} }}
    \subfloat{{\includegraphics[height=3.5cm,width=4cm]{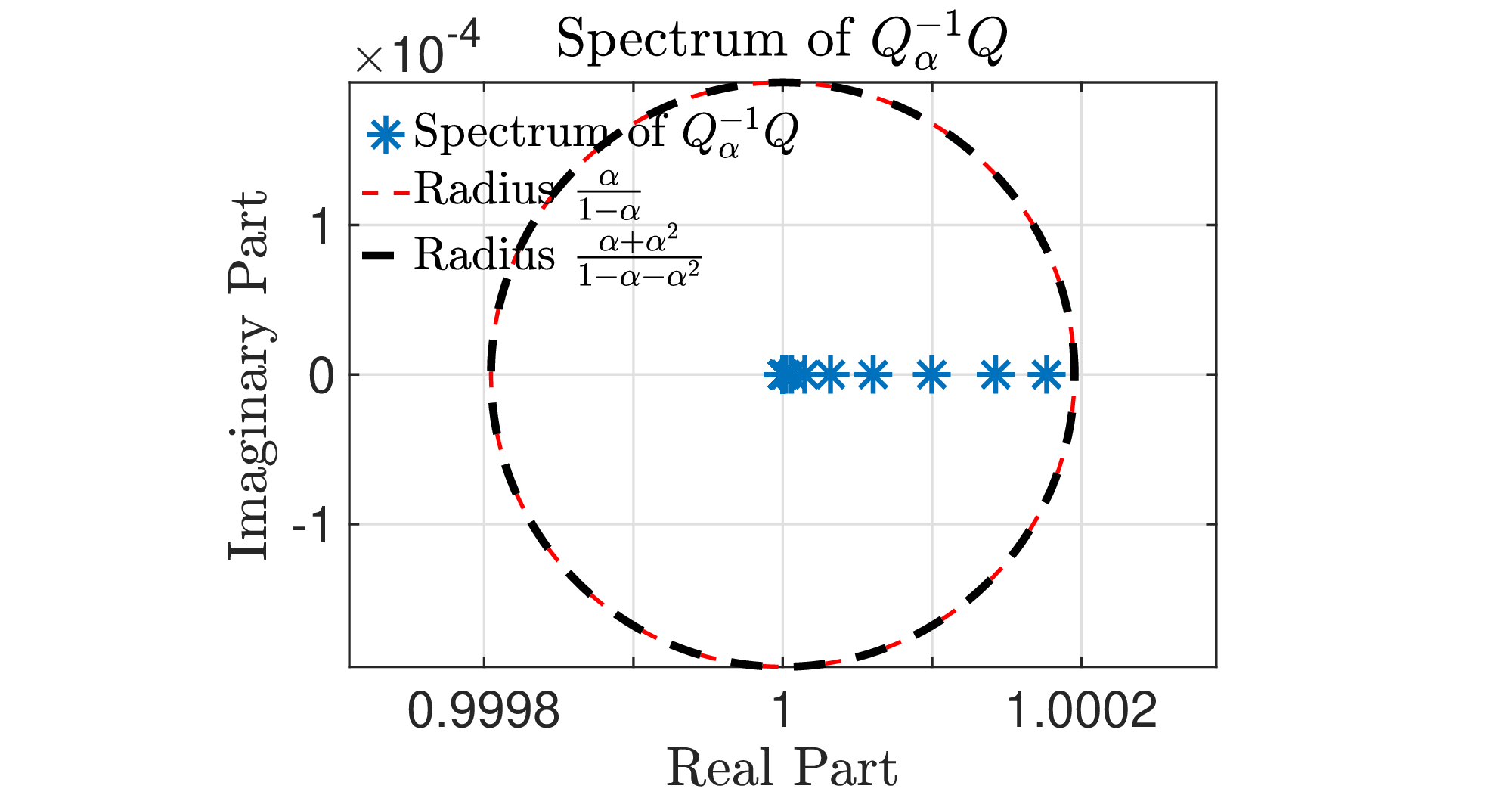} }}
    \subfloat{{\includegraphics[height=3.5cm,width=4cm]{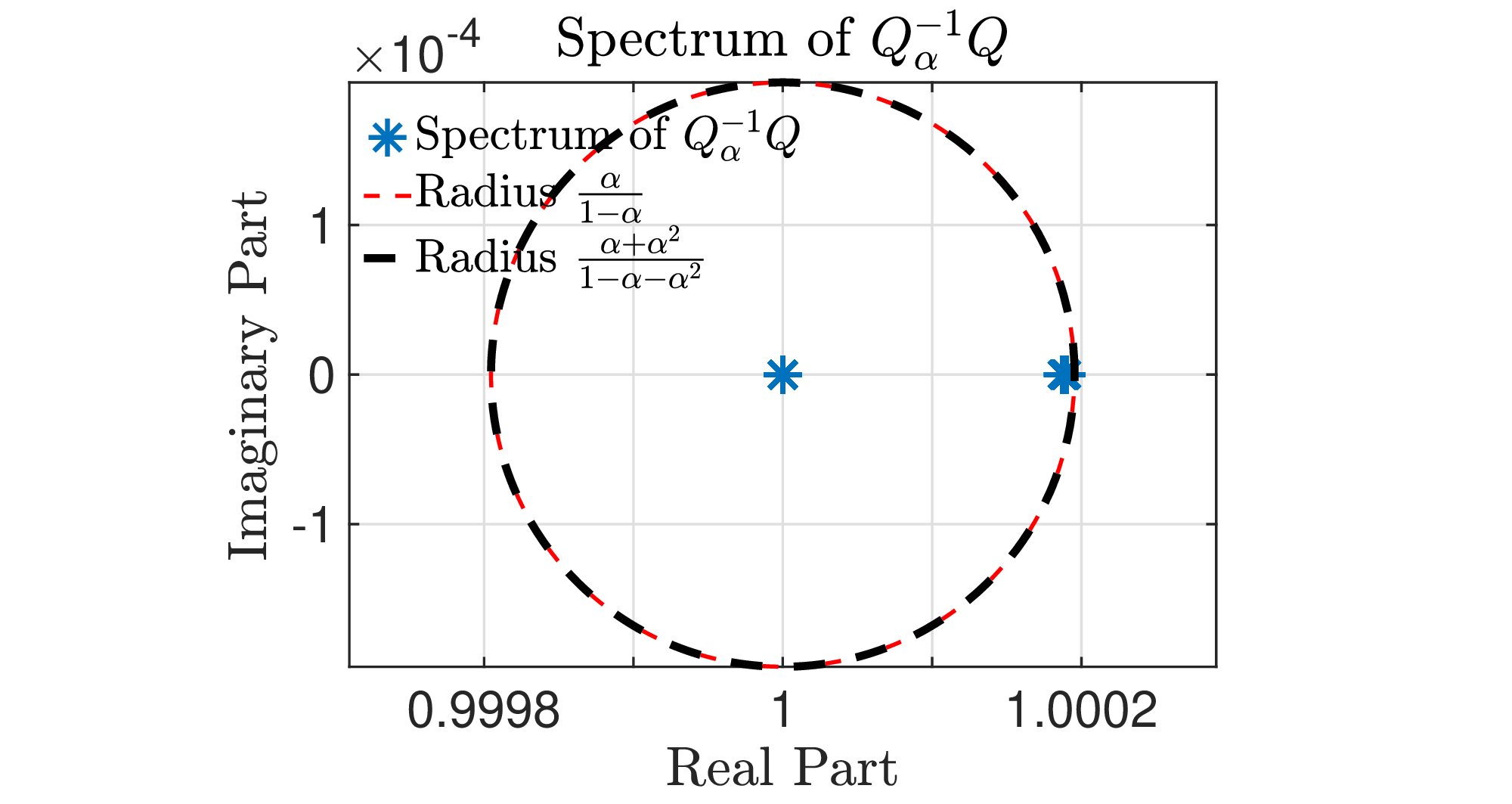} }}
    \caption{ First: convergence for different $\alpha$ with $(a, c)=(0.00001, 0)$; Second: convergence for different $T$ with $(a, c)=(0.00001, 0)$; Third: Spectrum of $Q_{\alpha}^{-1}Q$ and its bound for $(a, c)=(0.1, 0.1)$; Fourth: Spectrum of $Q_{\alpha}^{-1}Q$ and its bound for $(a, c)=(0.00001, 0.1)$.}
    \label{bdf2_gmres}
\end{figure}
In the third and fourth plots of Figure~\ref{bdf2_gmres}, we display the spectrum of the preconditioned system, validating Theorem~\ref{thm_bdf2} and Lemma~\ref{gmres_centered_at1_2nd_order}.
In the first two plots of Figure~\ref{bdf2_gmres_diff_t_dxdt}, we show the convergence of preconditioned GMRES with $h = \frac{1}{128}$, $\Delta t = 0.01$, and $\alpha = \alpha_{\text{opt}}$. The results indicate that the convergence is very fast and independent of the time window size.
\begin{figure}[h!]
    \centering
    \subfloat{{\includegraphics[height=3.5cm,width=4cm]{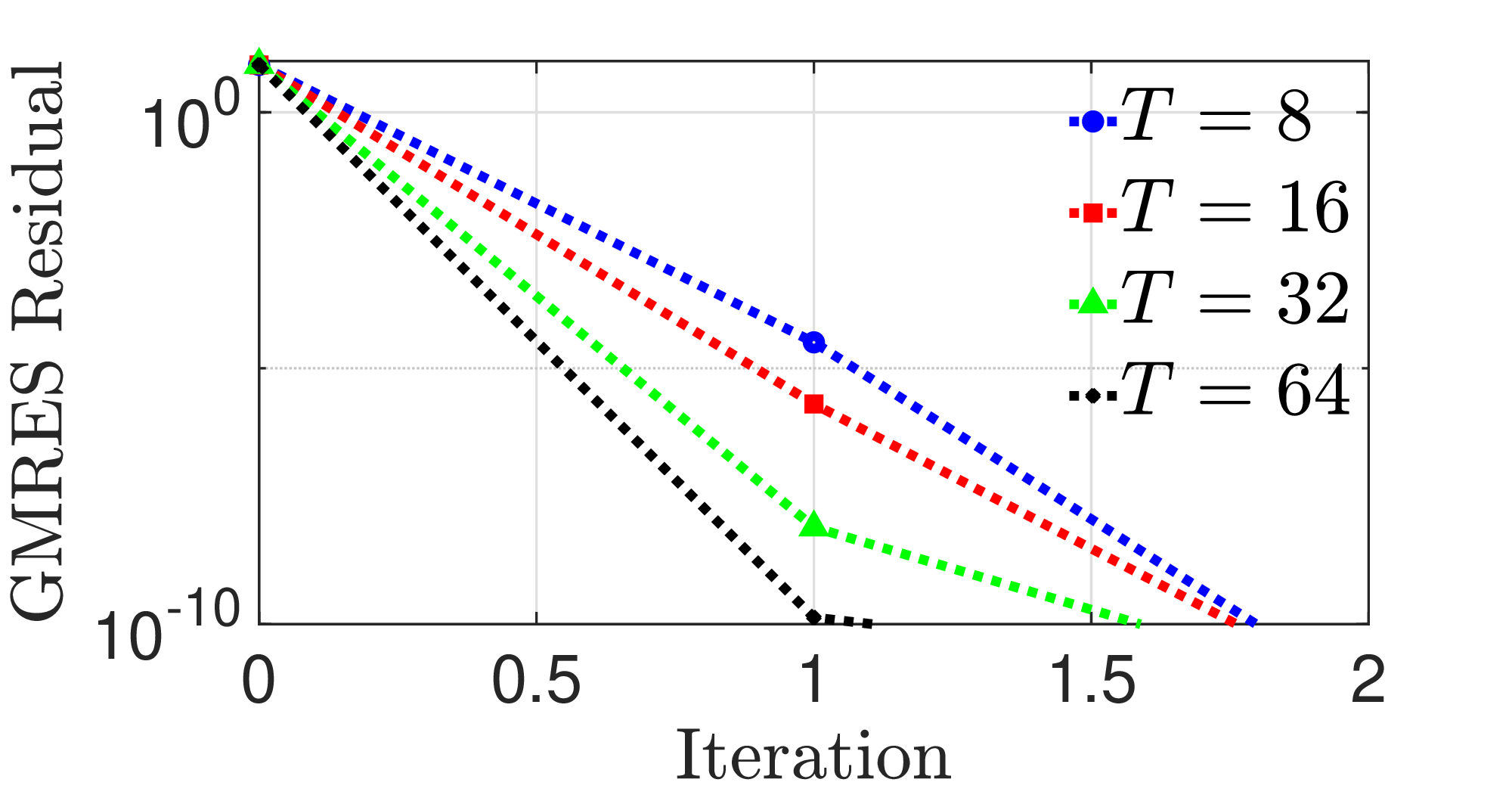} }}
    \subfloat{{\includegraphics[height=3.5cm,width=4cm]{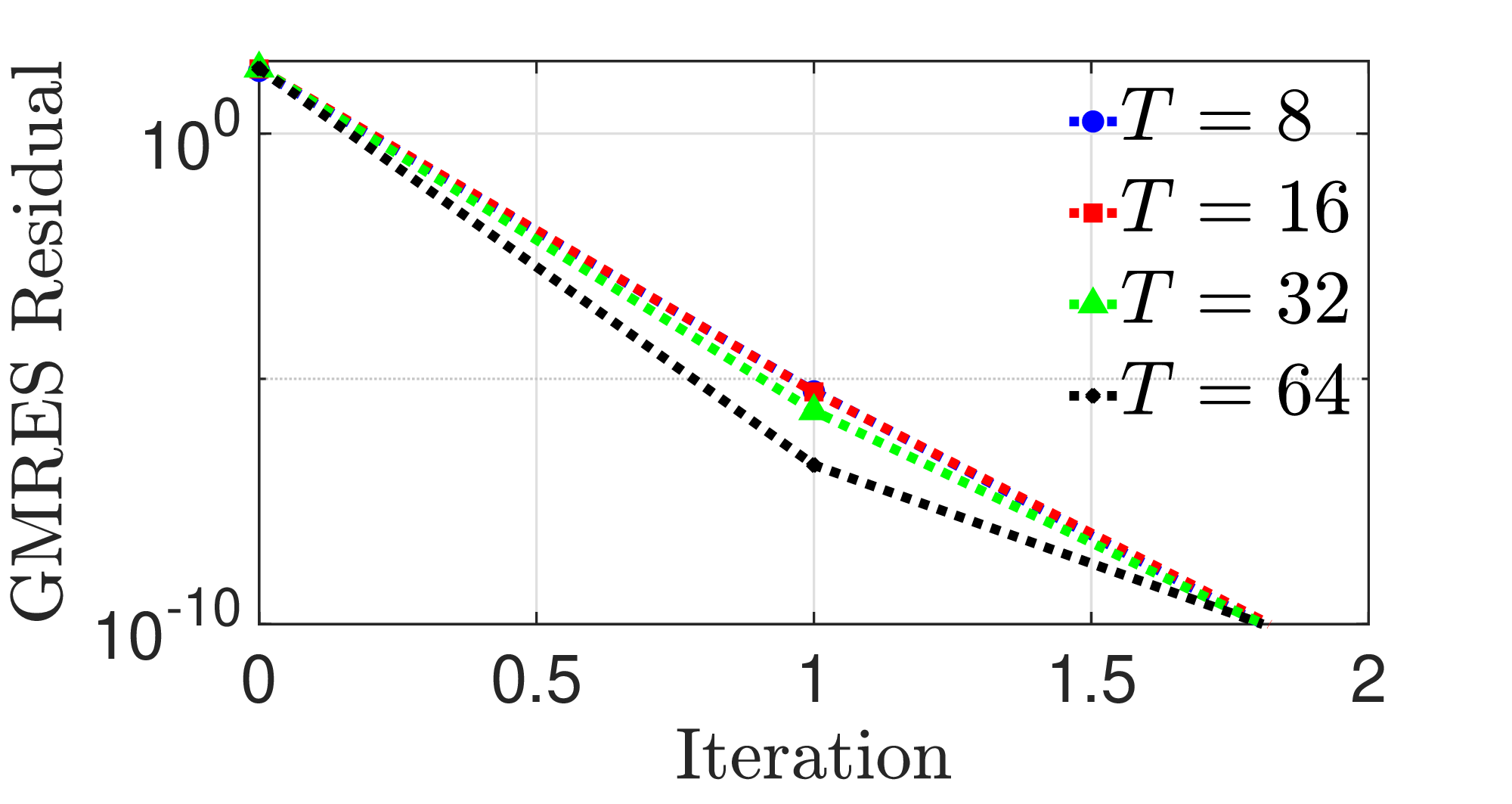} }}
    \subfloat{{\includegraphics[height=3.5cm,width=4cm]{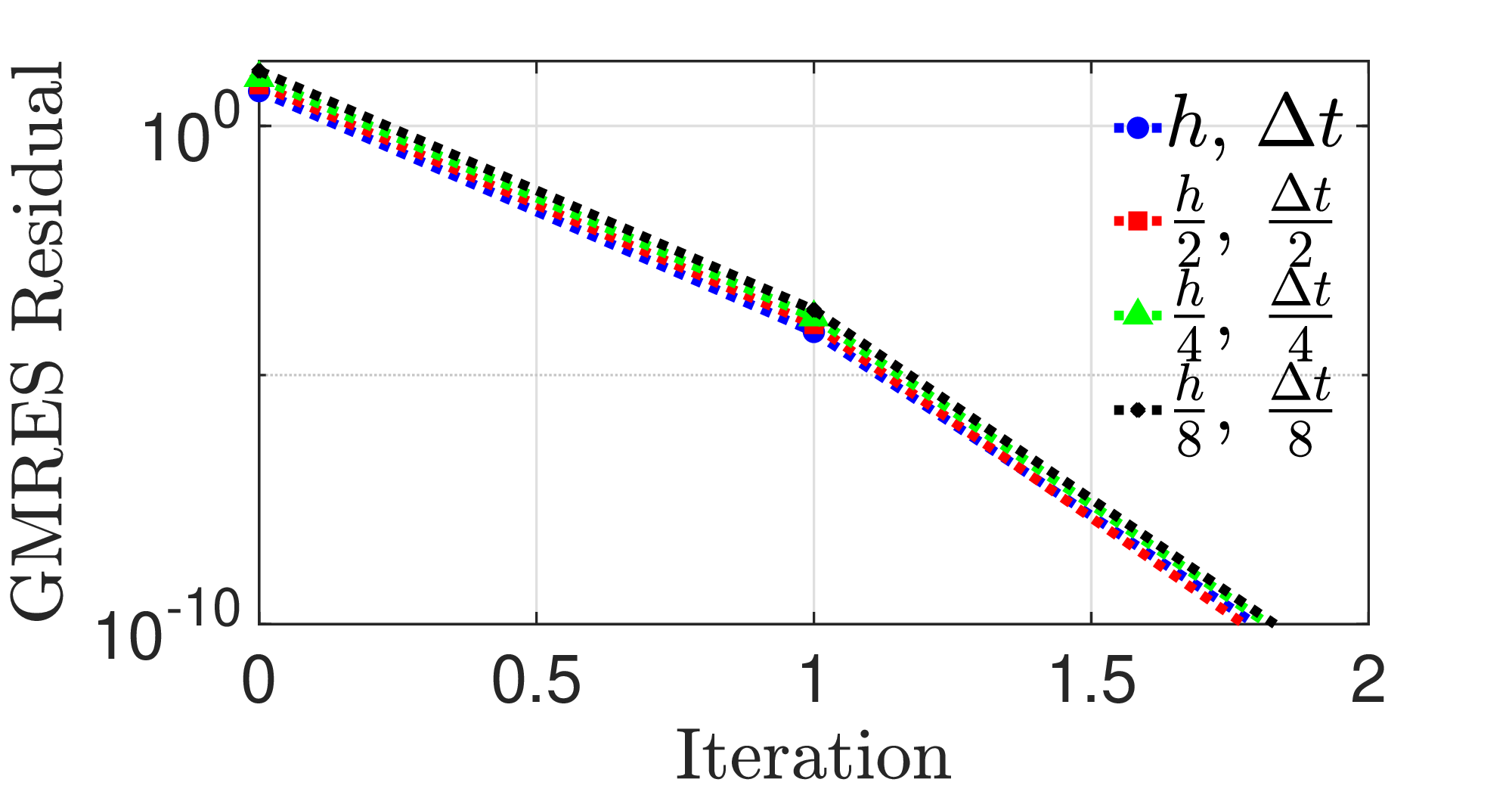} }}
    \subfloat{{\includegraphics[height=3.5cm,width=4cm]{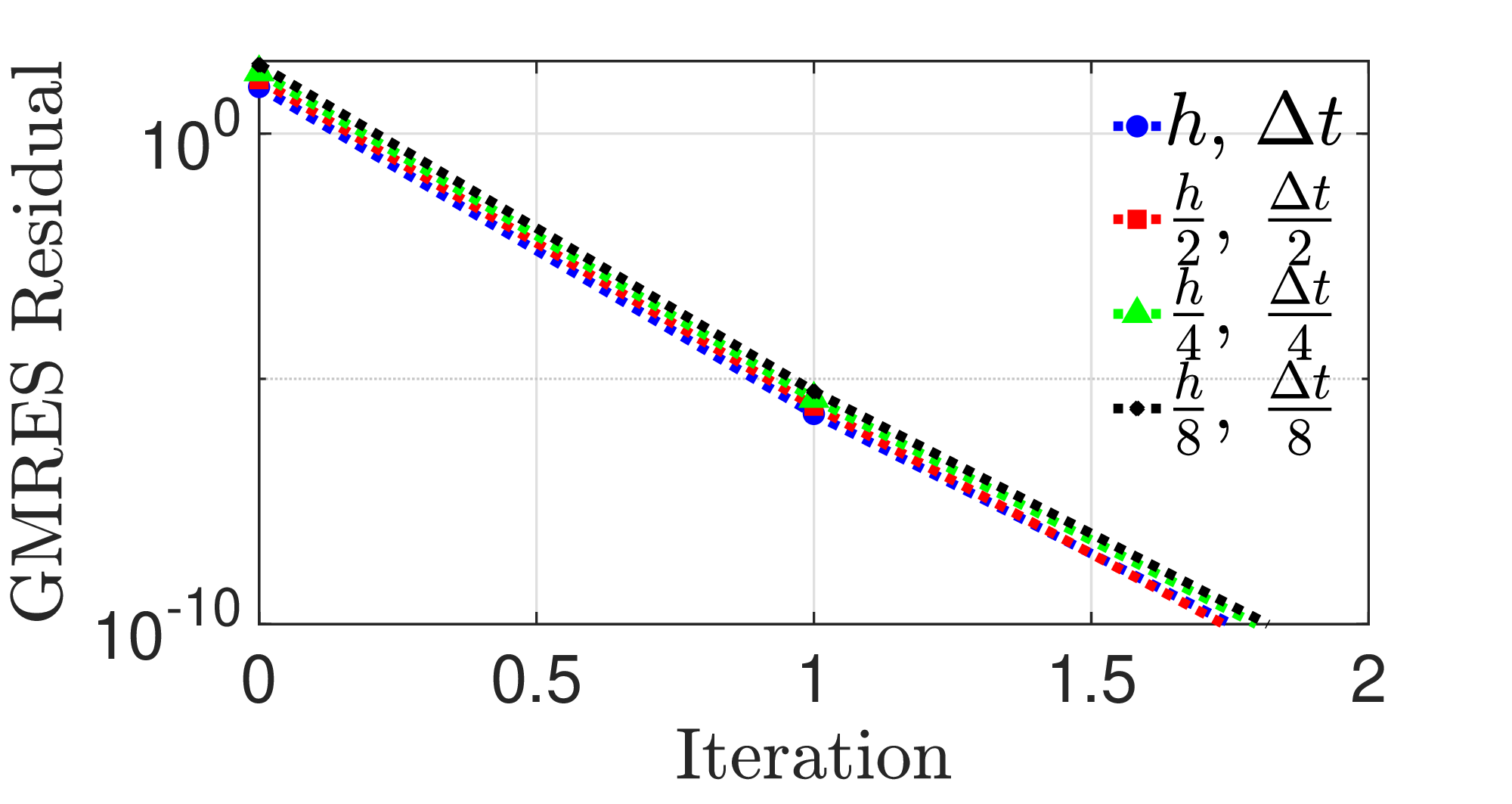} }}
    \caption{ First: convergence for different $T$ with $(a, c)=(0.1, 0.1)$; Second: convergence for different $T$ with $(a, c)=(0.00001, 0.1)$; Third: mesh independence for $(a, c)=(0.1, 0.1)$; Fourth: mesh independence for $(a, c)=(0.00001, 0.1)$.}
    \label{bdf2_gmres_diff_t_dxdt}
\end{figure}
In the last two plots of Figure~\ref{bdf2_gmres_diff_t_dxdt}, we demonstrate the mesh independence of the preconditioned GMRES by setting $T = 4$ and $\alpha = \alpha_{\text{opt}}$. The grid is initialized with $(h, \Delta t) = (1/32, 1/32)$. It is clear that the convergence is independent of the mesh parameters. 
\begin{figure}[h!]
    \centering
    \subfloat{{\includegraphics[height=3.5cm,width=4cm]{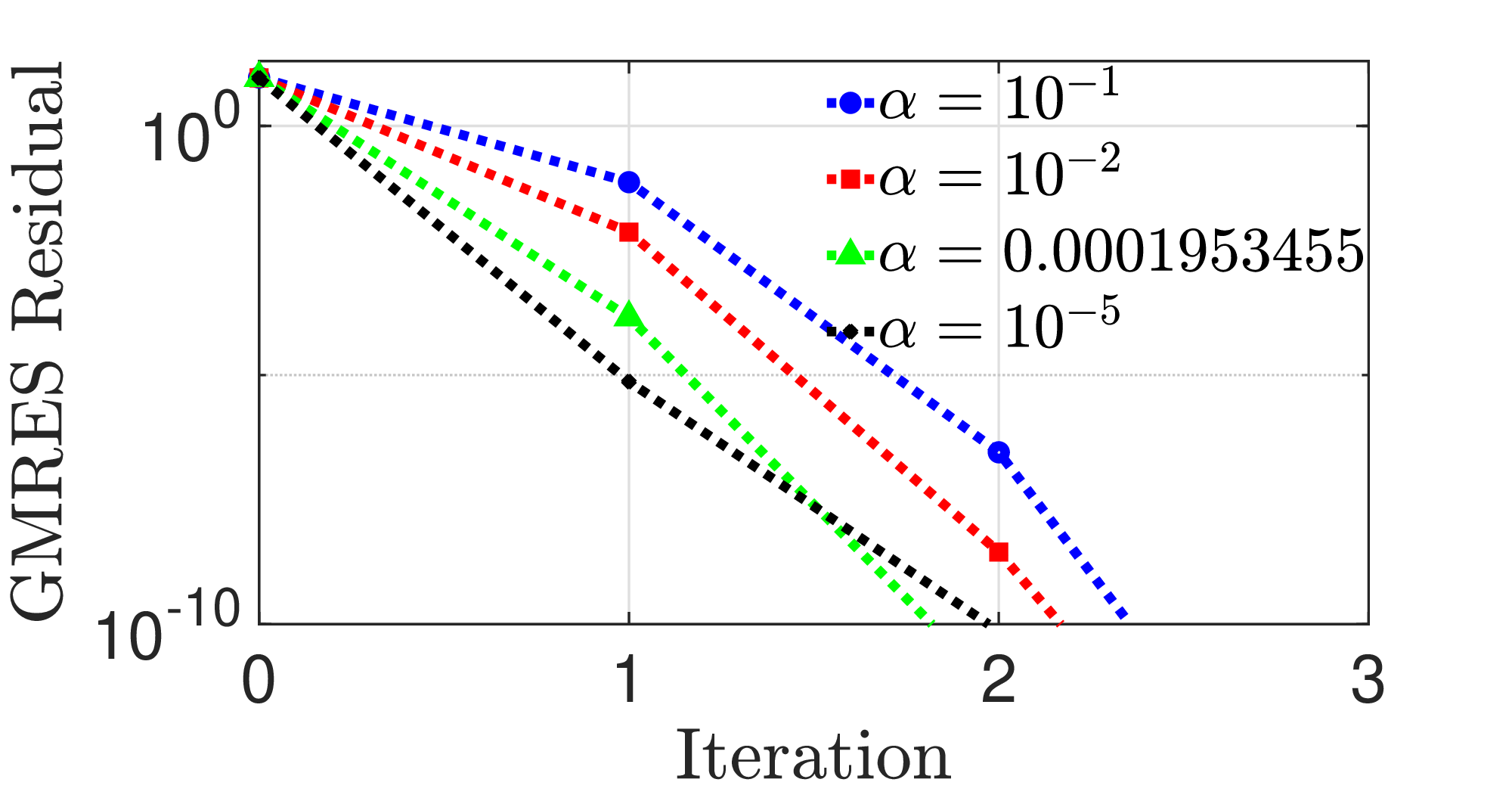} }}
    \subfloat{{\includegraphics[height=3.5cm,width=4cm]{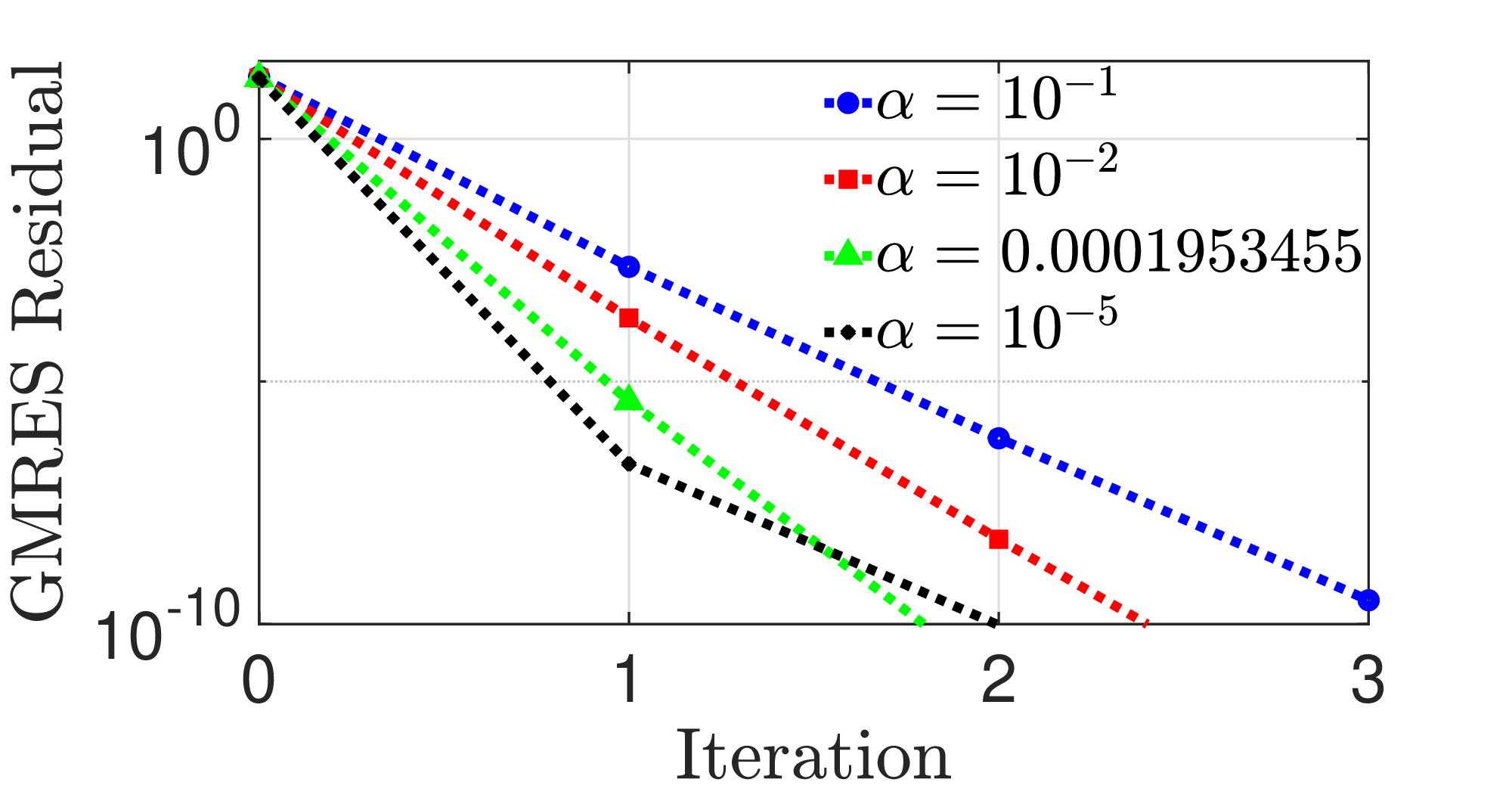} }}
    \subfloat{{\includegraphics[height=3.5cm,width=4cm]{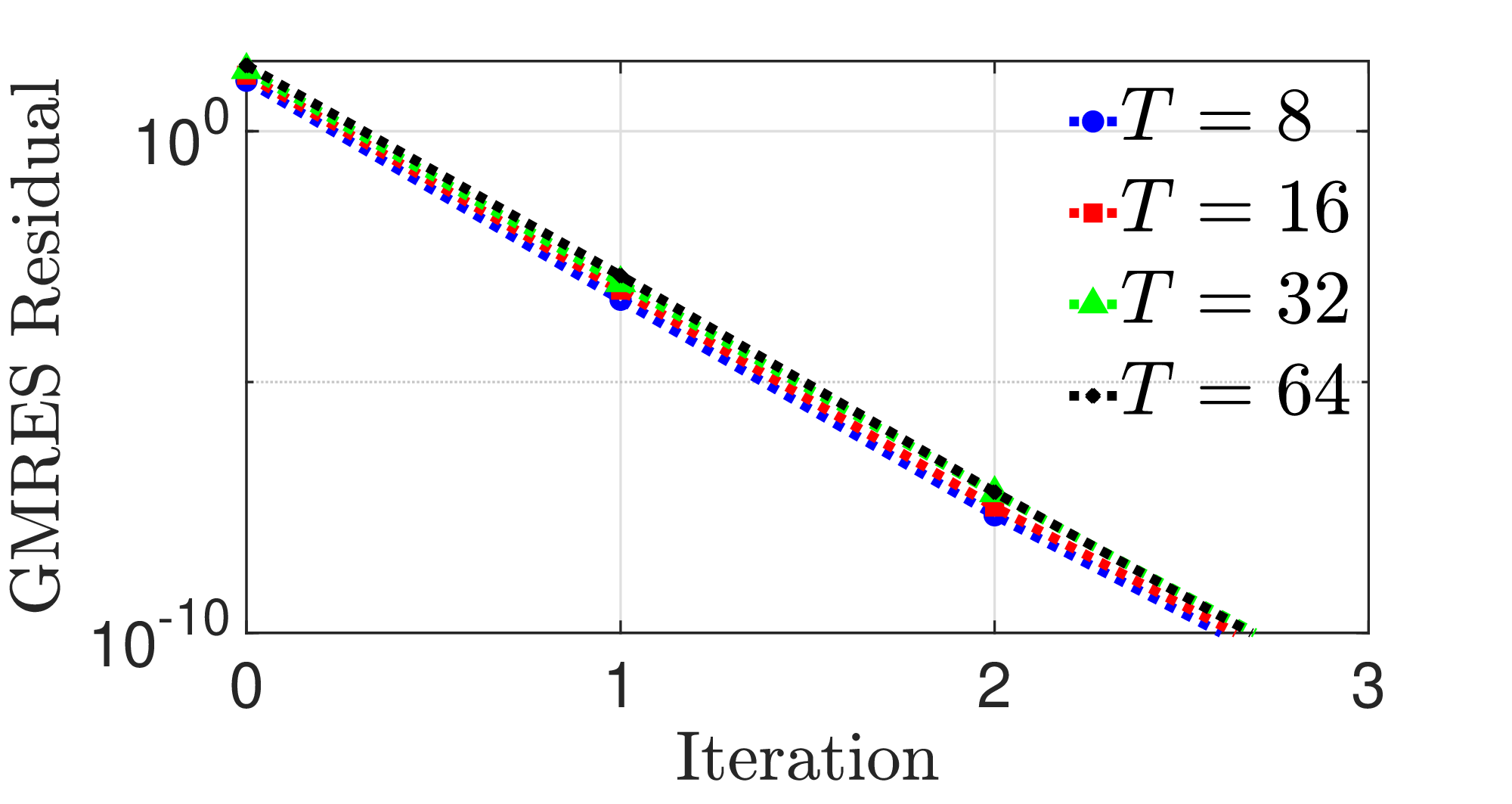} }}
    \subfloat{{\includegraphics[height=3.5cm,width=4cm]{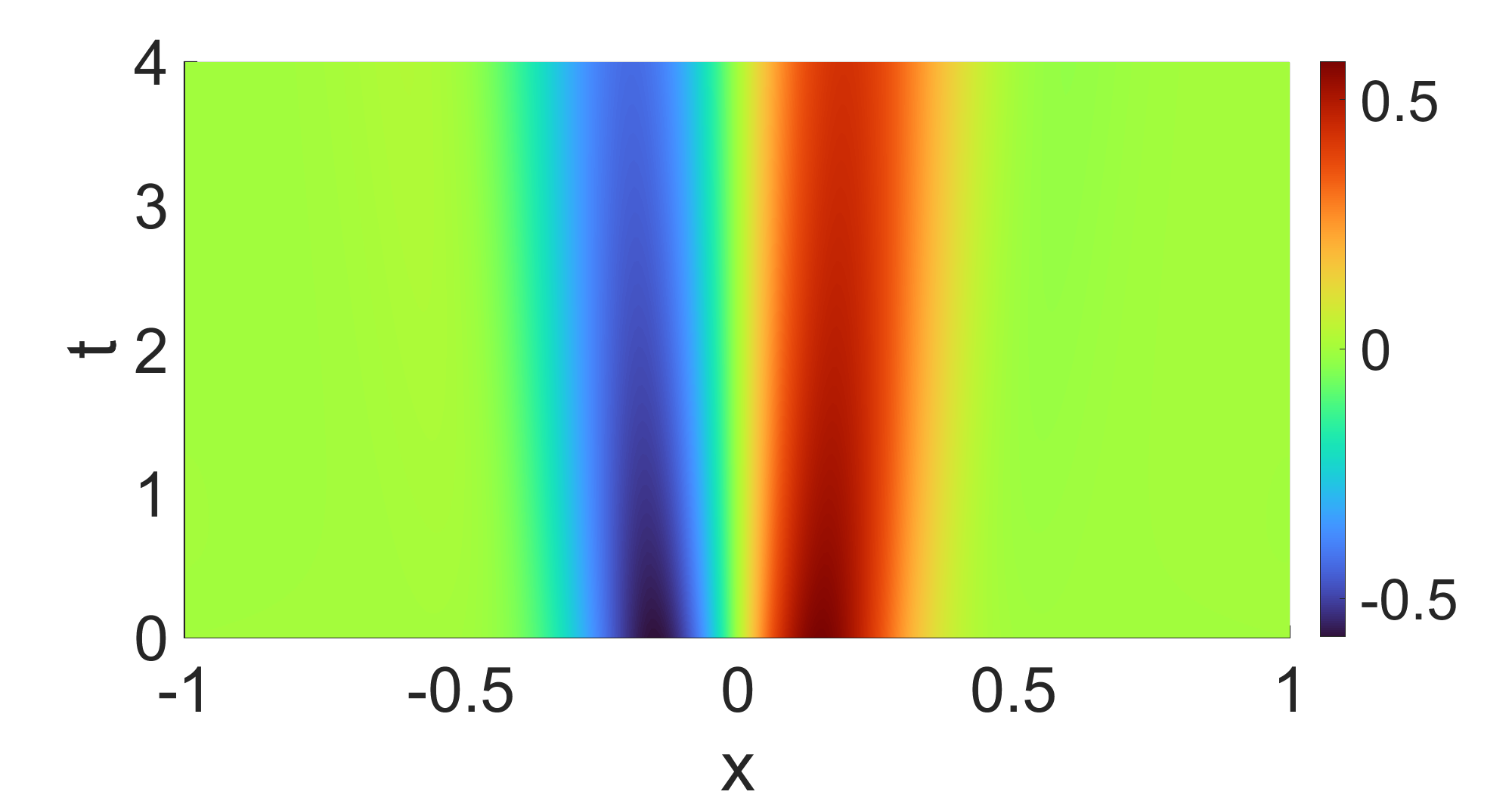} }}
    \caption{ First: convergence for different $\alpha$ with $(a, c)=(0.1, 0.1)$; Second: convergence for different $\alpha$ with $(a, c)=(0.00001, 0.1)$; Third: convergence for different $T$; Fourth: solution profile up-to $T=4$.}
    \label{bdf2_gmres_diff_alpha}
\end{figure}
In the first two plots of Figure~\ref{bdf2_gmres_diff_alpha}, we show the GMRES residuals of the preconditioned GMRES method for different values of $\alpha$, with $\Delta t = h = \frac{1}{128}$ and $T = 4$. The results show that convergence is fast, and becomes even faster as $\alpha$ decreases. 

\noindent Next, we consider the biharmonic heat equation $u_t = -10^{-5} \Delta^2 u$ with boundary conditions $u = 0$ and $\Delta u = 0$, and initial condition $u_0 = \sin(2\pi x) e^{-15x^2}$. We apply the Exp-ParaDiag method based on BDF2 to solve this equation. In the third plot of Figure~\ref{bdf2_gmres_diff_alpha}, we show the convergence of the proposed method for different values of $T$, using $h = \frac{1}{128}$, $\Delta t = 0.05$, and $\alpha = \alpha_{\text{opt}}$. The results indicate that convergence is fast and independent of the time window size. In the rightmost plot, we display the solution profile of the biharmonic heat equation with $h = \Delta t = \frac{1}{128}$, $\alpha = \alpha_{\text{opt}}$, and $T = 4$.

\subsubsection{Experiment in 2D}
In this section, we present numerical experiments for the Exp-ParaDiag method, used both as a fixed-point iteration and as a preconditioner in GMRES, for first-order and second-order time-stepping schemes in 2D. 
In the first two plots of Figure~\ref{exp_pdiag_diff_t_2d}, we display the error curves of the Exp-ParaDiag method used as a fixed-point iteration for different values of $T$, by varying $a$, with $h = \frac{1}{40}$, $\Delta t = 0.05$, and $\alpha = \alpha_{\text{opt}}$. It can be observed that convergence is rapid and independent of the time window size. 
\begin{figure}[h!]
    \centering
    \subfloat{{\includegraphics[height=3.5cm,width=4cm]{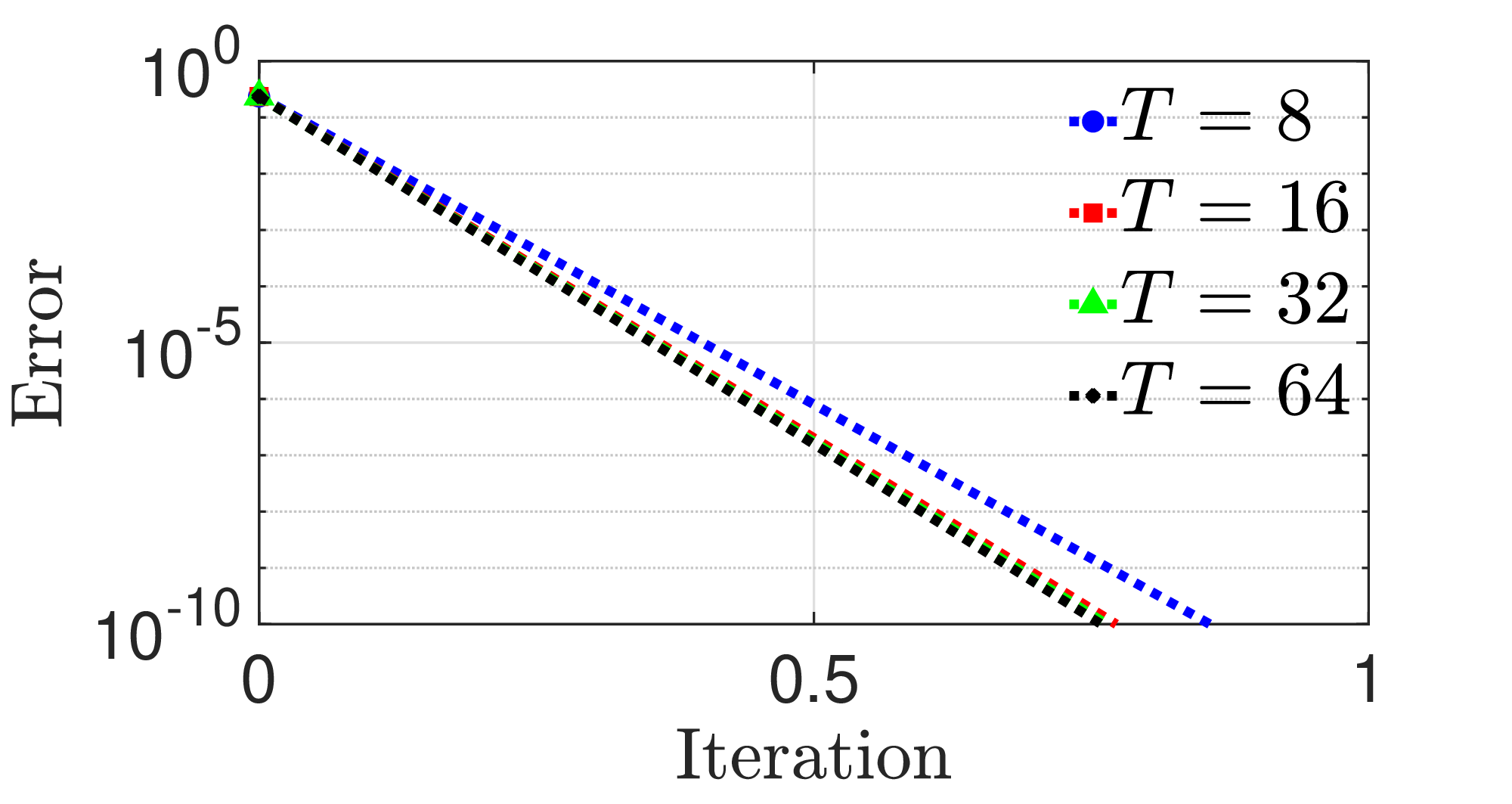} }}
    \subfloat{{\includegraphics[height=3.5cm,width=4cm]{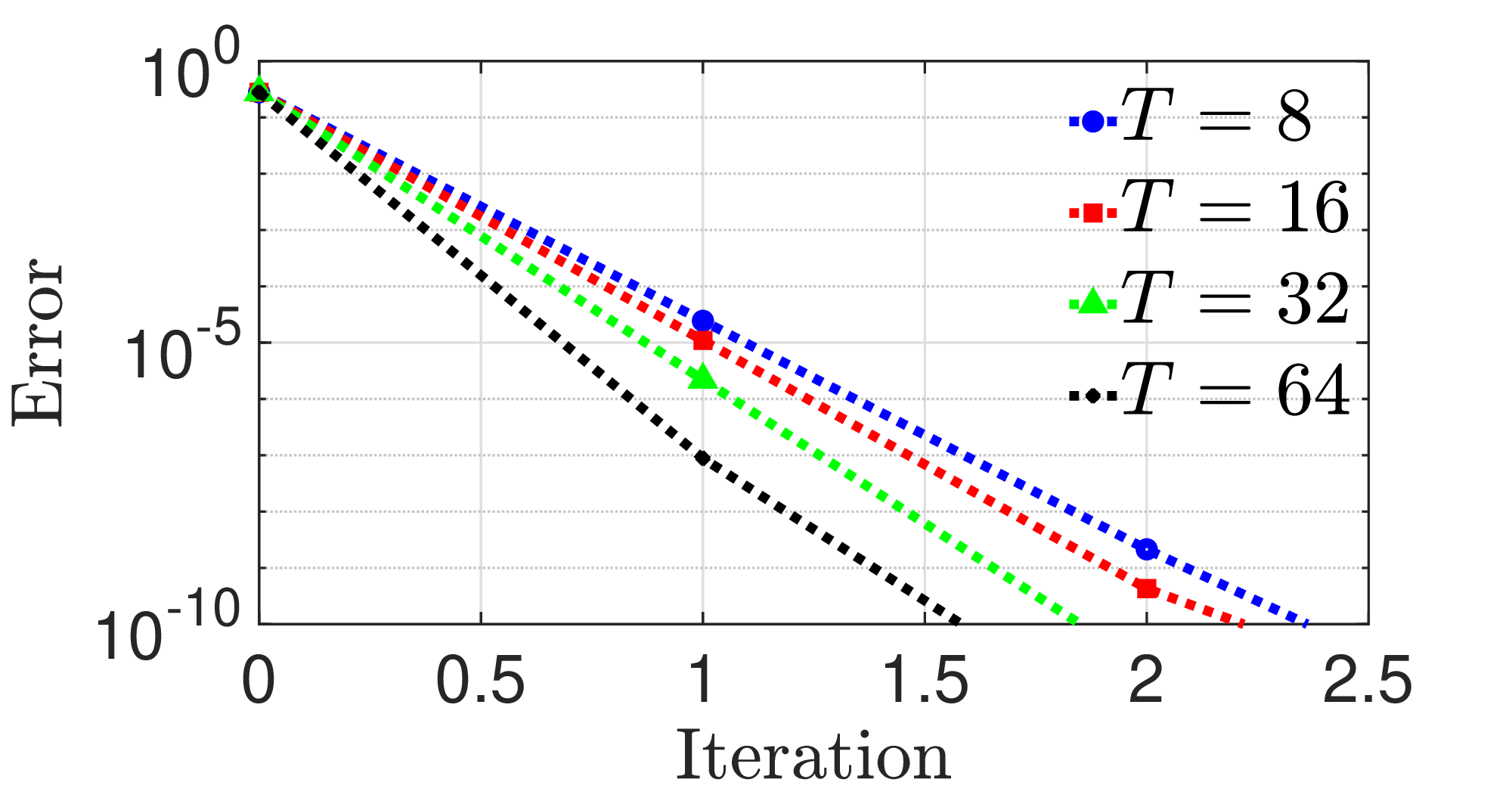} }}
    \subfloat{{\includegraphics[height=3.5cm,width=4cm]{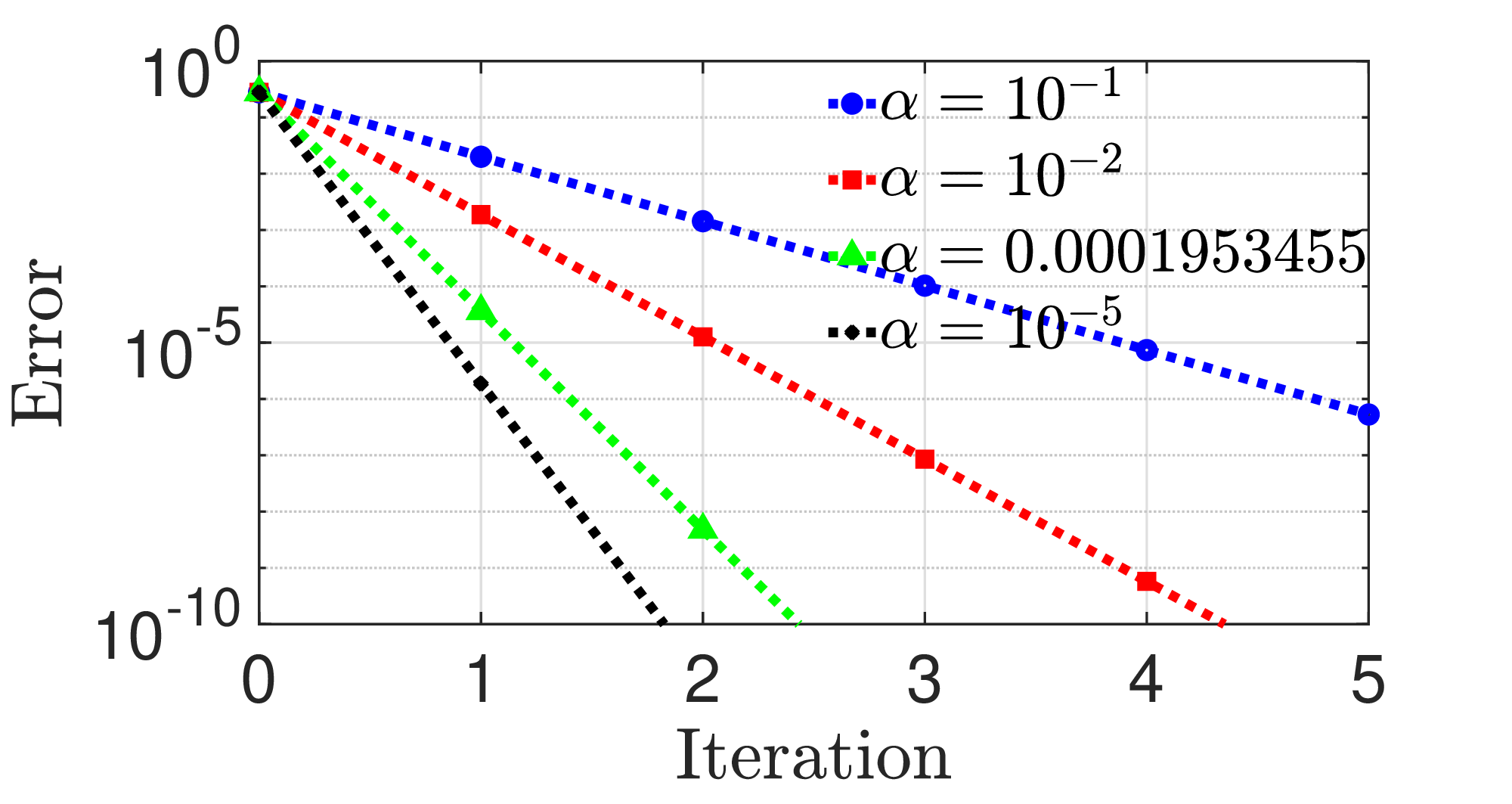} }}
    \subfloat{{\includegraphics[height=3.5cm,width=4cm]{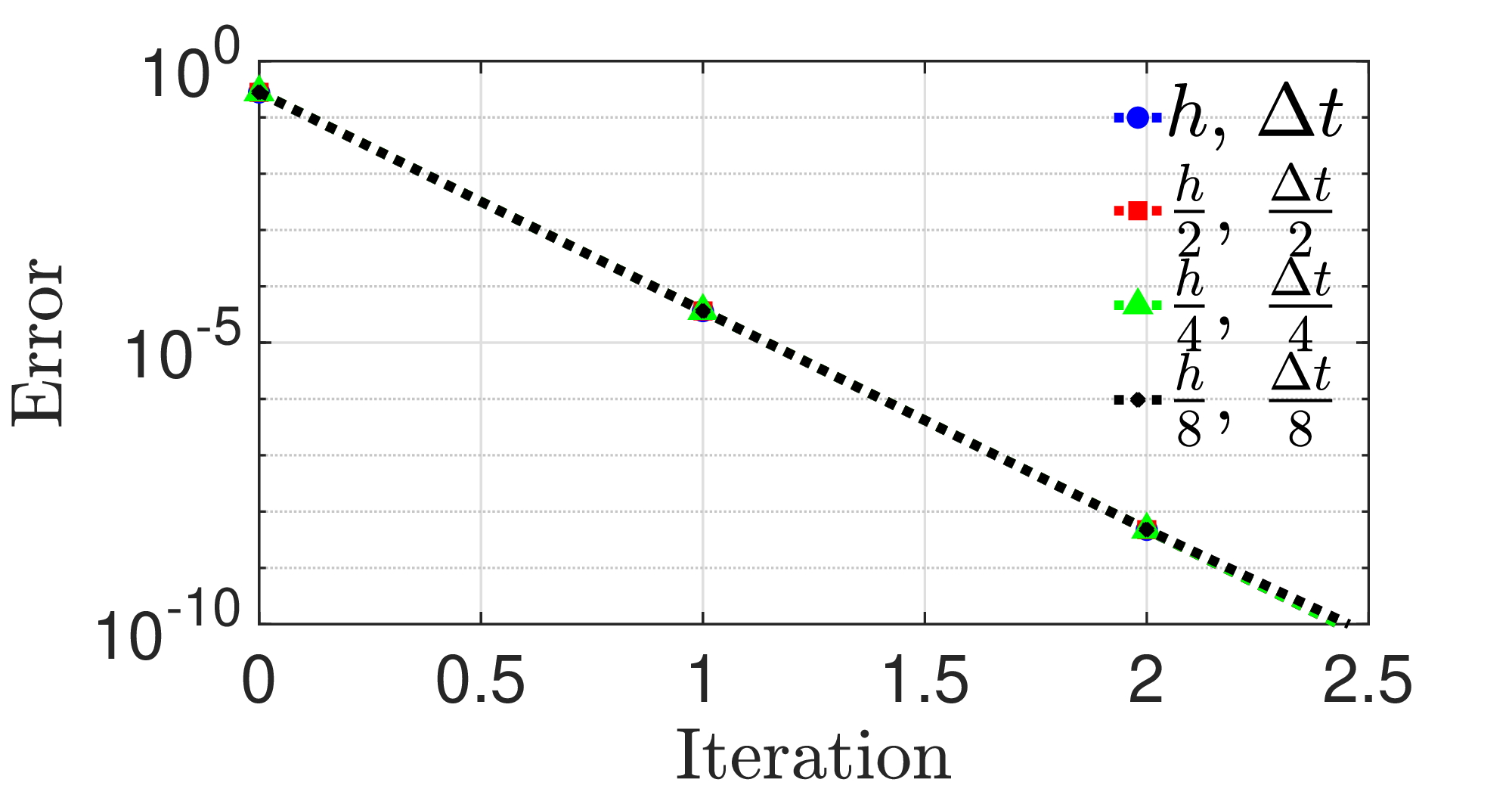} }}
    \caption{ First: convergence of Exp-ParaDiag with $(a, c)=(0.1, 0.1)$; Second: convergence of Exp-ParaDiag with $(a, c)=(0.00001, 0)$; Third: convergence for different $\alpha$ with $(a, c)=(0.00001, 0.1)$; Fourth: mesh independence with $(a, c)=(0.00001, 0.1)$.}
    \label{exp_pdiag_diff_t_2d}
\end{figure}
As observed in the 1D case, the method also converges faster in 2D for higher diffusion coefficients. Therefore, we present results for lower diffusion to illustrate the method's performance.
In the third plot, we show the convergence behavior of the Exp-ParaDiag method for different choices of $\alpha$, with $h = \frac{1}{40}$, $\Delta t = 0.05$, and $T = 4$. It can be observed that smaller values of $\alpha$ lead to faster convergence. In particular, values of $\alpha$ close to $\alpha_{\text{opt}}$ yield the best convergence in terms of iteration count. 
In the rightmost plot of Figure~\ref{exp_pdiag_diff_t_2d}, we present the convergence behavior of the Exp-ParaDiag method under mesh refinement. We start with a very coarse mesh, $(h, \Delta t) = (1/10, 1/10)$, with $T = 4$ and $\alpha = \alpha_{\text{opt}}$. The results indicate that convergence is independent of grid refinement.

\noindent Next, we present the convergence of preconditioned GMRES for a first-order accurate scheme. In the first two plots of Figure~\ref{gmres_bdf1_2d}, we display the error curves for different values of $T$ and varying $a$, using $h = \frac{1}{40}$, $\Delta t = 0.5$, and $\alpha = \alpha_{\text{opt}}$. It can be observed that the convergence is rapid and independent of the time window size.
\begin{figure}[h!]
    \centering
    \subfloat{{\includegraphics[height=3.5cm,width=4cm]{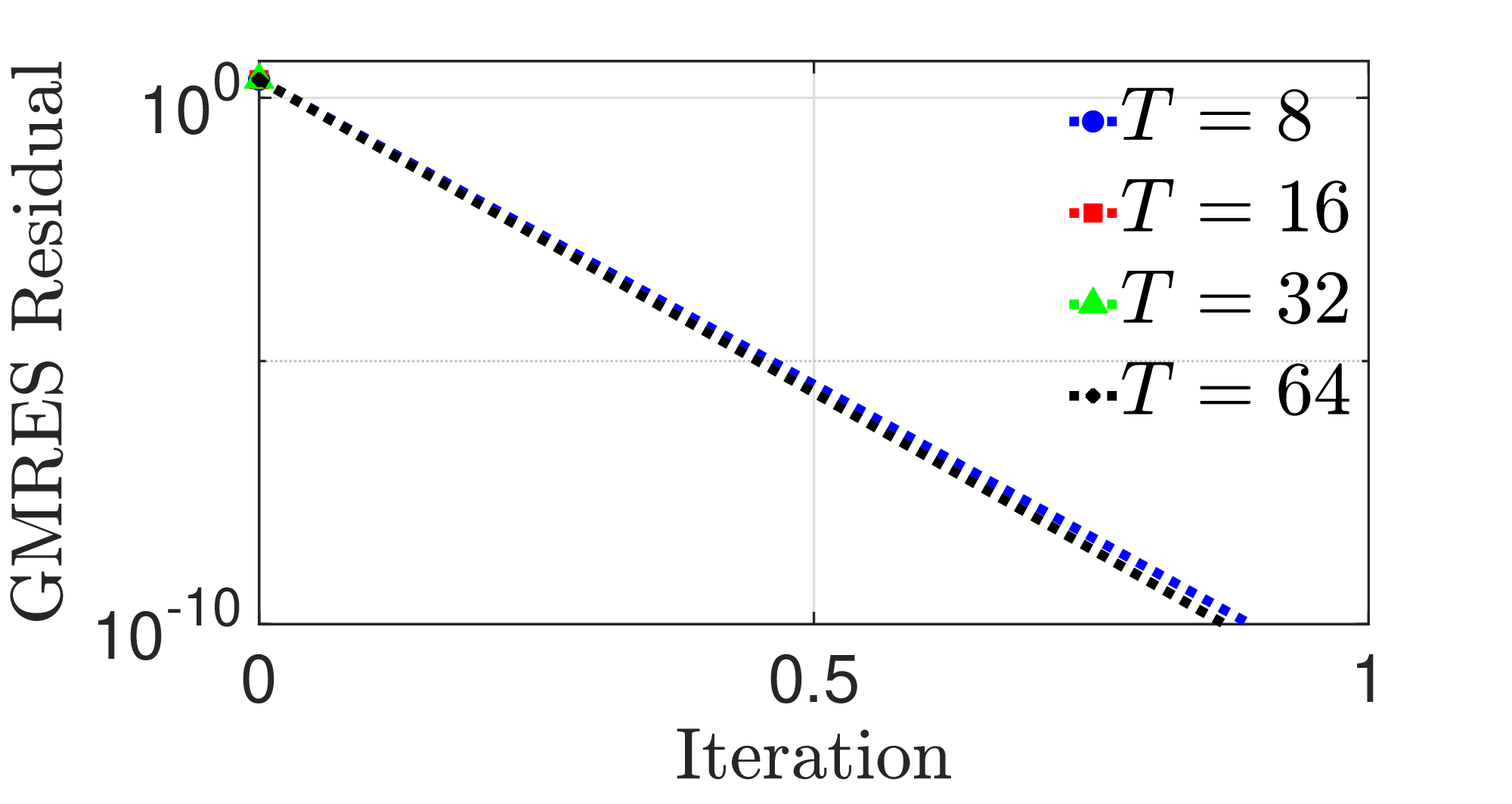} }}
    \subfloat{{\includegraphics[height=3.5cm,width=4cm]{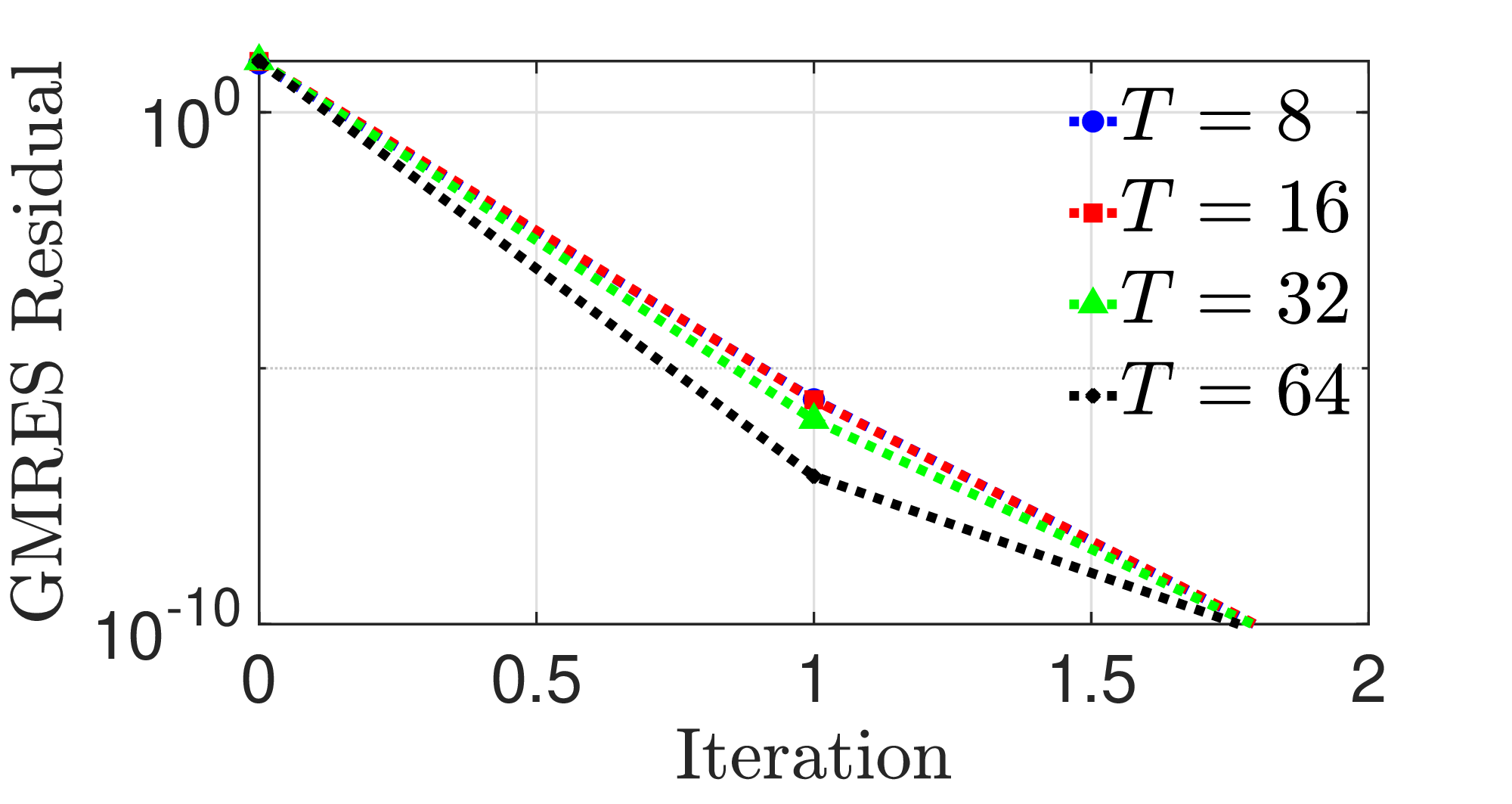} }}
    \subfloat{{\includegraphics[height=3.5cm,width=4cm]{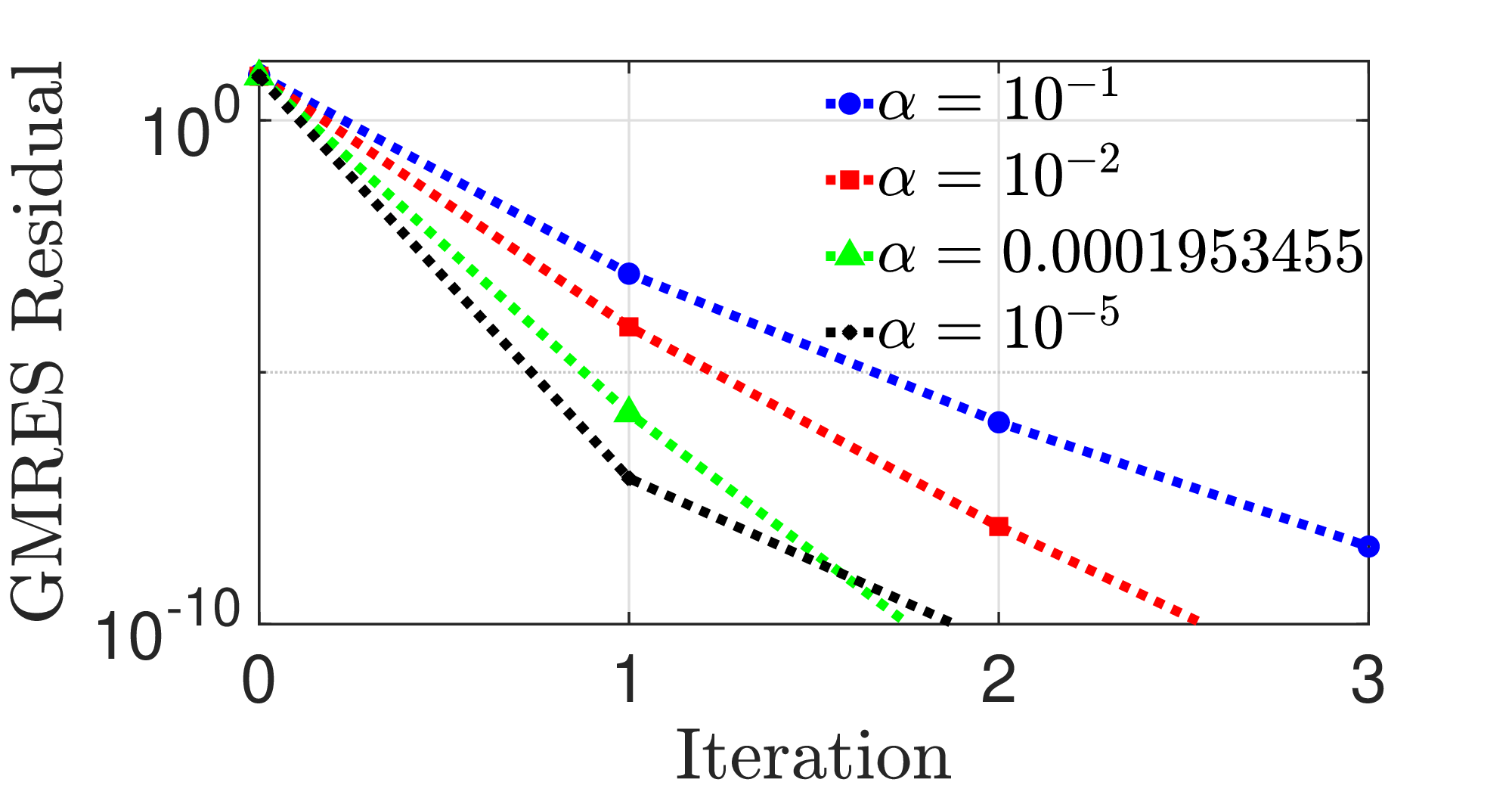} }}
    \caption{convergence of preconditioned GMRES; First: for different $T$ with $(a, c)=(0.1, 0.1)$; Second: for different $T$ with $(a, c)=(0.00001, 0)$; Third: for different $\alpha$ with $(a, c)=(0.00001, 0.1)$.}
    \label{gmres_bdf1_2d}
\end{figure}
In the third plot, we display the convergence behavior under different choices of $\alpha$, with $h = \frac{1}{40}$, $\Delta t = 0.05$, and $T = 4$. We observe a similar pattern to that seen in the fixed-point case.

\noindent We now examine the convergence behavior of the preconditioned GMRES method based on the BDF2 scheme. The first two plots in Figure~\ref{gmres_bdf2_2d} show error curves for various time horizons $T$ and different values of $a$, with parameters set to $h = \frac{1}{40}$, $\Delta t = 0.5$, and $\alpha = \alpha_{\text{opt}}$. The results demonstrate that the convergence remains unaffected by the length of the time window.
\begin{figure}[h!]
     \centering
     \subfloat{{\includegraphics[height=3.5cm,width=4cm]{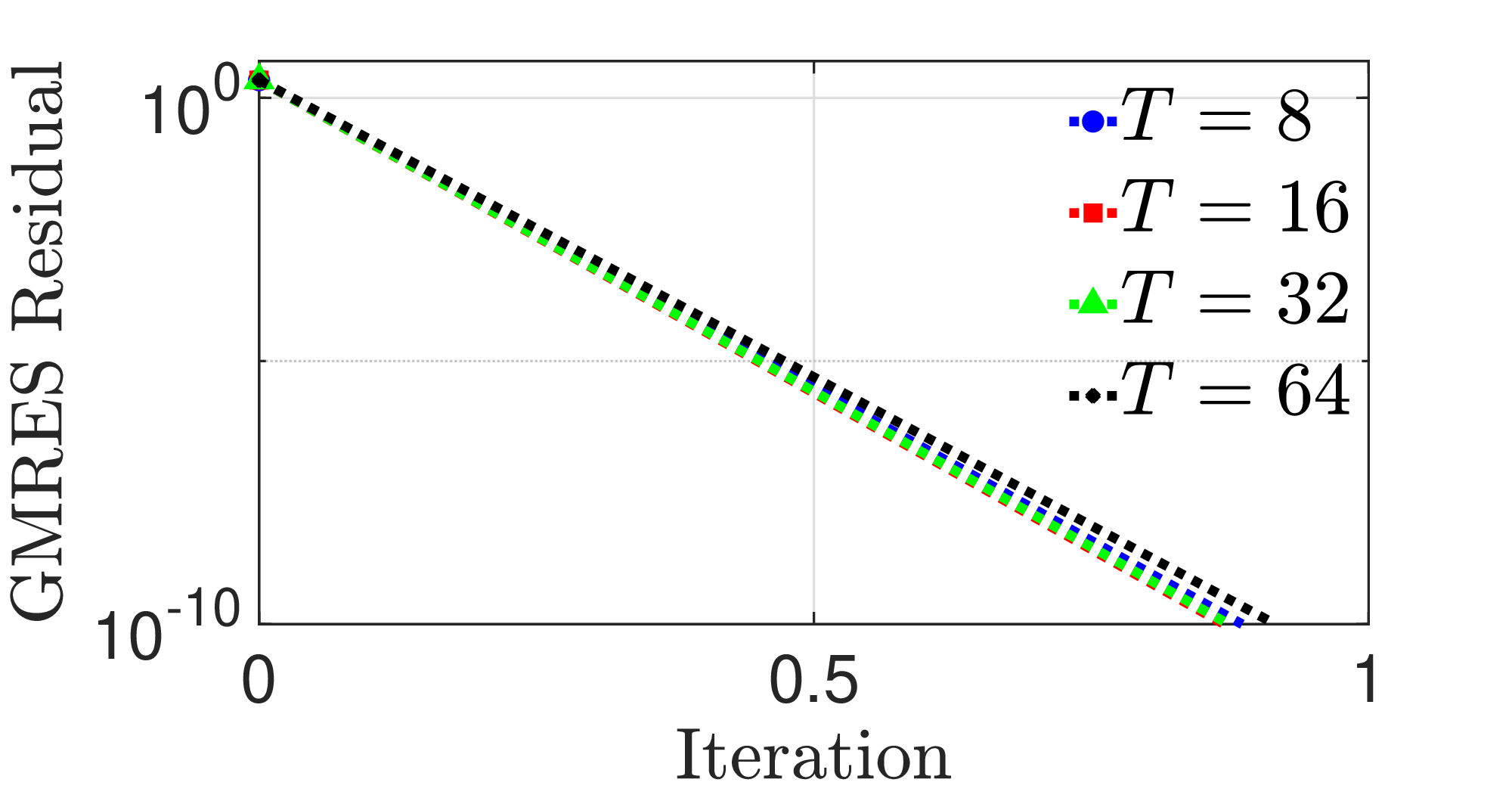} }}
     \subfloat{{\includegraphics[height=3.5cm,width=4cm]{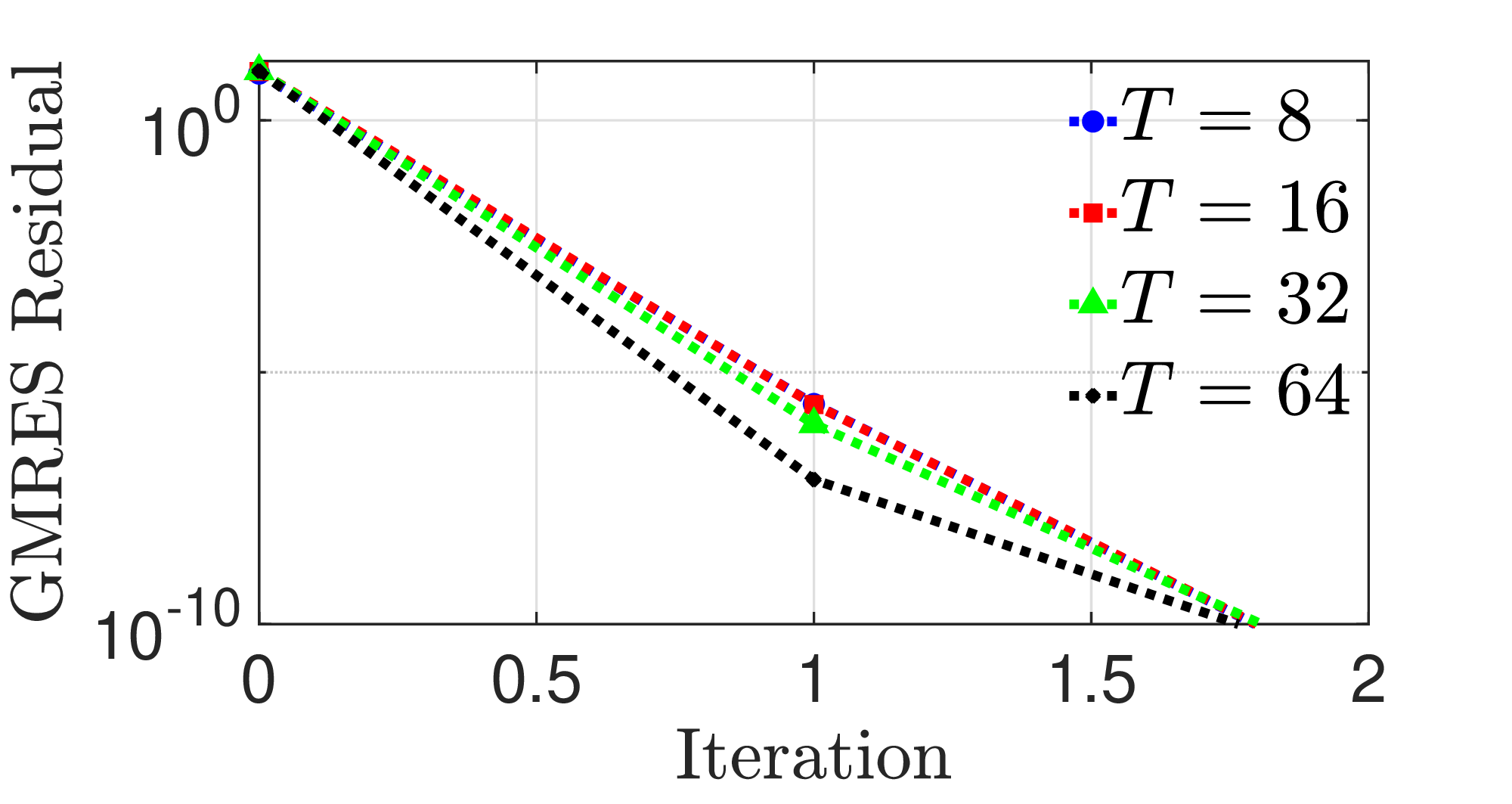} }}
     \subfloat{{\includegraphics[height=3.5cm,width=4cm]{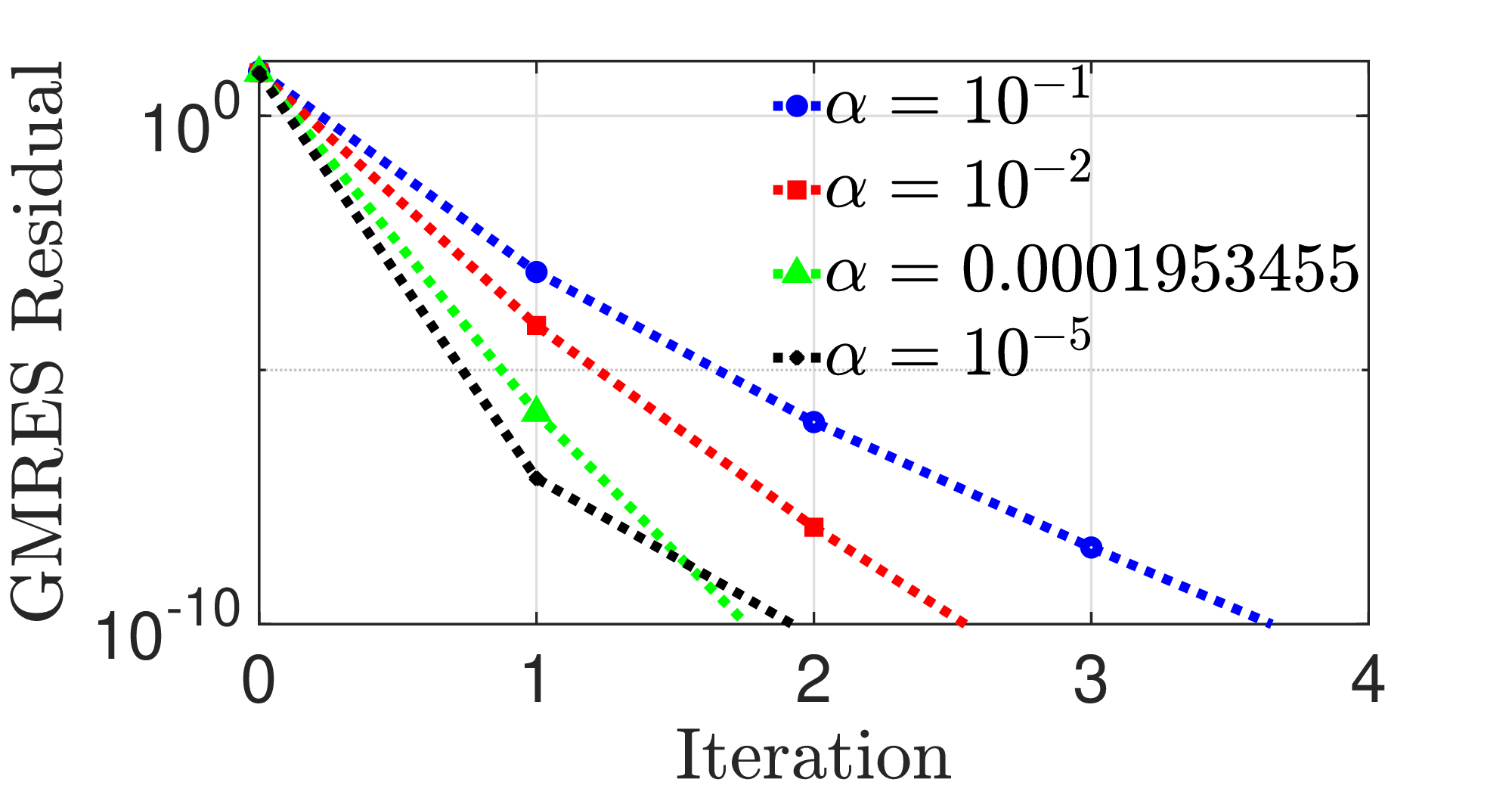} }}
     \caption{convergence of preconditioned GMRES; First: for different $T$ with $(a, c)=(0.1, 0.1)$; Second: for different $T$ with $(a, c)=(0.00001, 0)$; Third: for different $\alpha$ with $(a, c)=(0.00001, 0.1)$.}
     \label{gmres_bdf2_2d}
 \end{figure}
The third plot illustrates the convergence behavior for various values of $\alpha$, with $h = \frac{1}{40}$, $\Delta t = 0.05$, and $T = 4$. We observe that smaller values of $\alpha$ generally lead to faster convergence. In particular, the best performance is achieved when $\alpha$ is chosen close to the optimal value $\alpha_{\text{opt}}$.

\noindent We now turn our attention to the Schrödinger equation (SE), which we formulate by setting \( a = \mathrm{i} \) and \( c = 200\mathrm{i} \) in \eqref{model_problem_linear}. To evaluate the convergence properties of the preconditioned GMRES method based on the BDF2 scheme, we apply it to the SE with the initial condition
$
u_0(x, y) = \exp\left(-\frac{(x - 0.5)^2 + (y - 0.5)^2}{2\mu^2}\right) \exp\left(-5\mathrm{i}(x - 0.5)\right),
$
where $\mu = 0.05$.
\begin{figure}[h!]
     \centering
     \subfloat{{\includegraphics[height=3.5cm,width=4cm]{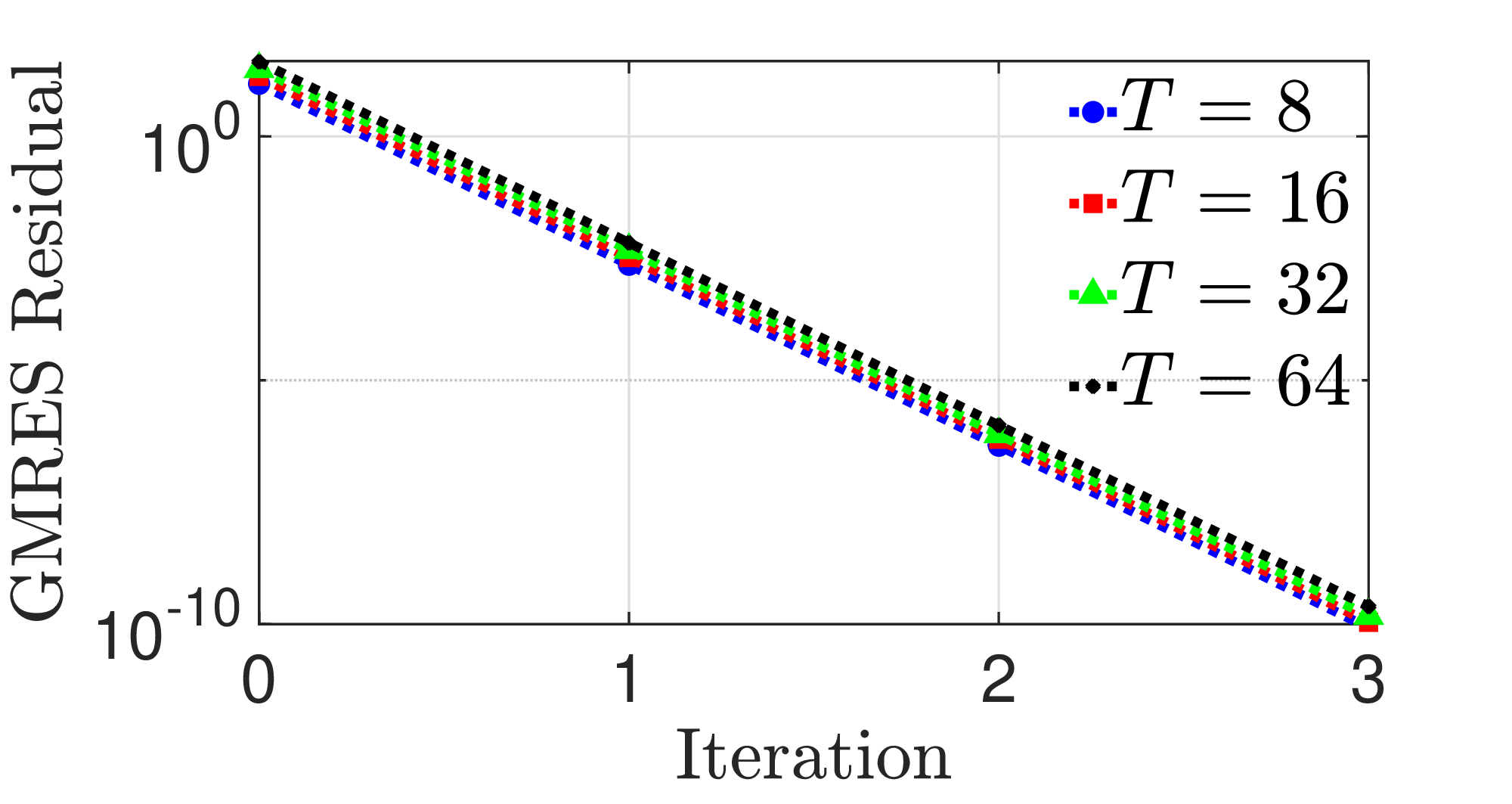} }}
     \subfloat{{\includegraphics[height=3.5cm,width=4cm]{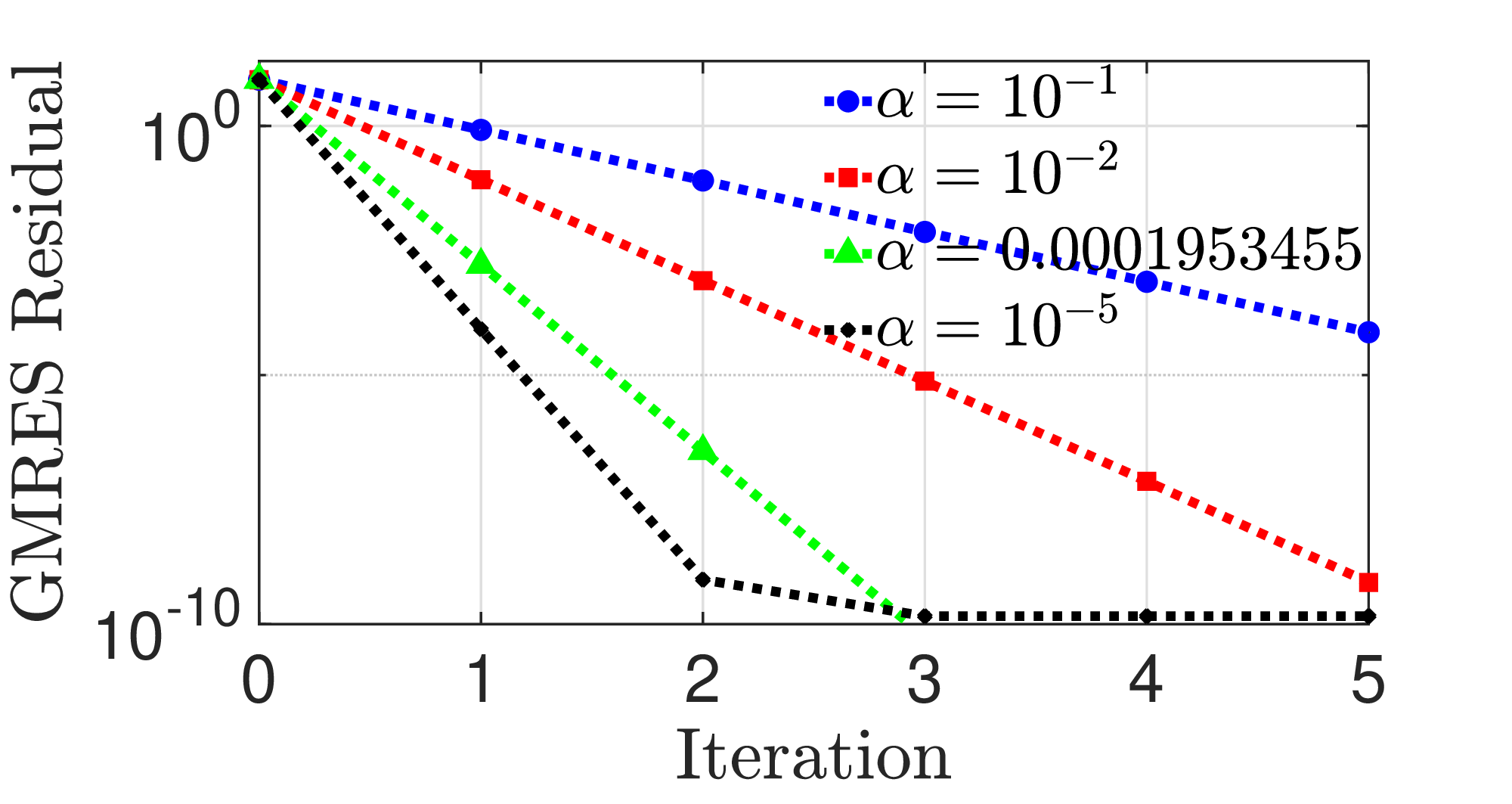} }}
     \subfloat{{\includegraphics[height=3.5cm,width=4cm]{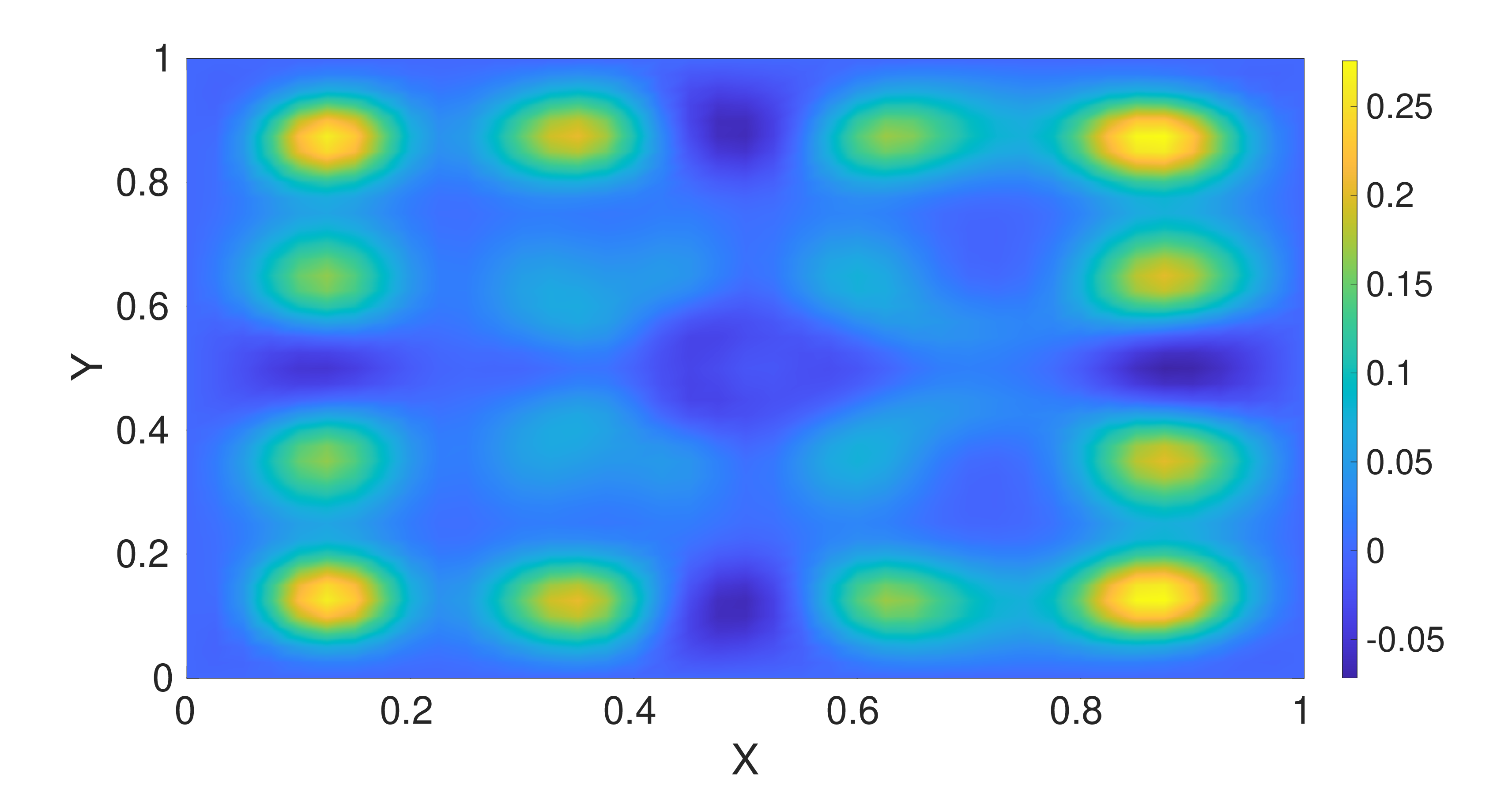} }}
     \subfloat{{\includegraphics[height=3.5cm,width=4cm]{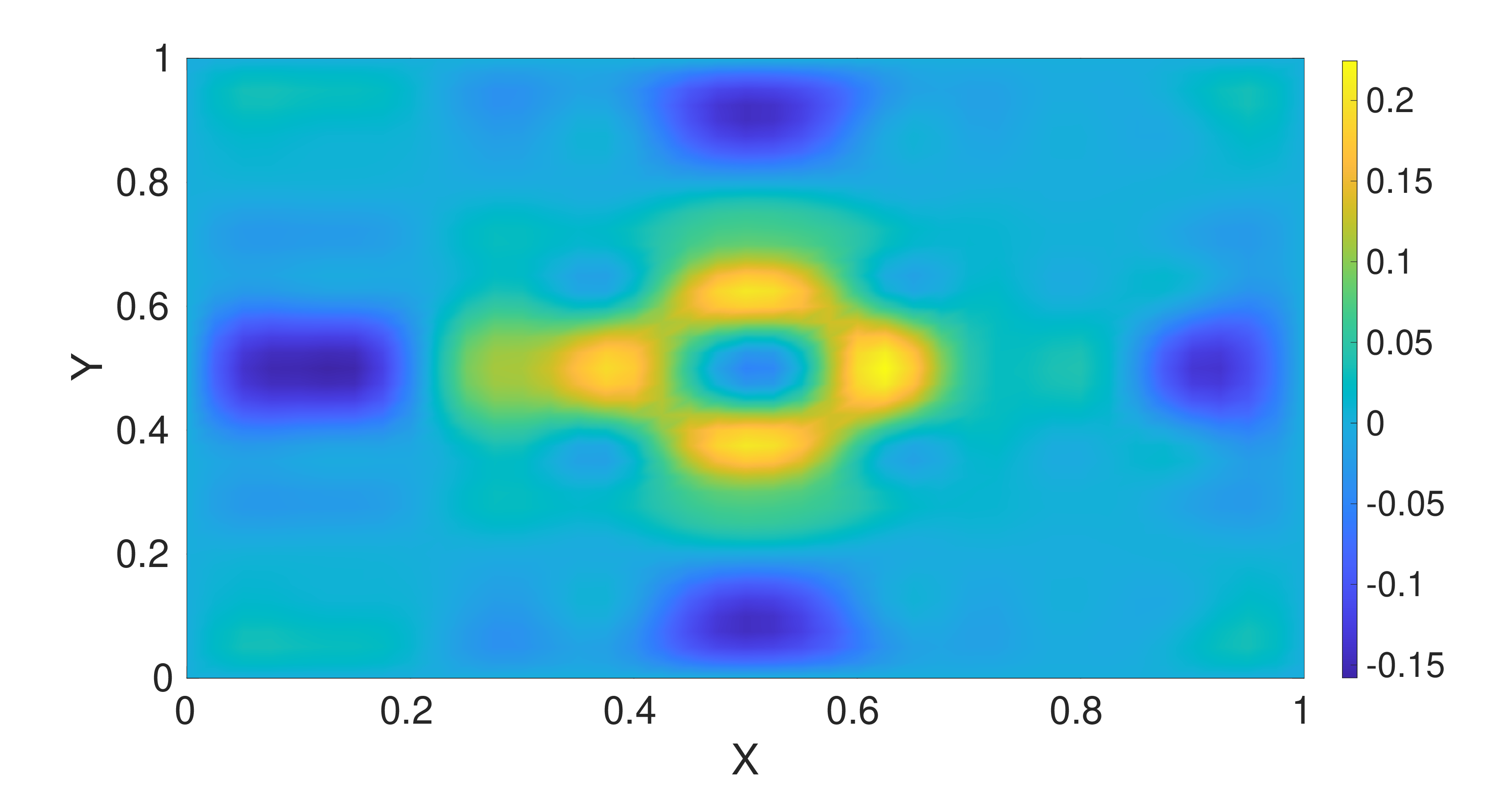} }}
     \caption{convergence of preconditioned GMRES; First: for different $T$; Second: for different $\alpha$; Third: solution of SE at $T=1$; Fourth: solution of SE at $T=4$.}
     \label{SE_gmres_bdf2_2d}
 \end{figure}
We fix \( h = \frac{1}{40} \) and \( \Delta t = 0.05 \). In the first plot of Figure~\ref{SE_gmres_bdf2_2d}, we present the convergence behavior of the method for various time window lengths using \( \alpha = \alpha_{\text{opt}} \). The results show that the method consistently converges in just three iterations, regardless of the window size. In the second plot, we examine the effect of different values of \( \alpha \) on convergence. As in previous cases, smaller values of \( \alpha \) lead to faster convergence.
In the third and fourth plots of Figure~\ref{SE_gmres_bdf2_2d}, we display the real part of the solution to the SE for \( T = 1 \) and \( T = 4 \), respectively, using \( \alpha = \alpha_{\text{opt}} \).

\begin{table}[ht]
\centering
\caption{CPU time comparison for preconditioned GMRES with BDF1 and BDF2 (in bracket CPU time for preconditioned BiCGStab)}
\label{cpu_time_combined}
\begin{tabular}{ccccc|cc}
\toprule
$N_x$ & $N_t$ & DoF & \multicolumn{2}{c|}{BDF1} & \multicolumn{2}{c}{BDF2} \\
\cmidrule(r){4-5} \cmidrule(l){6-7}
 & & & $a=10^{-1}$ & $a=10^{-5}$ & $a=10^{-1}$ & $a=10^{-5}$ \\
\midrule
\multirow{6}{*}{361}
  & 400  & 144400    & 0.78(0.55) & 0.77(0.55) & 0.83(0.59) & 0.85(0.58) \\
  & 800  & 288800    & 0.90(0.51) & 1.32(0.91) & 1.01(0.54) & 1.43(0.96) \\
  & 1600  & 577600   & 1.61(0.86) & 2.37(1.63) & 1.76(0.91) & 2.59(1.71) \\
  & 3200  & 1155200  & 3.02(1.52) & 4.47(3.03) & 3.24(1.66) & 4.96(3.22) \\
  & 6400  & 2310400  & 5.82(2.96) & 8.69(5.86) & 6.29(3.32) & 9.65(6.34) \\
  & 12800 & 4620800  & 11.41(5.87) & 23.19(17.37) & 12.92(7.09) & 26.06(20.35) \\
\midrule
\multirow{6}{*}{1521}
  & 400  & 608400    & 30.01(18.96) & 30.06(20.24) & 32.58(20.31) & 32.80(24.41) \\
  & 800  & 1216800   & 40.31(20.61) & 61.73(40.01) & 42.85(22.36) & 67.18(44.13) \\
  & 1600  & 2433600  & 86.53(40.00) & 130.30(84.65) & 88.57(42.74) & 141.52(90.69) \\
  & 3200  & 4867200  & 173.74(88.67) & 266.68(175.75) & 182.89(89.77) & 279.31(187.02) \\
  & 6400  & 9734400  & 352.91(173.26) & 534.92(357.87) & 369.05(184.30) & 563.47(377.44) \\
  & 12800 & 19468800 & 706.38(352.36) & 1428.76(1054.26) & 747.21(374.03) & 1510.21(1126.19) \\
\bottomrule
\end{tabular}
\end{table}
In Table~\ref{cpu_time_combined}, we report the CPU time (in seconds) for the preconditioned GMRES and BiCGStab methods applied to the purely diffusive equation, built on first and second order schemes with $\alpha = \alpha_{\text{opt}}$ and $\Delta t = 0.01$. BiCGStab exhibits similar convergence behavior to GMRES, as the spectrum of the preconditioned system is clustered around one, an ideal scenario for BiCGStab. Since both methods achieve comparable accuracy, we focus on comparing computational time, where BiCGStab demonstrates a clear advantage by requiring less time to converge.

\subsubsection{Experiment for Advection-Diffusion Equation}
In this section, we investigate the convergence behavior of the advection-diffusion equation (ADE) with dominant advection under periodic boundary conditions, specifically considering the equation
$
u_t = a \Delta u - 2 \nabla u.
$
\begin{figure}[h!]
    \centering
    \subfloat{{\includegraphics[height=3.5cm,width=4cm]{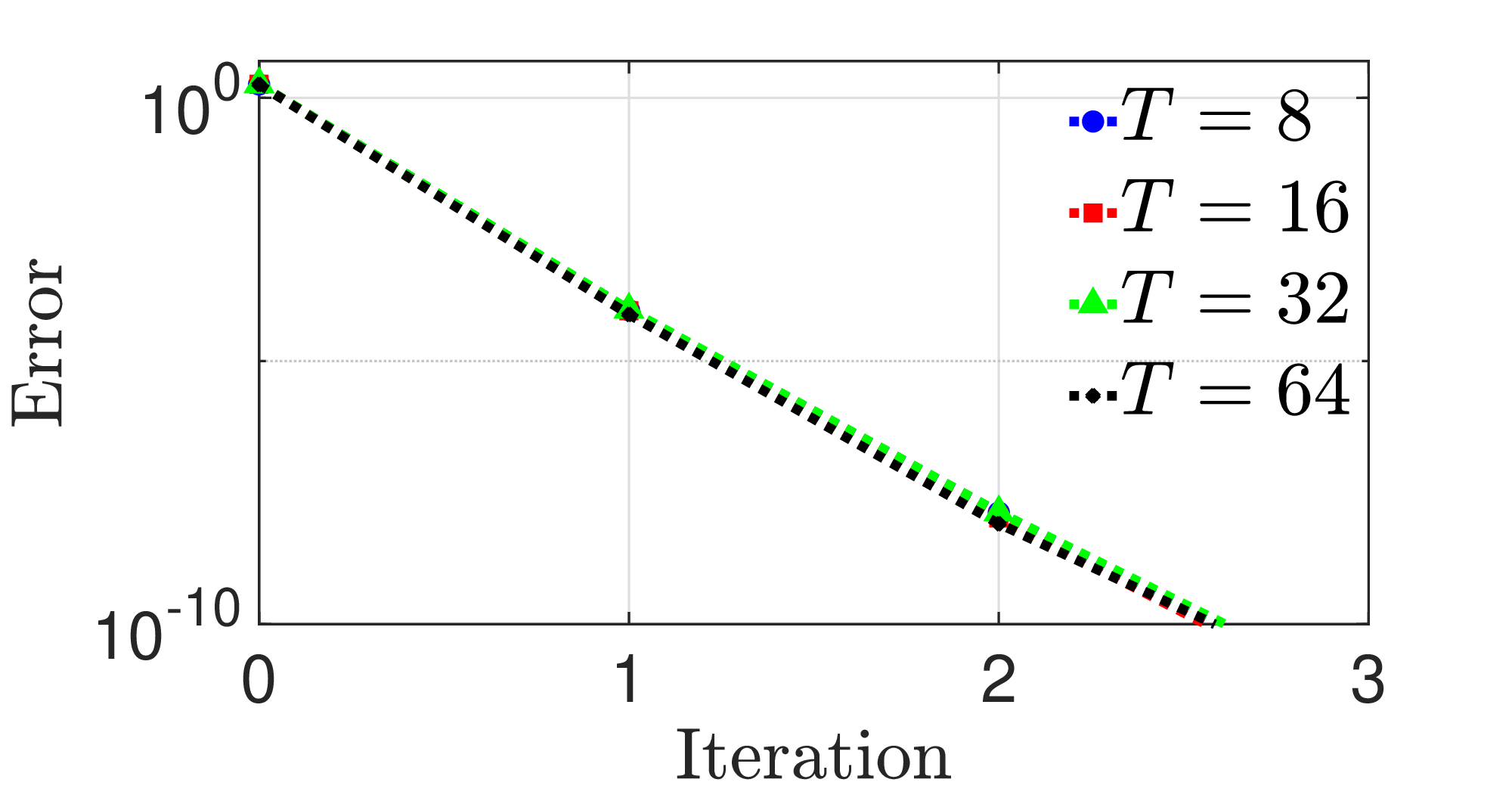} }}
    \subfloat{{\includegraphics[height=3.5cm,width=4cm]{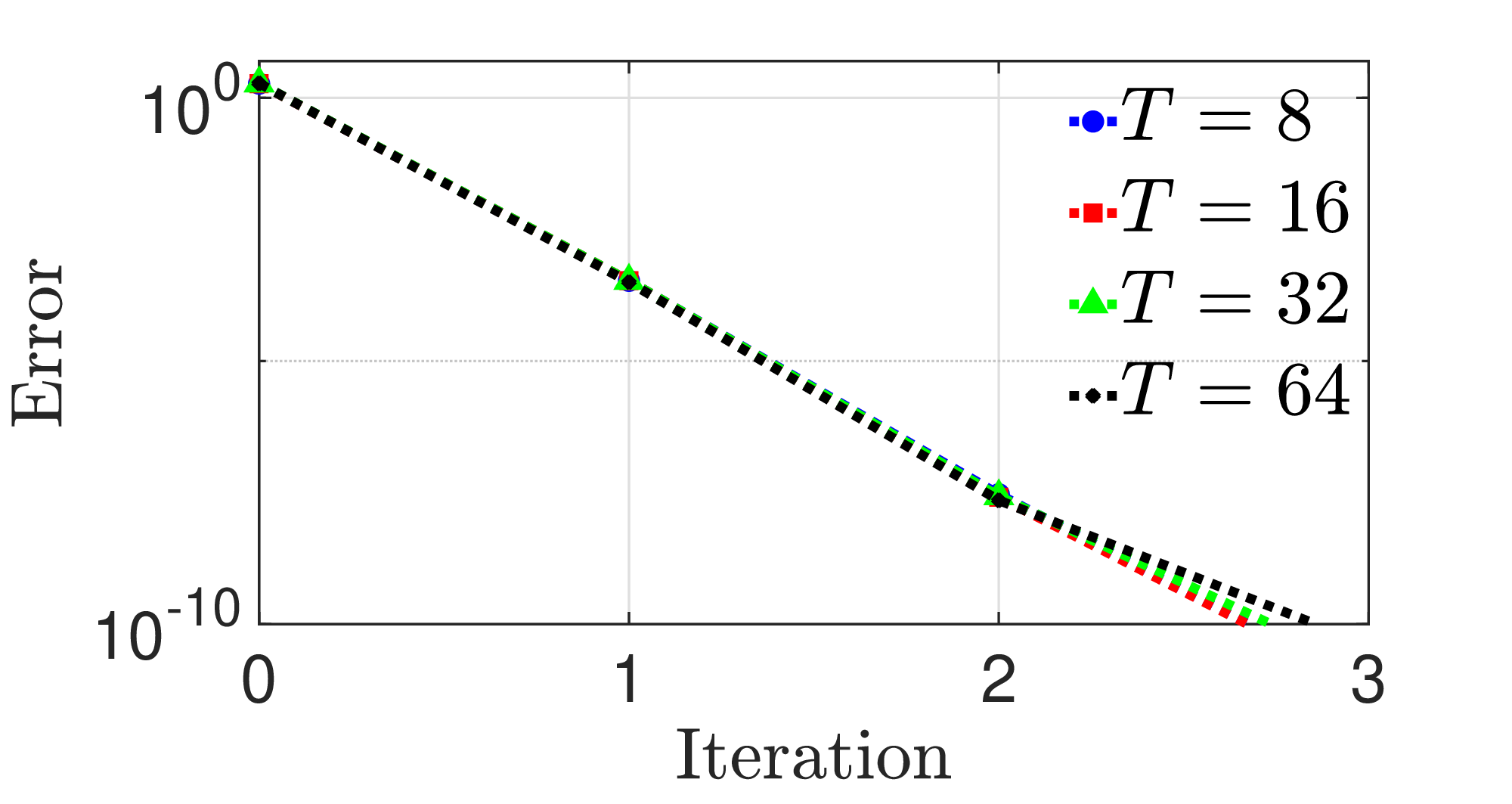} }}
    \subfloat{{\includegraphics[height=3.5cm,width=4cm]{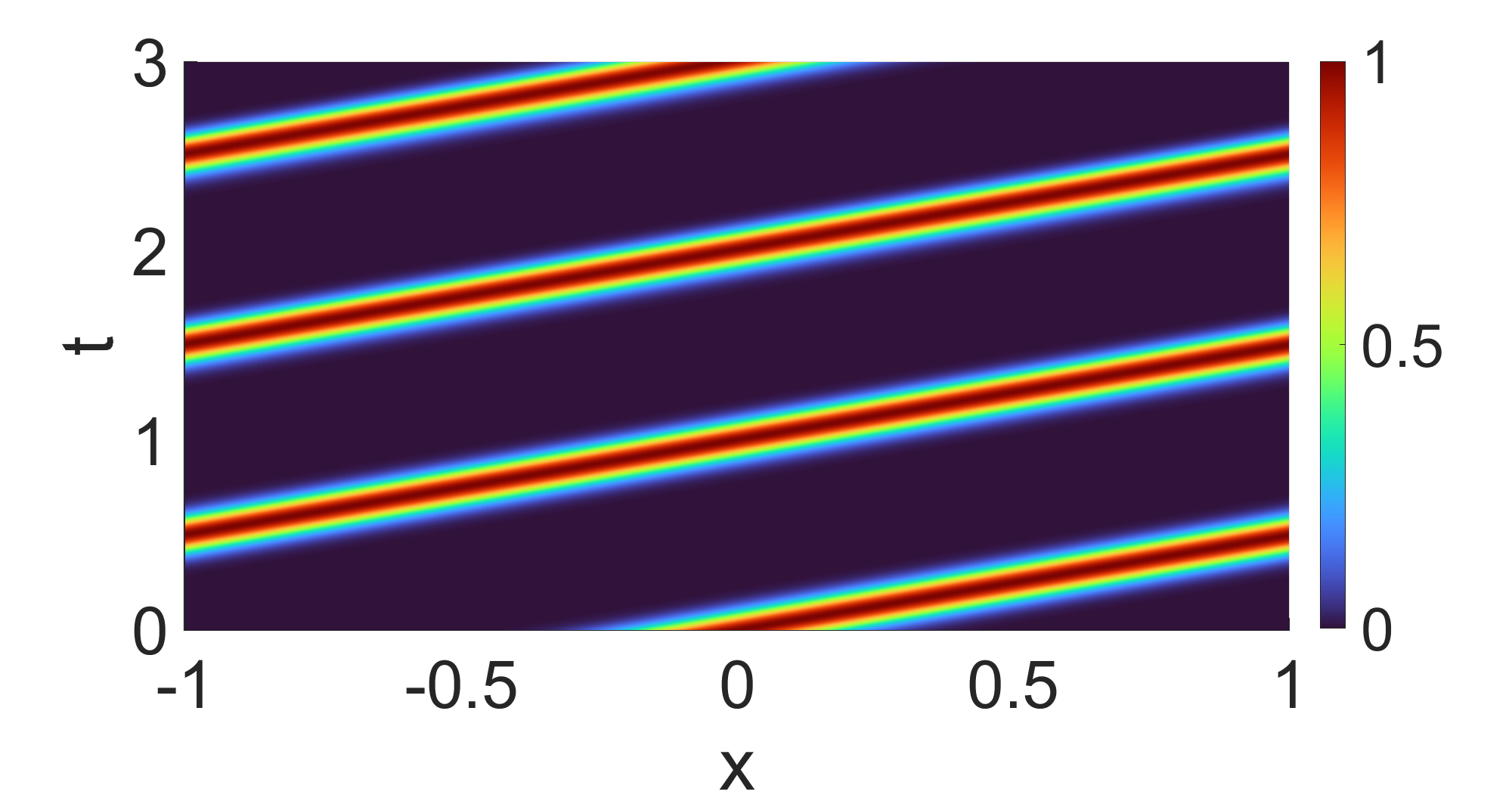} }}
    \subfloat{{\includegraphics[height=3.5cm,width=4cm]{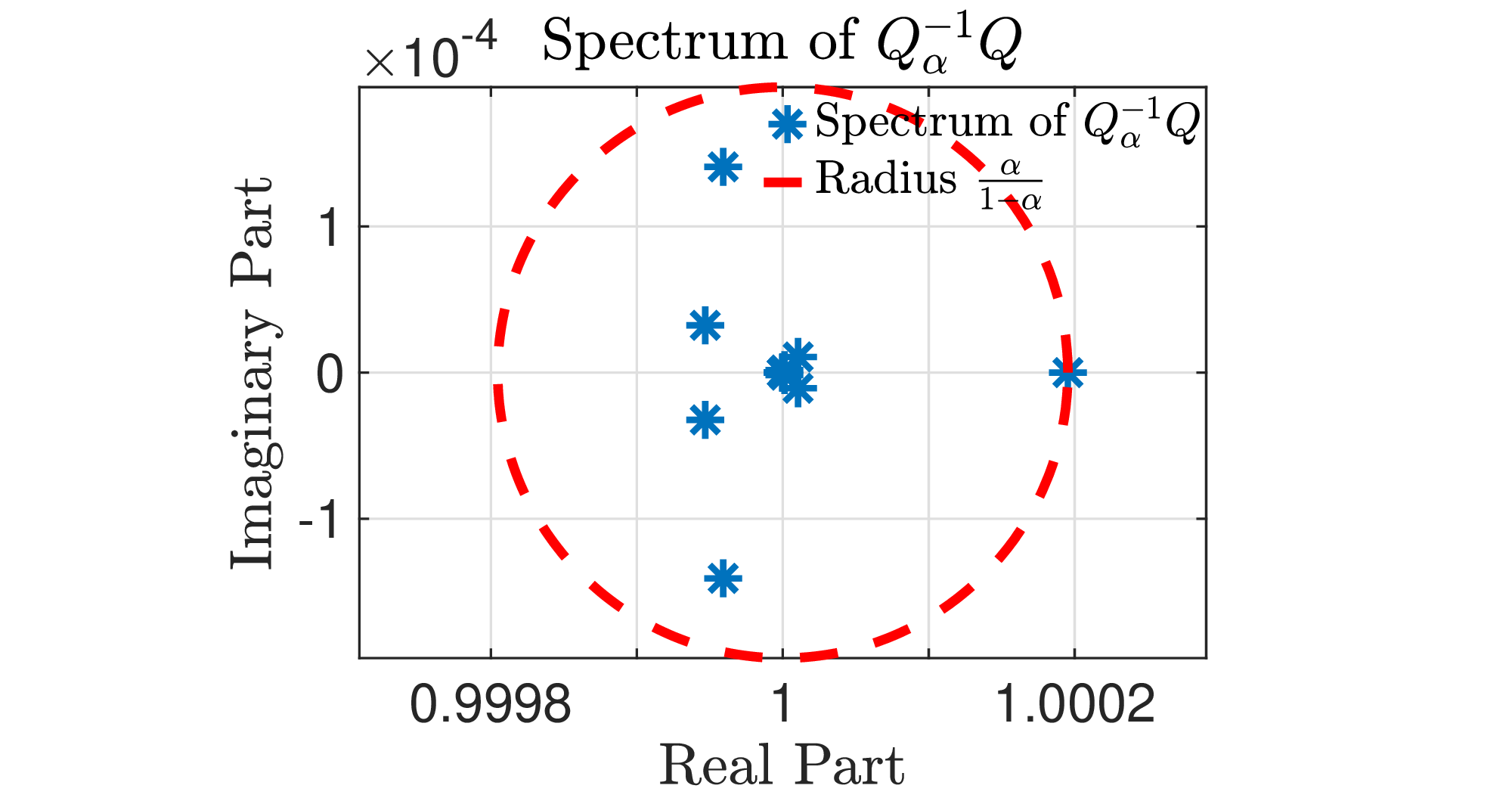} }}
    \caption{ First: convergence Exp-ParaDiag with $a=0.1$; Second: convergence of Exp-ParaDiag with $a=0.00001$; Third: Solution of ADE with $a=0.00001, h=\Delta t=1/128$ and $T=3$; Fourth: spectrum of the preconditioned system with $a=0.1$.}
    \label{ad_diff_t}
\end{figure}
In the first two plots of Figure~\ref{ad_diff_t}, we present the convergence behavior of the Exp-ParaDiag method applied to the ADE, using \( h = \Delta t = 1/128 \) and \( \alpha = \alpha_{\text{opt}} \), for varying time window sizes. The tests transition from moderate to low diffusion, i.e., an advection-dominated regime. We observe that the method exhibits rapid convergence even when advection is strongly dominant. In the third plot, we display the solution profile of the ADE up to \( T = 3 \), computed using the Exp-ParaDiag method with \( \alpha = \alpha_{\text{opt}} \).
\begin{figure}[h!]
    \centering
    \subfloat{{\includegraphics[height=3.5cm,width=4.5cm]{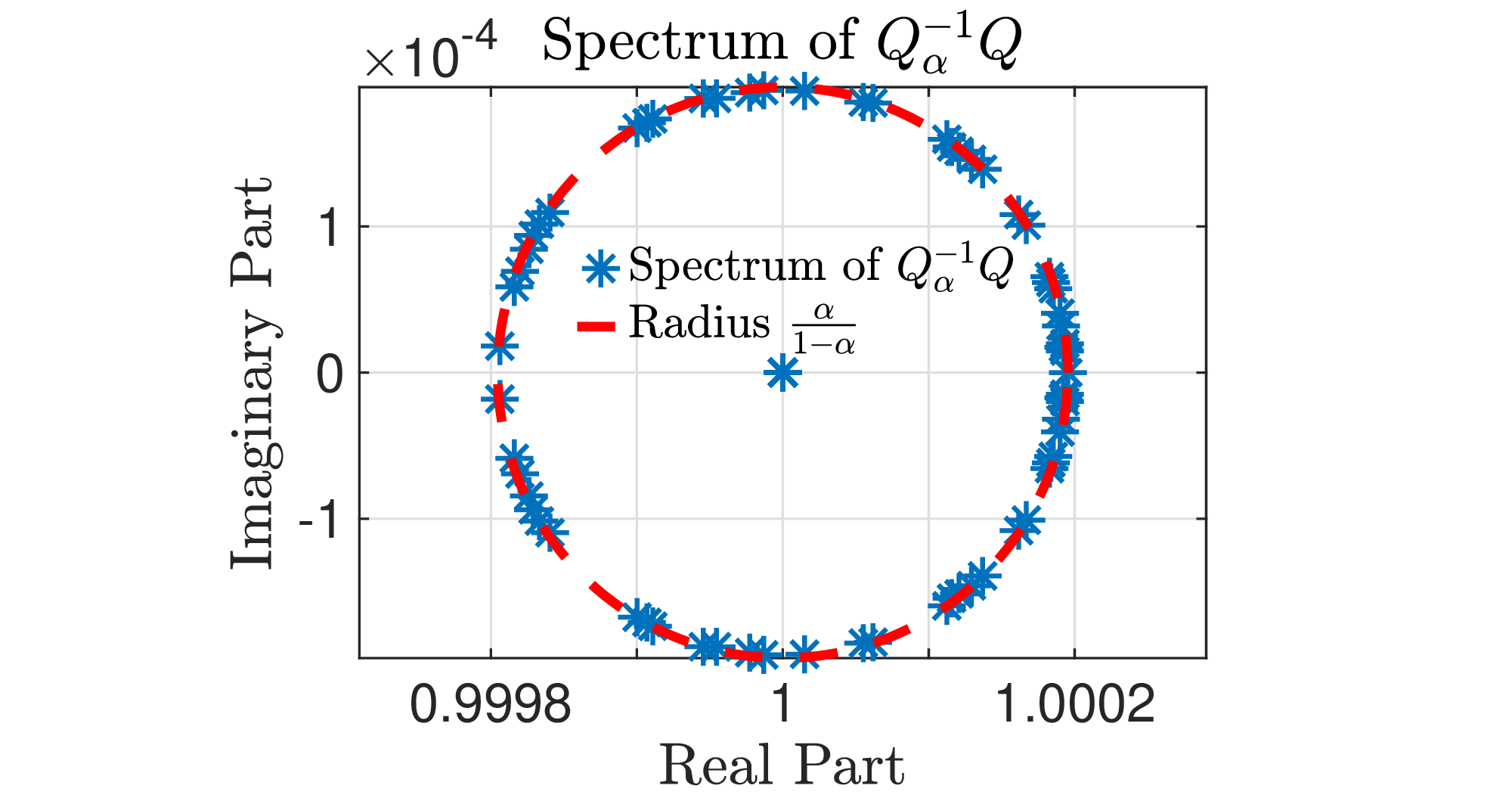} }}
    \subfloat{{\includegraphics[height=3.5cm,width=4cm]{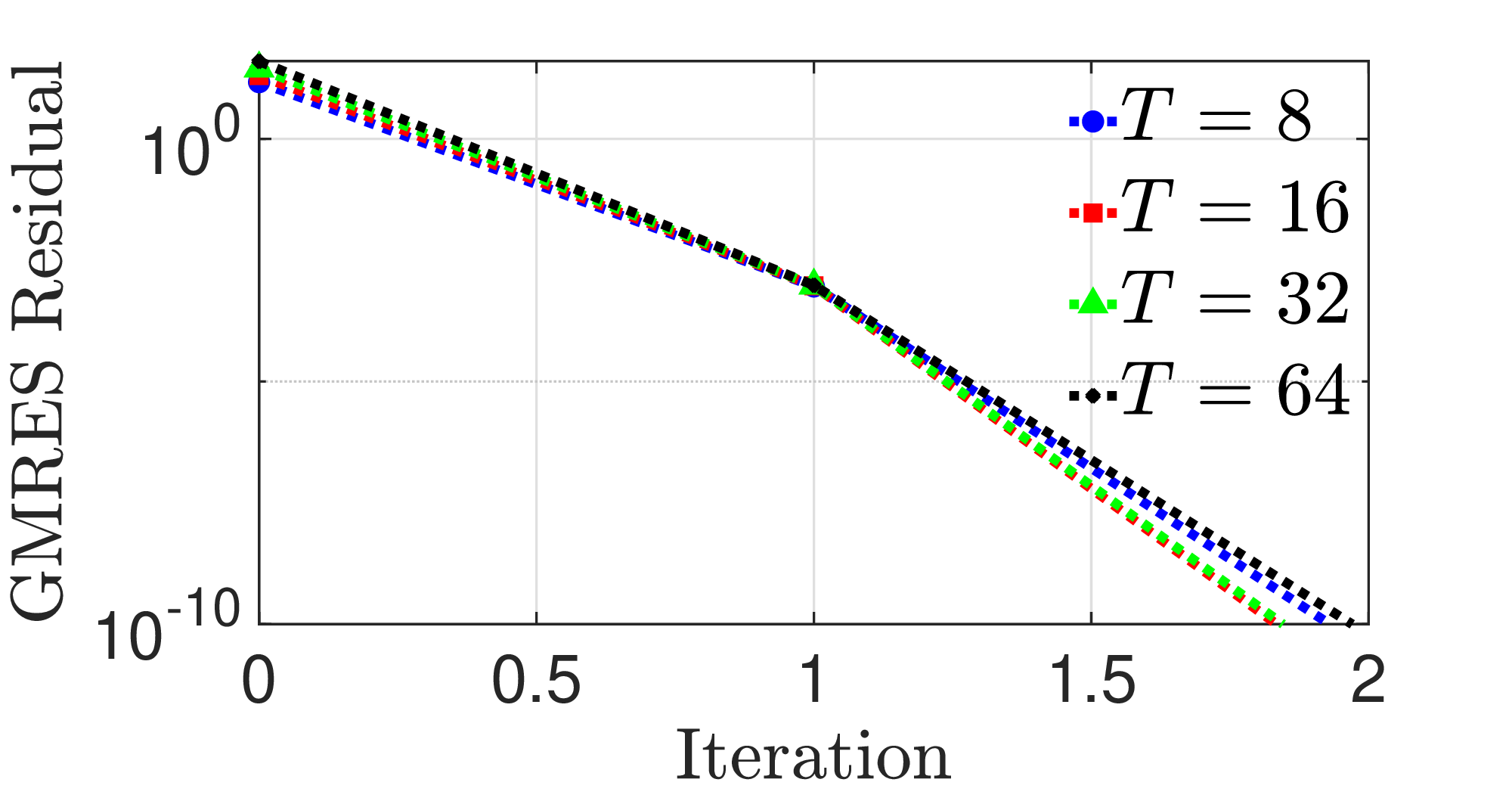} }}
    \subfloat{{\includegraphics[height=3.5cm,width=4cm]{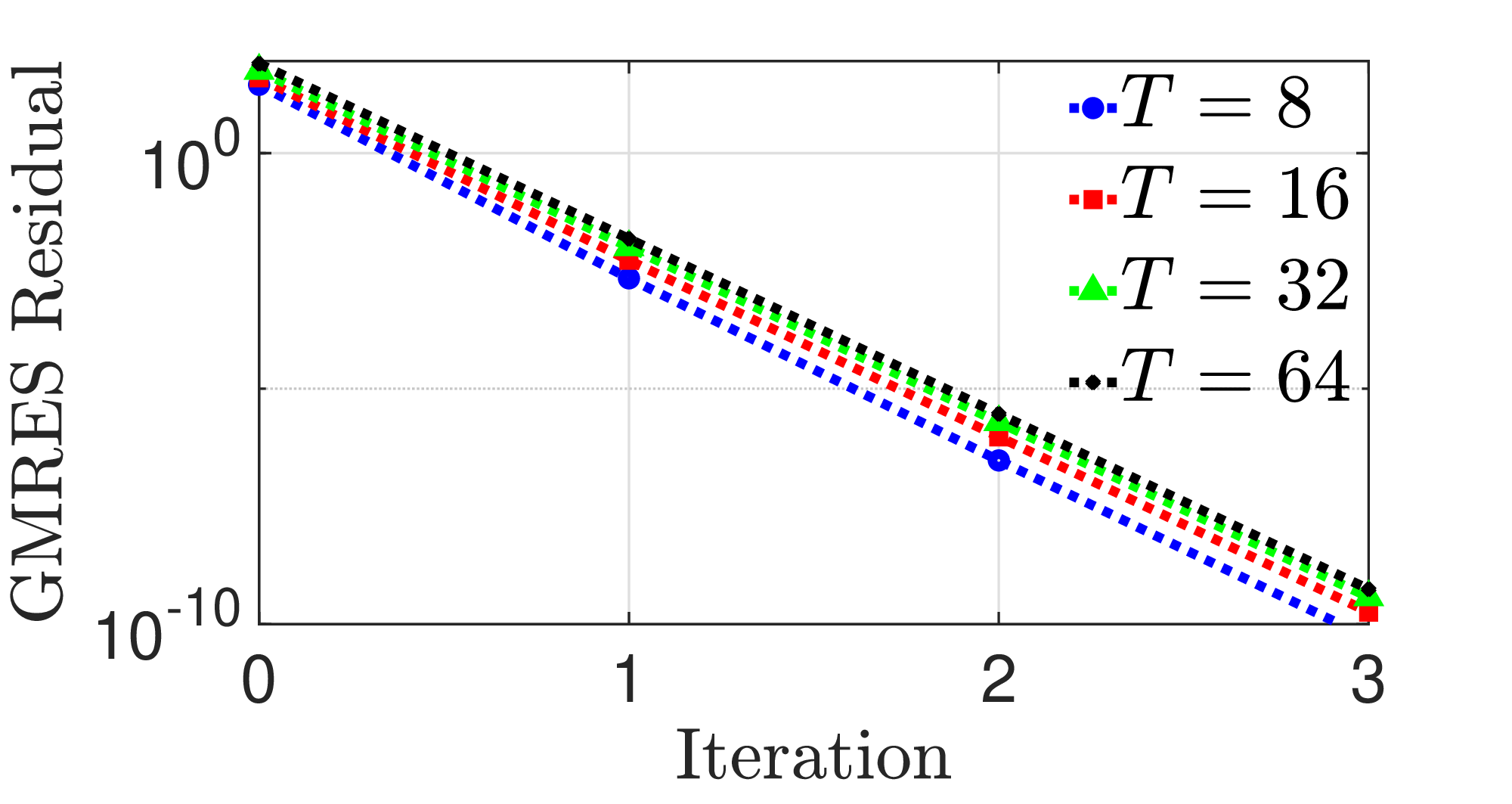} }}
    \caption{ First: spectrum of the preconditioned system with $a=0.00001$; Second: convergence of Exp-ParaDiag with $a=0.1$; Third: convergence of Exp-ParaDiag with $a=0.00001$.}
    \label{ad_gmres_diff_t}
\end{figure}
In the rightmost plot of Figure~\ref{ad_diff_t} and the leftmost plot of Figure~\ref{ad_gmres_diff_t}, we display the spectrum of \( \mathcal{Q}_{\alpha}^{-1} \mathcal{Q} \) for various diffusion coefficients under periodic boundary conditions, using \( N_t = 30 \), \( N_x = 65 \), and \( \alpha = \alpha_{\text{opt}} \). As noted in Remark~\ref{gmres_ade}, the quantity \( \frac{\alpha}{1 - \alpha} \) clearly acts as an upper bound for the spectral distribution. 
In the second and third plots of Figure~\ref{ad_gmres_diff_t}, we display the convergence behavior of the preconditioned GMRES method for varying time window sizes and diffusion coefficients with $h=1/128, \Delta t=0.01$ and $\alpha=\alpha_{\text{opt}}$. The results show that convergence is fast and remains independent of the time window size. This behavior is consistent with the spectral analysis in Figure~\ref{ad_diff_t}, where the spectrum of the preconditioned system \( \mathcal{Q}_{\alpha}^{-1} \mathcal{Q} \) is tightly bounded by \( \frac{\alpha}{1 - \alpha} \), as noted in Remark~\ref{gmres_ade}. The clustering of eigenvalues within this bound explains the rapid and robust convergence of GMRES across different parameter settings. 
Similar convergence behavior was observed for the Exp-ParaDiag method with BDF schemes of orders two to six; the details are omitted for brevity.

\subsubsection{Experiment for $\mathcal{O}((\Delta t)^s)$ Scheme for $3\leq s\leq 6$}
In this section, we present numerical experiments for the Exp-ParaDiag method based on BDF schemes of orders three to six, as developed in Section~\ref{bdfs}, applied to the purely diffusive case (\( c = 0 \)) of \eqref{model_problem_linear} in one spatial dimension. For all experiments, we fix \( \Delta t = 0.05 \), \( h = 1/128 \), and \( \alpha = \alpha_{\text{opt}} \). We examine the method's performance over various time window lengths and diffusion coefficients. The first two plots of Figure~\ref{gmres_1d_bdf3_bdf4} display the convergence behavior of Exp-ParaDiag based on the BDF3 scheme, while the last two plots depict the convergence for the BDF4 scheme. In all cases, the method exhibits rapid convergence.
 \begin{figure}[h!]
    \centering
    \subfloat{{\includegraphics[height=3.5cm,width=4cm]{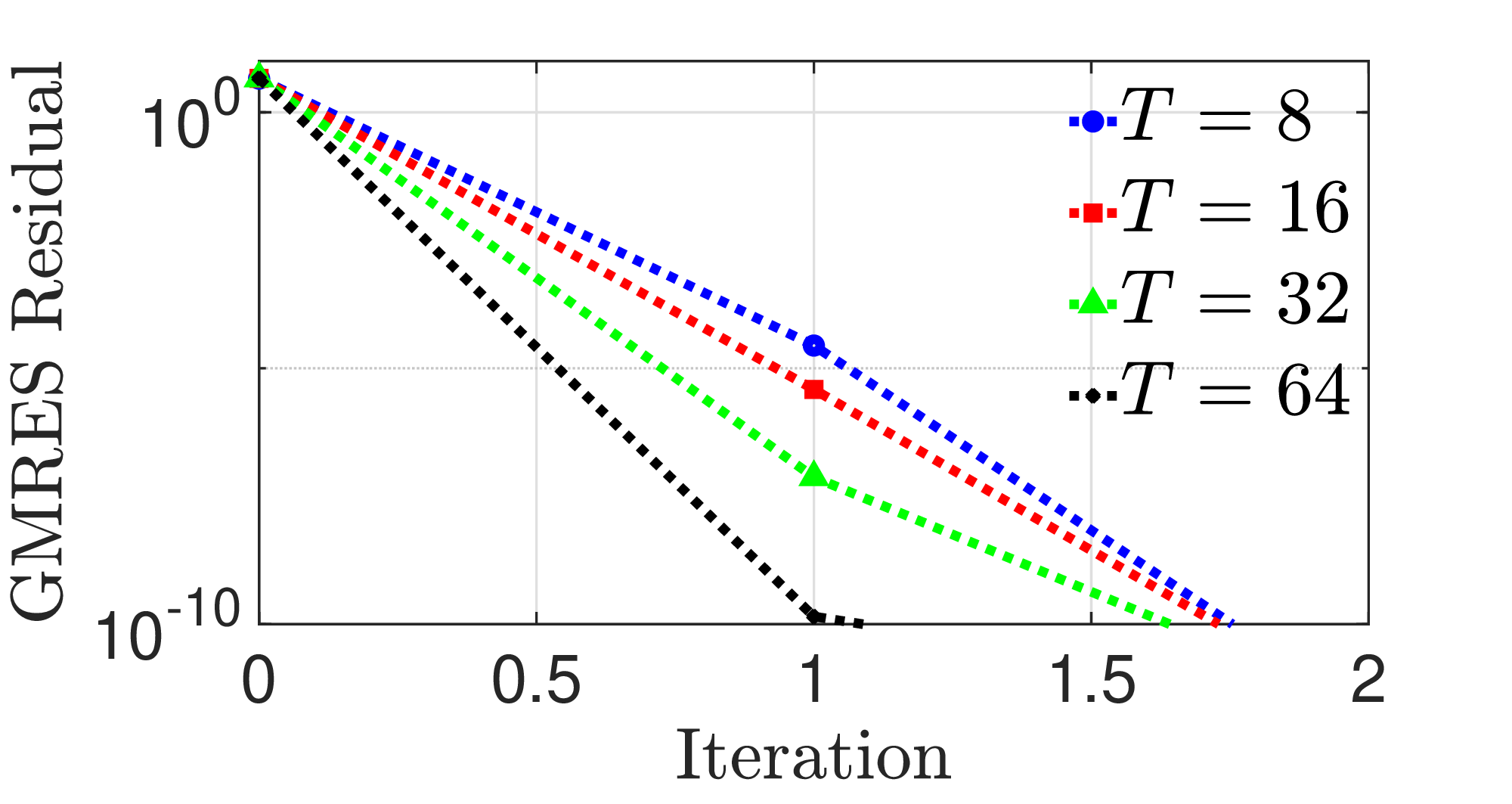} }}
    \subfloat{{\includegraphics[height=3.5cm,width=4cm]{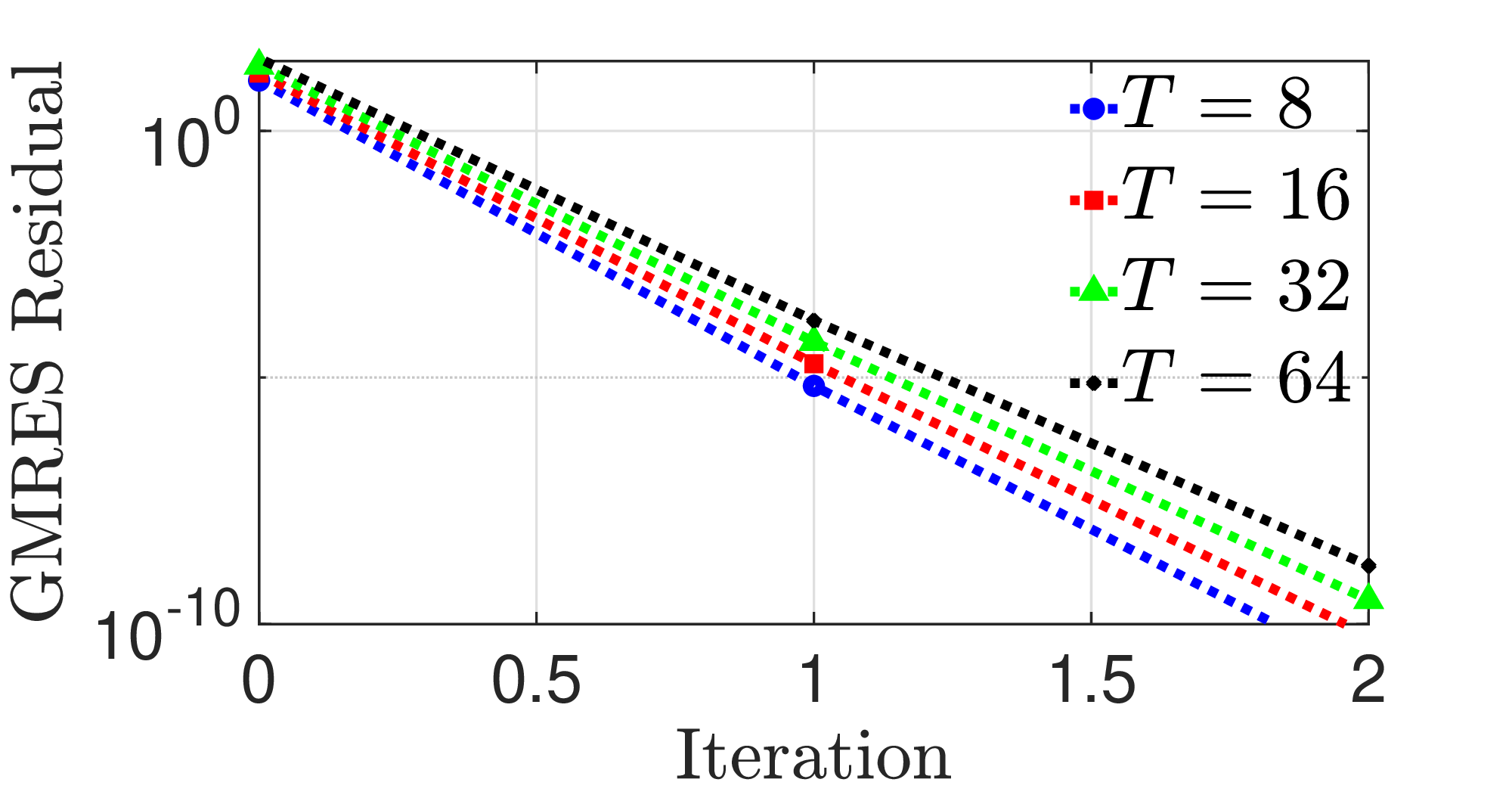} }}
    \subfloat{{\includegraphics[height=3.5cm,width=4cm]{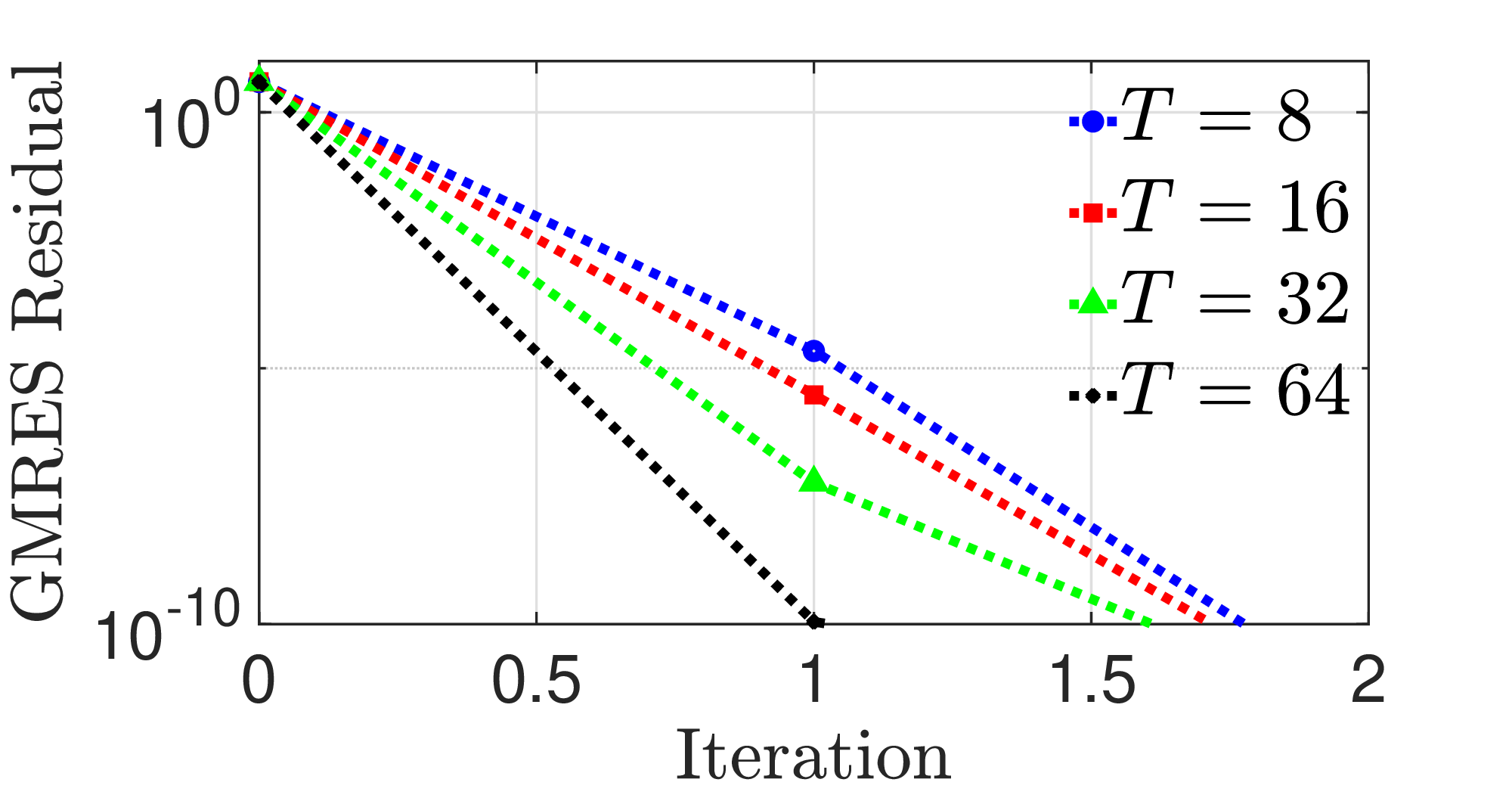} }}
    \subfloat{{\includegraphics[height=3.5cm,width=4cm]{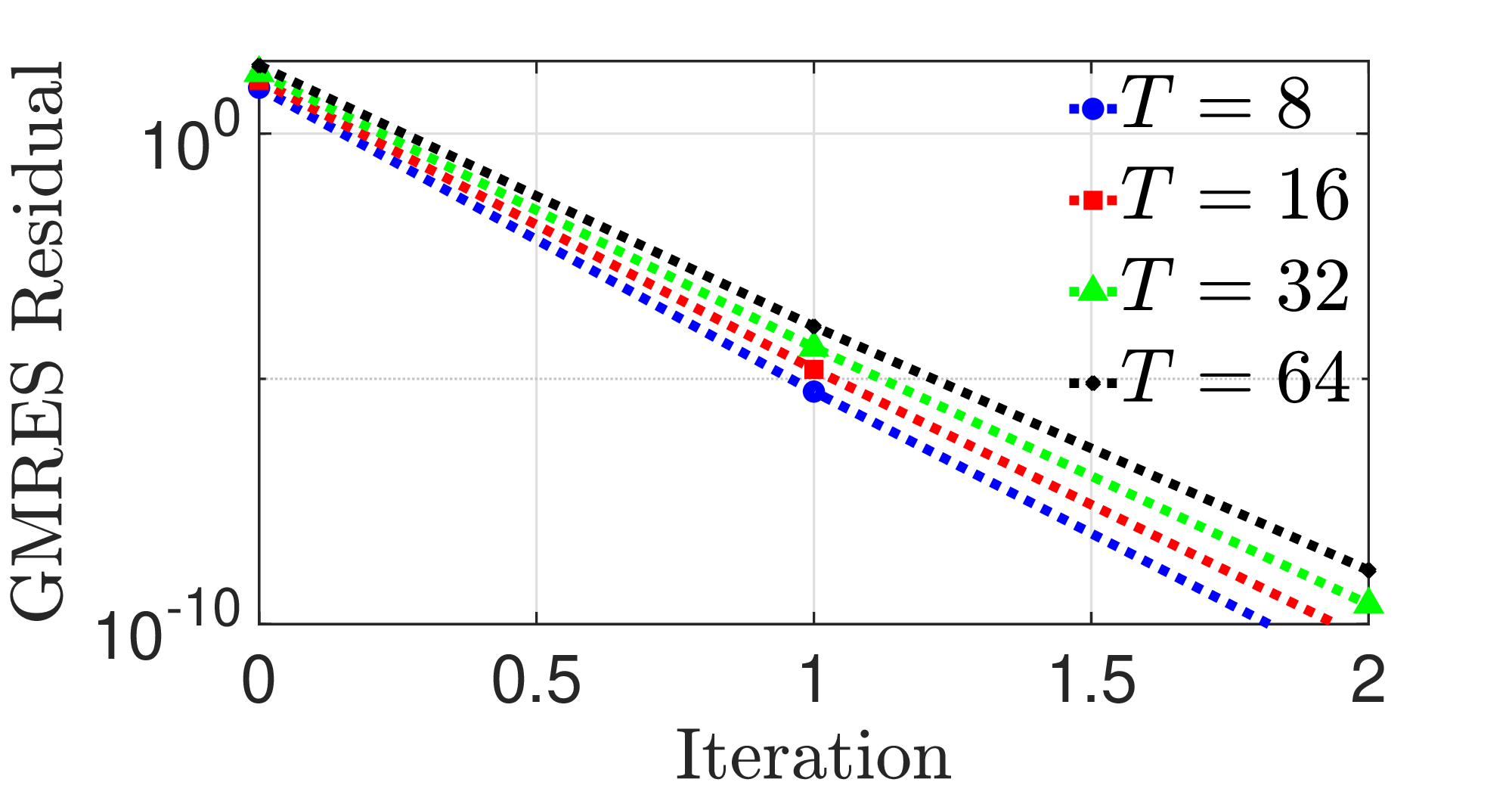} }}
    \caption{ First: Exp-ParaDiag for BDF3 with $a=0.1$; Second: Exp-ParaDiag for BDF3 with $a=0.00001$; Third: Exp-ParaDiag for BDF4 with $a=0.1$; Fourth: Exp-ParaDiag for BDF4 with $a=0.00001$.}
    \label{gmres_1d_bdf3_bdf4}
\end{figure}
The first two plots of Figure~\ref{gmres_1d_bdf5_bdf6} display the convergence trends of the Exp-ParaDiag method implemented with the BDF5 scheme, while the last two plots correspond to the results obtained using the BDF6 scheme. In all scenarios, the method converges efficiently and rapidly.
\begin{figure}[h!]
    \centering
    \subfloat{{\includegraphics[height=3.5cm,width=4cm]{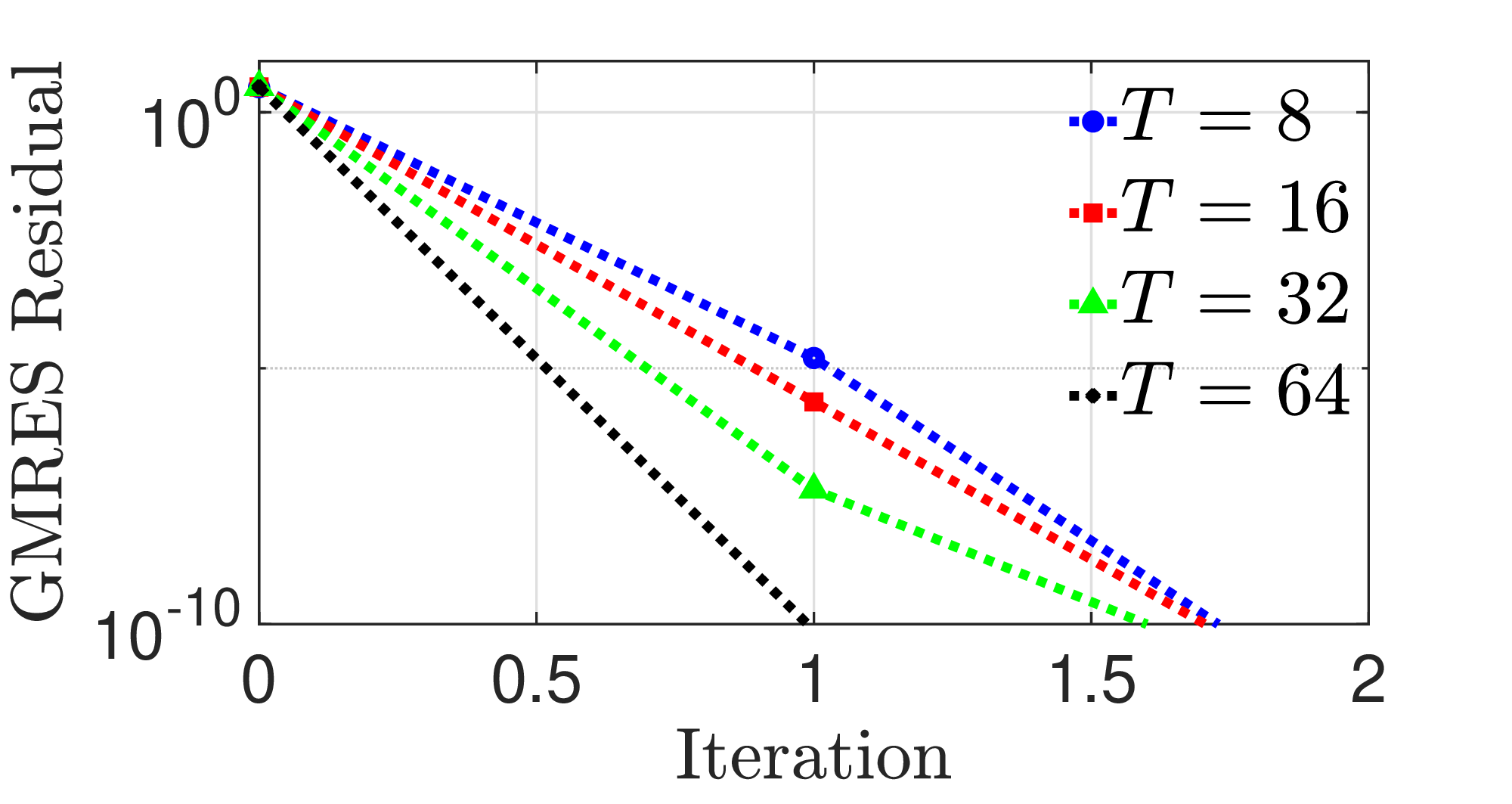} }}
    \subfloat{{\includegraphics[height=3.5cm,width=4cm]{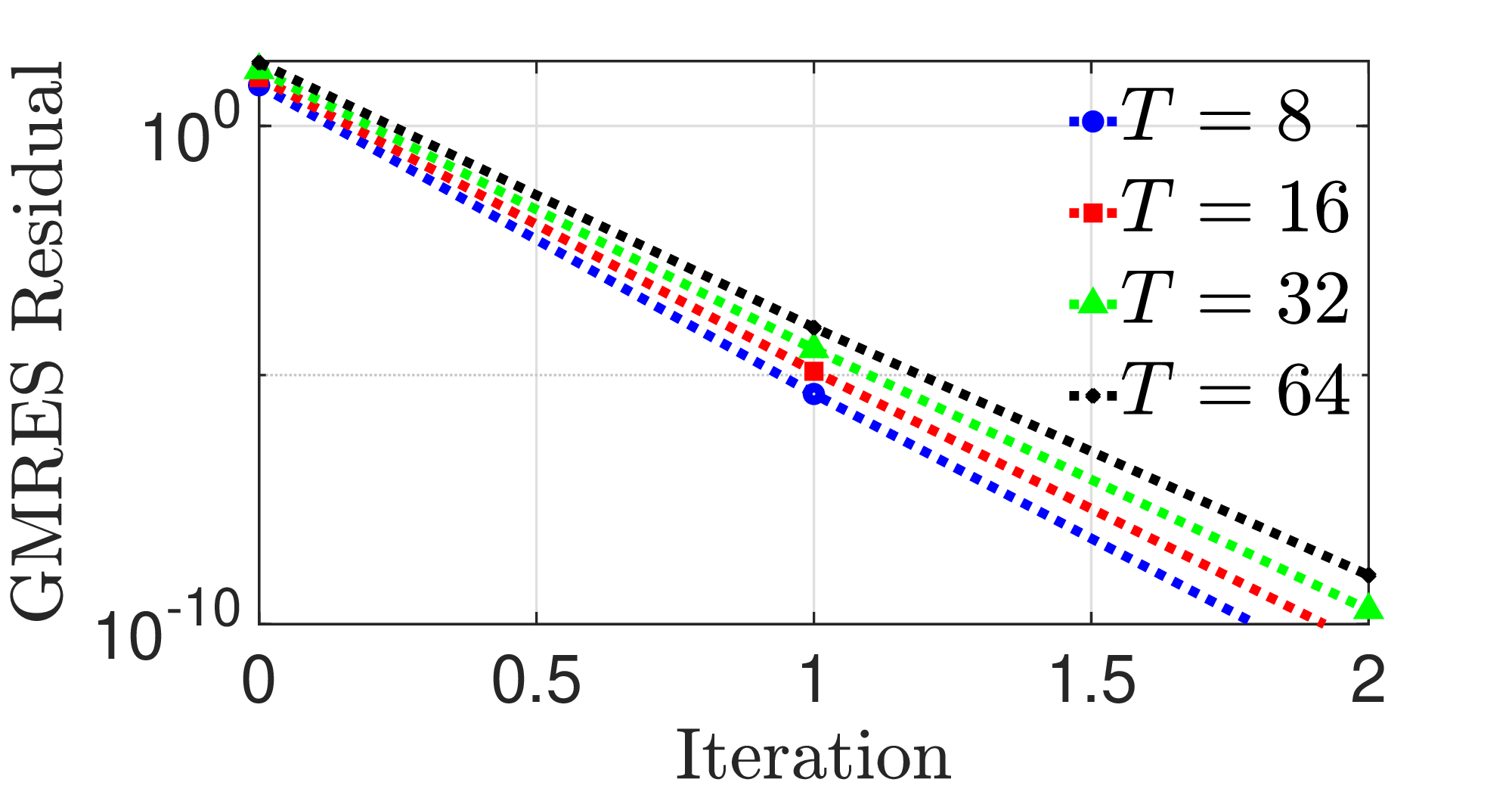} }}
    \subfloat{{\includegraphics[height=3.5cm,width=4cm]{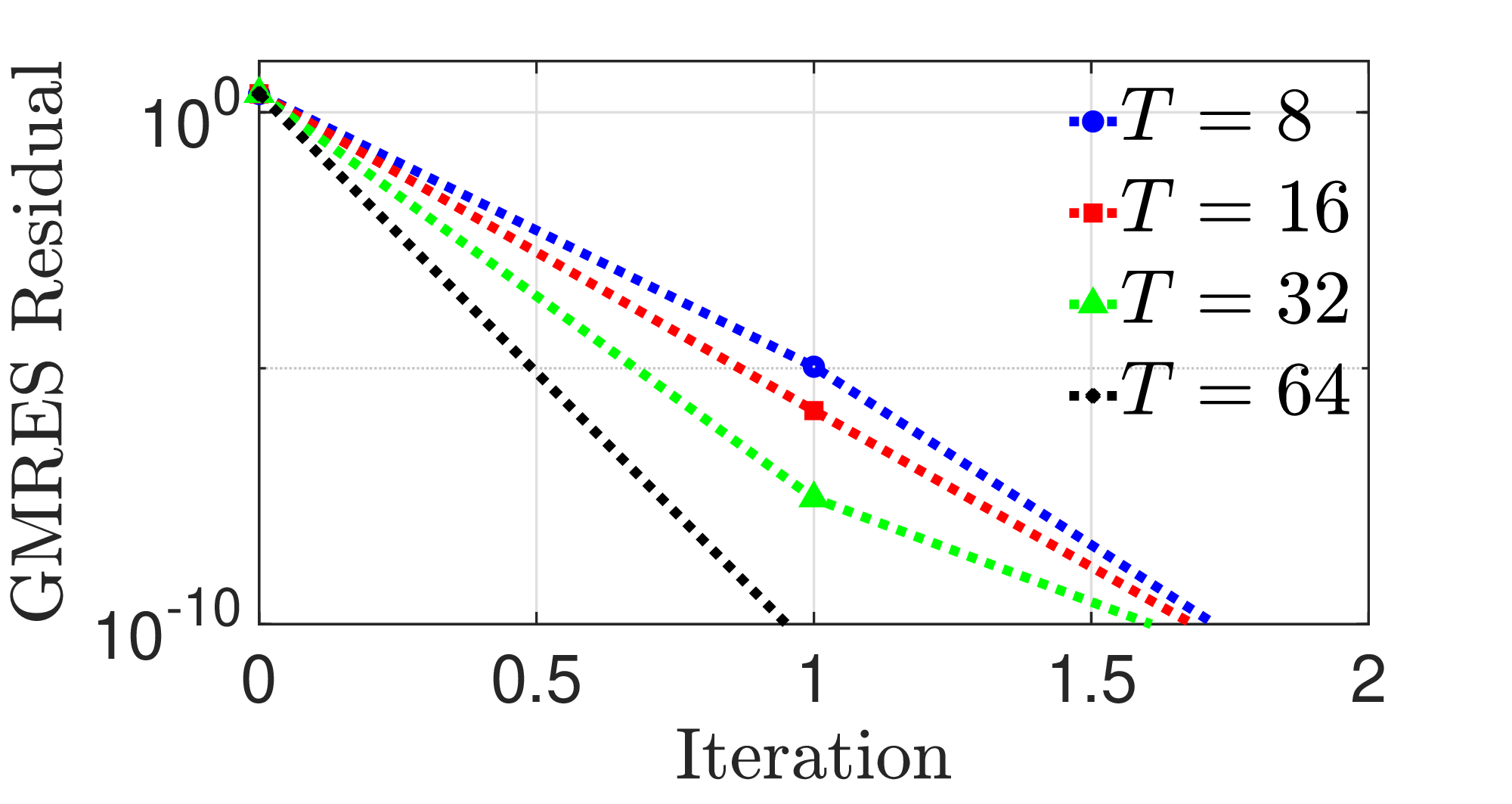} }}
    \subfloat{{\includegraphics[height=3.5cm,width=4cm]{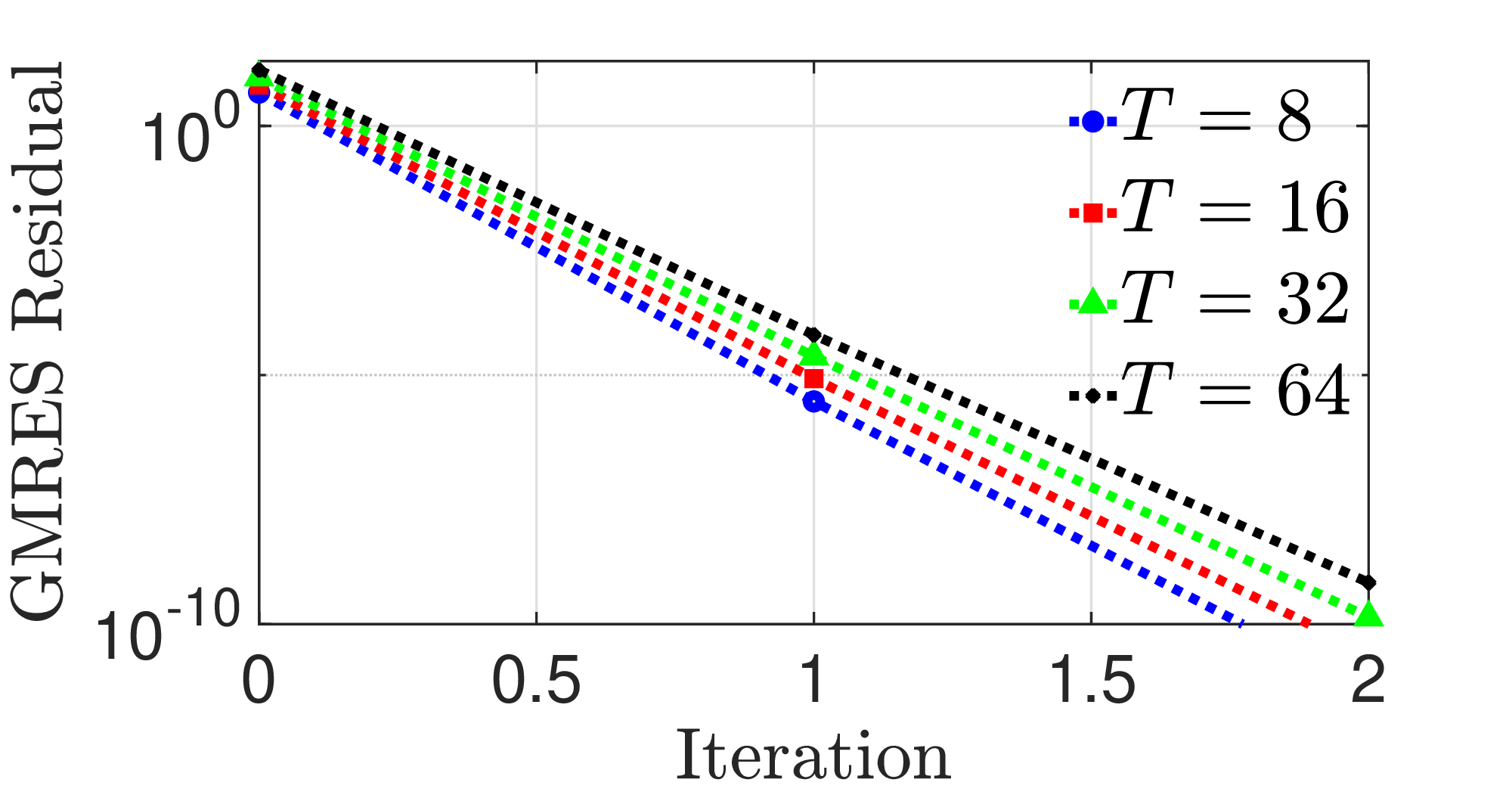} }}
    \caption{ First: Exp-ParaDiag for BDF5 with $a=0.1$; Second: Exp-ParaDiag for BDF5 with $a=0.00001$; Third: Exp-ParaDiag for BDF6 with $a=0.1$; Fourth: Exp-ParaDiag for BDF6 with $a=0.00001$.}
    \label{gmres_1d_bdf5_bdf6}
\end{figure}
The proposed Exp-ParaDiag method using higher-order BDF schemes exhibits mesh-independent behavior. Similar convergence patterns are observed in 2D, so we omit the details here (as the convergence graph is similar).

\subsection{Experiment in Non-linear Settings}
In this section, we show some numerical results for the Exp-ParaDiag method applied to a few nonlinear PDEs. We start with the Allen-Cahn equation, where we choose a source term $s_f(x,t)$ so that the exact solution is $u(x,t) = 0.5 e^{-t} \cos(2\pi x)$. The equation looks like this:
\begin{equation}\label{ac_eq}
\begin{cases}
    u_t=\epsilon^2\Delta u - (u^3-u) + s_f(x,t), \text{in}\; \Omega=(0, 1),\; \text{and}\; \epsilon=0.01,\\
    \frac{\partial u}{\partial n}=0, \;\text{on} \; \partial\Omega, \text{and}\;
    u(x, 0)=0.5\cos(2\pi x)\; \text{for}\; x\in \Omega.
    \end{cases}
\end{equation}
In Figure~\ref{sol_ac}, we plot the exact solution, the Exp-ParaDiag solution, and the pointwise error at $T = 0.1$ after the first iteration using \eqref{eq:exp_ParaDiag_steps_newton}. The rightmost plot shows the pointwise error after the second iteration, computed using $\widebar{\mathcal{N}}(\mathcal{U}^k) := I_t \otimes \diag(\mathcal{N}'(\mathbf{u}_{0}))$ in \eqref{eq:exp_ParaDiag_steps_newton}.
\begin{figure}[h!]
    \centering
    \subfloat{{\includegraphics[height=3.5cm,width=4cm]{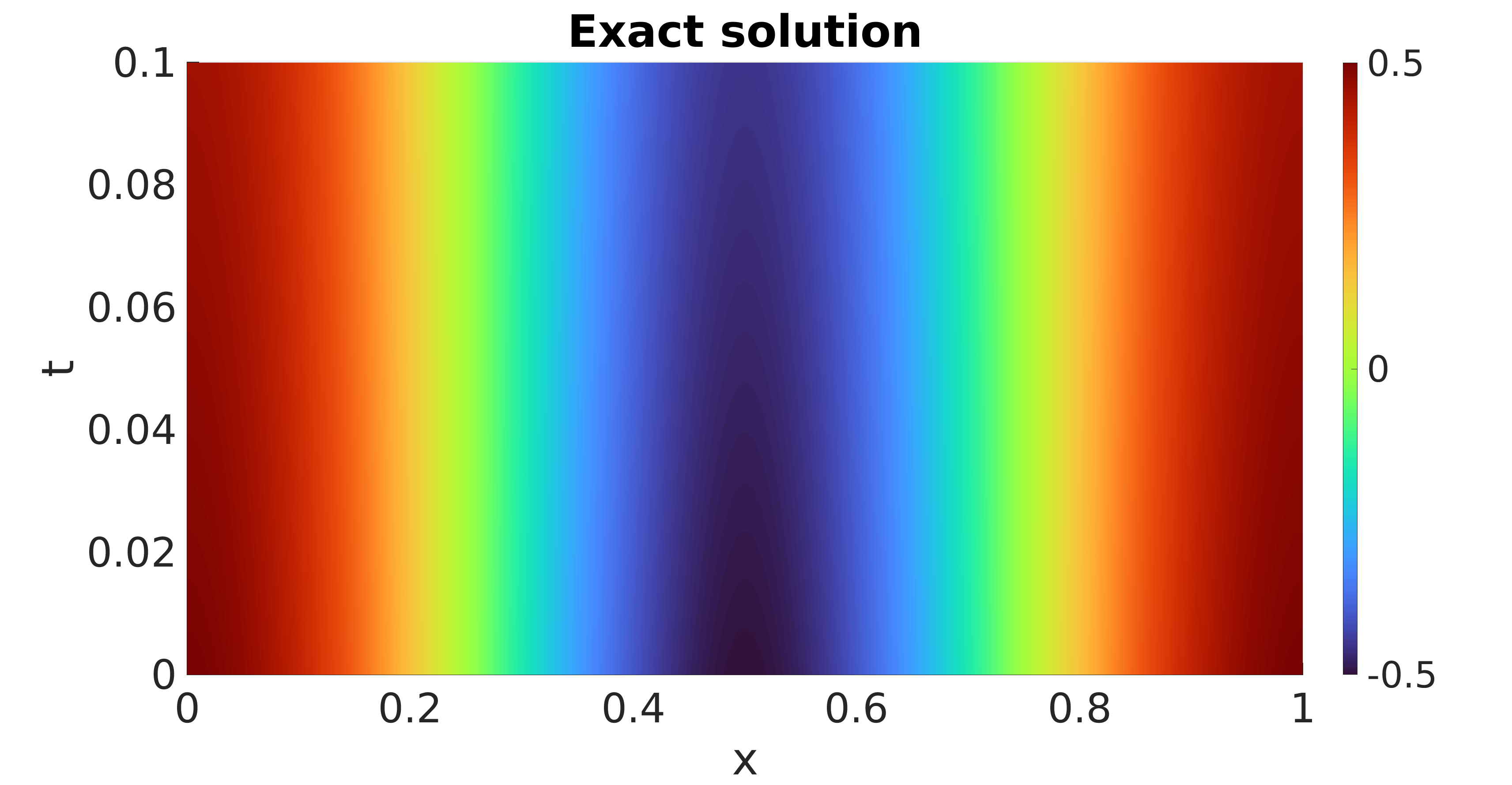} }}
    \subfloat{{\includegraphics[height=3.5cm,width=4cm]{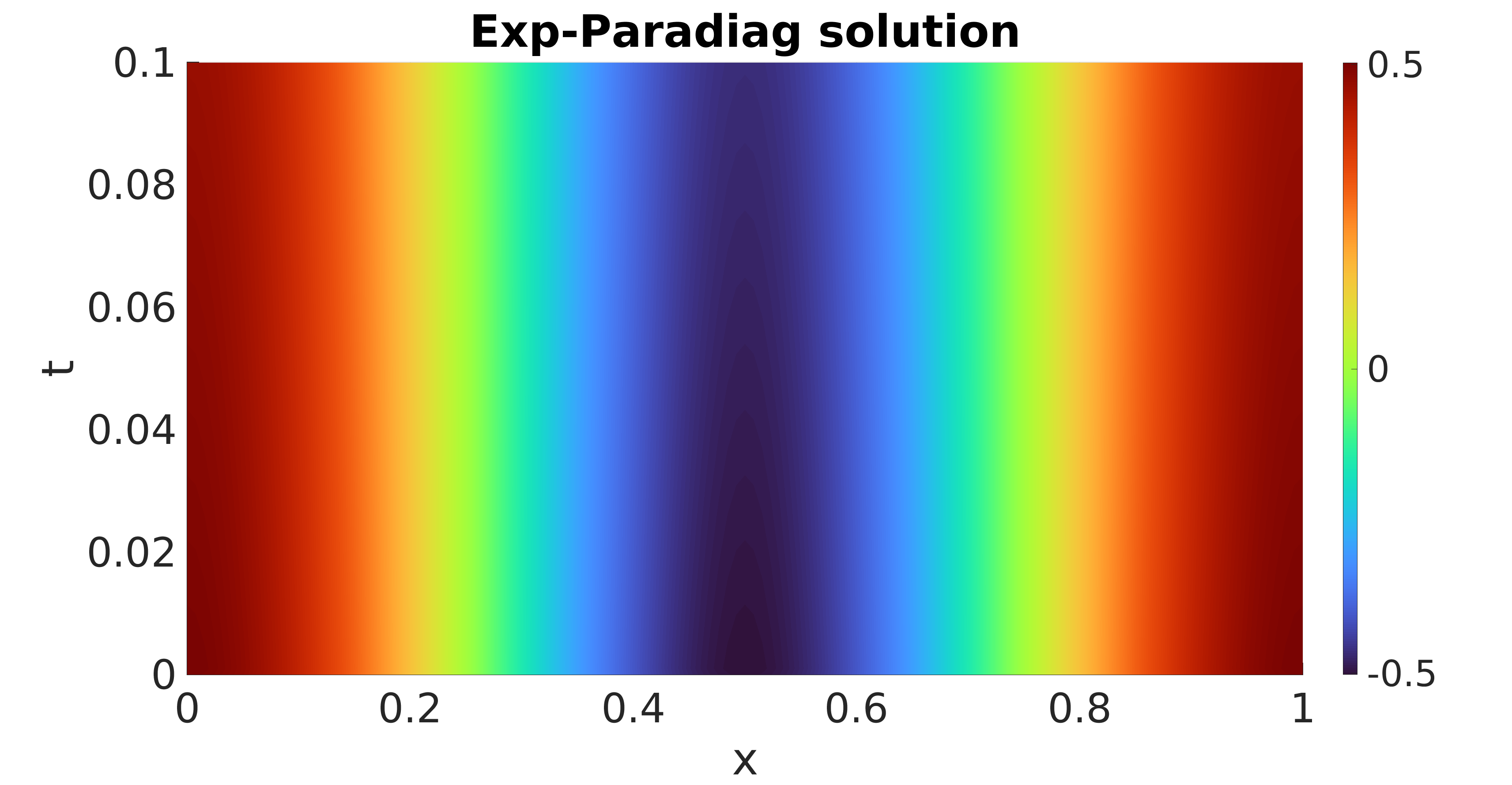} }}
    \subfloat{{\includegraphics[height=3.5cm,width=4cm]{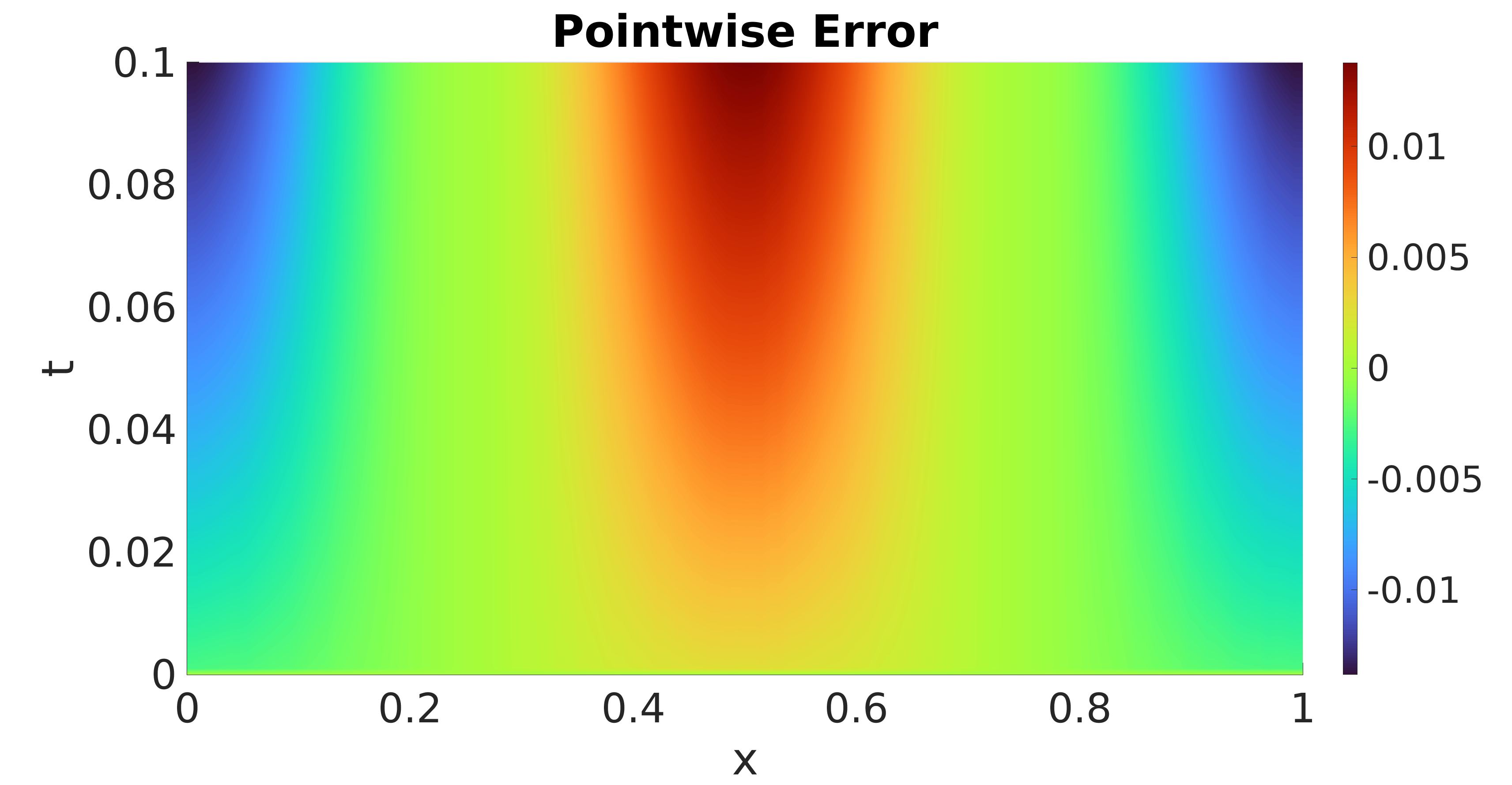} }}
    \subfloat{{\includegraphics[height=3.5cm,width=4cm]{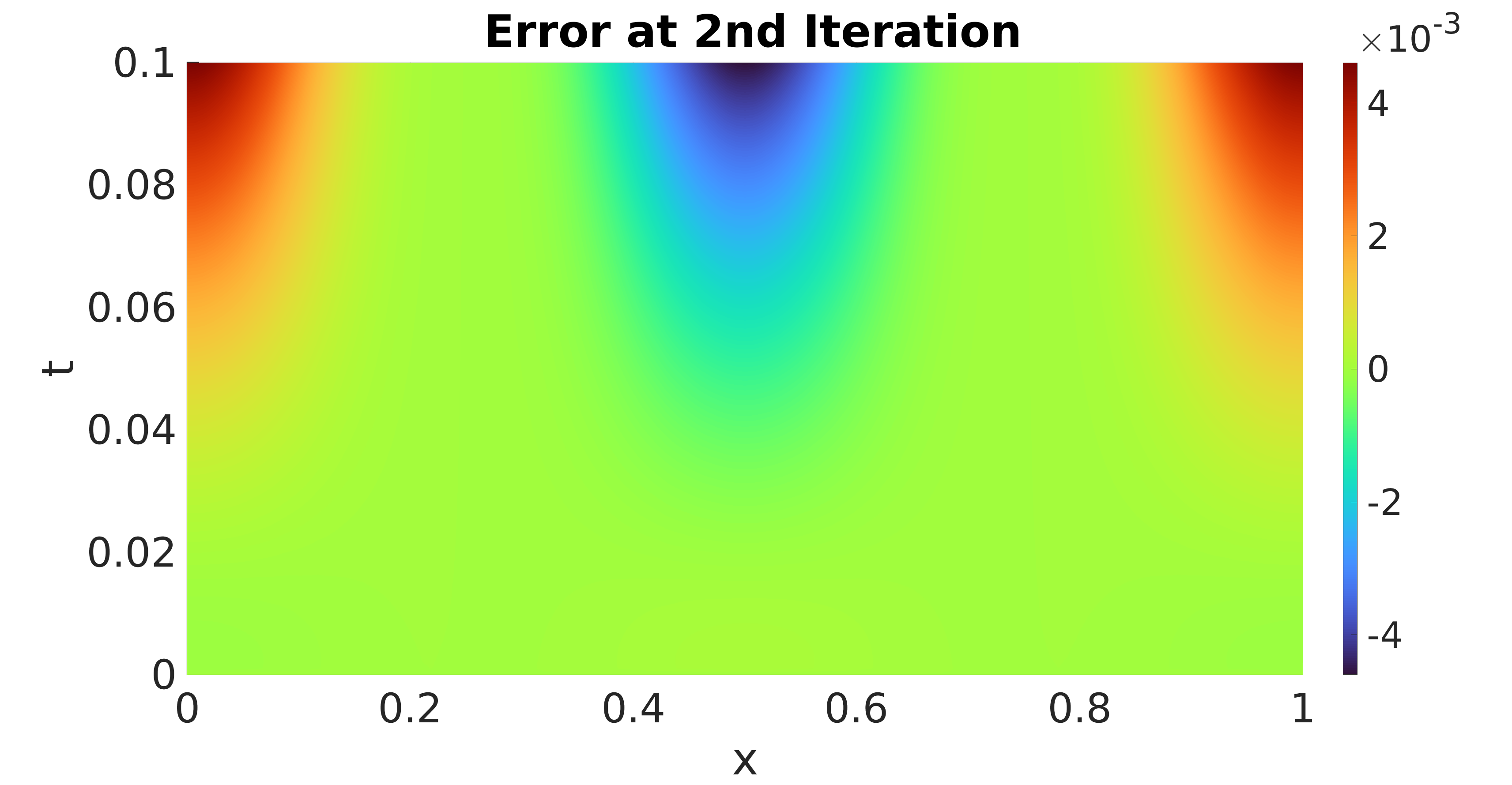} }}
    \caption{Solution and error comparision at $T=0.1$ with $h=1/128, \alpha=0.005$ and $N_t=100$. First: exact solution; Second: Exp-ParaDiag solution; Third: pointwise error after 1st iteration; Fourth: pointwise error after 2nd iteration for $\widebar{\mathcal{N}}(\mathcal{U}^k) := I_t \otimes \diag(\mathcal{N}'(\mathbf{u}_{0}))$ in \eqref{eq:exp_ParaDiag_steps_newton}.}
    \label{sol_ac}
\end{figure}
Next, we consider the following nonlinear reaction-diffusion equation:
\begin{equation}\label{nheat_pth}
\begin{cases}
    u_t=\Delta u - u^p, \text{in}\; \Omega=(0, 1),\; \text{and}\; p>1,\\
    \frac{\partial u}{\partial n}=0, \;\text{on} \; \partial\Omega, \text{and}\;
    u(x, 0)=0.1\cos(2\pi x)\; \text{for}\; x\in \Omega.
    \end{cases}
\end{equation}
The first two plots in Figure \ref{conv_1st_example} show the convergence of Exp-ParaDiag under different Jacobian approximations and varying levels of nonlinearity.
\begin{figure}[h!]
    \centering
    \subfloat{{\includegraphics[height=3.5cm,width=4.5cm]{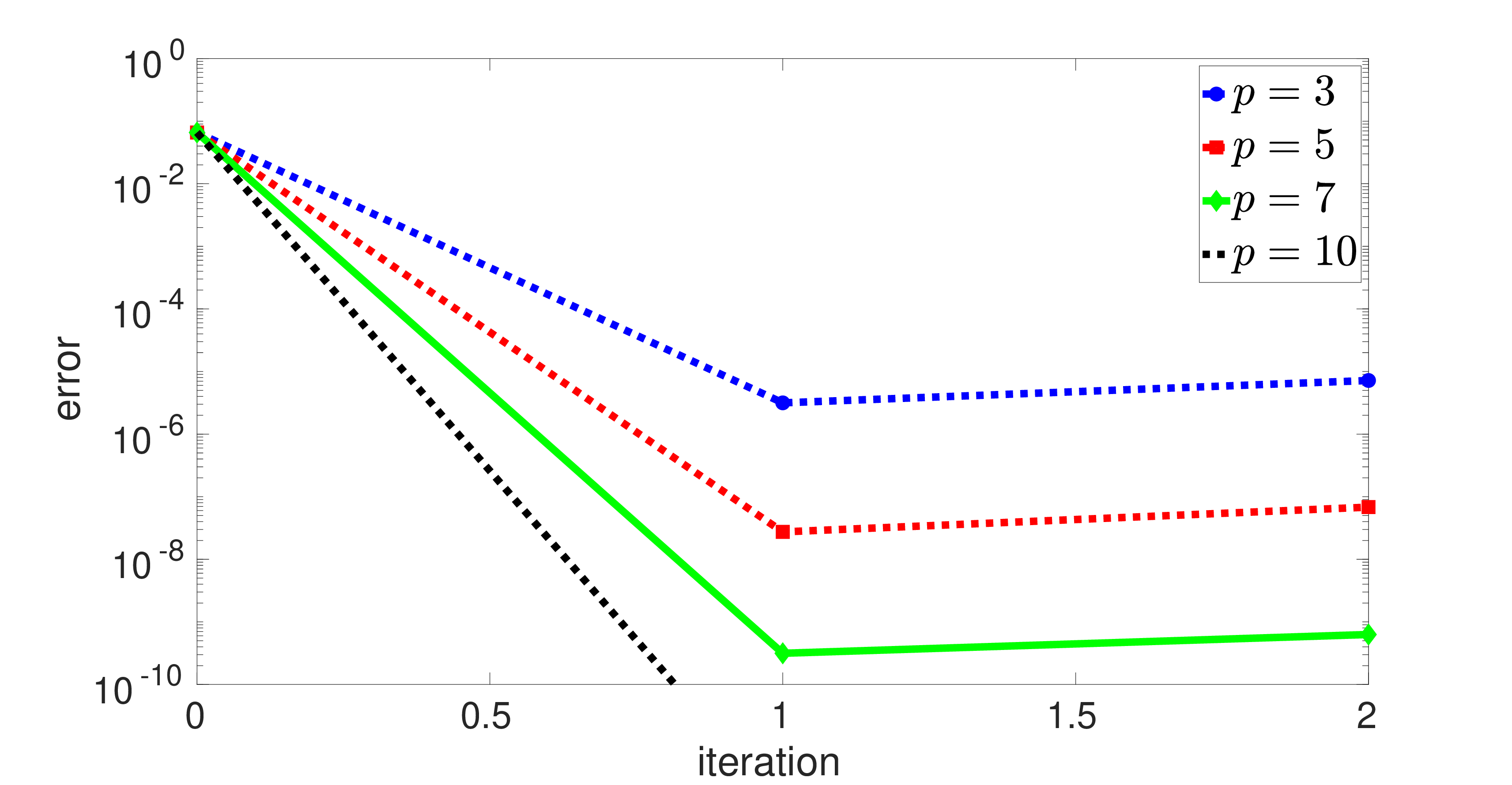} }}
    \subfloat{{\includegraphics[height=3.5cm,width=4.5cm]{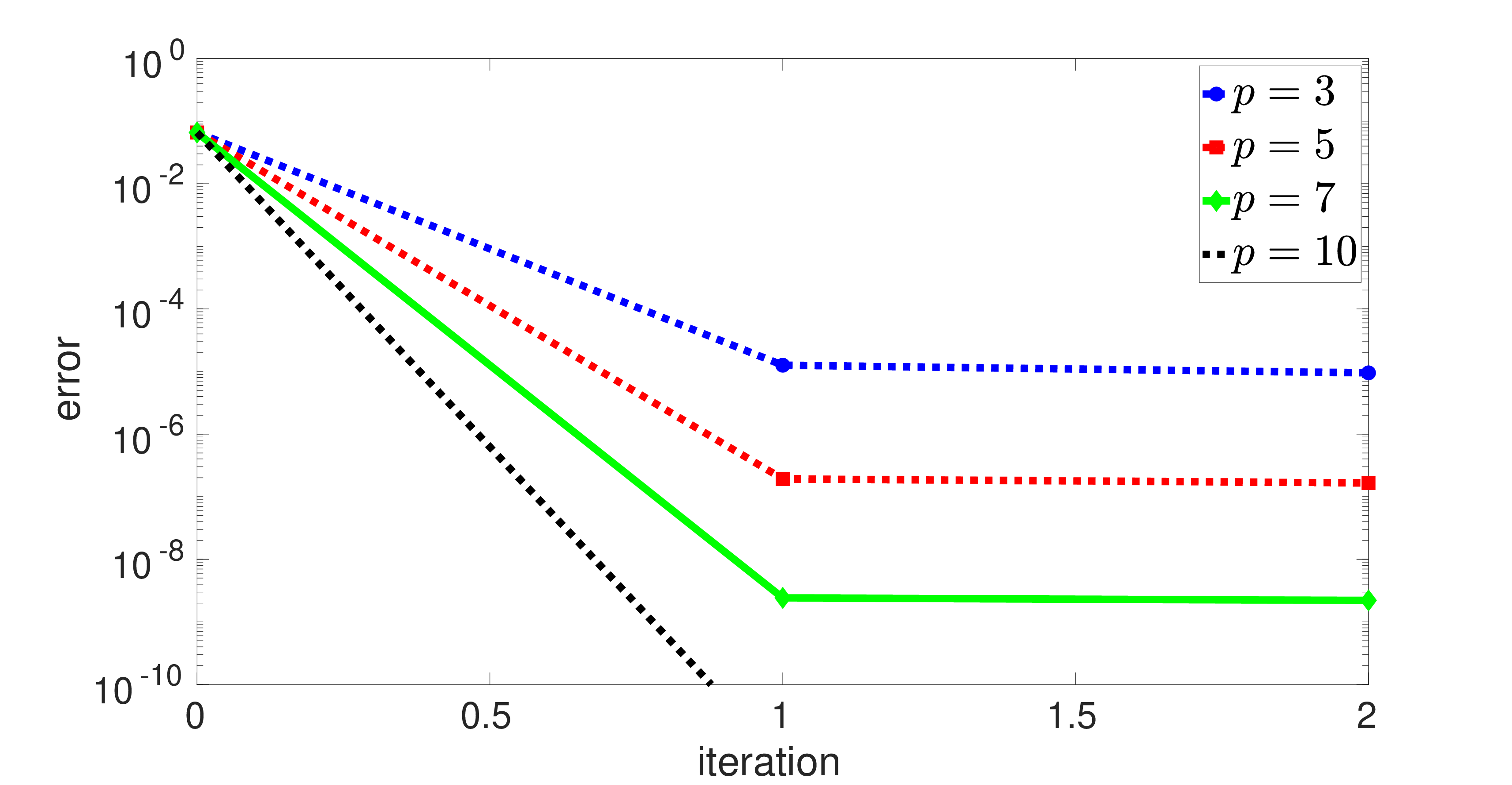} }}
    \subfloat{{\includegraphics[height=3.5cm,width=4.5cm]{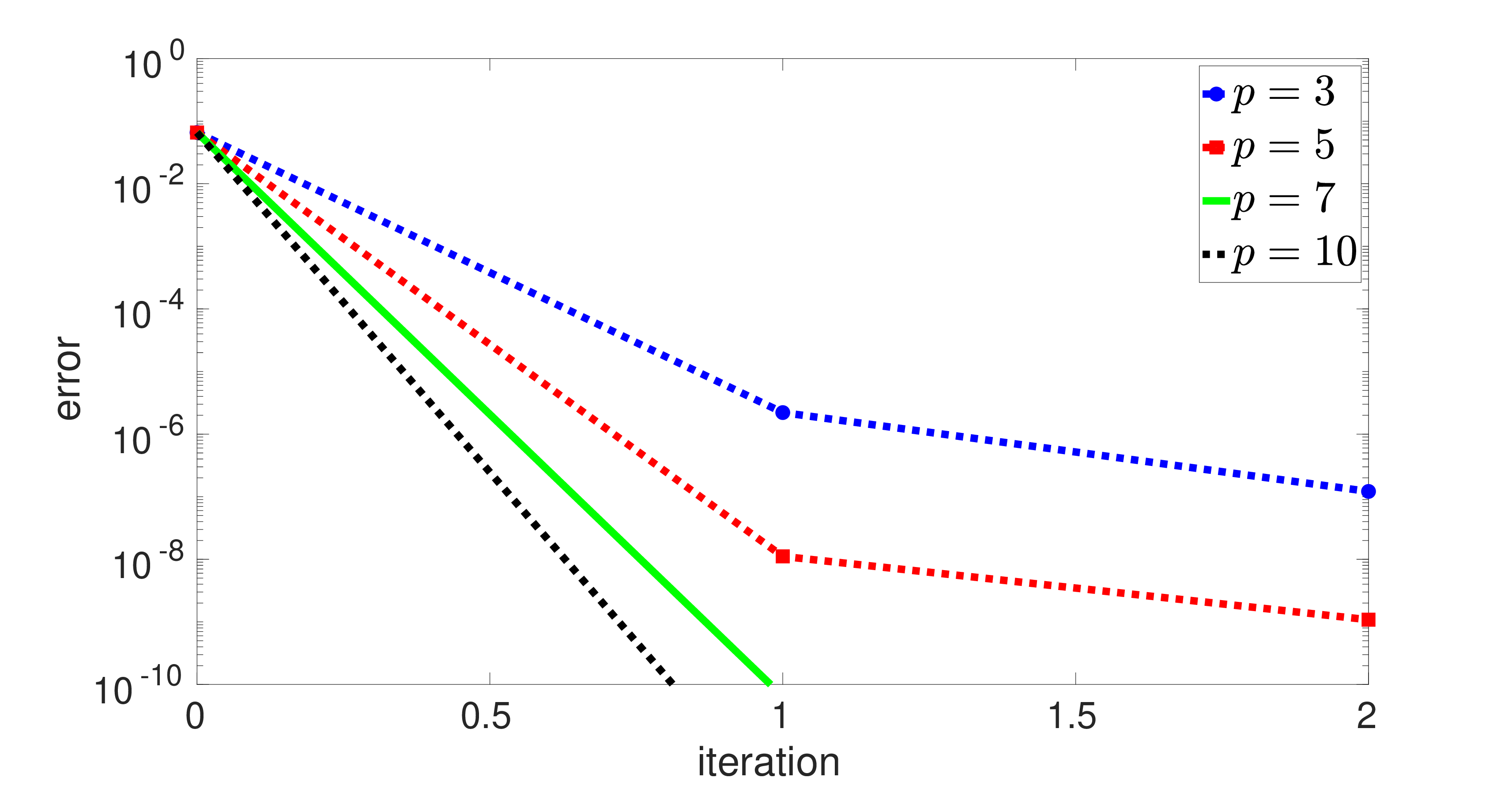} }}
    \caption{Convergence of Exp-ParaDiag for \eqref{nheat_pth} with $h=1/128$, $N_t=500$, $\alpha=0.005$, and $T=0.5$. Left: scheme \eqref{eq:exp_ParaDiag_steps_newton}; Middle: Jacobian approximation at initial solution, i.e., $\widebar{\mathcal{N}}(\mathcal{U}^k) := I_t \otimes \diag(\mathcal{N}'(\mathbf{u}_{0}))$ in \eqref{eq:exp_ParaDiag_steps_newton}; Right: scheme \eqref{fully_disc_wr_nonlinear_imextype}.}
    \label{conv_1st_example}
\end{figure}
We observe that convergence improves as \( p \) increases, primarily because for \( u_0 \in (-1, 1) \), the contribution of the nonlinear terms diminishes with larger \( p \). Additionally, Exp-ParaDiag demonstrates good convergence even when the Jacobian is approximated using the initial solution.

\noindent Next, we propose an alternative way to build an Exp-ParaDiag for the nonlinear scheme \eqref{fully_discrete_nonlinear} by treating the nonlinear term evaluated at earlier iterations. The exact procedure is given as: 
\begin{equation}\label{fully_disc_wr_nonlinear_imextype}
\begin{cases}
\mathbf{u}_{n}^k = e^{\Delta t \mathcal{L}_h} \mathbf{u}_{n-1}^k  +\Delta t  \mathcal{N}(\mathbf{u}_{n}^{k-1})\;, n=1, 2, \cdots, N_t,\\
\mathbf{u}^k_0 = \alpha \mathbf{u}^k_{N_t} - \alpha \mathbf{u}^{k-1}_{N_t} + \mathbf{u}_0.
\end{cases}
\end{equation}
In this case, the all-at-once system takes the form 
\begin{equation}\label{disc_imextype_scheme}
    (I_t \otimes I_x)\mathcal{U}^k - (C_0^{\alpha} \otimes A)\mathcal{U}^k = \widehat{\mathbf{r}}^{k-1}:=\Delta t \diag \left( \mathcal{N}(\mathbf{u}_{1}^{k-1}), \mathcal{N}(\mathbf{u}_{2}^{k-1}),\cdots, \mathcal{N}(\mathbf{u}_{n}^{k-1}) \right) +\mathbf{b}^{k-1}.
\end{equation}
Then we can apply the usual three steps to solve \eqref{disc_imextype_scheme}, which is obvious at this point, so we skip those details here. The convergence of this procedure applied to \eqref{nheat_pth} is shown in the rightmost plot of Figure~\ref{conv_1st_example}. Next we take $\mathcal{N}(u)=-\exp(u)$ and $\mathcal{L}=a\Delta$ in \eqref{model_problem} with boundary and initial condition specified in \eqref{nheat_pth}. Based on that consideration, we proceed with the application of procedure \eqref{fully_disc_wr_nonlinear_imextype}.
\begin{figure}[h!]
    \centering
    \subfloat{{\includegraphics[height=3.5cm,width=4.5cm]{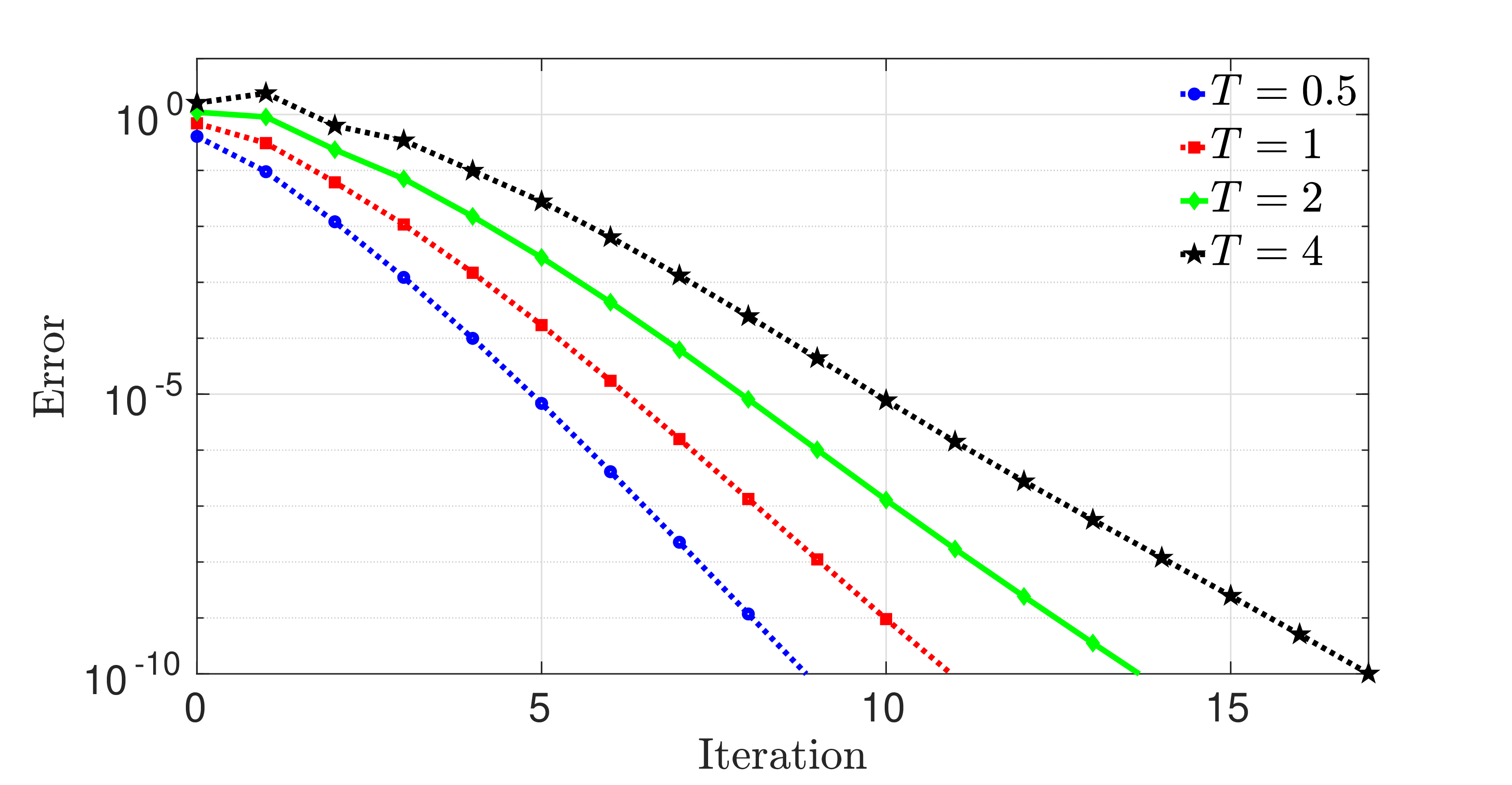} }}
    \subfloat{{\includegraphics[height=3.5cm,width=4.5cm]{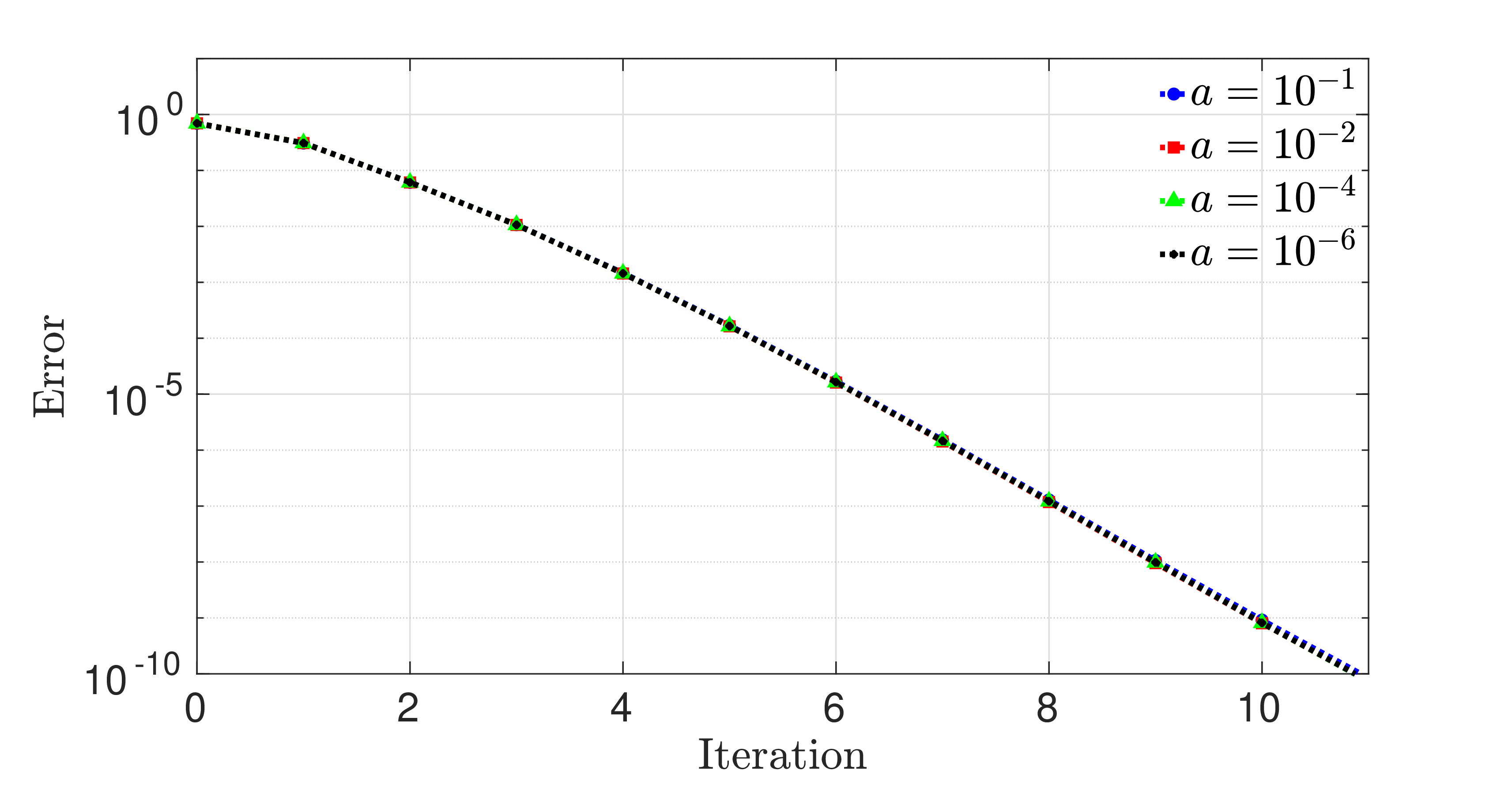} }}
    \subfloat{{\includegraphics[height=3.5cm,width=4.5cm]{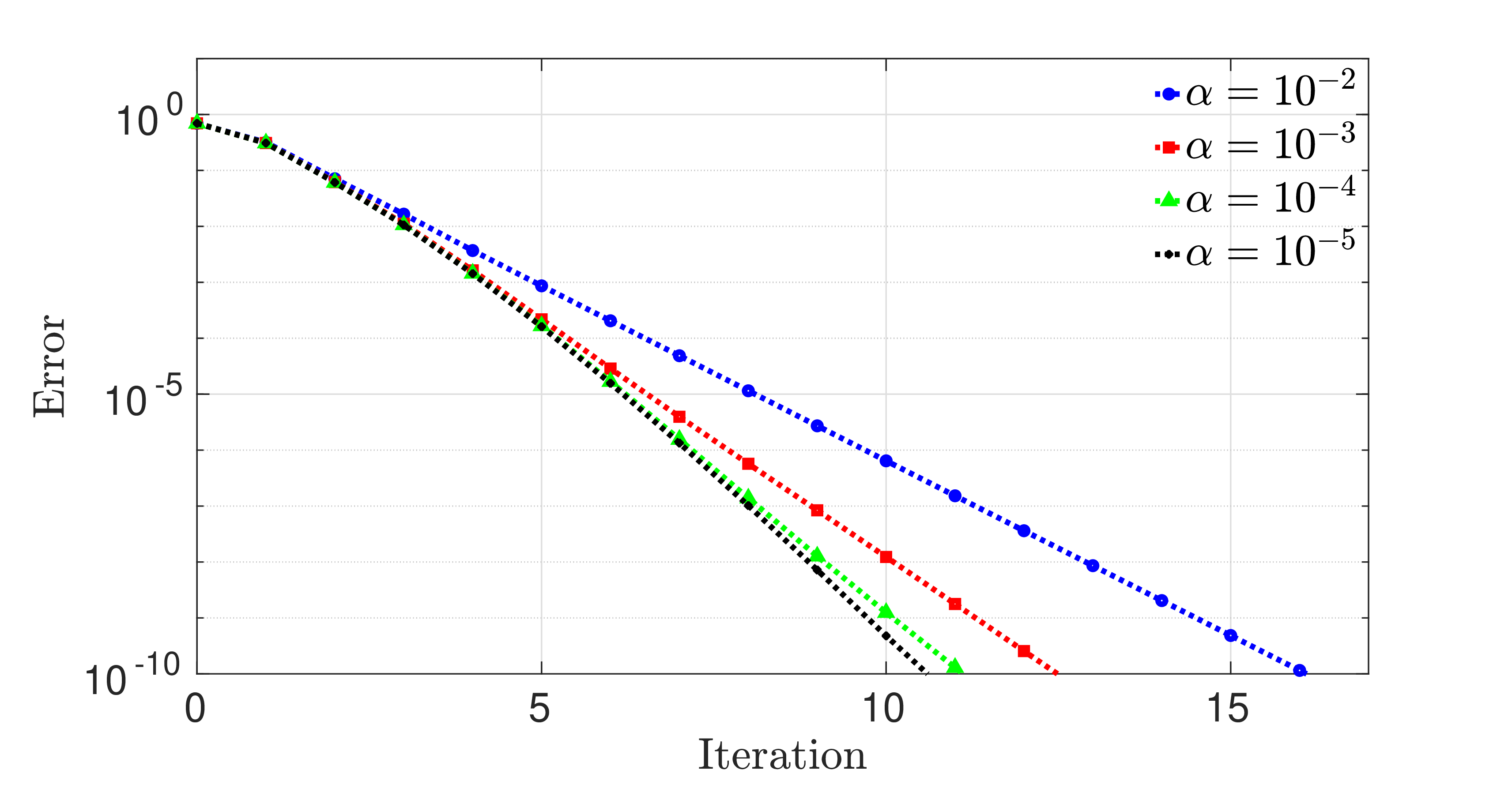} }}
    \caption{Convergence of Exp-ParaDiag with $h=1/256$ and $\Delta t=10^{-3}$. Left: different $T$; Middle: different diffusion coefficient $a$; Right: different $\alpha$.}
    \label{conv_expu}
\end{figure}
In the leftmost plot of Figure~\ref{conv_expu}, we show the convergence of Exp-ParaDiag \eqref{disc_imextype_scheme} for different time window sizes with $a = 1$ and $\alpha = 0.00005$. We observe that the convergence is robust. The middle plot displays error curves for different diffusion coefficients with $T = 1$ and $\alpha = 0.00005$, showing that the convergence is independent of the choice of $a$. In the rightmost plot, we illustrate the convergence for various values of the free parameter $\alpha$, with $T = 1$ and $a = 0.01$. It can be seen that the convergence remains consistent.

\noindent Next we consider the Fisher equation with $\mathcal{N}(u)=u-u^2$ and $\mathcal{L}=a\Delta$ in \eqref{model_problem} with homogeneous Dirichlet boundary and initial condition $u(x,0)=\sech^2(10x)$ in $\Omega=(-1, 1)$.
\begin{figure}[h!]
    \centering
    \subfloat{{\includegraphics[height=3.5cm,width=4cm]{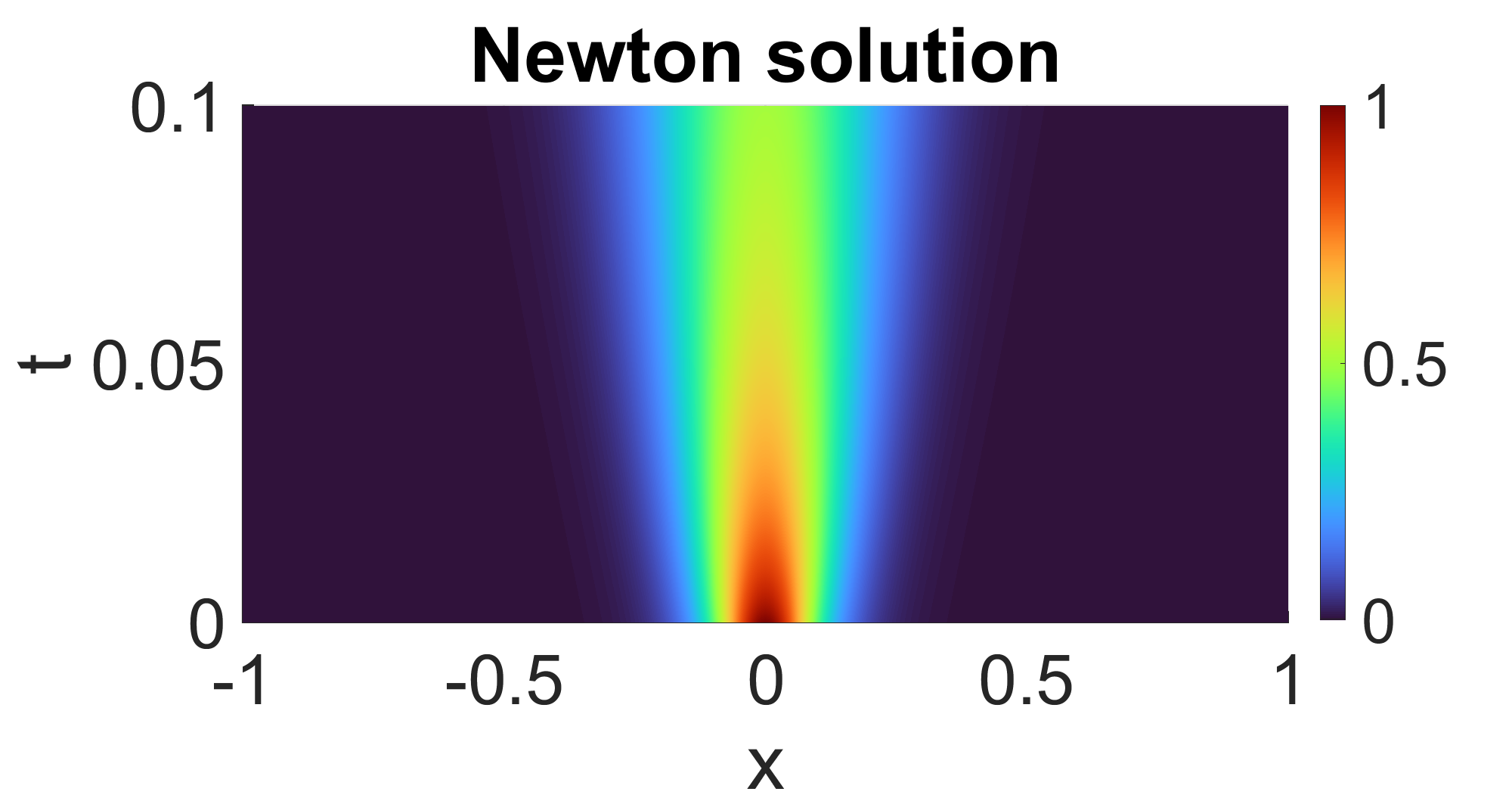} }}
    \subfloat{{\includegraphics[height=3.5cm,width=4cm]{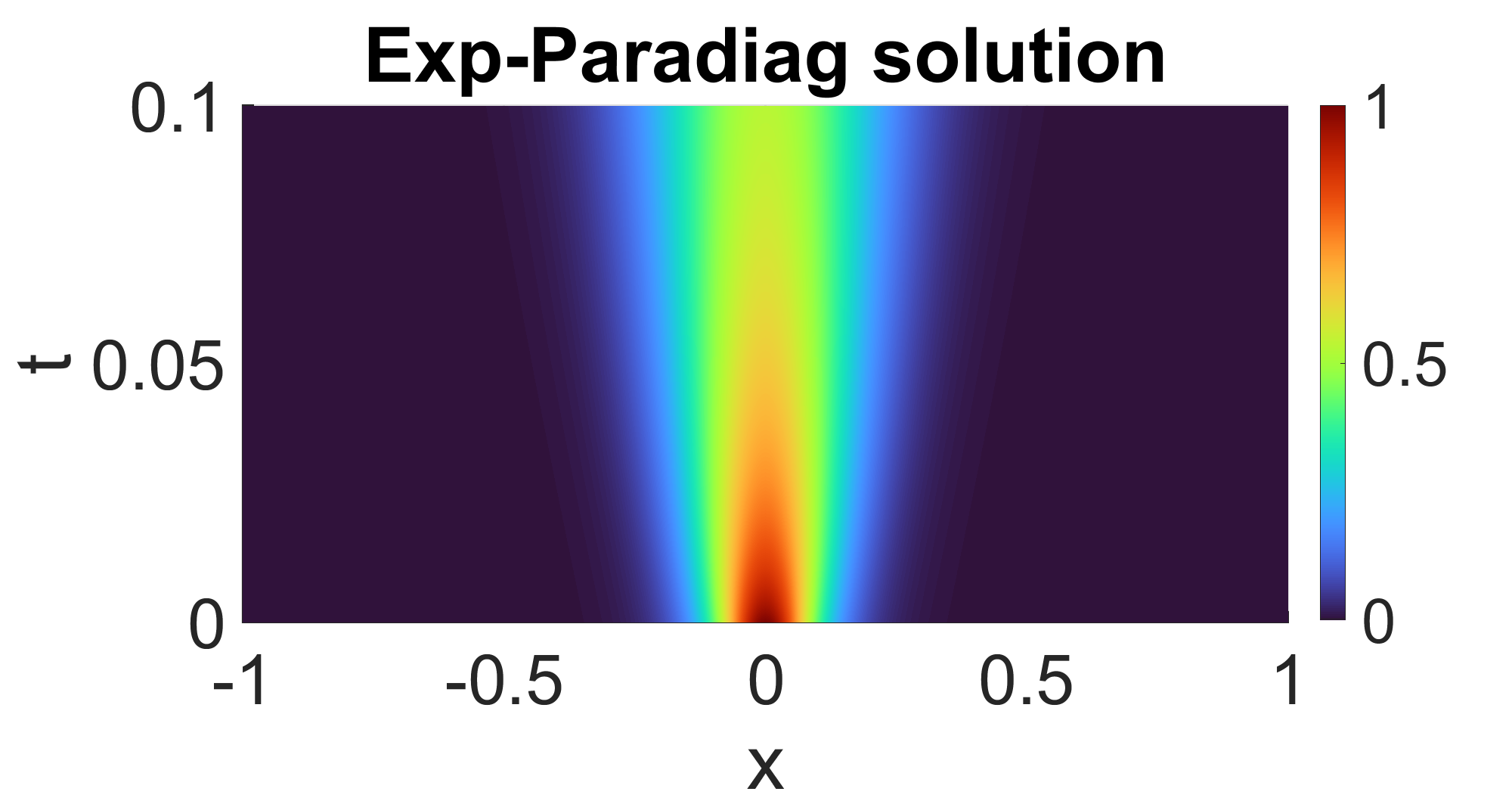} }}
    \subfloat{{\includegraphics[height=3.5cm,width=4cm]{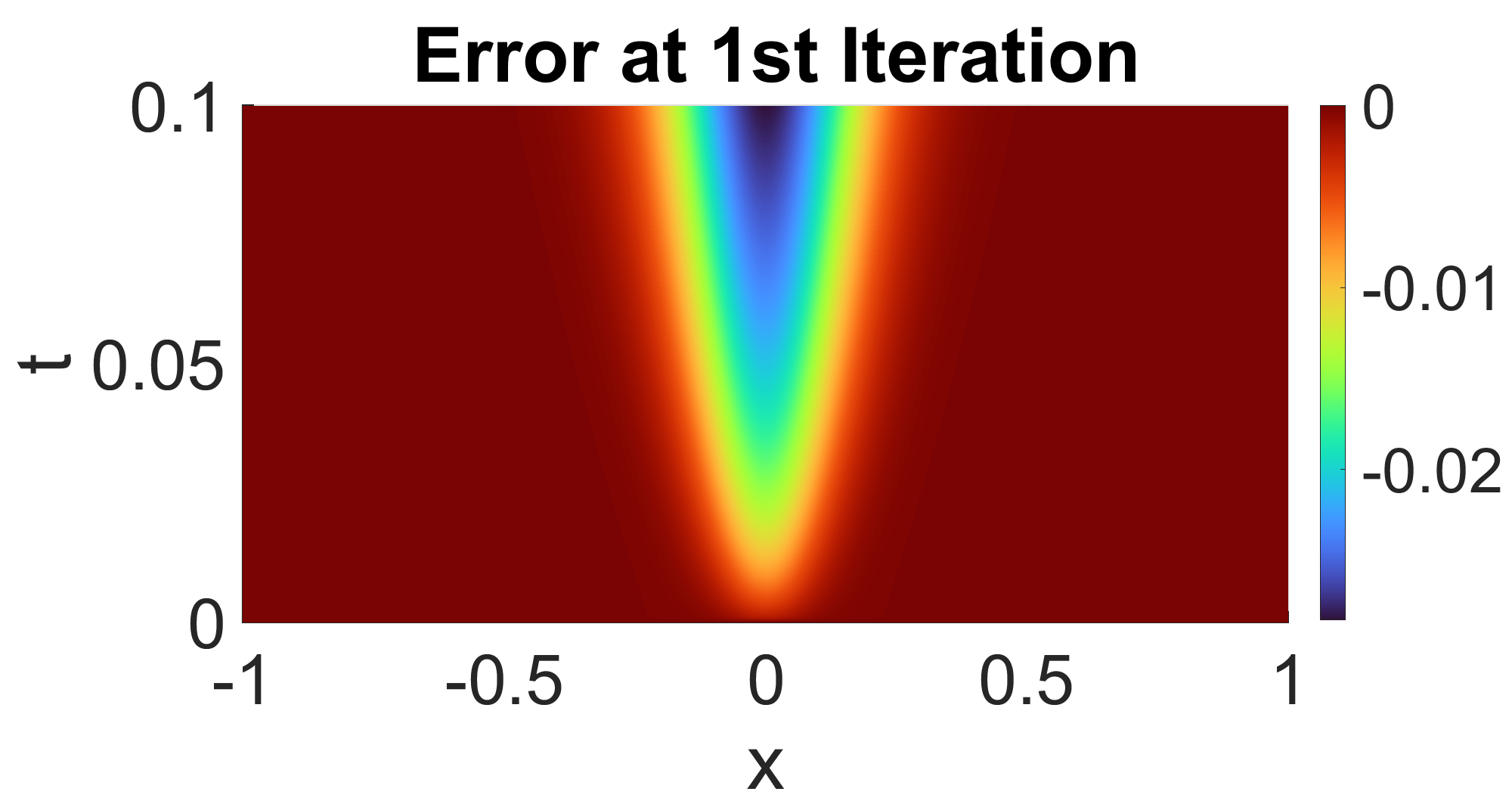} }}
    \subfloat{{\includegraphics[height=3.5cm,width=4cm]{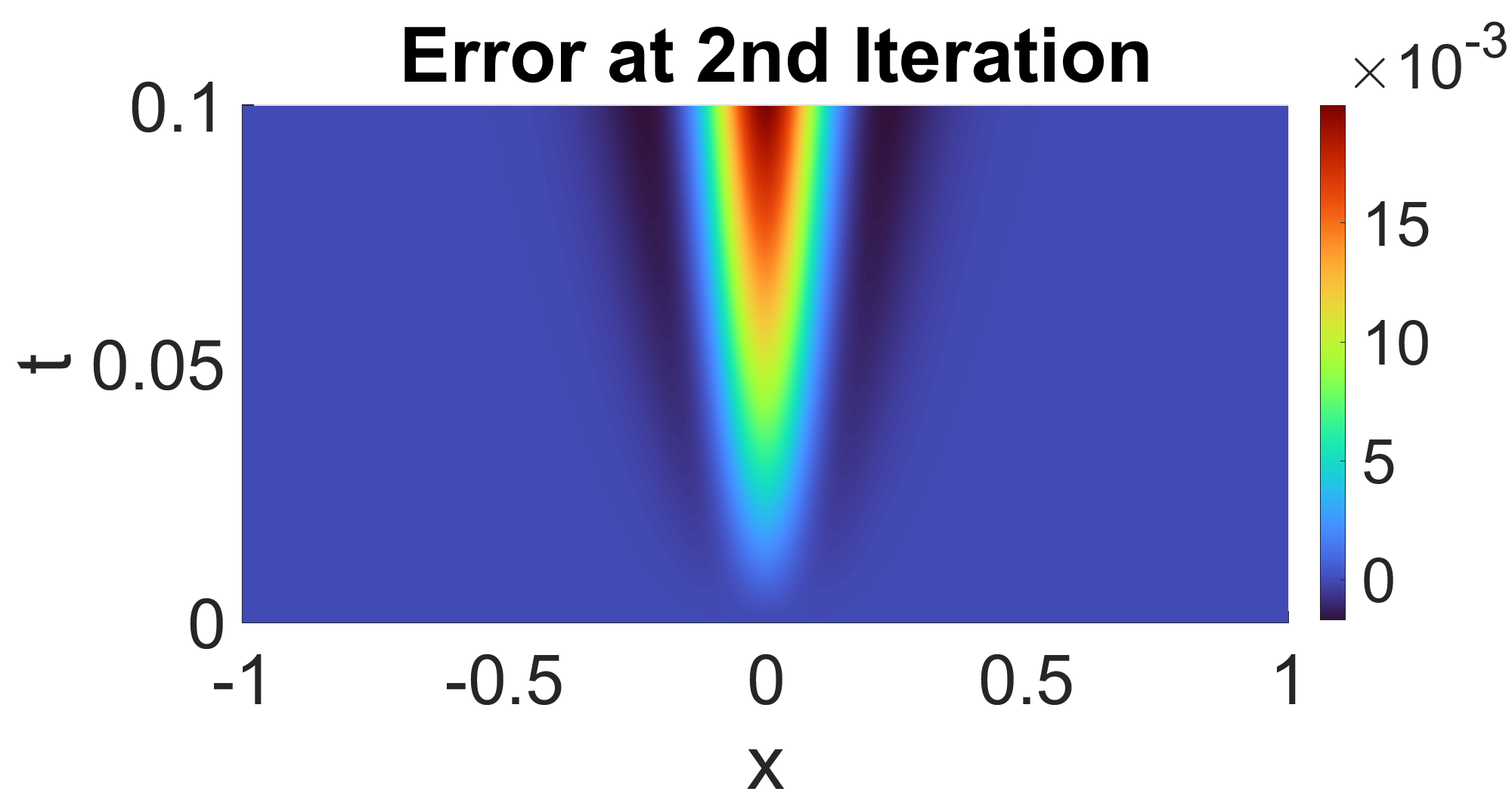} }}
    \caption{Solution and error comparision at $T=0.1$ with $h=1/128, \alpha=0.0005, a=0.1$ and $N_t=100$. First: Newton solution; Second: Exp-ParaDiag solution; Third: error after first iteration; Fourth: error after second iteration for $\widebar{\mathcal{N}}(\mathcal{U}^k) := I_t \otimes \diag(\mathcal{N}'(\mathbf{u}_{0}))$ in \eqref{eq:exp_ParaDiag_steps_newton}.}
    \label{sol_fisher}
\end{figure}
Figure~\ref{sol_fisher} shows the solution using the Newton method, the Exp-ParaDiag approximation, and the pointwise error at $T = 0.1$ after the second iteration of \eqref{eq:exp_ParaDiag_steps_newton}. The rightmost panel displays the pointwise error after the first iteration, computed using $\widebar{\mathcal{N}}(\mathcal{U}^k) := I_t \otimes \diag(\mathcal{N}'(\mathbf{u}_0))$ in \eqref{eq:exp_ParaDiag_steps_newton}.
We now apply \eqref{fully_disc_wr_nonlinear_imextype} to the Fisher equation. Figure~\ref{conv_fisher} shows the convergence of Exp-ParaDiag \eqref{disc_imextype_scheme} under various settings. The left plot demonstrates robust convergence across different time window sizes with $a = 0.01$, and $\alpha = 0.0005$. The middle plot confirms convergence is unaffected by changes in $a$ for $T = 1$, $N_t = 1000$, and$\alpha = 0.0005$. The right plot shows consistent behavior across different $\alpha$ values for $T = 1, N_t=1000$, and $a = 0.1$.
\begin{figure}[h!]
    \centering
    \subfloat{{\includegraphics[height=3.5cm,width=4.5cm]{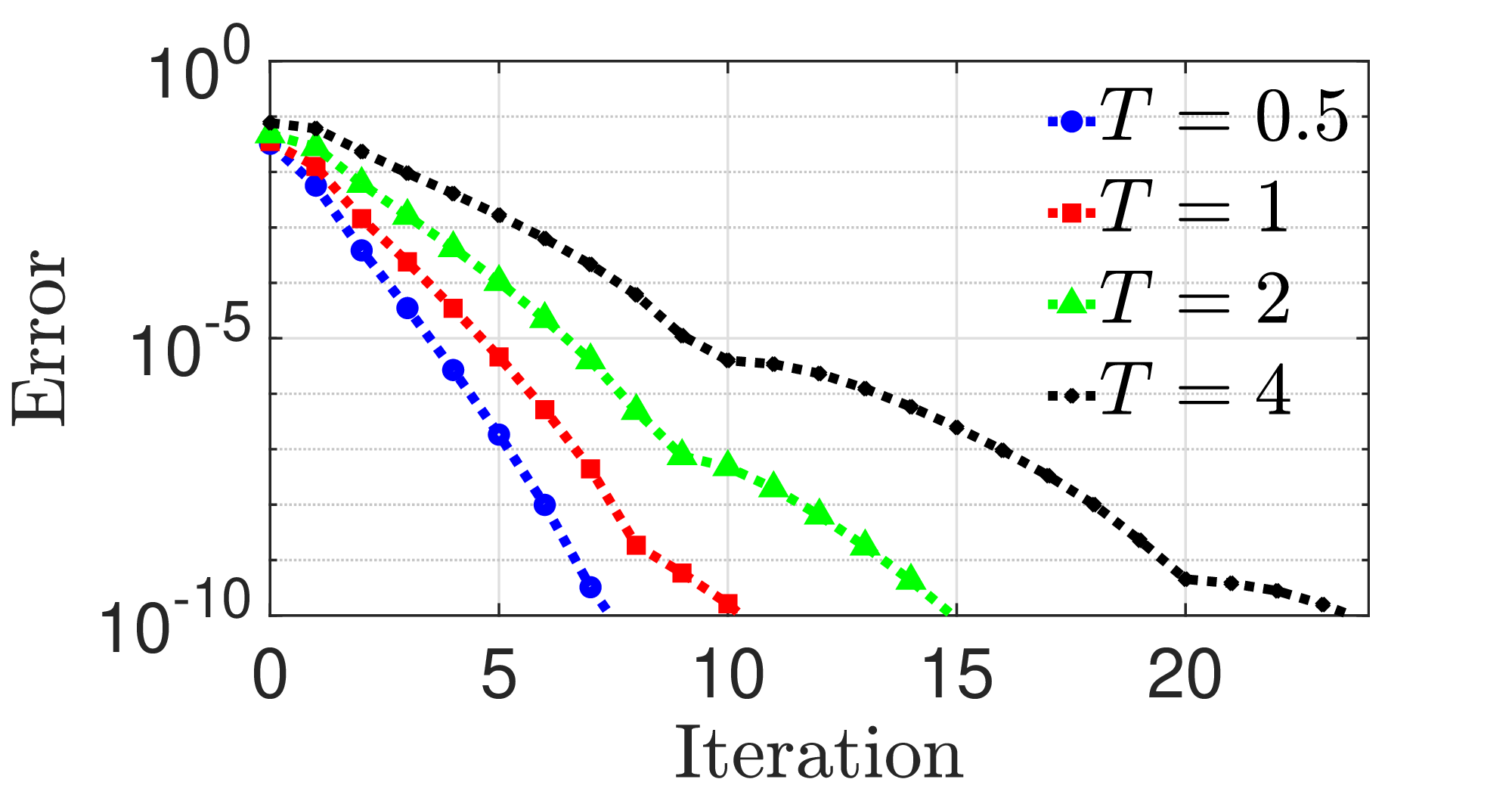} }}
    \subfloat{{\includegraphics[height=3.5cm,width=4.5cm]{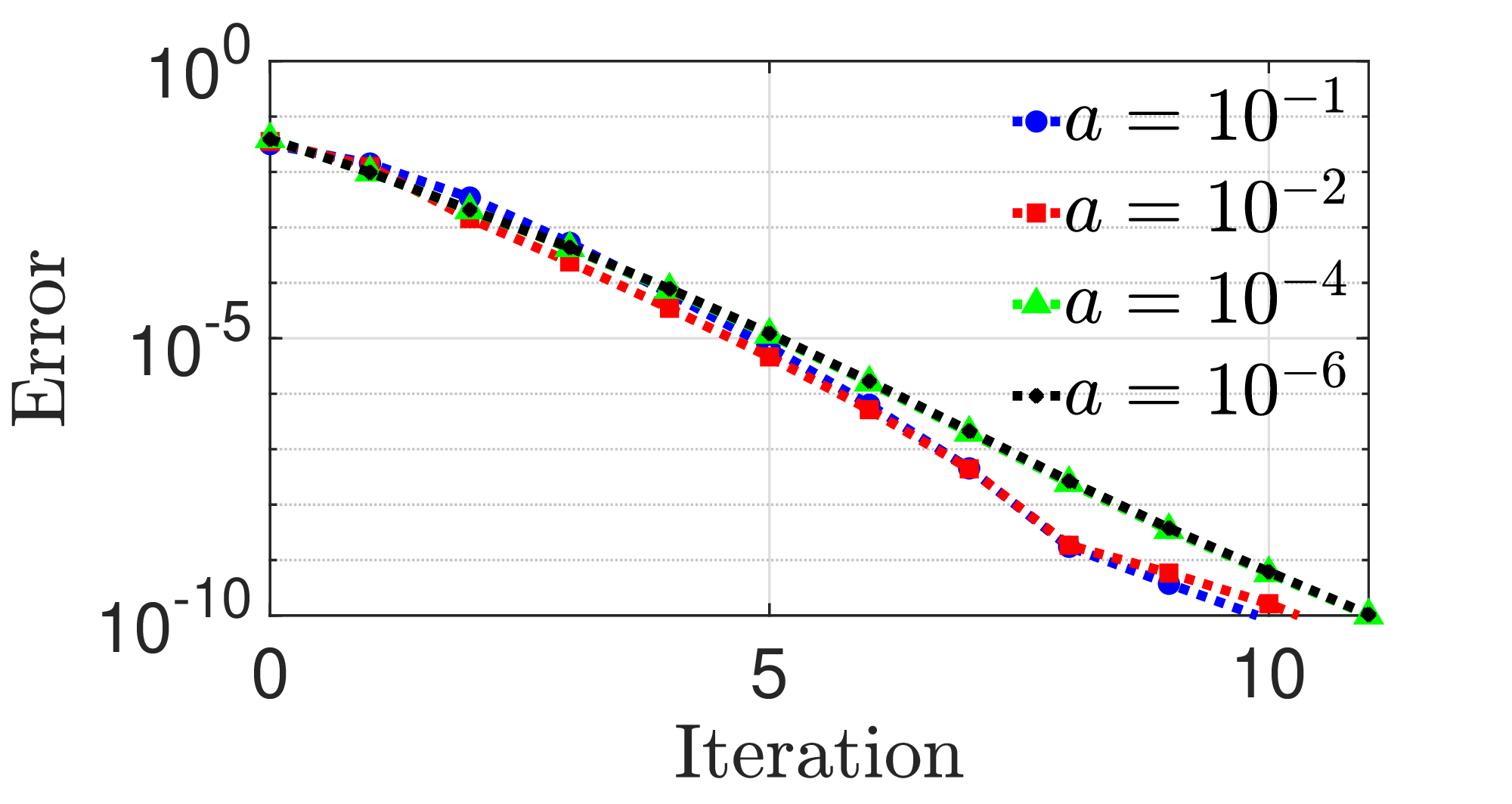} }}
    \subfloat{{\includegraphics[height=3.5cm,width=4.5cm]{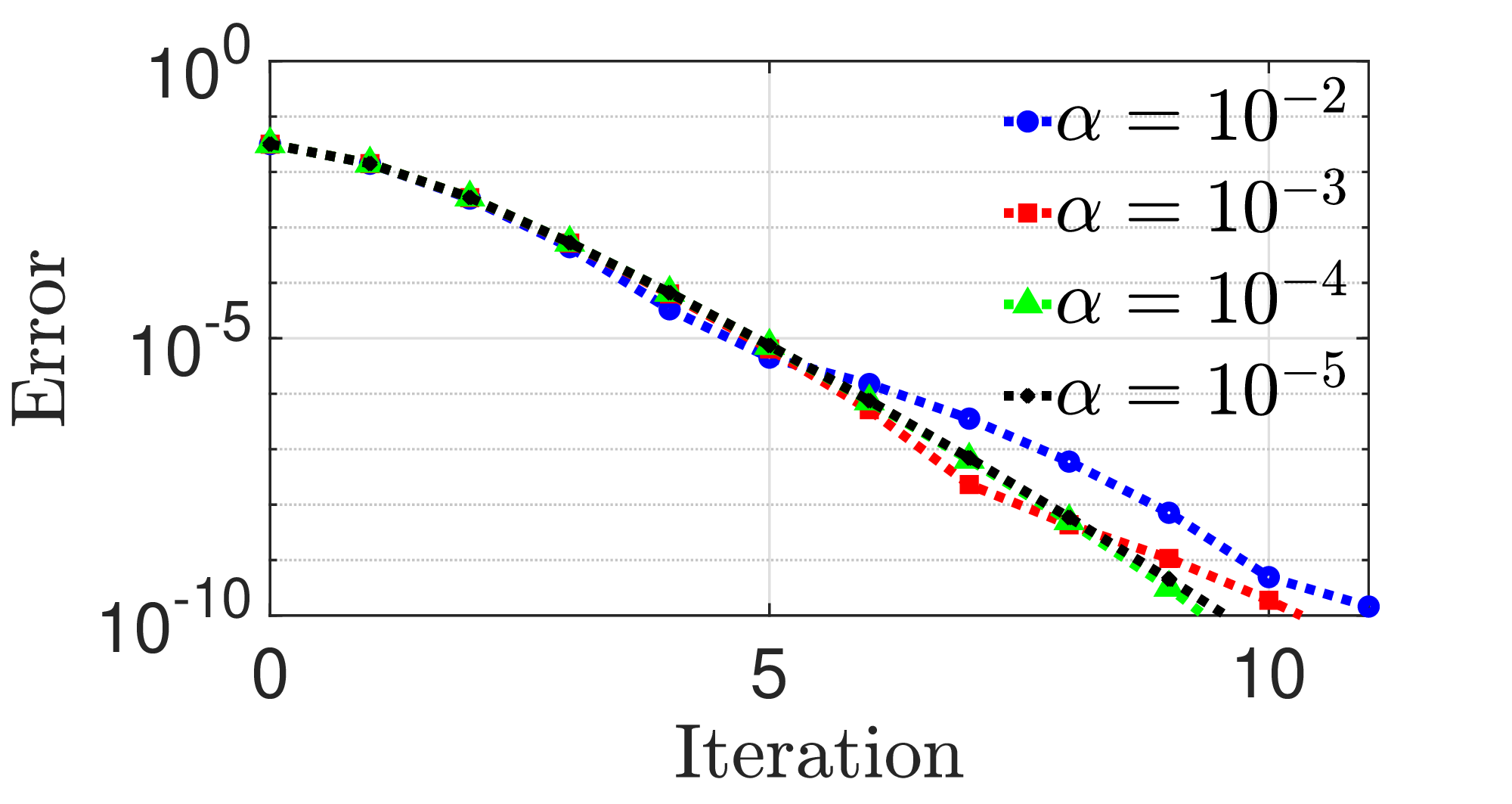} }}
    \caption{Convergence of Exp-ParaDiag with $h=1/128$ and $\Delta t=10^{-3}$. Left: different $T$; Middle: different diffusion coefficient $a$; Right: different $\alpha$.}
    \label{conv_fisher}
\end{figure}
\subsubsection{Experiments in 2D}
In two dimensions, we consider the Allen–Cahn equation in the unit square with homogeneous Dirichlet boundary conditions and an initial condition given by $u_0 = 0.1 \sin(2\pi x) \sin(2\pi y)$. The first three plots of Figure~\ref{ac_2d} present a comparison between the Newton solution and the Exp-ParaDiag solution using \eqref{disc_imextype_scheme}, along with their error at time $T = 0.5$, for parameters $\epsilon = 0.01$ and $\alpha = 0.005$.
\begin{figure}[h!]
    \centering
    \subfloat{{\includegraphics[height=3.5cm,width=4cm]{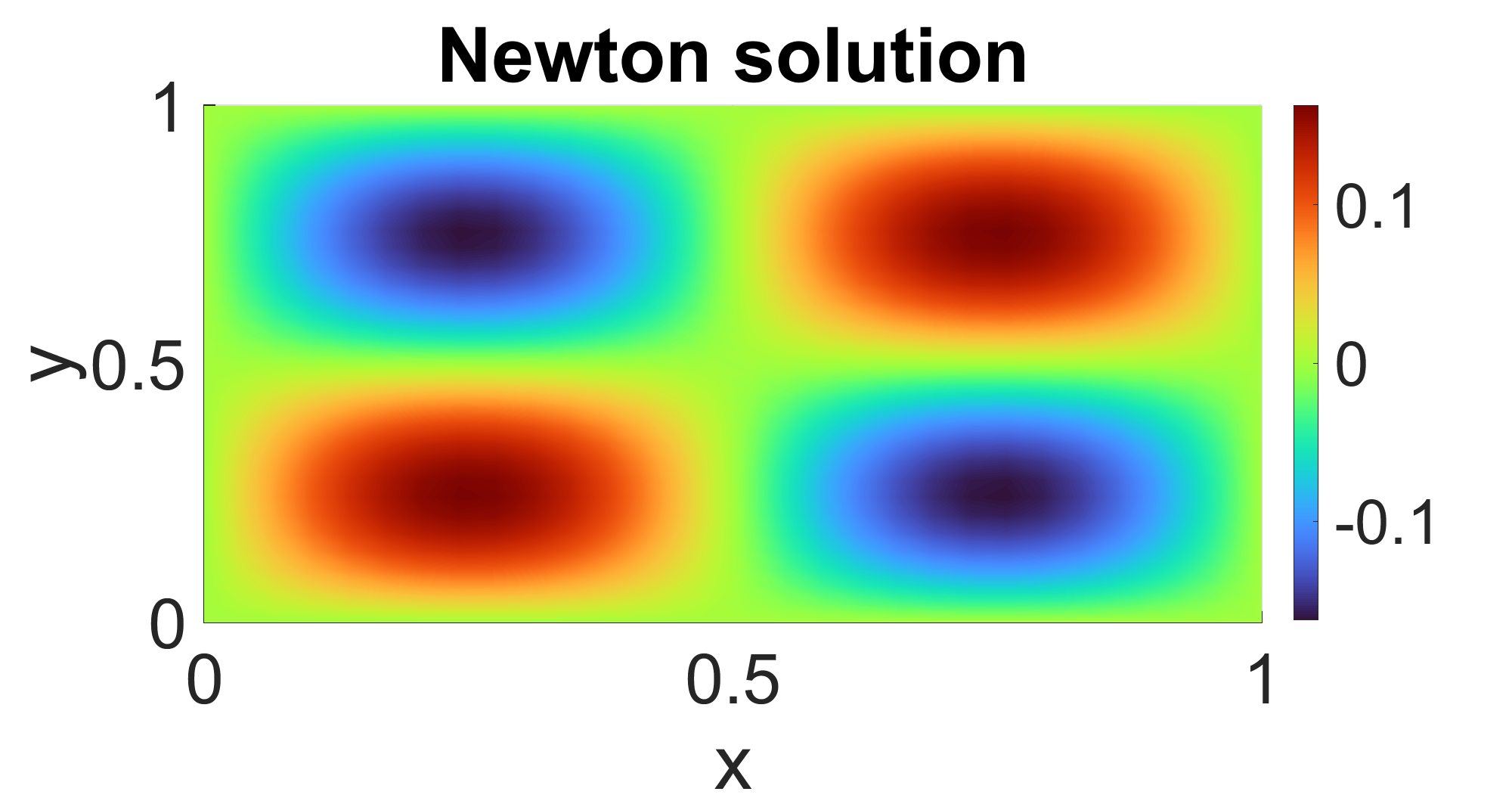} }}
    \subfloat{{\includegraphics[height=3.5cm,width=4cm]{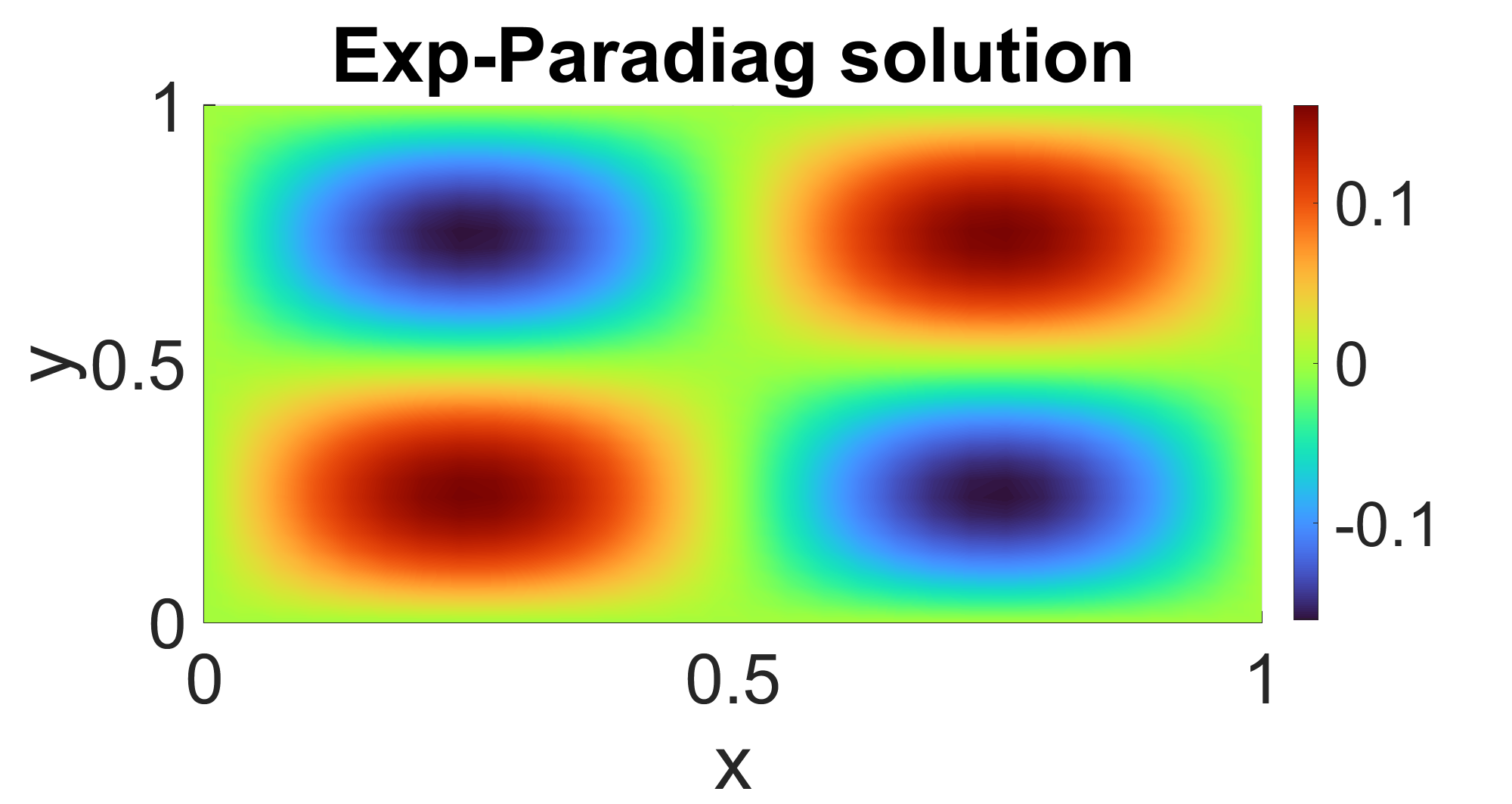} }}
    \subfloat{{\includegraphics[height=3.5cm,width=4cm]{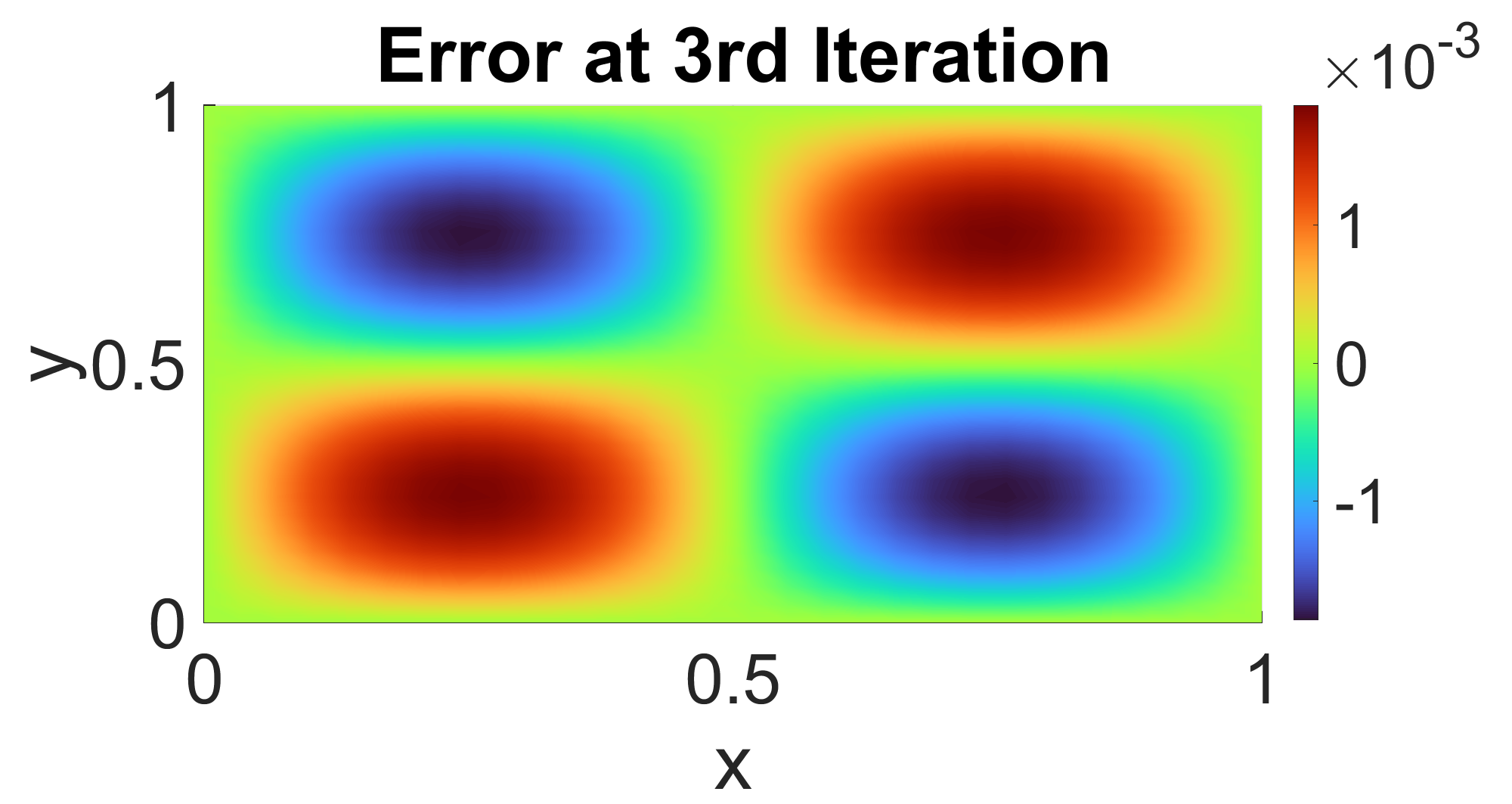} }}
    \subfloat{{\includegraphics[height=3.5cm,width=4cm]{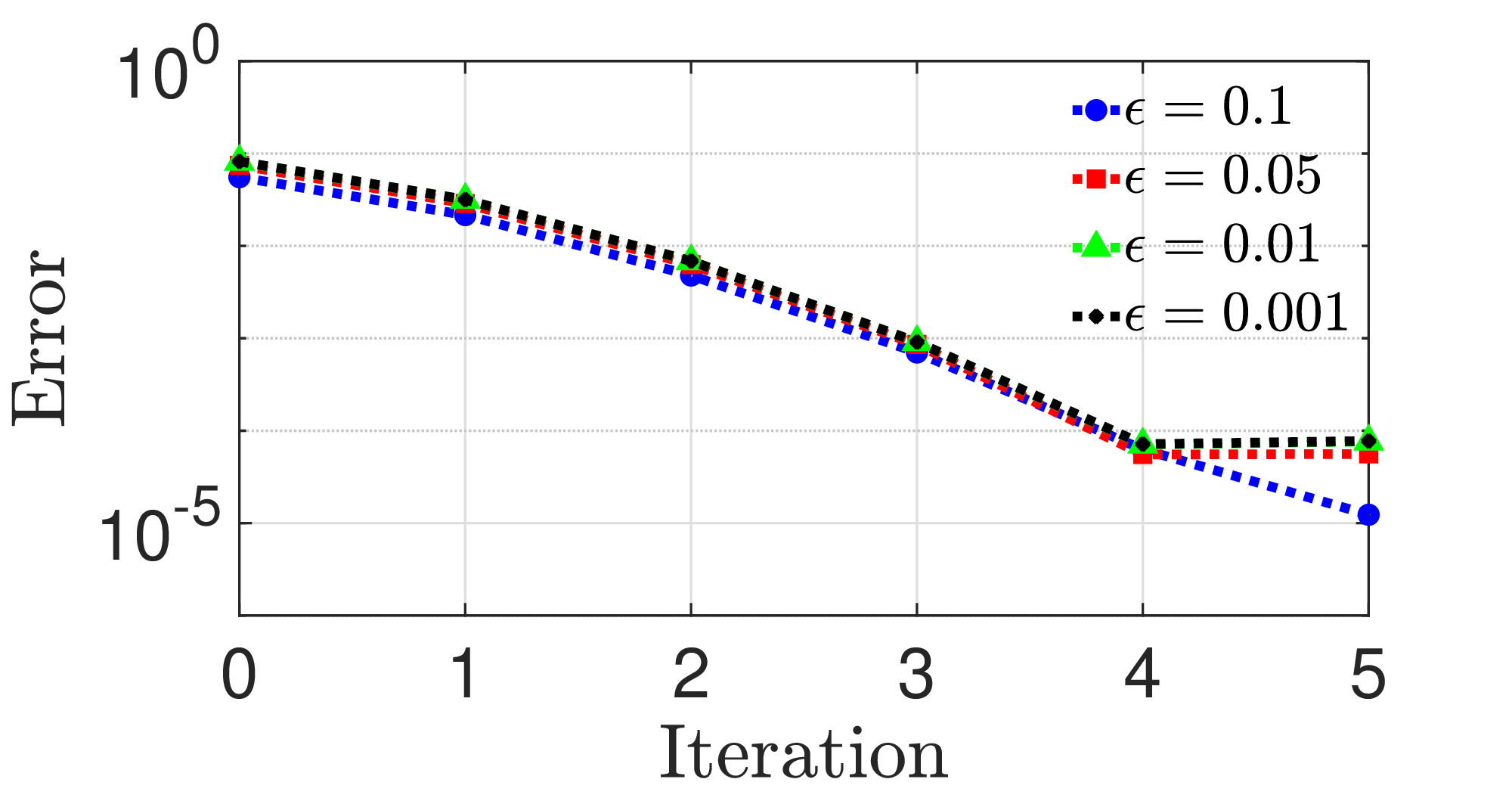} }}
    \caption{Convergence of Exp-ParaDiag with $h=1/30$ and $\Delta t=10^{-3}$ at $T=0.5$. First: Newton solution; Second: Exp-ParaDiag solution; Third: Pointwise error; Fourth: Convergence for different $\epsilon$.}
    \label{ac_2d}
\end{figure}
In the rightmost plot, we show the convergence of the Exp-ParaDiag method using \eqref{disc_imextype_scheme} for different values of $\epsilon$, with $T = 0.5$ and $\alpha = 0.005$.

\section{Conclusion}
In this work, we have introduced the Exponential-ParaDiag (Exp-ParaDiag) method, a novel time-parallel integration framework that combines the efficiency of exponential integrators with the ParaDiag paradigm. By incorporating exponential time discretization into the ParaDiag structure, the proposed approach provides a scalable and computationally efficient solver for both stiff and non-stiff PDEs.

The Exp-ParaDiag algorithm has been analyzed from two complementary viewpoints: as a fixed-point iteration and as a preconditioner for the GMRES method. This dual interpretation offers a comprehensive understanding of its convergence properties and computational characteristics. Moreover, we have extended the formulation to achieve sixth-order temporal accuracy through higher-order exponential integrators. The method has also been generalized to nonlinear problems, for which convergence has been established under suitable assumptions. 
Extensive numerical experiments confirm the theoretical analysis and demonstrate the robustness, scalability, and efficiency of the proposed method. 

Future research will focus on improving the computational efficiency of the method through low-rank and rational Krylov approximations of the matrix exponential, implementation of a distributed parallel computing framework, and extending this work to include linear problems with time-dependent coefficients, wave equations, and spatial fractional settings.

\section*{Acknowledgment} The authors would like to thank the School of Mathematics at IISER Thiruvananthapuram for the excellent research environment. NC acknowledges SERB, Department of Science and Technology, India (CRG/2022/006421), for their support.

\bibliographystyle{siam}
\bibliography{pdiag_exp_bibfile}

\end{document}